\documentclass[a4paper,11pt,reqno]{amsart}
\usepackage[a4paper,hscale=0.7,vscale=0.75,centering]{geometry}

\usepackage{mathrsfs} 
\usepackage{stmaryrd} 
\usepackage{bm} 

\usepackage{amsthm}
\usepackage{hyperref}
\usepackage{amssymb}
\usepackage{latexsym}
\usepackage{amsmath}
\usepackage{amsfonts}
\usepackage{mathabx}
\usepackage[dvips]{graphicx}
\usepackage[utf8]{inputenc}

\usepackage[LGR,T1]{fontenc}%
\usepackage[francais,english]{babel}
\usepackage{url}
\usepackage{enumerate}
\usepackage{hyperref}
\usepackage[usenames,dvipsnames]{color}

\newcommand{\new}[1]{\textcolor{black}{#1}}

\usepackage{fancyhdr}
\newtheorem{Theorem}{Theorem}
\newtheorem{Definition}[Theorem]{Definition}
\newtheorem{Proposition}[Theorem]{Proposition}

\newtheorem{Assumption}{Assumption}
\newtheorem{Lemma}[Theorem]{Lemma}
\newtheorem{Corollary}[Theorem]{Corollary}
\newtheorem{Remark}[Theorem]{Remark}
\newtheorem{Example}[Theorem]{Example}

\makeatletter \@addtoreset{equation}{section}

\@addtoreset{Theorem}{section}

\@addtoreset{Assumption}{section}

\def \Sum{\displaystyle\sum}

\def\esssup{{\rm ess}\!\sup\limits}

\newcommand{\cA}{\mathcal{A}}

\newcommand{\cC}{\mathcal{C}}
\newcommand{\cD}{\mathcal{D}}

\newcommand{\cF}{\mathcal{F}}

\newcommand{\cH}{\mathcal{H}}
\newcommand{\cI}{\mathcal{I}}
\newcommand{\cJ}{\mathcal{J}}
\newcommand{\cK}{\mathcal{K}}

\newcommand{\cM}{\mathcal{M}}

\newcommand{\cO}{\mathcal{O}}
\newcommand{\cP}{\mathcal{P}}

\newcommand{\cR}{\mathcal{R}}
\newcommand{\cS}{\mathcal{S}}

\newcommand{\cU}{\mathcal{U}}
\newcommand{\cV}{\mathcal{V}}
\newcommand{\cW}{\mathcal{W}}
\newcommand{\cX}{\mathcal{X}}
\newcommand{\cY}{\mathcal{Y}}
\newcommand{\cZ}{\mathcal{Z}}

\newcommand{\E}{\mathbb{E}}

\renewcommand{\P}{\mathbb{P}}

\newcommand{\R}{\mathbb{R}}

\def \proof{{\noindent \bf Proof. }}
\def \eproof{\hbox{ }\hfill$\Box$}

\newcommand{\ud}{\mathrm{d}}

\newcommand{\HYP}[1]
    {\ensuremath{(\bm{H#1})}}

\newcommand{\1}{{\bf 1}}
    
\newcommand{\HP}[1] 
    {\ensuremath{\mathscr{H}^{#1}}}    

\newcommand{\NH}[2] 
    { \ensuremath{ \|#2\|_{\mathcal{H}^{#1}} } }

\newcommand{\NS}[2]
    {\ensuremath{\|#2\|_{ \mathcal{S}^{#1} }}}

\newcommand{\esp}[1]{\ensuremath{\mathbb{E}\!\left[#1\right] }}

\newcommand{\EFp}[2]
    {\ensuremath{
     \mathbb{E}_{#1}\!\left[#2\right] }}
     
\newcommand{\cc}[1]{\ensuremath{ \langle #1 \rangle} } 
\newcommand{\law}[1]{\ensuremath{ [#1]} } 
\newcommand{\hesp}[1]{\ensuremath{ 
\widehat{\mathbb{E}}\left[#1\right]} 
} 

\newcommand{\dzeki}{\ensuremath{ \partial^2_{\zeta,\chi} } }

\newcommand{\NHt}[2]{\ensuremath{ \|#2\|_{\mathcal{H}^{#1},t} }}
\newcommand{\NSt}[2]{\ensuremath{ \|#2\|_{\mathcal{S}^{#1},t} }}

\newcommand{\NL}[2]{\ensuremath{ \|#2\|_{#1} }}
\newcommand{\Nt}[1]{\ensuremath{ \vvvert#1\vvvert }}

\title
[Classical solutions to the master equation]
{A Probabilistic approach to classical solutions of the master equation for large population 
equilibria}

\author[Jean-Fran\c{c}ois Chassagneux, Dan Crisan and Fran\c{c}ois Delarue]{Jean-Fran\c{c}ois Chassagneux$^*$, Dan Crisan$^{\dagger}$ and Fran\c{c}ois Delarue$^{\ddagger}$ }

\begin{document}

\maketitle

\renewcommand{\thefootnote}{\fnsymbol{footnote}}
 \footnotetext[1]{Department of Mathematics, Imperial College London. {\sf j.chassagneux@imperial.ac.uk}}
 \footnotetext[2]{Department of Mathematics, Imperial College London. {\sf d.crisan@imperial.ac.uk}}
 \footnotetext[3]{Laboratoire Jean-Alexandre Dieudonn\'e, 
 Universit\'e de Nice Sophia-Antipolis. {\sf delarue@unice.fr}}

 \renewcommand{\thefootnote}{\arabic{footnote}}

\begin{abstract}
\new{We analyze a class of nonlinear partial differential equations (PDEs) defined on $\R^d \times \cP_2(\R^d),$ where $\cP_2(\R^d)$ is the Wasserstein space of probability measures on 
$\R^d$ with a finite second-order moment. We show that such equations  admit a  classical solutions for sufficiently small time intervals. Under additional constraints, we prove that their solution can be extended to arbitrary large intervals. These nonlinear PDEs arise in the recent developments         in the theory of large population stochastic control. More
precisely they are the so-called  \textit{master equations}    
corresponding to asymptotic equilibria for a large population 
of controlled players with mean-field interaction and subject
to minimization constraints. The results in the paper are deduced by exploiting  this connection. In particular,  we study the differentiability with respect to  the  initial condition of the flow generated 
by a forward-backward stochastic system of McKean-Vlasov type.    
As a byproduct, we prove that the decoupling field 
generated by the forward-backward system
is a classical solution of the corresponding master equation.  Finally, we give several applications to mean-field games and to the control of McKean-Vlasov diffusion processes. }
\end{abstract}

\vspace{2mm}

\noindent decoupling field; Wasserstein space; master equation.
\vspace{2mm}

\noindent {\bf MSC Classification (2000):}  
Primary 93E20; secondary 60H30, 60K35.

 \section{Introduction}

{The theory of} large population stochastic control 
describes asymptotic equilibria among a large population 
of controlled players  
{with mean field interaction} 
and subject to minimization constraints. 
It has received a lot of interest since the earlier works  on mean-field games  
of Lasry and Lions \cite{MFG1,MFG2,MFG3} and of
Huang, Caines and Malham\'e \cite{HuangCainesMalhame2}. Mean-field game theory is 
{the} branch of large population stochastic control theory that corresponds to the case when 
equilibria inside the population are understood in the sense of Nash and thus describe consensus between the players  
{that make the best decision they can, taking into account the 
current states
of the others in the game.
We cover this class of control problems  in Section \ref{subse mfg}.} There are other 
types of large population equilibria in the literature yielding different types of asymptotic control problems. 
As an example, the case when players 
obey a common 
 policy controlled by a single center of decision is investigated in 
 \cite{carmona:delarue:aop,carmona:delarue:lachapelle}. We cover this distinct control problem  
 in Section \ref{subse ussr}. 
 
\new{Lasry and Lions described equilibria by means of a fully-coupled forward-backward system consisting of two partial differential equations: a (forward) 
Fokker-Planck equation describing the dynamics of the population and a (backward) Hamilton-Jacobi-Bellman 
equation describing the optimization constraints. 
In his seminal lectures at the \textit{Coll\`ege de France}, Lions noticed that 
the flow of measures solving the
Fokker-Planck equation (that is the forward part of the system) can be interpreted as the 
characteristic trajectories of a nonlinear PDE.
 The equilibrium of a large population 
of players with mean field interaction is characterized through a nonlinear partial differential equation set on an enlarged state space that contains both the private position
of a typical player and the distribution of the population. The solution of the  PDE  contains all the necessary information to entirely describe the equilibria of the game and, \textcolor{black}{on the model of the Chapman-Kolmogorov equation for the evolution of a Markov semi-group}, it is called the \textit{master equation} of the game. This equation has the form }\footnote{\new{The master equation is introduced here in its forward form. However in its application to mean field games it is used in its backward form, see equation (\ref{eq:master:PDE}). }} 
\begin{equation}\label{form0}
 \partial_t u(t,x,\mu)=Au(t,x,\mu)+ f\bigl(x,u(t,x,\mu), Bu(t,x,\mu),
 {\nu} 
 \bigr) +\int_{\R^d} \bigl[ Cu(t,x,\mu) \bigr](\cdot) \ud \mu(\cdot), 
\end{equation}
\new{for $t>0$ and $(x,\mu)\in \R^d \times \cP_{2}(\R^d),$ where $\cP_2(\R^d)$ is the Wasserstein space of probability measures on 
$\R^d$ with a finite second-order moment.  In (\ref{form0}), 
$\nu$ is the image of $\mu$ by the mapping 
$\R^d \ni x \mapsto (x,u(t,x,\mu))$; moreover,
$A$ and $B$ are differential operators that differentiate in the $x$ variable, respectively at the second and first order, whilst $C$ is a non-local operator that involves differentiation in the $\mu$ variable. The notion of differentiation in the measure variable follows Lions' definition (see \cite{cardaliaguet}).}  

%

Since \new{its introduction} in Lions' lectures, there have been only a few papers on the master equation. 
In the notes he wrote following Lions' lectures (see \cite{cardaliaguet}), 
Cardaliaguet 
discusses 
the particular case 
when players have
deterministic trajectories, and where the solutions to the master equation is understood in the 
viscosity sense. 
In this framework,
the existence of  
classical solutions has just been investigated for short time 
horizons by Gangbo and Swiech 
in the preprint \cite{gangbo:swiech}. Recently, in the 
independent works  \cite{bensoussan:yam,carmona:delarue:lyons,gomes:saude,gomes:voskanyan}
and with different approaches,
several authors
revisited, mostly heuristically, the  
master equation in the  case when the dynamics of the players
are stochastic.
A few months ago,
in a lecture at the \textit{Coll\`ege de France} 
\cite{Lions:video},
 Lions gave an outline of 
 a proof, based on PDE arguments, for investigating the master 
 equation rigorously in the latter case. 
 In \cite{bensoussan:yam,carmona:delarue:lyons}, 
 the notion of master equation 
 is extended to  other types of stochastic control problem
\new{ with players that obey a common 
 policy controlled by a single center of decision.
 } 
%

The goal of \new{this} paper is 
\new{to develop a probabilistic analysis of the class of equations \eqref{form0}. 
We seek \emph{classical} solutions for a class of PDEs that incorporates the master equations }
for both types of policies (individual or collective) \new{and for players with dynamics that can be either deterministic or stochastic. } 
Beyond  \new{their} purely theoretical  \new{interest},  
classical solutions (\new{as opposed to} viscosity solutions)  \new{are} expected to be 
\new{of use when handling} approximated equilibria
\new{in a variety of} situations: 
For instance, \new{they help in proving} the   
convergence of the equilibria, when computed 
over finite systems of players, toward the
equilibria of the 
asymptotic game. This is indeed a challenging 
question that remains partially open\new{.\footnote{\new{See however the recent advances 
in \cite{fischer,lacker}}.}} 
Similarly, the analysis of numerical schemes 
for computing the equilibria certainly benefits 
from robust regularity estimates \new{for} the solution of the master equation.

One of the reason for using  a probabilistic approach is that there has been an expanding 
literature in probability theory on forward-backward systems, which have been widely used in
stochastic control. 
Although mostly limited to the finite dimension, the existing theory gives a 
helpful insight into 
the general mechanism for deriving the master equation. One of the most noticeable result\new{s} 
is that a forward-backward system 
may be decoupled by means of a \textit{decoupling field} provided  \new{the system} is uniquely solvable,  
see \new{e.g.} 
\cite{ma:wu:zhang:zhang,Ma_sHJB}. 
\new{More precisely}, the decoupling field 
allows one to express the backward 
component of the solution as a function of the forward one. When the coefficients
of the forward-backward system are deterministic, \new{the decoupling field} satisfies 
(in a suitable sense) a quasilinear PDE.  
In the case of mean-field games,
the forward-backward system \new{consists} of two coupled PDEs, one 
of Fokker-Planck type and another one of Hamilton-Jacobi-Bellman type, 
and the \new{corresponding} quasilinear PDE is nothing but the master equation. 

\new{Another} reason for \new{analysing} the master equation by means of probabilistic arguments is that 
equilibria in large population stochastic control problems driven by 
\new{either} individual or collective policies may be characterized as solutions of finite-dimensional 
forward-backward systems of the McKean-Vlasov type, see  \cite{carmona:delarue:aop,carmona:delarue:sicon}. 
The reformulation is based 
either on
the connection between Hamilton-Jacobi-Bellman equations and 
backward SDEs
or on the stochastic Pontryagin principle, see  
\cite{fleming:soner,yong:zhou}
for the basic mechanisms in the non McKean-Vlasov framework. 
This reformulation has a crucial role as it allows one to reduce the infinite-dimensional system 
made of the Fokker-Planck equation and of the Hamilton-Jacobi-Bellman equation
\new{to} a finite dimensional system. 
The price to pay is that the coefficients of the finite dimensional 
system may depend upon the law of the solution, in the spirit of McKean's theory  \new{of} nonlinear 
SDEs. Inspired by 
Pardoux and Peng's work \cite{pardoux:peng}
on the connection between backward SDEs and classical solutions to semilinear PDEs, we 
 \new{develop a}  systematic approach for 
 \new{analyzing} the smoothness
of the solution of the master equation by investigating the smoothness of the flow 
generated by the solution of the 
McKean-Vlasov
forward-backward system with respect to the initial input. However, because of the McKean-Vlasov nonlinearity, 
\new{the analysis is far from a straightforward adaptation of the classical result of Pardoux and Peng} \cite{pardoux:peng}. 
The main issue is that the independent variable includes a probability measure, which  
requires a non-trivial extension of the notion of differentiability with respect to a probability measure. 
 
Several notions of derivatives with respect to a probability measure have been introduced in the literature. 
For example, the notion of Wasserstein derivative has been discussed within the context of optimal transport, 
see the monograph by Villani \cite{villani}. An alternative, though connected, approach 
was suggested by Lions, see \cite{cardaliaguet}. 
Generally speaking, Lions' approach  consists in lifting (\new{in a canonical manner})  functions defined on the 
Wasserstein space \new{$\cP_2(\R^d)$} (\new{the} space of probability measures on $\R^d$, with finite second-order moments 
\new{endowed with the Wasserstein metric}) 
into functions defined on 
\new{$L^2(\Omega,\mathcal{A},\mathbb{P};\R^d),$} 
\new{the space of square integrable $d$-dimensional random variables defined on the probability space $(\Omega,\cA,\P)$}. 
In this way, 
\new{the operation of differentiation}  with respect to a probability measure 
\new{is defined as the Fr\'echet differentiation in $L^2(\Omega,\mathcal{A},\mathbb{P};\R^d)$.} 
This approach is especially suited to  \new{the mean field games framework}. Indeed, 
the probabilistic representation we use 
yields a canonical lifted representation of the equilibria on  \new{$L^2(\Omega,\mathcal{A},\mathbb{P};\R^d)$}
that carries the underlying noise.
\new{The} McKean-Vlasov forward-backward system that 
 \new{models} the equilibria \new{consists of a}  forward component  \new{describing} the dynamics of the population 
and  \new{a} backward one \new{describing} the dynamics of the solution of the master equation 
along the state of the population.  
Any perturbation in \new{$L^2(\Omega,\mathcal{A},\mathbb{P};\R^d)$} of the initial condition of the forward component thus generates 
a perturbation in \new{$L^2(\Omega,\mathcal{A},\mathbb{P};\R^d)$} of the solution of the master equation. 
Using this strategy, the smoothness of the solution of the master equation \new{is deduced} 
by investigating the smoothness of the flow generated by the McKean-Vlasov forward-backward system  
with respect to an initial condition in \new{$L^2(\Omega,\mathcal{A},\mathbb{P};\R^d)$}. 

In the sequel, we apply this strategy to general forward-backward systems of equations of McKean-Vlasov type. 
Under suitable assumption, we prove that existence and uniqueness of solutions hold\new{s for the system} and that the corresponding 
\textit{decoupling field}
is the unique classical solution of \new{the time reversed version of the PDE \eqref{form0}}. 
\new{ To do this we prove first the} smoothness of the decoupling field by using the notion of differentiation described above. 
\new{Next}, we apply a tailor-made chain rule on the Wasserstein space to identify 
the structure of the PDE from the coefficients of the forward-backward system. 
In general, the result holds  
\new{for sufficiently small time intervals}, as \new{it is usually the case} with forward-backward processes. 

Inspired by \cite{del02}, 
we then show that, provided we have an \textit{a priori} estimate for the gradient of the 
solution of the master equation, existence and uniqueness of a classical solution 
may be extended, via an inductive argument, to arbitrary large time intervals. 
This requires the identification of  a suitable space of solutions that is \new{left} invariant along the induction, which 
is one of the most technical issue\new{s} of the paper. 
In the framework of large population stochastic control, 
we identify three classes of examples under which the \textit{a priori} bound for the gradient is shown to hold.  
The \new{first two belong to} the framework of mean-field games. 
To bound the gradient in each of them, we combine
either convexity (in the first example) or ellipticity 
(in the second example) with the so-called Lasry\new{-}Lions condition, used for guaranteeing uniqueness 
of the equilibria, see \cite{cardaliaguet}. 
\new{To} the best of our knowledge, 
except the aforementioned video by Lions \cite{Lions:video},
the solvability of the master equation in the classical sense is, in both cases, a new result\footnote{As far as we understand the sketch of the proof in 
\cite{Lions:video}, the underlying arguments are 
reminiscent of the way in which we use convexity in the first class of examples.}.
The third example concerns the situation when players
obey a common center of decision, in which case the stochastic control problem
may be reformulated as an optimization problem over controlled McKean-Vlasov diffusion processes.
In \new{this} last example,  the proof mainly relies on convexity. 

\new{I}n a parallel work to ours
\new{made available recently}, Buckdahn
\textit{et al.}
\cite{buck:li:rain:peng} 
adopted a similar approach \new{to} study 
forward flows, proving that 
the semigroup of a standard McKean-Vlasov stochastic differential equation
with smooth coefficients
 \new{is the} classical solution  \new{of} a  linear PDE 
 \new{defined on $\R^d \times \cP_2(\R^d)$}. 
\new {The results in \cite{buck:li:rain:peng} do not cover nonlinear PDEs of the type \eqref{form0} 
that include master equations for large population equilibria. }  
It must be also noticed that a crucial assumption is made therein 
on the smoothness of the coefficients, which restrict rather drastically the scope 
of application\new{s}. 
We avoid this, however, we do pay a heavy price for working under 
more tractable assumptions, see Remark \ref{rem:comparaison:buck}  \new{below}. 

\new{We treat here}  systems of players driven by idiosyncratic (or independent) noises. 
Motivated by practical applications, see  
\cite{carmona:delarue:lyons,GueantLasryLions.pplnm},
\new{in subsequent work, the players will be driven by an additional} common source of noise, in which case 
the McKean-Vlasov interaction in the forward-backward equations under consideration
becomes random itself, 
as it then stands for the conditional distribution of the population given the common source of randomness.

The paper is organized as follows. The general set-up together with the main results are described in Section 
\ref{se intro}. The chain rule on the Wasserstein space is discussed in Section \ref{se chain rule}. \new{The} smoothness 
of the flow of a McKean-Vlasov forward-backward system is investigated in small time in Section \ref{se smoothness}. 
In Section \ref{se app}, we provide some applications to large population stochastic control. The proofs of
some technical results are given in Appendix.

\section{General step-up and overview of the results}
\label{se intro}
Let $(\Omega,\cA,\P)$ be a probability space supporting a $d$-dimensional Brownian motion 
$(W_{t})_{t \geq 0}$ and a square integrable random variable $\xi$, independent of $(W_{t})_{t \geq 0}$. 
We denote by $(\cF^{\xi,W}_t)_{t \ge 0}$ the augmented  
filtration generated by $\xi$ and $(W_{t})_{t \geq 0}$. 
For a given terminal time $T>0$, we consider the following system of equations:
\begin{align}
\left\{
\begin{array}{rcl}
X_s &= &\xi + \int_0^s b(X_r, Y_r,Z_r,\P_{(X_r, Y_r)} )\ud r +\int_0^s \sigma(X_r, Y_r,\P_{(X_r, Y_r)}) \ud W_r,
\\
Y_s &= &g(X_T,\P_{X_T} ) + \int_s^T f(X_r, Y_r,Z_r,\P_{(X_r, Y_r)}) \ud r
- \int_s^T Z_r \ud W_r,
\end{array}
\right. \ s \in [t,T]
\label{eq intro XYZ}
\end{align}
\new{The processes $X$, $Y$ and $Z$ are $d$, $m$ and 
$m \times d$ dimensional, respectively. The coefficients $b: \R^d\times \R^m \times \R^{m\times d} \times \cP_{2}(\R^{d}\times \R^m)\rightarrow \R^d,$
$\sigma: \R^d\times  \R^m \times \cP_{2}(\R^{d}\times \R^m)\rightarrow \R^{d\times d},$
$f: \R^d \times \R^m \times \R^{m\times d} \times \cP_{2}(\R^{d}\times \R^m)\rightarrow \R^m$ 
and $g: \R^d \times \cP_{2}(\R^{d})\rightarrow \R^m$ are measurable functions that 
satisfy conditions that will be imposed below. $\P_{(X_r, Y_r)}$ denotes the law of $(X_r, Y_r)$. The system \eqref{eq intro XYZ} is called a 
forward-backward system of McKean-Vlasov type.} 
Notice that, for simplicity, the coefficients $b$, $\sigma$ and $f$ are time homogeneous and $X$ has same dimension 
as the noise $W$. These constraints can however be lifted and a similar analysis will apply.   

{
{In the following, we will} show that, under convenient assumptions, 
there exists a unique solution  {of} the {forward-backward} system \eqref{eq intro XYZ} together with a 
\textit{decoupling field} 
$U: [0,T]\times \R^d \times \cP_{2}(\R^{d})\rightarrow \R^m$ to \eqref{eq intro XYZ}. Namely,
 $U$ is a function such that 
 \begin{align}
Y_s = U(s, X_s, \P_{X_s})\,, \;0 \leq s \le T\,.
\label{eq thedecouplingproperty}
 \end{align}
}
\new{Finally, we will show 
\[
(t,x,\mu)\in[0,T]\times \R^d \times \cP_{2}(\R^{d})\to U(T-t,x,\mu)\
\] 
is a classical solution of the equation \eqref{form0}. }

\subsection{Definition of $U$}
{
The construction 
of the decoupling field $U$  is typically 
discussed under the assumption that the existence and uniqueness 
of the solution of the system \eqref{eq intro XYZ} is holds. 
See, e.g. \cite{carmona:delarue:aop,carmona:delarue:ecp,carmona:delarue:sicon}, 
for conditions under which this holds for an arbitrary time horizon $T$. 
We adopt here a different approach: We first focus on the case where $T$
is sufficiently small so that
the existence and uniqueness of the solution of the system \eqref{eq intro XYZ} 
hold.
This helps us construct the decoupling 
field $U$ for the same time horizon and, therefore deduce the existence of a unique local solution 
of PDE \eqref{form0}. Secondly we use results from \cite{carmona:delarue:aop,carmona:delarue:sicon} 
to pass from a small time to an arbitrary time horizon and there justify 
the existence of a unique global solution to \eqref{form0}.}

A common \new{strategy to introduce} the decoupling field consists in letting the initial time in \eqref{eq intro XYZ} vary. 
Without any loss of generality, we can assume that
$(\Omega,\cA,\P)$ is equipped with
a filtration 
$({\mathcal F}_{t})_{t \geq 0}$ (satisfying the usual condition)
such that $(\Omega,{\mathcal F}_{0},\P)$ is rich enough 
to carry $\R^d$-valued random variables with any arbitrary distribution
in $\cP_{2}(\R^d)$
and 
$(W_{t})_{t \geq 0}$ is an $({\mathcal F}_{t})_{t \geq 0}$-Brownian motion. 
In particular, $\xi$ in 
\eqref{eq intro XYZ}
may be taken as an $\cF_{0}$-measurable square-integrable random variable. 

%
In the sequel, we often use the symbol $\mu$ to denote the law of $\xi$. 
\new{We will use the notation}  
$\law{\Theta} := \P_{\Theta}$ \new{to denote the law of the} random variable $\Theta$
(so then $\mu = \law{\xi}$). 
Within this set-up,
we consider the following version of \eqref{eq intro XYZ} 
with the forward component starting at time $t$ from $\xi \in L^2(\Omega,\cF_{t},\P;\R^d)$:
\begin{align}
\left\{
\begin{array}{rcl}
X^{t,\xi}_s &=& \xi + \int_t^s b \bigl(\theta^{t,\xi}_r,\law{\theta^{t,\xi,(0)}_r} 
\bigr)\ud r +\int_t^s \sigma\bigl(\theta^{t,\xi,(0)}_r,\law{\theta^{t,\xi,(0)}_r}\bigr) \ud W_r,
\\
Y^{t,\xi}_s &=& g\bigl(X^{t,\xi}_T,\law{X^{t,\xi}_T} \bigr) + \int_s^t 
f\bigl(\theta^{t,\xi}_r,\law{\theta^{t,\xi,(0)}_r}\bigr) \ud r
- \int_t^s Z^{t,\xi}_r \ud W_r,
\end{array}
\right. \ \ s\in[t,T]
 \label{eq X-Y t,xi}
\end{align}
with $\theta^{t,\xi} = (X^{t,\xi},Y^{t,\xi},Z^{t,\xi})$ and $\theta^{t,\xi,(0)} 
= (X^{t,\xi},Y^{t,\xi})$.

A crucial remark for the subsequent analysis is to notice that the Yamada--Watanabe
theorem extends to equations of the same type as \eqref{eq X-Y t,xi}. 
More precisely, one can prove that, whenever 
pathwise uniqueness holds, solutions are also unique in law \cite[Example 2.14]{kur14}. 
As a consequence, it  \new{follows} that the law of $(X^{t,\xi},Y^{t,\xi})$ only depends upon 
the law of $\xi$. In other words, $\law{(X^{t,\xi}_r,Y^{t,\xi}_r)}$ is \new{a function of}  $\law{\xi}=\mu$. 
Given $\mu \in \cP_{2}(\R^d)$, it thus makes sense to consider $(\law{(X^{t,\xi}_r,Y^{t,\xi}_r)})_{t \leq r \leq T}$
without specifying the choice of the lifted random variable $\xi$ that has $\mu$ as distribution. 
We then introduce, for any $x \in \R^d$,
a stochastic flow associated to the system  \eqref{eq X-Y t,xi}, defined as
 \begin{align}
\left\{
\begin{array}{rcl}
X^{t,x,\mu}_s &=& x + \int_t^s b\bigl(\theta_r^{t,x,\mu}, 
\law{\theta^{t,\xi,(0)}_r}\bigr) \ud r 
+\int_t^s \sigma\bigl(\theta^{t,x,\mu,(0)}_r, \law{\theta^{t,\xi,(0)}_r}\bigr) \ud W_r,
\\
Y^{t,x,\mu}_s &=& g\bigl(X^{t,x,\mu}_T,\law{X^{t,\xi}_T} \bigr) + \int_s^t 
f\bigl(\theta^{t,x,\mu}_r, \law{\theta^{t,\xi,(0)}_r}\bigr) \ud r
- \int_t^s Z^{t,x,\mu}_r \ud W_r,
\end{array}
\right.
 \label{eq X-Y t,x,mu}
\end{align}
with $\theta^{t,x,\mu} = (X^{t,x,\mu}, Y^{t,x,\mu},Z^{t,x,\mu})$ and 
$\theta^{t,x,\mu,(0)} = (X^{t,x,\mu}, Y^{t,x,\mu})$.

We now have all the \new{ingredients} to give the definition of a 
decoupling field to \eqref{eq X-Y t,xi} on $[0,T]\times\R^d \times \cP_2(\R^d)$.
For the following definition, assume for the moment that, for any $(t,x,\mu) \in 
[0,T]\times\R^d \times \cP_2(\R^d)$ and any random 
variable $\xi \in L^2(\Omega,{\mathcal F}_{t},\P;\R^d)$ with distribution $\mu$,
\eqref{eq X-Y t,xi} has a unique (progressively-measurable) solution $(X^{t,\xi}_{s},Y^{t,\xi}_{s},Z^{t,\xi}_{s})_{s \in [t,T]}$
such that $$\sup_{s \in [t,T]} \vert X_{s}^{t,\xi} \vert^2,
\quad  
\sup_{s \in [t,T]} \vert Y_{s}^{t,\xi} \vert^2, \quad \textrm{and}
\quad \int_{t}^T \vert Z_{s}^{t,\xi} \vert^2 \ud s,$$
are integrable. Assume also that \eqref{eq X-Y t,x,mu} has a 
unique (progressively-measurable) solution 
$(X^{t,x,\xi}_{s},Y^{t,x,\xi}_{s},Z^{t,x,\xi}_{s})_{s \in [t,T]}$
such that 
$$\sup_{s \in [t,T]} \vert X_{s}^{t,x,\xi} \vert^2,
\quad 
\sup_{s \in [t,T]} \vert Y_{s}^{t,x,\xi} \vert^2, \quad 
\textrm{and} \quad 
\int_{t}^T \vert Z_{s}^{t,x,\xi} \vert^2 \ud s,$$
are integrable. Then, we may let:
 
 \begin{Definition}[The decoupling field $U$] \label{le dec field}
\new{The function $U : [0,T] \times \R^d \times \cP_{2}(\R^d)\mapsto\R^m$ defined as    
\begin{align}
\label{eq:decoupling:field}
U(t,x,\mu) =Y_{t}^{t,x,\mu},\ \ \ (t,x,\mu)\in[0,T] \times \R^d \times \cP_{2}(\R^d) \end{align}
is called the decoupling field of the forward-backward system \eqref{eq X-Y t,xi} (or, equivalently, of the  corresponding stochastic flow \eqref{eq X-Y t,x,mu}).}
\end{Definition}

The decoupling property \eqref{eq thedecouplingproperty} of $U$ is proved in Proposition
\ref{pr thedecouplingproperty} below, 
under assumptions that guarantee existence and uniqueness 
to \eqref{eq X-Y t,xi} and \eqref{eq X-Y t,x,mu}.

\vspace{5pt}

\new{Recall now that the $2$-Wasserstein distance $W_{2}$, defined on $\cP_{2}(\R^k)$, $k \geq 1$ is given by
\begin{equation*}
\begin{split}
&W_{2}(\mu,\nu) 
= \inf_{\gamma} \biggl[ \int_{(\R^k)^2}
\vert u-v\vert^2 \gamma(\ud u, \ud v)  ; \ \gamma( \cdot \times \R^k) = \mu, \
\gamma( \R^k \times \cdot) = \nu \biggr]^{1/2}.
\end{split}
\end{equation*}
 } As already \new{mentioned}, a very convenient way to prove strong existence 
and uniqueness to \eqref{eq X-Y t,xi} and \eqref{eq X-Y t,x,mu}
consists in working  \new{first with small time horizons}. For $T$ \new{sufficiently} small, there exists a unique solution to the systems \eqref{eq X-Y t,xi} and
\eqref{eq X-Y t,x,mu} under the following assumption:

\begin{Assumption}[\HYP{0}(i)]
\notag There exists a constant $L>0$ such that the 
mappings $b$, $\sigma$, $f$ and $g$ are 
$L$-Lipschitz continuous in all the variables, the distance on $\cP_{2}(\R^d \times \R^m)$, respectively 
$\cP_{2}(\R^d)$ being the $2$-Wasserstein distance. 
\end{Assumption}
 

The existence of a local solution to the systems \eqref{eq X-Y t,xi} and
\eqref{eq X-Y t,x,mu} under assumption \HYP{0}(i) is not new (see for instance \cite[Proof of Lemma 2]{carmona:delarue:ecp}). 
The proof consists of a  straightforward adaption of the results in \cite{del02} 
for \new{classical forward backward stochastic differential equations (FBSDEs)}. 
\new{To be precise, one} shows that the systems \eqref{eq X-Y t,xi} and
\eqref{eq X-Y t,x,mu} 
are uniquely solvable under assumption \HYP{0}(i) provided $T \leq c$ for a constant $c:=c(L)>0$. 
Examples \new{where the result can be extended to long time horizons} will be discussed in Section \ref{se app}. 
\vspace{2mm}

{It is quite illuminating to observe that} the system \eqref{eq X-Y t,x,mu} can be rewritten as a classical coupled  FBSDE with time dependent coefficients,
\new{as follows}
\begin{align}
\left\{
\begin{array}{rcl}
X^{t,x,\mu}_s &=& x + \int_t^s \hat{b}_{t,\mu}(r,\theta^{t,x,\mu}_r) \ud r 
+\int_t^s \hat{\sigma}_{t,\mu}(r,\theta^{t,x,\mu,(0)}_r) \ud W_r,
\\
Y^{t,x,\mu}_s & = & \hat{g}_{t,\mu}(X^{t,x,\mu}_T ) + \int_s^t 
\hat{f}_{t,\mu}(r,\theta^{t,x,\mu}_r) \ud r
- \int_t^s Z^{t,x,\mu}_r \ud W_r,
\end{array}
\right. \label{eq reec mkvfbsde}
\end{align}
with 
$(\hat{b}_{t,\mu},\hat{f}_{t,\mu},\hat{\sigma}_{t,\mu},\hat{g}_{t,\mu})(r,x,y,z) := (b,f,\sigma,g)
(x,y,z,\law{\theta^{t,\xi,(0)}_r})$.
Basically, \new{for} this new set of coefficients, the dependence upon the measure is frozen
since $\mu$ and $\law{\theta^{t,\xi,(0)}}$ \new{are fixed} and do not 
depend on $x$.
In particular, when replacing $x$ by $\xi$ in 
\eqref{eq X-Y t,x,mu} and 
\eqref{eq reec mkvfbsde},
for some random variable $\xi$ with $\mu$ as distribution, uniqueness of
solutions to the classical (time-inhomogeous) FBSDE  \eqref{eq reec mkvfbsde} implies that
$(X^{t,\xi,\mu},Y^{t,\xi,\mu},Z^{t,\xi,\mu})=\theta^{t,\xi}$.
Then, the representation \eqref{eq reec mkvfbsde} allows us to characterize the decoupling field of the system \eqref{eq X-Y t,xi} as follows:    

Under \HYP{0}(i),
we know from the classical theory of coupled FBSDEs \cite{del02} 
that, for $T$ \new{sufficiently} small, for any $t \in [0,T]$, there exists a continuous decoupling field $\hat{U}_{t,\mu}: 
[t,T] \times \R^d \ni (s,x) \mapsto \hat{U}_{t,\mu}(s,x)$
\textcolor{black}{to}
\eqref{eq reec mkvfbsde} such that 
$Y^{t,x,\mu}_s = \hat{U}_{t,\mu}(s,X^{t,x,\mu}_s)$ for $s \in [t,T]$, the representation remaining 
true when $x$ is replaced by an ${\mathcal F}_{t}$-measurable square-integrable random variable (see
\cite[Corollary 1.5]{del02}). In particular, choosing $s=t$, we get 
$U(t,x,\mu) = \hat{U}_{t,\mu}(t,x)$. 
We deduce that 
{
\begin{Proposition} 
\label{pr thedecouplingproperty}
Given $(t,x,\mu) \in [0,T] \times \R^d \times \cP_{2}(\R^d)$
and $\xi \in L^2(\Omega,{\mathcal F}_{t},\P;\R^d)$, with $[\xi]=\mu$, 
we have, for $T$ small enough, for all $s \in [t,T]$, 
\begin{align*}
\hat{U}_{t,\mu} (s,x) = U(s,x,\law{X_s^{t,\xi}})\,,\;
Y^{t,x,\mu}_s = U(s,X_s^{t,x,\mu},\law{X_s^{t,\xi}})  \; \text{ and }
Y^{t,\xi}_s = U(s,X_s^{t,\xi},\law{X_s^{t,\xi}}).
\end{align*}
\end{Proposition}
}
\proof 
By uniqueness of the solution to \eqref{eq X-Y t,xi}, the processes
$$(X^{s,X_{s}^{t,\xi}}_{u},Y_{u}^{s,X_{s}^{t,\xi}})_{u \in [s,T]}
\quad
\textrm{and}
\quad 
(X^{t,\xi}_{u},Y_{u}^{t,\xi})_{u \in [s,T]}$$
coincide, so that $[(X_{u}^{s,X_{s}^{t,\xi}},Y_{u}^{s,X_{s}^{t,\xi}})]
= [(X^{t,\xi}_{u},Y_{u}^{t,\xi})]$ for $u \in [s,T]$. We deduce that 
$$\hat{U}_{t,\mu}(s,\cdot) = \hat{U}_{s,[X_{s}^{t,\xi}]}(s,\cdot),$$ 
which is the first equality. 
Now, the second equality follows from the fact that 
$$Y_{s}^{t,x,\mu} = \hat{U}_{t,\mu}(s,
X_{s}^{t,x,\mu}) = \hat{U}_{s,[X_{s}^{t,\xi}]}(s,X_{s}^{t,x,\mu}).$$ The last inequality
is obtained by inserting $X^{t,\xi}_s$ instead of $x$ in the first equality and observing
$\hat{U}_{t,\mu}(s,X^{t,\xi}_s)=\hat{U}_{t,\mu}(s,X^{t,\xi,\mu}_s)=Y^{t,\xi,\mu}_s=Y^{t,\xi}_s$.
%
%
\eproof

\vspace{2mm}

\subsection{Smoothness of $U$}

Introduce now the additional assumption: 
\begin{Assumption}[\HYP{0}(ii)] 
\new{The functions $b$, $\sigma$, $f$ and $g$ are twice differentiable in the
variables 
 $x$, $y$ and $z$, the derivatives of order $1$ and $2$ being uniformly bounded and uniformly Lipschitz-continuous
in the variables $x$, $y$ and $z$, uniformly in the parameter $\mu$}.
\end{Assumption} 

Under \HYP{0}(i--ii),
we know from the classical theory of coupled FBSDEs \cite{del02} 
that the decoupling field $\hat{U}_{t,\mu}$
is once continuously differentiable in time and twice continuously differentiable in \new{ the $x$ variable}
on $[t,T] \times \R^d$. It also 
satisfies $Z^{t,x,\mu}_s = \hat{V}_{t,\mu}(s,X^{t,x,\mu}_s)$ with 
\begin{align}
\hat{V}_{t,\mu}(s,x) := \partial_x \hat{U}_{t,\mu}(s,x) \hat{\sigma}_{t,\mu} \big(s,x, \hat{U}_{t,\mu}(s,x) \big) 
\; , 
\quad s \in [t,T], \; x \in \R^d.  
\label{eq de hvmu}
\end{align}
Notice that, \new{throughout \textcolor{black}{the paper}}, gradients of real-valued functions are expressed as row vectors. \new{In particular}, the
term $\partial_{x} \hat{U}_{t,\mu}(s,x)$ is thus an $m \times d$ matrix as $\hat{U}_{t,\mu}$
takes values in $\R^m$. 

Moreover, the decoupling field is a classical solution of the following quasi-linear PDE (or
system of quasi-linear PDEs since $m$ may be larger than $1$)\footnote{\new{For any matrix $a$ 
we will denote its transpose  by $a^{\dagger}$ and its trace by ${\rm Tr}(a)$. }}:
\begin{align}
&\partial_s \hat{U}_{t,\mu}(s,x) +
 \partial_x\hat{U}_{t,\mu}(s,x) 
 \hat{b}_{t,\mu} \big (s,x,\hat{U}_{t,\mu}(s,x),\hat{v}_{t,\mu}(s,x) \big) \label{eq PDE hat u}
 \\
&\hspace{5pt}
 + \frac12 
 {\rm Tr}\bigl[\partial^2_{xx}\hat{U}_{t,\mu}(s,x)   \big( \hat{\sigma}_{t,\mu} \hat{\sigma}_{t,\mu}^{\dagger} \big)
 \big (s,x,\hat{U}_{t,\mu}(s,x)\big)  \bigr]
+ \hat{f}_{t,\mu}\big (s,x,\hat{U}_{t,\mu}(s,x),\hat{v}_{t,\mu}(s,x) \big)  = 0, \nonumber
\end{align}
on $[t,T]\times \R^d$ with the \new{final} boundary condition $\hat{U}_{t,\mu}(T,x) = \hat{g}_{t,\mu}(x)$, $x \in \R^d$. 
\new{In \eqref{eq PDE hat u},} the trace 
reads as the $m$ dimensional vector 
 $({\rm Tr}[\partial^2_{xx}\hat{U}^i_{t,\mu}(s,x)  ( \hat{\sigma}_{t,\mu} \hat{\sigma}_{t,\mu}^{\dagger})
 (s,x,\hat{U}_{t,\mu}(s,x))  ])_{1 \leq i \leq m}$. 
Recalling the link between $\hat{U}_{t,\mu}(t,\cdot)$ and $U(t,\cdot,\mu)$, it is
then 
 clear that the function $U(t,\cdot,\mu)$ is twice continuously differentiable in \new{ the $x$ variable}. 
\vspace{5pt}

A more challenging question is the smoothness in the direction of the measure. 
Generally speaking, 
we will show that $U(t,x, \cdot)$ is \new{twice differentiable in the measure direction}, in a suitable sense,  
provided that the coefficients \new{of the system \eqref{eq intro XYZ}}
are \new{also} regular in the \new{measure} direction. 
For the reader's convenience,
we provide next a brief introduction to the notion of regularity with 
respect to the measure argument, further details being given in 
Section \ref{se chain rule}. 

Given a function $V:\cP_2(\R^d)\to\R$,  
we call the 
\textit{lift of $V$} on the probability space $(\Omega,{\mathcal A},\P)$\footnote{For notational convenience, the lifting procedure is done onto the same probability space that carries the driving Brownian motion $W$. Alternatively, one can use an arbitrary rich enough atomless probability space, see \cite{cardaliaguet} and Section 
\ref{se chain rule} for details.} 
the mapping ${\mathcal V}:L^2(\Omega,\cA,\P)\to \R $ defined by  
\begin{equation*}
{\mathcal V}(X) = V\bigl( \law{X} \bigr), \quad X \in L^2(\Omega,\cA,\P).
\end{equation*}
\new{Following Lions} (see \cite{cardaliaguet}), the mapping $V$ is then said to be differentiable (resp. $\cC^1$) on the 
Wasserstein space $\cP_2(\R^d)$ if the lift ${\mathcal V}$ is 
\new{Fr\'echet} differentiable (resp. {Fr\'echet} differentiable with continuous derivatives) on $L^2(\Omega,\cA,\P)$. 
{The main feature of this approach is that the Fr\'echet derivative 
$D {\mathcal V}(X)$, when seen as an element of $L^2(\Omega,\cA,\P)$ via Riesz' theorem, can be represented as }
\begin{equation*}
D {\mathcal V}(X) = \partial_{\mu} V \bigl( \law{X} \bigr)(X),
\end{equation*}
{where $\partial_{\mu} V(\law{X}) : \R^d \ni v \mapsto \partial_{\mu} V(\law{X})(v) \in \R^d$ is in $L^2(\R^d,\law{X};\R^d)$.
In this way the tangent space to ${\mathcal P}_{2}(\R^d)$ at a probability measure $\mu$ 
is identified with a subspace of $L^2(\R^d,\mu;\R^d)$. }

Note that the map 
 {$\partial_{\mu} V(\mu) : \R^d \ni v \mapsto 
\partial_{\mu} V(\mu)(v) \in \R^d$} 
is uniquely 
defined up to a $\mu$-negligible Borel subset. Choosing a version for each $\mu$ might be a problem
for handling it as a function of the joint variable $(v,\mu)$. 
In the next section, we will present conditions under 
which
a continuous version 
of $\partial_{\mu} U(\mu)(\cdot)$
can be identified,
such a version being uniquely defined  
on the support of 
$\mu$.
The next step is then to discuss the smoothness of the map 
$\R^d \times \cP_{2}(\R^d) \ni (v,\mu)
\mapsto 
\partial_{\mu} V(\mu)(v)$. We  say that $V$ is \textit{partially $\cC^2$} if the mapping
${\mathcal P}_{2}(\R^d) \times \R^d \ni (\mu,v) \mapsto \partial_{\mu} 
V(\mu)(v)$
{
is continuous at any point $(\mu,v)$ such that $v \in \textrm{Supp}(\mu)$} and if, 
for any $\mu \in {\mathcal P}_{2}(\R^d)$, the mapping 
$\R^d \ni v \mapsto \partial_{\mu} V(\mu)(v)$ is differentiable, its derivative 
being jointly continuous with respect to $\mu$ and $v$
{
at any point $(\mu,v)$ such that $v \in \textrm{Supp}(\mu)$.}
The 
gradient is then denoted by $\partial_{v} [\partial_{\mu} V(\mu)](v) \in \R^{d 
\times d}$. 

\new{Note that}, $\partial_{\mu} V(\mu)(v)$ is a $d$-dimensional row vector and 
$\partial_{v} [\partial_{\mu} V(\mu)](v)$ is a $d \times d$ matrix. 


\subsection{Solution of a Master PDE}
\label{subse:solution:PDE}
In Section \ref{se chain rule}, we prove a chain rule for functions defined on the space $\cP_2(\R^d)$ 
which are \emph{partially} $\cC^2$ in the above sense.
\new{Applying the chain rule} to $U(t,x,\cdot)$, \new{we get}:
\begin{align}
&U\bigl(t,x,\law{X_s^{t,\xi}}\bigr) - U\bigl(t,x,\law{\xi}\bigr) \nonumber
\\
&= \int_t^s \hesp{\partial_\mu U\bigl(t,x,\law{X^{t,\xi}_r}\bigr)
\bigl(
\cc{X^{t,\xi}_r}\bigr)  
b\bigl(\cc{\theta^{t,\xi}_r},\law{\theta^{t,\xi,(0)}_r}\bigr)} \ud r \label{eq:chain:rule:PDE}
 \\
 &\hspace{15pt} + \frac12 \int_t^s \hesp{
 \text{Tr} \bigl[
 \partial_v \bigl[\partial_\mu U\bigl(t,x,\law{X^{t,\xi}_r}\bigr)\bigr]\bigl(\cc{X^{t,\xi}_r}\bigr)
 \bigl( \sigma
  \sigma^{\dagger}\bigr)\bigl(  
\cc{\theta^{t,\xi,(0)}_r},\law{\theta^{t,\xi,(0)}_r}\bigr)
\bigr]
} \ud r. \nonumber
 \end{align}
 The above identity relies on new notations. Indeed, in order to distinguish the original randomness in the dynamics 
of \eqref{eq X-Y t,xi}, which has a physical meaning, from the randomness used to represent the derivatives
on the Wasserstein space, we will represent the derivatives on the Wasserstein space on another probability space, denoted by 
$(\hat{\Omega},\hat{\mathcal A},\hat{\P})$. \new{$(\hat{\Omega},\hat{\mathcal A},\hat{\P})$ is}  a copy of the original space 
$(\Omega,{\mathcal A},\P)$. In particular, for a random variable $\xi$ defined on 
$(\Omega,{\mathcal A},\P)$, we denote by $\cc{\xi}$ its copy on $\hat{\Omega}$. 
All the expectations in the above expression
may be translated into expectations under $\E$. \new{Nevertheless, we will refrain from doing this to avoid ambiguities 
between ``lifts'' and}
random variables constructed 
on the original space $(\Omega,{\mathcal A},\P)$. 
\new{We will state conditions} 
under which the expectations in \eqref{eq:chain:rule:PDE}
are indeed well-defined. 

Notice that, in \eqref{eq:chain:rule:PDE}, we used the same convention as in 
\eqref{eq PDE hat u} for denoting gradients. The term $\partial_\mu U(t,x,\law{X^{t,\xi}_r})
(\cc{X^{t,\xi}_r})$ is thus seen as an $m \times d$ matrix and the trace term 
$\text{Tr} [
 \partial_v [\partial_\mu U](t,x,\law{X^{t,\xi}_r})(\cc{X^{t,\xi}_r})
 ( \sigma
  \sigma^{\dagger})(  
\cc{\theta^{t,\xi,(0)}_r},\law{\theta^{t,\xi,(0)}_r})
]$
as a vector of dimension $m$.   

Combined with the analysis of the smoothness of $U$, 
we will then 
show that 
{the function
$[0,T] \times \R^d \times \cP_{2}(\R^d) 
\ni (t,x,\mu) \mapsto U(t,x,\mu)$
solves, up to a time reversal, a PDE of the form \eqref{form0}}.
{For the time being, 
we present a formal calculation to deduce this claim, 
the complete argument being given in 
Section \ref{se smoothness}}. The basic observation is that, 
{in 
the framework of Proposition 
\ref{pr thedecouplingproperty}}, the
 time-increments of $U$ may be expanded as 
\begin{equation}
\label{eq:expansion:master:PDE}
\begin{split}
U(s+h,x,\law{X^{t,\xi}_s}) - U(s,x,\law{X^{t,\xi}_s}) &= 
U(s+h,x,\law{X^{t,\xi}_s}) - U(s+h,x,\law{X^{t,\xi}_{s+h}})
\\
&\hspace{15pt}+ \hat{U}_{t,\mu}(s+h,x) - \hat{U}_{t,\mu}(s,x),
\end{split}
\end{equation}
for $t \leq s \leq s+h \leq T$. 
Applying the chain rule to the difference term 
$U(s+h,x,\law{X^{t,\xi}_s}) - U(s+h,x,\law{X^{t,\xi}_{s+h}})$
on the right hand side of the previous equality 
and assuming that the derivatives in the chain rule are continuous
\new{in time} so that we can let $h$ tend to 
$0$, we obtain  
\begin{align}
&\bigl[ \frac{\ud}{\ud h} \bigr]_{\vert h=0} U(s+h,x,\law{X^{t,\xi}_s}) 
= 
- \hesp{
b\bigl(\cc{\theta^{t,\xi}_s},\law{\theta^{t,\xi,(0)}_s}\bigr)
\partial_\mu U(s,x,\law{X^{t,\xi}_s})
\bigl(\cc{X^{t,\xi}_s}\bigr) 
} \label{eq:time:derivative:U}
\\
&\hspace{15pt} - \frac12  \hesp{
 {\rm Tr} \bigl[ 
  \partial_v [\partial_\mu U(s,x,\law{X^{t,\xi}_s})]\bigl(\cc{X^{t,\xi}_s}\bigr) 
 \bigl( \sigma \sigma^{\dagger} \bigr) 
 \bigl(\cc{\theta^{t,\xi,(0)}_s},\law{\theta^{t,\xi,(0)}_s}\bigr)\bigr]}
  + \partial_{s} \hat{U}_{t,\mu}(s,x). \nonumber
\end{align}
Choosing $s=t$, we deduce that $U$ is right differentiable in time
\textcolor{black}{(it is then
differentiable in time provided that the right-hand side is continuous in time)}.
Recalling \eqref{eq PDE hat u}
together with the notation $\mu = \law{\xi}$
and using the transfer theorem
to express the expectations that appear in the chain rule 
as integrals over $\R^d$,
we then get that $U$ is a solution to
the equation
\begin{equation}
\label{eq:master:PDE}
\begin{split}
 \partial_t U(t,x,\mu) &+ 
   \partial_xU(t,x,\mu) b \big(x,U(t,x,\mu),\partial_{x}^\sigma U(t,x,\mu),\nu) 
   \\
  &+ \frac12 {\rm Tr} 
  \bigl[\partial^2_{xx}U(t,x,\mu) \bigl( \sigma \sigma^{\dagger} \bigr) 
 \bigr]
+ f\big (x,U(t,x,\mu),\partial_{x}^\sigma U(t,x,\mu), \nu \big)  
\\
&+ 
\int_{\R^d}
\partial_\mu U(t,x,\mu)(v) 
b\bigl(v,U(t,v,\mu),\partial_{x}^\sigma U(t,v,\mu),\nu\bigr) \ud\mu(v)
\\
& +  \frac12  \int_{\R^d} {\rm Tr} \bigl[ 
\partial_v \bigl[ \partial_\mu U(t,x,\mu)\bigr] (v)
\bigl( \sigma \sigma^{\dagger}\bigr)\bigl(v,U(t,v,\mu),\nu\bigr)
\bigr] \ud\mu(v) = 0,
\end{split}
\end{equation}
for $(t,x,\mu) \in [0,T] \times \R^d \times \cP_{2}(\R^d)$,
with the terminal condition $U(T,x,\mu) = g(x,\mu)$,
where $\nu$ is the law
of $(\xi,U(t,\xi,\mu))$ when $[\xi]=\mu$ and 
$$\partial_{x}^\sigma U(t,x,\mu) =  \partial_{x} U(t,x,\mu) \sigma(x,U(t,x,\mu),\nu).$$
In particular, $u(t,\cdot,\cdot) = U(T-t,\cdot,\cdot)$
satisfies the equation \eqref{form0}, 
the operators $A$, $B$ and $C$ therein being defined as follows:
\begin{eqnarray}
\label{operators}
A u(t,x,\mu)&=&\partial_x u(t,x,\mu) b \big(x,u(t,x,\mu),\partial_{x}^\sigma u(t,x,\mu),\nu \big)  \nonumber 
\\
&&+ \frac12 {\rm Tr} 
  \bigl[\partial^2_{xx} u(t,x,\mu) \bigl( \sigma \sigma^{\dagger} \bigr) \bigl(x,u(t,x,\mu),\nu \bigr)
 \bigr]  \nonumber
 \\
B u(t,x,\mu)&=&  \partial_{x} u(t,x,\mu) \sigma(x,u(t,x,\mu),\nu)  \nonumber \\
C u(t,x,\mu)(v)& =& \partial_\mu u(t,x,\mu)(v) 
b\bigl(v,u(t,v,\mu),\partial_{x}^\sigma u(t,v,\mu),\nu\bigr)\nonumber\\
&& + {1\over 2} {\rm Tr} \bigl[ 
\partial_v \bigl[ \partial_\mu u(t,x,\mu)\bigr] (v)
\bigl( \sigma \sigma^{\dagger}\bigr)\bigl(v,u(t,v,\mu),\nu\bigr)
\bigr], \ v \in \R^d,
 \nonumber
\end{eqnarray} 
with the initial condition $u(0,x,\mu) = g(x,\mu)$, and with 
the same convention as above for 
the meaning of $\nu$ and of 
$\partial_{x}^\sigma u$. 
\color{black}

\color{black}

{Our} first main result is that, for small time horizons, all 
the partial derivatives that appear above make sense as continuous functions 
whenever the coefficients are sufficiently smooth. In this sense, 
$U$ is a ``classical'' solution of \eqref{eq:master:PDE}, see Theorem 
\ref{main:thm:short:time} right below. We can actually prove that it is the unique
one to satisfy suitable growth conditions, see Theorem
\ref{main:thm:uniqueness}. 
{Our} second main result is the extension to arbitrarily large time horizons for three classes of population equilibria. 
We refer the reader
to Subsection \ref{subse:main:results} for a short account of the second result and to 
Section \ref{se app} for complete statements.

\subsection{Assumptions}
\label{subse:assumption}
For an $L^2$ space, we use the notation  
$\NL{2}{\cdot}$ as a generic notation for the corresponding $L^2$-norm.  
For a linear mapping $\Upsilon$ on an $L^2$ space, we let 
$ \Nt{\Upsilon} := \sup_{\NL{2}{\upsilon} = 1} \NL{2}{\Upsilon({\upsilon})}$,
and for a bilinear form on an $L^2$ space, we let in the same way 
$ \Nt{\Upsilon} := \sup_{\NL{2}{\upsilon_1} = 1, \NL{2}{\upsilon_2} = 1} 
\NL{2}{\Upsilon({\upsilon}_1,{\upsilon}_2)}$. 

For a function $h$ from \textcolor{black}{a} product space of the form $\R^k \times {\mathcal P}_{2}(\R^{l})$
into $\R$, where 
$k,l \geq 1$, we denote by $\partial_{w} h(w,\mu)$ the derivative (if it exists) 
of $h$ with respect to the Euclidean variable $w$ and 
by $D{\mathcal H}(w,\chi)$ the Fr\'echet derivative of the 
lifted mapping  ${\mathcal H} : L^2(\Omega,\cA,\P;\R^l) \ni \chi \mapsto h(w,\law{\chi})$. 
The Fr\'echet derivative is seen as a linear form on $L^2$. 
{Concerning the first order differentiability  of the coefficients, we shall 
assume:}

%
{
\begin{Assumption}{\HYP{1}}
{In addition to  
\HYP{0}{\rm (i)}},
the mappings $b$, $f$, 
$\sigma$, $g$ are differentiable in $(w=(x,y,z),\mu)$\footnote{Here $\mu$ stands for the generic symbol to denote the measure argument.}
with jointly continuous derivatives
in $(w,\mu)$\footnote{Under the standing assumptions on the joint continuity of 
the derivatives, it is easily checked that the joint differentiability is equivalent  
to partial differentiability in each 
of the two directions $w$ and $\mu$.}
in the following sense: There exist a constant $\tilde{L}$ (in addition to the constant $L$ 
defined in $\HYP{0}\textrm{\rm (i)}$),
a constant $\alpha \geq 0$ 
and a functional $\Phi_{\alpha} : [L^2(\Omega,\cA,\P;\R^l)]^2  \ni (\chi,\chi') \mapsto \Phi_{\alpha}(\chi,\chi') \in \R_{+}$,
continuous at any point $(\chi,\chi)$ of the diagonal,
such that, for all $\chi,\chi' \in L^2(\Omega,\cA,\P;\R^l)$,
 \begin{equation}
 \label{phialpha}
\Phi_{\alpha}(\chi,\chi') \leq  
{\mathbb E}
\Bigl\{ \bigl( 1+ \vert \chi \vert^{2\alpha} + \vert \chi' \vert^{2\alpha} 
+ \| \chi \|_{2}^{2\alpha}  \bigr)
 \vert \chi - \chi' \vert^2 \Bigr\}^{1/2} \quad \text{when} \ \chi \sim \chi',
\end{equation}
and, for $h$ matching 
{any of the coordinates}\footnote{For the presentation of the assumption, it is here easier to take $h$ as a real-valued function, which explains
why we identify $h$ with a coordinate of $b$, $f$, $\sigma$ or $g$; however, 
we will sometimes say --rather abusively-- that $h$ matches $b$, $f$, $\sigma$ or $g$.}
of $b$, $f$, $\sigma$ or $g$,
for all
$w,w' \in \R^k$ and 
 $\chi,\chi' \in L^2(\Omega,\cA,\P;\R^l)$,
 \begin{align*}
&\Nt{D{\mathcal H}(w,\chi)} \le L \; , \quad  
 \Nt{D{\mathcal H}(w,\chi)-D{\mathcal H}(w',\chi')} \le \tilde{L} \bigl(|w- 
w'| + 
 \Phi_{\alpha}(\chi,\chi')
 \bigr), 
 \end{align*}
and
\begin{align*}
|\partial_{w} h(w,[\chi])| \leq L \; , \quad
|\partial_{w} h(w,[\chi])- \partial_{w}h(w',[\chi']) |
 \le \tilde{L} \bigl(|w-w'| + 
 \Phi_{\alpha}(\chi,\chi')
 \bigr).
 \end{align*}
Moreover, for any $\chi \in L^2(\Omega,\cA,\P;\R^l)$, 
the family $(D \cH(w,\chi))_{w \in \R^k}$, identified
by Riesz' theorem 
with a collection of elements in 
$L^2(\Omega,\cA,\P;\R^l)$, is uniformly square integrable.
\end{Assumption}
}
\begin{Remark}
{
(i) In particular we note that
$|D {\mathcal H}(w,\chi) \cdot \chi'|\le L \NL{2}{\chi'}$ (where 
`$\cdot$' denotes the action of the duality).}

{(ii) Notice that the right-hand side in  
 \eqref{phialpha} might not be finite. Actually,
we shall make use of \eqref{phialpha} when $\chi$ and $\chi'$  
coincide outside a bounded subset $\R^l$, namely $\chi(\omega) = \chi'(\omega)$ whenever 
$\vert \chi(\omega) \vert$ and $\vert \chi'(\omega) \vert$ 
are larger than some prescribed $R \geq 0$, in which cases the right-hand side in \eqref{phialpha} is finite. 
For instance, choosing $\chi = \chi'$,
we get from \eqref{phialpha} that $\Phi_{\alpha}$ is zero on the diagonal. 
Notice also that, when $\alpha=0$, we can directly choose 
$\Phi_{\alpha}(\chi,\chi') = \E[\vert \chi - \chi' \vert^2]^{1/2}$.}

{(iii) Proposition \ref{prop:lipschitz:lifted} below shows that, under \HYP{1}, 
the function $\R^l \ni v \mapsto \partial_{\mu} h(w,\mu)(v)$ admits, for any 
$w \in \R^k$ and $\mu \in {\mathcal P}_{2}(\R^l)$, a continuous version.
%
It allows to represent $D {\mathcal 
H}(w,\chi)$, when identified with an element of $L^2(\Omega,\cA,\P;\R^l)$
by Riesz' theorem, in the form
$\partial_{\mu}h(w,[\chi])(\chi)$. We stress that 
{such a continuous version of $\R^l \ni v \mapsto \partial_{\mu} h(w,\mu)(v)$ 
is uniquely defined on the support of $\mu$}.
Reexpressing the bounds in \HYP{1}, it satisfies
\begin{equation}
\label{eq:to compare}
\begin{split}
&{\mathbb E} \bigl[ \bigl\vert 
\partial_{\mu} h(w,[\chi])(\chi)
\bigr\vert^2 \bigr]^{1/2} \leq {L},
\\
&{\mathbb E} \bigl[ \bigl\vert 
 \partial_{\mu} h (w,[\chi])(\chi)
- \partial_{\mu} h (w',[\chi']) (\chi')
\bigr\vert^2 \bigr]^{1/2}
\leq \tilde{L} \bigl\{ 
\vert w- w' \vert 
+ 
\Phi_{\alpha}(\chi,\chi') \bigr\},
\end{split}
\end{equation}
Moreover, the uniform square integrability property is equivalent to say that the family 
$(\partial_{\mu} h(w,[\chi])(\chi))_{w \in \R^k}$ is uniformly square integrable
for any $\chi \in L^2(\Omega,\cA,\P;\R^d)$. }

{
(iv) The uniform integrability assumption plays a major role in our analysis. Taking into account the fact that 
all the $(D \cH(w,\chi))_{w \in \R^k}$ have a norm less than $L$, this amounts to require that 
\begin{equation*}
\lim_{\P(A) \rightarrow 0, A \in \cA}
\sup_{w \in \R^k} \sup_{\Lambda \in L^2(\Omega,\cA,\P;\R^l) : \| \Lambda \|_{2} \leq L}
\bigl|D {\mathcal H}(w,\chi) \cdot \bigl( \Lambda {\mathbf 1}_{A} \bigr) \bigr\vert
= 0.
\end{equation*}
We stress the fact that it is automatically satisfied when 
$\alpha=0$ in  \eqref{phialpha}. Indeed, we shall prove in 
\eqref{eq:lineargrowth} below that, whenever 
$\alpha=0$, there exists a constant 
$C \geq 0$ such that, for all $w \in \R^k$, 
$\vert D {\mathcal H}(w,\chi) \vert$ (identified with a random variable)
is less than $C(1+ \vert \chi \vert + \| \chi \|_{2})$.  }\end{Remark}

{Concerning the second order differentiation of the coefficients}, we shall assume:

\begin{Assumption}[\HYP{2}] 
{In addition to  \HYP{1},} all the mappings \new{$(x,y,z)\mapsto b(x,y,z,\mu)$, $(x,y,z)\mapsto f(x,y,z,\mu)$, 
$(x,y)\mapsto \sigma(x,y,\mu)$ and $x\mapsto g(x,\mu)$ are 
twice differentiable for any $\mu \in {\mathcal P}_{2}(\R^l)$} 
the second-order derivatives being jointly continuous 
in $(x,y,z)$ and $\mu$. Moreover, 
for $h $ equal to 
{any of the coordinates of} $b$, $f$, $\sigma$ or $g$,
for any $w \in \R^k$ and 
$\mu \in {\mathcal P}_{2}(\R^l)$,
with the appropriate dimensions $k$ and $l$,
{there exists a continuously differentiable version of the 
mapping $\R^l \ni v \mapsto 
\partial_{\mu} h(w,\mu)(v)$ 
such that
the mapping $\R^k \times \R^l \ni (w,v) \mapsto 
\partial_{\mu} h(w,\mu)(v)$ is differentiable
(in both variables)
at any point $(w,v)$ such that $v \in \textrm{\rm Supp}(\mu)$, 
the partial derivative
$\R^k \times \R^l \ni (w,v)
\mapsto 
\partial_{v} [\partial_{\mu} h(w,\mu)](v)$
being continuous at any $(w,v)$
such that $v \in \textrm{\rm Supp}(\mu)$
and 
the partial derivative
$\R^k \times \textrm{\rm Supp}(\mu) \ni (w,v)
\mapsto 
\partial_{w} [\partial_{\mu} h(w,\mu)](v)$
being continuous in $(w,v)$}.   
With the same constants $\tilde{L}$ and $\alpha$ as in 
\HYP{1}, for $w \in \R^k$ and $\chi \in L^2(\Omega,{\mathcal A},\P;\R^l)$,
\begin{equation*}
\bigl\vert \partial_{ww}^2 h(w,[\chi])
\bigr\vert
+
{\mathbb E} \bigl[ \bigl\vert \partial_{w} \bigl[\partial_{\mu} h(w,[\chi])\bigr](\chi)
\bigr\vert^2 \bigr]^{1/2} 
+
{\mathbb E} \bigl[ \bigl\vert \partial_{v} \bigl[\partial_{\mu} h(w,[\chi])\bigr](\chi)
\bigr\vert^2 \bigr]^{1/2} \leq \tilde{L},
\end{equation*}
and, for $w,w' \in \R^k$ and $\chi,\chi' \in L^2(\Omega,{\mathcal A},\P;\R^l)$,
\begin{equation*}
\begin{split}
&\bigl\vert \partial_{ww}^2 h(w,[\chi]) - \partial_{ww}^2 
h (w',[\chi'])(\chi') \bigr\vert
\\
&\hspace{15pt} 
+ {\mathbb E} \bigl[ \bigl\vert \partial_{w}\bigl[ \partial_{\mu} h (w,[\chi])\bigr] (\chi)
- \partial_{w} \bigl[ \partial_{\mu} h (w',[\chi'])\bigr] (\chi')
\bigr\vert^2 \bigr]^{1/2} 
\\
&\hspace{15pt} + {\mathbb E} \bigl[ \bigl\vert \partial_{v} \bigl[ \partial_{\mu} h (w,[\chi])\bigr](\chi)
- \partial_{v} \bigl[ \partial_{\mu} h (w',[\chi']) \bigr](\chi')
\bigr\vert^2 \bigr]^{1/2} 
\\
&\leq \tilde{L} \bigl\{ 
\vert w- w' \vert 
+ 
\Phi_{\alpha}(\chi,\chi')
%
\bigr\},
\end{split}
\end{equation*}
In \HYP{2}, we include the assumption:
\end{Assumption}
\begin{Assumption}[\HYP{\sigma}]
The function $\sigma$ is bounded by $\tilde{L}$.
\end{Assumption}
{Note that \HYP{2} contains \HYP{0}(ii) (and obviously 
\HYP{0}(i)
and
\HYP{1}).}

\begin{Remark}
%
The specific form of \HYP{1} and \HYP{2} is dictated by our desire to establish 
results for arbitrary large horizons. 
Generally speaking, such results are established by means of a recursive argument, 
which consists in using the current value 
$U(t,\cdot,\cdot)$ of the decoupling field at time $t$ as a new boundary condition, or put it differently in letting 
$U(t,\cdot,\cdot)$ play at time $t$ the role of $g$ at time $T$
when the FBSDEs  \eqref{eq X-Y t,xi} and \eqref{eq X-Y t,x,mu} are considered on $[0,t]$ instead of $[0,T]$.  
A delicate point in this construction is to choose a space of boundary 
conditions which is stable, namely in which $U(t,\cdot,\cdot)$ remains along the recursion. We remark that we cannot prove that boundary conditions with globally 
Lipschitz derivatives in the measure argument are stable, even in small time. 
One of the contribution of the paper is thus to identify a space of terminal 
conditions which are indeed stable and which permits to apply the recursion method.
\end{Remark}

\begin{Remark}
\label{rem:comparaison:buck}
The reader may compare \HYP{0}, \HYP{1} and \HYP{2} with the assumptions
in \cite{buck:li:rain:peng}. 
We first point out that, 
in \cite{buck:li:rain:peng}, 
the first 
$L^2$ bound in \eqref{eq:to compare} is assumed to hold in $L^{\infty}$. 
The example 
$h([\xi])=\| \xi \|_{2}$, for which the derivative 
has the form $\partial_{\mu} h([\xi])(v) = v/\| \xi \|_{2}$, 
shows that asking $\partial_{\mu} h$ to be in $L^\infty$
is rather restrictive. 
We also observe that, differently from 
\cite{buck:li:rain:peng},
we do not require  
the coefficients to admit second-order derivatives of the 
type 
$\partial^2_{\mu\mu}$. 
The reason is that we here 
establish the chain rule for functions 
from $\cP_{2}(\R^d)$ to $\R$ that may not have second-order derivatives of the type 
$\partial^2_{\mu \mu}$, see
Theorem \ref{thm:partial:C2}.
\end{Remark}

\subsection{Main results: from short to long time horizons and application to control}

\label{subse:main:results}
Inspired by Assumptions 
{\HYP{0}(i)}, \HYP{1} and \HYP{2}, we let:
\begin{Definition}
\label{eq:sol:admissible}
Given non-negative real numbers $\beta,a,b$, with $a<b$, 
we denote by $\cD_{\beta}([a,b])$ the space of functions $V : [a,b] \times \R^d \times \cP_{2}(\R^d) 
\ni (t,x,\mu) \mapsto V(t,x,\mu) \in \R^m$ for which we can find a constant $C \geq 0$ such that

\textit{(i)} For any $t \in [a,b]$, the function $V(t,\cdot,\cdot) : \R^d \times 
\cP_{2}(\R^d) \ni (x,\mu) \mapsto V(t,x,\mu)$ satisfies the same assumption 
as $g$ in \HYP{0}{\rm (i)}, \HYP{1}
and \HYP{2}, but with $\alpha$ replaced by $\beta$ and with $L$ and $\tilde{L}$ replaced by $C$ (and thus 
with $w=x \in \R^d$, $v \in \R^d$ and $\chi \in L^2(\Omega,\cA,\P;\R^d)$
in the various inequalities where these letters appear);

\textit{(ii)} For any $x \in \R^d$ and $\mu \in \cP_{2}(\R^d)$, the function 
$[a,b] \ni t \mapsto V(t,x,\mu)$ is differentiable, 
{the derivative being 
continuous  with respect to $(t,x,\mu)$ on the set $[a,b] \times \R^d \times \cP_{2}(\R^d)$}. 
Moreover, the functions
\begin{eqnarray*}
[a,b] \times \R^d
\times L^2(\Omega,\cA,\P;\R^d)
\ni (t,x,\xi)&\mapsto& 
\partial_{x} V(t,x,[\xi]) \in \R^d, \\ 
\ [a,b]\times\R^d\times L^2(\Omega,\cA,\P;\R^d)
\ni (t,x,\xi) &\mapsto& 
\partial_{\mu} V(t,x,[\xi])(\xi) \in L^2(\Omega,\cA,\P;\R^d),\\
\ [a,b] \times \R^d
\times L^2(\Omega,\cA,\P;\R^d)
\ni (t,x,\xi) &\mapsto& 
\partial_{x}^2 V(t,x,[\xi]) \in \R^d,\\ 
\ [a,b] \times \R^d
\times L^2(\Omega,\cA,\P;\R^d)
\ni (t,x,\xi) &\mapsto& 
\partial_{x} [ \partial_{\mu} V(t,x,[\xi])](\xi)\in L^2(\Omega,\cA,\P;\R^d)\\
\ [a,b] \times \R^d
\times L^2(\Omega,\cA,\P;\R^d)
\ni (t,x,\xi) &\mapsto& 
\partial_{v} [
\partial_{\mu} V(t,x,[\xi])](\xi) \in L^2(\Omega,\cA,\P;\R^d) 
\end{eqnarray*}
are continuous. 
\end{Definition}
For the reader's convenience, when $[a,b]=[0,T]$, we will simply use the notation $\cD_\beta$
for $\cD_{\beta}([0,T])$.

\vspace{2mm}
The set $\bigcup_{\beta \geq 0}
\cD_{\beta}$ is the space we use below for investigating existence and uniqueness
of a ``classical'' solution to \eqref{eq:master:PDE}. 
For short time horizons, our main result
takes the following form 
(see Theorem 
\ref{prop:partial:C^2:u}):

\begin{Theorem}
\label{main:thm:short:time}
Under Assumption \HYP{2}, there exists a constant 
$c=c(L)$ ($c$ not depending upon $\tilde{L}$ nor $\alpha$) such that, 
for $T \leq c$, the function 
$U$ defined in \eqref{eq:decoupling:field} is in $\cD_{\alpha+1}$
and satisfies the PDE \eqref{eq:master:PDE}.
\end{Theorem}

Uniqueness holds in the class $\bigcup_{\beta \geq 0} \cD_{\beta}$:
\begin{Theorem}
\label{main:thm:uniqueness}
Under \textcolor{black}{ \HYP{0}{\rm (i)}} and \HYP{\sigma}, there exists at most one solution 
to the PDE \eqref{eq:master:PDE} in the class $\bigcup_{\beta \geq 0} \cD_{\beta}$
(regardless of how large the time horizon $T$ is). 
\end{Theorem}

\new{The} extension to \new{arbitrarily large time horizons} will be discussed in Section \ref{se app}. 
The principle for extending the result from small to long horizons has been already \new{covered} in 
several papers, including \cite{del02,ma:wu:zhang:zhang}. Basically, the principle is to prove that, \new{following the recursive step}, 
the decoupling field remains in a space of admissible boundary conditions 
for which the length of the interval of solvability can be bounded from below. 
Generally speaking, this requires, first, to 
\new{identify} a class of functions on $\R^d \times \cP_{2}(\R^d)$ \new{left invariant by the recursive step} 
and, second, to control the Lipschitz constant
of the decoupling field, uniformly along the \new{recursion.} In the current framework, the Lipschitz constant means the Lipschitz constant in 
both the space variable and the measure argument. 

As
suggested by 
Theorem
\ref{main:thm:short:time},
we are not able to prove that \new{the space $\cD_{\beta}$, for a fixed 
$\beta \geq 0$ is left invariant by the recursion step}. 
\new{In particular, for} the case $\beta=0$, this means that, even in short time, 
we \new{cannot} prove that
the decoupling field has Lipschitz derivatives in the direction of the measure 
when the terminal condition $g$ has Lipschitz derivatives.
This \new{difficulty}  motivates the specific form of 
the local Lipschitz assumption in \HYP{1} and 
\HYP{2}. Indeed, Theorem
\ref{main:thm:short:time} shows that
 the set $\bigcup_{\beta \geq 0} \cD_{\beta}$ is preserved by the dynamics
of the master PDE under \HYP{2} although 
none of the sets $\cD_{\beta}$ has been shown to be stable. 
\new{More precisely,} 
we allow the 
exponent $\alpha$ in \HYP{1} and \HYP{2} to increase by $1$ at each step \new{of the recursion}, 
Theorem
\ref{main:thm:short:time}
guaranteeing that
the set $\bigcup_{\beta \geq 0} \cD_{\beta}$ is indeed stable along the induction.


In Section \ref{se app}, we give three examples when the Lipschitz constant of the decoupling field can be indeed controlled. 
First, we consider the forward-backward system deriving from the tailor-made version of the 
stochastic Pontryagin principle 
for mean-field games. Then, we establish a Lipschitz estimate of $U$, 
in the case when the extended Hamiltonian of the control problem 
is convex in both the state and control variables
 and when the Lasry-Lions monotonicity 
condition that guarantees uniqueness of the equilibrium is satisfied (see 
\cite{cardaliaguet}). 
We then interpret $U$ as the gradient in space of the solution of 
the master equation that arises in the theory of mean-field games
and,  
as a byproduct, we get that, in this framework, the master equation for mean-field games
is solvable. 
Second, we propose another approach to handle 
the master equation for mean-field games when the extended Hamiltonian is not convex in $x$. 
We directly express the solution of the master equation as 
the decoupling field of a forward-backward system of the McKean-Vlasov type. 
We then prove the required Lipschitz estimate of $U$ 
 when the cost functionals are bounded in $x$ and are linear-quadratic in $\alpha$, 
 the volatility is non-degenerate 
 and the Lasry-Lions condition is in force. 
Third, we consider the forward-backward system deriving from the stochastic 
Pontryagin principle, when 
applied to the control of McKean-Vlasov diffusion processes.  Then, we establish 
a similar estimate
for the Lipschitz control of $U$, but under a
stronger convexity assumption of the extended Hamiltonian --namely, convexity must
hold in the state and control variables and also 
in the direction of the measure-- (in which case there is no need of the Lasry-Lions condition).
Again, this permits \new{us} to deduce that the master equation associated to 
the control problem \new{has a global classical solution}.

We may summarize with the following statement 
(again, we refer to Section \ref{se app} for a complete account):

\begin{Theorem}
\label{main:thm:3}
We can find general examples taken from large population stochastic control such that,
for a given $T>0$, 
\eqref{eq X-Y t,xi}
and \eqref{eq X-Y t,x,mu} have a unique solution and the decoupling field 
$U$ belongs to $\bigcup_{\beta \geq 0} \cD_{\beta}$ and 
satisfies the PDE \eqref{eq:master:PDE}. In particular, the corresponding forward equation \eqref{form0} has a unique classical solution on $[0,\infty)$.
\end{Theorem}
 
\subsection{Frequently used notations.}
For two random variables $X$ and $X'$, 
the relationship $X \sim X'$ means that $X$ and $X'$ have the same distribution. 
The conditional expectation given ${\mathcal F}_{t}$ is denoted by $\E_{t}$. 
Let $t\in [0,T)$. For a progressively-measurable process $(X_{s})_{s \in [t,T]}$ with values in 
$\R^{l}$, for some integer $l \leq 1$, we let
\begin{equation}
\label{eq:norms:NHS}
\begin{split}
 \NHt{p}{X} &:= \E_{t}\biggl[ \left(\int_t^T \vert X_s \vert^2 \ud s\right)^{p/2} \biggr]^{1/p},\quad \NSt{p}{X} := \E_{t} \Bigl[ 
 \sup_{s \in [t,T]}|X_s|^p \Bigr]^{1/p},
\\
\NH{p}{X} &:= \E \biggl[ \left(\int_t^T \vert X_s \vert^2\ud s\right)^{p/2} \biggr]^{1/p},\quad \NS{p}{X} := \E \Bigl[ \sup_{s \in [t,T]}|X_s|^p\Bigr]^{1/p}.
\end{split}
\end{equation}
In particular, we denote by $\cS^{p}([t,T];\R^l)$ the space of continuous and adapted
 random processes from $[t,T]$ to $\R^l$
 with a finite \new{norm}  $\NS{p}{\cdot}$  and by $\cH^{p}([t,T];\R^l)$ the space 
 of progressively-measurable processes from 
 $[t,T]$ to $\R^l$
 with a finite norm $\NH{p}{\cdot}$.
 
In the sequel, the generic letter $C$ is used for denoting constants the value of which 
may often vary from line to line. Constants
whose precise values have a fundamental role in the analysis will 
be denoted by letters \new{distinct from $C$}.


\def \di{\dot{\i}}
\def \bi{\bar{\i}}
\def \X{\mathbb X}
\def \B{\bar{B}}
\def \V{\mathbb V}
\def \Z{\mathbb Z}
\def \R{\mathbb R}
\def \E{\mathbb E}
\def \N{\mathbb N}
\def \P{\mathbb P}
\def \mfz{\mathfrak z}
\def \mfv{\mathfrak v}
\def \mfx{\mathfrak x}

\providecommand{\U}[1]{\protect\rule{.1in}{.1in}}
\newcommand{\bmf}[1]{\boldsymbol{#1}}

\section{Chain rule -- application to the proof of Theorem \ref{main:thm:uniqueness}}
\label{se chain rule}
In this section, we 
discuss the chain rule used in \eqref{eq:chain:rule:PDE}
and \new{apply it to prove} Theorem \ref{main:thm:uniqueness}. Namely, we provide a chain rule for $(U(\mu_{t}))_{t \geq 0}$ where 
$U$ is an $\R$-valued smooth functional defined on the space 
${\mathcal P}_{2}(\R^d)$  and 
$(\mu_{t})_{t \geq 0}$ is the flow of marginal measures of an $\R^d$-valued
It\^o process $(X_{t})_{t \geq 0}$.  

There are two strategies to expand $(U(\mu_{t}))_{t \geq 0}$.  
The first one consists{, for a given $t >0$,} in 
dividing the interval $[0,t]$ into \new{sub-intervals} of length $h = t/N$, 
for some integer $N \geq 1$, and then in splitting
the difference $U(\mu_{t})-U(\mu_{0})$ accordingly:
\begin{equation*}
U(\mu_{t}) - U(\mu_{0}) = \sum_{i=0}^{N-1} \bigl[ U(\mu_{i h}) - U(\mu_{(i-1)h}) \bigr].
\end{equation*}
\new{The differences $U(\mu_{i h}) - U(\mu_{(i-1)h})$ are expanded} by applying 
Taylor's formula at order 2. Since the \new{order of the} remaining terms in the Taylor expansion are expected to be
smaller than the step size $h$, we can derive the chain rule by 
letting $h$ tend to $0$. 
This strategy fits the original proof of It\^o's differential calculus 
and is presented in details in \cite[Section 6]{buck:li:rain:peng}
and in \cite[Section 6]{carmona:delarue:lyons}. 

\new{An alternative}  strategy consists in approximating the dynamics differently. Instead of 
discretizing in time as in the previous \new{strategy}, it is conceivable to reduce the 
space dimension by approximating the flow $(\mu_{t})_{t \geq 0}$  \new{with} the flow of 
empirical measures 
\begin{equation*}
\biggl( \bar{\mu}_{t}^N = \frac{1}{N} \sum_{\ell=1}^N \delta_{X_{t}^\ell} \biggr)_{t \geq 0},  \ \  N\ge 1,
\end{equation*}
where $(X^1_{t})_{t \geq 0},\dots, (X^N_{t})_{t \geq 0}$ stand for $N$ independent copies of $(X_{t})_{t \geq 0}$. Letting 
\begin{equation}
\label{eq:23:10:4}
\forall x^1,\dots,x^N \in \R^d, \quad u^N(x^1,\dots,x^N) 
= U \biggl( \frac{1}{N} \sum_{\ell=1}^N \delta_{x^\ell} \biggr),
\end{equation}
\new{we} expand $(u^N(X_{t}^1,\dots,X_{t}^N))_{t \geq 0}$ by standard It\^o's formula. 
Letting $N$ tend to the infinity, we then expect to recover the same chain rule as the one obtained by the first method. 
Here $u^N$ \new{is interpreted as a}  finite dimensional projection of $U$.

The first strategy  mimics the proof of the 
standard chain rule. The second one gives \new{an insight into the significance} 
of the differential 
calculus on the space of probability measures
introduced by Lions in \cite{cardaliaguet}. Both strategies require some 
smoothness conditions on $U$: Clearly, $U$ must be twice differentiable in some 
suitable sense. \new{From this viewpoint,} the strategy by 
particle approximation  \new{is} advantageous: Taking benefit of the finite 
dimensional framework,  \new{by using a }  standard 
mollification argument \new{it works under weaker} 
smoothness conditions required on the coefficients. 
In particular, 
differently from \cite{buck:li:rain:peng,carmona:delarue:lyons},
we do not require the existence of $\partial^2_{\mu\mu} U$ to prove
the chain rule, see
Theorem \ref{thm:partial:C2}.

\subsection{Full ${\mathcal C}^2$ regularity} 
\label{subse:ito:full}
We first remind the reader of the \new{notion of} 
lifted version of $U$. On $L^2(\Omega,\cA,\P;\R^d)$
(the $\sigma$-field $\cA$ being prescribed), we let 
\begin{equation*}
{\mathcal U}(X) = U\bigl( \law{X} \bigr), \quad X \in L^2(\Omega,\cA,\P;\R^d). 
\end{equation*}
Instead of $(\Omega,\cA,\P)$, we could use $(\hat{\Omega},\hat{\cA},\hat{\P})$, but since
no confusion is possible with the ``physical'' random variables that appear in 
\eqref{eq X-Y t,xi} and \eqref{eq X-Y t,x,mu}, we \new{continue}  to work on 
$(\Omega,\cA,\P)$. 

\new{
Following Lions' approach (see \cite[Section 6]{cardaliaguet}), the mapping $U$ is said to be differentiable on the 
Wasserstein space if the lift ${\mathcal U}$ is differentiable in the sense of Fr\'echet on $L^2(\Omega,\cA,\P)$. 
By Riesz' theorem, the Fr\'echet derivative 
$D {\mathcal U}(X)$, seen} 
as an element of $L^2(\Omega,\cA,\P;\R^d)$, can be represented as 
\begin{equation*}
D {\mathcal U}(X) = \partial_{\mu} U \bigl( \law{X} \bigr)(X),
\end{equation*}
where $\partial_{\mu} U(\law{X}) : \R^d \ni v \mapsto \partial_{\mu} U(\law{X})(v) \in \R^d$ is in $L^2(\R^d,\mu;\R^d)${, see 
\cite[Section 6]{cardaliaguet}.}
Recall that, as a gradient, $\partial_{\mu} U(\law{X})(v)$ will be seen as a row vector. 

A natural question to investigate is the joint regularity of the function $\partial_{\mu} U$ with respect to the variables
$\mu$ and $v$. This requires a preliminary analysis for choosing a `canonical version' of the mapping 
$\partial_{\mu} U(\mu) : \R^d \ni v \mapsto \partial_{\mu} U(\mu)(v) \in \R^d$, which is \textit{a priori} defined just as an element of 
$L^2(\R^d,\mu;\R^d)$. In this perspective, a reasonable strategy consists in choosing a continuous version of the derivative 
if such a version exists. 
For instance, whenever $D {\mathcal U}$ is Lipschitz continuous, we know from \cite{carmona:delarue:aop} that, for any $\mu \in {\mathcal P}_{2}(\R^d)$, there exists a Lipschitz continuous version of the function 
$\R^d \ni v \mapsto \partial_{\mu} U(\mu)(v)$, with 
a Lipschitz constant independent of $\mu$. This result is made precise in Proposition \ref{prop:lipschitz:lifted} below.

{The choice of a continuous version is especially meaningful when the support of $\mu$ is the entire 
$\R^d$, 
in which case the continuous version is uniquely defined. 
Whenever the support of $\mu$ is strictly 
included in $\R^d$, 
some precaution is however needed, as 
the continuous version may be arbitrarily defined outside the support 
of $\mu$. 
To circumvent this difficulty, 
one might be tempted to look for a version of 
$\partial_{\mu} U(\mu) : \R^d \ni v \mapsto \partial_{\mu} U(\mu)(v) \in \R^d$,
for each $\mu \in {\mathcal P}_{2}(\R^d)$, such that 
the global mapping
\begin{equation}
\label{eq:ito:def:partialU:liminf:2}
{\mathcal P}_{2}(\R^d) \times \R^d \ni
(\mu,v) \mapsto \partial_{\mu} U(\mu)(v) \in \R^d
\end{equation}
is continuous. Noticing, by means of a convolution argument, that the set
$\{ \mu' \in {\mathcal P}_{2}(\R^d) : \textrm{\rm supp}(\mu')
=\R^d\}$
is dense
in 
${\mathcal P}_{2}(\R^d)$, this would indeed permit 
to uniquely determine the value
of 
$\partial_{\mu} U(\mu)(v)$
for $v$
outside the support of $\mu$ (when it is strictly included in $\R^d$).} 

{
Unfortunately, in the practical cases handled below, 
the best we can do is to find a version 
of 
$\partial_{\mu} U(\mu) : \R^d \ni v \mapsto \partial_{\mu} U(\mu)(v) \in \R^d$,
for each $\mu \in {\mathcal P}_{2}(\R^d)$, such that 
the global mapping
\eqref{eq:ito:def:partialU:liminf:2}
is 
continuous at the points $(\mu,v)$ such that $v \in \textrm{\rm supp}(\mu)$.  
}

{The fact that 
$\partial_{\mu} U(\mu) : \R^d \ni v \mapsto \partial_{\mu} U(\mu)(v) \in \R^d$ is not uniquely determined outside the
support of $\mu$ is not a problem 
for investigating the differentiability of $\partial_{\mu} U(\mu)$ in $v$. 
It is an issue only 
for investigating the differentiability in $\mu$.} 
We thus say that the chosen version 
of $\R^d \ni v \mapsto \partial_{\mu} U(\mu)(v)$ is differentiable in 
$v$, 
{
for a given $\mu \in {\mathcal P}_{2}(\R^d)$,} 
if the mapping $\R^d \ni v \mapsto \partial_{\mu} U(\mu)(v)$ is differentiable in the standard sense, the derivative being denoted by $\R^d \ni v \mapsto \partial_{v} [\partial_{\mu} U(\mu)](v)$
(which belongs to $\R^{d \times d}$). Note that
{\it there is no Schwarz' theorem for exchanging the derivatives as $U$} does 
not depend on $v$.

{Now, if we can find a jointly continuous version of the global mapping 
\eqref{eq:ito:def:partialU:liminf:2}, 
$\partial_{\mu} U$ is said to be differentiable in $\mu$,
\textcolor{black}{at $v \in \R^d$}, if the lifted 
mapping $L^2(\Omega,\cA,\P;\R^d) \ni X \mapsto 
\partial_{\mu} U(\law{X})(v) \in \R^d$ is differentiable in the Fr\'echet 
sense. Then, according to the previous discussion, the derivative 
can be interpreted as a mapping $\R^d \ni v' \mapsto \partial_{\mu} 
[\partial_{\mu} U(\law{X})(v)](v') \in \R^{d \times d}$ in $L^2(\R^d,\mu;\R^{d 
\times d})$, which we will denote
by $\R^d \ni v' \mapsto \partial_{\mu}^2 U(\law{X})(v,v')$.
In a first step, we will prove  
It\^o's formula
when this additional assumption on the smoothness 
of $\partial_{\mu} U$
in $\mu$ is in force.
More precisely, we will} say that $U$ is \emph{fully} ${\mathcal C}^2$ 
if the global mapping 
$\partial_{\mu} U$ in
 \eqref{eq:ito:def:partialU:liminf:2}
 is continuous and the mappings 
\begin{equation*}
\begin{split}
&{\mathcal P}_{2}(\R^d) \times \R^d \ni (\mu,v) \mapsto \partial_{\mu} U(\mu)(v),\\
&{\mathcal P}_{2}(\R^d) \times \R^d \ni (\mu,v) \mapsto \partial_{v} [ \partial_{\mu} U(\mu)](v),\\ 
&{\mathcal P}_{2}(\R^d) \times \R^{d} \times \R^d \ni (\mu,v,v') \mapsto \partial_{\mu}^2 U(\mu)(v,v'),
\end{split}
\end{equation*}
 are continuous for the product topologies, the space ${\mathcal P}_{2}(\R^d)$ being endowed with the 2-Wasserstein distance. 
 
Under suitable assumptions, it can be checked that 
\emph{full ${\mathcal C}^2$} regularity implies twice Fr\'echet differentiability of the 
lifting ${\mathcal U}$. As we won't make use of such a result, we refrain from 
providing its proof in the paper. 
We will be 
  much more interested in a possible converse: Can we expect to recover that 
  $U$ is ${\mathcal C}^2$ regular (with respect to $v$ and
$\mu$), given the fact that ${\mathcal U}$ has some Fr\'echet or G\^ateaux differentiability properties at the second-order? We answer this (more challenging) question
in Subsection \ref{sec:ito:frechet} below. 

To clarify the significance of the notion of \emph{full ${\mathcal C}^2$} regularity, 
we now make the connection between the derivatives of $u^N$ and those of $U$:

\begin{Proposition}
\label{prop:derivatives}
Assume that $U$ is ${\mathcal C}^1$. Then, for any $N \geq 1$, the function $u^N$ is 
differentiable on 
$\R^N$ and, for all $x^1,\dots,x^N \in \R^d$, the mapping 
\begin{eqnarray*}
\R^d\ni x^i&\mapsto&\partial_{x^{i}} u^N(x^1,\dots,x^N)\in\R^d
\end{eqnarray*}
reads
\begin{equation*}
\partial_{x^{i}} u^N(x^1,\dots,x^N) = \frac{1}{N} \partial_{\mu} U
\biggl( \frac{1}{N} \sum_{\ell=1}^N \delta_{x^{\ell}} \biggr)(x^i).
\end{equation*}
If, moreover, $U$ is fully ${\mathcal C}^2$, then, for any $N \geq 1$, the function $u^N$ is ${\mathcal C}^2$ on 
$\R^N$ 
and, for all $x^1,\dots,x^N \in \R^d$,  the mapping 
\begin{eqnarray*}
\R^{d}\times \R^{d}\ni( x^i,x^j)&\mapsto&\partial^2_{x^i x^j} u^N(x^1,\dots,x^N)\in\R^{d\times d}
\end{eqnarray*}
 satisfies 
\begin{equation*}
\begin{split}
&\partial^2_{x^i x^j} u^N(x^1,\dots,x^N) 
= \frac{1}{N} \partial_{v} \biggl[ \partial_{\mu} U\biggl( \frac{1}{N} \sum_{\ell=1}^N \delta_{x^{\ell}} \biggr) \biggr] (x^i) \delta_{i,j}
+ \frac{1}{N^2} \partial_{\mu}^2 U \biggl( \frac{1}{N} \sum_{\ell=1}^N \delta_{x^{\ell}} \biggr)(x^i,x^j). 
\end{split}
\end{equation*}
\end{Proposition}

\begin{proof} The formula for the first order derivative has been already proved in \cite{carmona:delarue:aop}.  
It remains to deduce the formula for the second order derivative. When $i \not =j$, it is a direct 
consequence of the first order formula. When $i=j$, the computations require some precaution as differentiability is
simultaneously investigated 
in the directions of $\mu$ and $v$ in $\partial_{\mu} U(\mu)(v)$, but, by the joint continuity of the second-order derivatives
$\partial_{\mu}^2 U(\mu)(v)$ and $\partial_{v} \partial_{\mu} U(\mu)(v,v')$, they are easily handled.  
\qed
\end{proof}
\begin{Remark}
\label{rem:symmetry:derivatives}
Assume that $U$ is fully ${\mathcal C}^2$. Then,
for any $X \in L^2(\Omega,{\mathcal A},\P;\R^d)$
and
$Y,Z \in L^\infty(\Omega,{\mathcal A},\P;\R^d)$, 
the mapping
\begin{equation*}
\begin{split}
\vartheta : \R^2 \ni (h,k) \mapsto 
U\bigl( \bigl[X+ h Y + k Z \bigr] 
\bigr) \in \R
\end{split}
\end{equation*}
is of class ${\mathcal C}^2$ on $\R^2$, with 
\begin{equation*}
\begin{split}
\frac{\ud}{\ud h}
\bigl[ \frac{\ud \vartheta}{\ud k}
\bigr](h,k)
&= \frac{\ud}{\ud h} {\mathbb E} \bigl[ \partial_{\mu} U\bigl( \bigl[ 
X+hY + kZ \bigr] \bigr)(X+hY+kZ) Z \bigr]
 \\ 
&= {\mathbb E} \bigl[ 
{\rm Tr} \bigl(
\partial_{v}
\partial_{\mu} U\bigl( \bigl[ 
X+hY + kZ \bigr] \bigr)(X+hY + kZ) Z \otimes Y \bigr) \bigr]
\\
&\hspace{15pt} + 
{\mathbb E} \hat{\mathbb E} \bigl[ 
{\rm Tr} \bigl(
\partial_{\mu}^2 U\bigl( \bigl[ 
X +hY + kZ\bigr] \bigr)(X+hY + kZ,\hat X+h \hat Y + k \hat Z) Z \otimes \hat{Y} \bigr)\bigr],
\end{split}
\end{equation*}
the triplet $(\hat{X},\hat{Y},\hat{Z})$ denoting 
a copy
of $(X,Y,Z)$ on $(\hat{\Omega},\hat{\mathcal A},\hat{\P})$
and the tensorial product operating on $\R^d$.

By Schwarz' Theorem, the roles of $Z$ and $Y$ can be exchanged, 
which means (choosing $h=k=0$) that 
\begin{equation}
\label{eq:ito:symmetry:2}
\begin{split}
&{\mathbb E} \bigl[ {\rm Tr} \bigl(\partial_{v}
\partial_{\mu} U\bigl( [ 
X ] \bigr)(X) Z \otimes Y \bigr) \bigr]
+ 
{\mathbb E} \hat{\mathbb E} \bigl[
{\rm Tr} \bigl( 
\partial_{\mu}^2 U\bigl( [ 
X ] \bigr)(X,\hat X) Z \otimes \hat{Y} \bigr)\bigr]
\\
&\hspace{15pt}
={\mathbb E} \bigl[ {\rm Tr} \bigl( \partial_{v}
\partial_{\mu} U\bigl([ 
X] \bigr)(X) Y \otimes Z \bigr) \bigr]
+ 
{\mathbb E} \hat{\mathbb E} \bigl[ 
{\rm Tr} \bigl(
\partial_{\mu}^2 U\bigl([ 
X] \bigr)(X,\hat X) Y \otimes \hat{Z} \bigr) \bigr].
\end{split}
\end{equation}
Choosing $Y$ of the form $\varepsilon \varphi(X)$
and $Z$ of the form $\varepsilon \psi(X)$, with $\P(\varepsilon=1)=
 \P(\varepsilon=-1)=1/2$ and $\varepsilon$ independent of 
 $X$,
 and considering two bounded 
 Borel measurable functions 
 $\varphi$ and $\psi : \R^d \rightarrow \R^d$, we deduce that 
\begin{equation}
\label{eq:ito:symmetry:1}
\begin{split}
&{\mathbb E} \bigl[
{\rm Tr} \bigl(
 \partial_{v}
\partial_{\mu} U\bigl( [ 
X] \bigr)(X) \varphi(X) \otimes \psi(X) \bigr) \bigr]
={\mathbb E} \bigl[
{\rm Tr} \bigl( \partial_{v}
\partial_{\mu} U\bigl([ 
X ] \bigr)(X) \psi(X) \otimes \varphi(X) \bigr) \bigr],
\end{split}
\end{equation}
from which we deduce that 
 $\partial_{v}
\partial_{\mu} U([ 
X])(X)$ takes values in the set of symmetric 
matrices of size $d$. By continuity, it means that 
 $\partial_{v} \partial_{\mu} U( \mu)(v)$
 is a symmetric matrix for any $v \in \R^d$ when 
 $\mu$ has the entire $\R^d$ as support. 
 By continuity in $\mu$, we deduce that 
 $\partial_{v} \partial_{\mu} U( \mu)(v)$
 is a symmetric matrix for any $v \in \R^d$ and 
 any $\mu \in {\mathcal P}_{2}(\R^d)$. 

Now, choosing $Y$ and $Z$ of the form $\varphi(X)$
and $\psi(X)$ respectively and plugging 
\eqref{eq:ito:symmetry:1} into 
\eqref{eq:ito:symmetry:2}, 
we deduce that 
\begin{equation*}
\begin{split}
{\mathbb E} \hat{\mathbb E} \bigl[
{\rm Tr} \bigl( 
\partial_{\mu}^2 U\bigl( [ 
X ] \bigr)(X,\hat X) \varphi(X) \otimes \psi(\hat{X}) \bigr) \bigr]
&=
{\mathbb E} \hat{\mathbb E} \bigl[ 
{\rm Tr} \bigl(
\partial_{\mu}^2 U\bigl( [ 
X ] \bigr)(X,\hat X) \psi(X) \otimes \varphi(\hat{X}) \bigr) \bigr]
\\
&=
{\mathbb E} \hat{\mathbb E} \bigl[ 
{\rm Tr} \bigl(
\partial_{\mu}^2 U\bigl( [ 
X ] \bigr)(\hat X,X) \psi(\hat X) \otimes \varphi(X) \bigr) \bigr]
\\
&=
{\mathbb E} \hat{\mathbb E} \bigl[ 
{\rm Tr} \bigl( \bigl[
\partial_{\mu}^2 U\bigl( [ 
X ] \bigr)(\hat X,X) \bigr]^{\dagger}  \varphi(X) \otimes 
\psi(\hat X)  \bigr) \bigr],
\end{split}
\end{equation*}
from which we deduce that
 $\partial^2_{\mu} U([ 
X])(X,\hat{X}) = 
(\partial^2_{\mu} U([ 
X])(\hat{X},X))^\dagger
$. By the same argument as above, 
we finally deduce that 
 $\partial_{\mu}^2 U( \mu)(v,v')
  =
 (\partial_{\mu}^2 U( \mu)(v',v))^\dagger 
 $, for any $v,v' \in \R^d$ and 
 any $\mu \in {\mathcal P}_{2}(\R^d)$. 
\end{Remark}
\color{black}

\subsection{The chain rule for $U$ fully ${\mathcal C}^2$}
We consider an $\R^d$-valued It\^o process
\begin{equation*}
\ud X_{t} = b_{t} \ud t + \sigma_{t} \ud W_{t}, \quad X_{0} \in L^2(\Omega,\cA,\P),
\end{equation*}
where $(b_{t})_{t \geq 0}$ and $(\sigma_{t})_{t \geq 0}$ are progressively-measurable processes 
with values in $\R^d$ and, respectively, $\R^{d \times d}$ respectively
with respect to the (augmented) filtration generated by $W$, such that 
\begin{equation}
\label{eq:23:10:2}
\forall T >0, \quad \E \biggl[ \int_{0}^T \bigl( \vert b_{t} \vert^2 + \vert \sigma_{t} \vert^4 \bigr) \ud t \biggr] < + \infty.
\end{equation}
The following is the main result of this section
\begin{Theorem}
\label{thm:full}
Assume that $U$ is fully ${\mathcal C}^2$ and that, for any compact subset 
${\mathcal K} \subset {\mathcal P}_{2}(\R^d)$, 
\begin{equation}
\label{eq:compact:bound}
\sup_{\mu \in {\mathcal K}}
\biggl[ 
\int_{\R^d} \bigl\vert  \partial_{\mu} U(\mu)(v) \bigr\vert^2 \ud \mu(v) 
+
\int_{\R^d} \Bigl\vert \partial_{v} \bigl[ \partial_{\mu} U(\mu) \bigr](v) \Bigr\vert^2 \ud \mu(v) 
\biggr]
< + \infty, 
\end{equation}
Then, letting $\mu_{t} := [X_{t}]$ and $a_{t} := \sigma_{t} (\sigma_{t})^{\dagger}$, for any $t \geq 0$,
\begin{equation}
\label{eq:thm:chain:rule}
\begin{split}
U (\mu_{t}) &= U  ( \mu_{0}) + 
\int_{0}^t {\mathbb E} \bigl[ \partial_{\mu} U  ( \mu_{s} )(X_{s}) b_{s}  \bigr] \ud s 
+ 
\frac{1}{2} \int_{0}^t {\mathbb E} \bigl[ {\rm Tr} \bigl( 
\partial_{v} \bigl( \partial_{\mu} U(\mu_{s}) \bigr)(X_{s})
a_{s}\bigr) 
\bigr] \ud s. 
\end{split}
\end{equation}
\end{Theorem}
\noindent The proof relies on a mollification argument captured in the proof of the following result:
\begin{Proposition}
\label{prop:mollif}
Assume that the chain rule \eqref{eq:thm:chain:rule} holds for any function $U$ that is fully ${\mathcal C}^2$ with 
first and second order derivatives that are bounded and uniformly continuous (with respect to both the space and measure variables). Then Theorem \ref{thm:full} holds, in other words, the chain rule \eqref{eq:thm:chain:rule} is valid for any fully ${\mathcal C}^2$ function $U$ satisfying \eqref{eq:compact:bound}. 
\end{Proposition}  

\begin{proof}[Proof of Proposition \ref{prop:mollif}.]
Let $U$ be a \emph{fully} ${\mathcal C}^2$ function that satisfies \eqref{eq:compact:bound}. 
We `mollify' $U$ in such a way that its mollification is bounded with bounded 
first and second order derivatives. Let  $\varphi : \R^d \rightarrow \R^d$ be a
smooth function with compact support and, for arbitrary $\mu \in {\mathcal P}_{2}(\R^d)$ define 
\begin{equation*}
\forall \mu \in {\mathcal P}_{2}(\R^d), \quad \bigl( U \star \varphi \bigr)(\mu) := U \bigl( \varphi \sharp \mu
\bigr),
\end{equation*}
where $\varphi \sharp \mu$ denotes the image of $\mu$ by $\varphi$.
The lifted version of $U \star \varphi$ is nothing but 
${\mathcal U} \circ \varphi$, where (with an abuse of notation) $\varphi$ is canonically lifted as $\varphi : 
L^2(\Omega,\cA,\P;\R^d) \ni X \mapsto \varphi(X)$. It is then quite standard to check that: 
\begin{equation}
\label{eq:composition}
\begin{split}
&\partial_{\mu} \bigl[ U \star \varphi \bigr](\mu)(v) =  
{
\Bigl(
\Sum_{k=1}^d
\Bigl[ \partial_{\mu} U 
\bigl(\varphi \sharp \mu \bigr) \bigl( \varphi(v) \bigr) \Bigr]_{k} \frac{\partial \varphi_{k}}{\partial x_{i}}(v) \Bigr)_{i=1,\dots,d}},  
\\
&\partial_{\mu}^2  \bigl[ U \star \varphi \bigr](\mu)(v,{v'}) =  
{
\Bigl( 
\sum_{k,\ell=1}^d
\Bigl[ \partial_{\mu}^2 U 
\bigl(\varphi \sharp \mu \bigr) \bigl( \varphi(v),\varphi(v') \bigr) 
\Bigr]_{k,\ell}
\frac{\partial \varphi_{k}}{\partial x_{i}}(v)
\frac{\partial \varphi_{\ell}}{\partial x_{j}}(v')
\Bigr)_{i,j=1,\dots,d}},   
\\
&\partial_{v} \Bigl[ \partial_{\mu} \bigl[ U \star \varphi \bigr](\mu)(v)
\Bigr]
= 
{ \Bigl( \sum_{k=1}^d
\Bigl[ 
\partial_{\mu} U 
\bigl(\varphi \sharp \mu \bigr) \bigl( \varphi(v) \bigr) 
\Bigr]_{k} \frac{\partial^2 \varphi_{k}}{\partial x_{i} \partial x_{j}}(v)}
\\
&\hspace{110pt} 
{+ 
\sum_{k,\ell=1}^d
\Bigl[ \partial_{v} \bigl[ \partial_{\mu} U 
\bigl(\varphi \sharp \mu \bigr) \bigr] 
\bigl( \varphi(v) \bigr)
\Bigr]_{k,\ell}
\frac{\partial \varphi_{k}}{\partial x_{i}}(v) 
\frac{\partial \varphi_{\ell}}{\partial x_{j}}(v) 
\Bigr)_{i,j=1,\dots,d}}.
\end{split}
\end{equation}
Recall from 
Remark 
\ref{rem:symmetry:derivatives} that the second-order derivatives that 
appear in \eqref{eq:composition} have some symmetric structure. 
Now, since $\varphi$ is compactly supported, the mapping ${\mathcal P}_{2}(\R^d) \ni \mu \mapsto 
\varphi \sharp \mu$ has a relatively compact range\footnote{Tightness is obvious. 
By boundedness of $\varphi$, any subsequence converging in the weak sense is also convergent 
with respect to $W_{2}$.} (in ${\mathcal P}_{2}(\R^d)$). By the continuity of $U$ and its derivatives, we deduce that 
$U \star \varphi$ and its first and second order derivatives are bounded and uniformly continuous on the whole space. 

Assume now that the chain rule has been proved for any bounded and uniformly continuous $U$ with bounded and uniformly continuous derivatives of order 1 and 2. Then, for some $U$ just satisfying the assumption of Theorem 
\ref{thm:full}, we can apply the chain rule to $U \star \varphi$, for any $\varphi$ as above. In particular, we can apply 
the chain rule to $U \star \varphi_{n}$ for any $n \geq 1$, where $(\varphi_{n})_{n \geq 1}$ is a sequence of 
compactly supported smooth functions such that 
$(\varphi_{n},\partial_{x} \varphi_{n},\partial^2_{xx}[\varphi_{n}]_1,\dots,\partial^2_{xx}[\varphi_{n}]_d)(v) \rightarrow (v,{I_{d}},0,\dots,0)$ 
uniformly on compact sets as  $n \rightarrow  \infty$, 
{
$I_{d}$ denoting the identity matrix of 
size $d$.} 
In order to pass to the limit in the chain rule \eqref{eq:thm:chain:rule}, the only thing is to verify some almost sure (or pointwise) convergence in the underlying expectations and to check the corresponding uniform integrability argument. 

Without any loss of generality, we can assume that there exists a constant $C$ such that 
\new{
\begin{align} \label{eq bound varphin}
\vert \varphi_{n}(v) \vert \leq C \vert v \vert,\quad \vert \partial_{x} \varphi_{n}(v) \vert \leq C \quad
\text{and} \quad \vert \partial_{xx}^2 [\varphi_{n}(v)]_k \vert \leq C\;, \; 1 \le k \le d\,, 
\end{align}
}
for any $n \geq 1$ and $v \in \R^d$ and that $\varphi_{n}(v)=v$ 
for any $n \geq 1$ and any $v$ with $\vert v \vert \leq n$. Then, for any $\mu \in {\mathcal P}_{2}(\R^d)$ and any random variable $X$ with $\mu$ as distribution,
it holds
\begin{equation*}
W_{2}^2 \bigl( \varphi_{n} \sharp \mu,\mu \bigr) \leq {\mathbb E} \bigl[ \vert \varphi_{n}(X) - X \vert^2 {\mathbf 1}_{\{\vert X \vert \geq n\}} \bigr]
\leq C {\mathbb E} \bigl[ \vert X \vert^2 {\mathbf 1}_{\{\vert X \vert \geq n\}} \bigr],
\end{equation*}
which tends to $0$ as $n \rightarrow \infty$. By continuity of $U$ and its partial derivatives and by 
\eqref{eq:composition}, it is easy to deduce that, a.s.,  
\begin{equation}
\label{eq:24:10:2}
\begin{split}
&U \star \varphi_{n}(\mu) \rightarrow U (\mu),
\quad \partial_{\mu} \bigl[ U \star \varphi_{n}\bigr](\mu)(X) \rightarrow 
\partial_{\mu} U (\mu )(X),
\\
&\partial_{v} \bigl[ \partial_{\mu} \bigl( U \star \varphi_{n}\bigr) \bigr](\mu)(X) \rightarrow 
\partial_{v} \bigl[ \partial_{\mu} U (\mu)\bigr](X).
\end{split}
\end{equation}
Moreover, 
we notice that
\begin{equation}
\label{eq:24:10:1}
\sup_{n \geq 1} {\mathbb E} \Bigl[ 
\bigl\vert \partial_{\mu} \bigl[ U \star \varphi_{n}\bigr](\mu )(X)
\bigr\vert^2 + 
\bigl\vert \partial_{v} \bigl[ \partial_{\mu} \bigl( U \star \varphi_{n}\bigr) (\mu ) \bigr](X)
\bigr\vert^2 \Bigr] < \infty.
\end{equation}
Indeed, by \eqref{eq:composition} \new{ and \eqref{eq bound varphin}}, it is enough to check that 
\begin{equation*}
\begin{split}
&\sup_{n \geq 1} \biggl[ \int_{\R^d} 
\Bigl\vert \partial_{\mu} U \bigl( \varphi_{n} \sharp \mu \bigr)(v) \Bigr\vert^2 \ud 
\bigl( \varphi_{n} \sharp
\mu \bigr)(v)
+ \int_{\R^d} \Bigl\vert \partial_{v} \bigl[ \partial_{\mu}  U \bigl(  \varphi_{n} \sharp \mu \bigr)
\bigr](v) \Bigr\vert^2 
\ud \bigl( \varphi_{n} \sharp \mu \bigr)(v)
\biggr] 
 < \infty,
\end{split}
\end{equation*}
which follows directly from \eqref{eq:compact:bound}, noticing that the 
sequence $(\varphi_{n} \sharp \mu)_{n \geq 1}$ lives in a compact subset 
of ${\mathcal P}_{2}(\R^d)$ as it is convergent.

%
By \eqref{eq:24:10:2} and \eqref{eq:24:10:1} and by a standard uniform integrability argument, we deduce that, for any 
$t \geq 0$ and any $s \in [0,t]$ such that ${\mathbb E}[
\vert b_{s} \vert^2 + \vert \sigma_{s} \vert^4] < \infty$,
\begin{equation*}
\begin{split}
&\lim_{n \rightarrow + \infty} 
 {\mathbb E} \bigl[  \partial_{\mu} (U \star \varphi_{n})  ( \law{X} )(X) b_{s} \bigr] 
 = {\mathbb E} \bigl[  \partial_{\mu} U  ( \law{X} )(X) b_{s}  \bigr],
\\
&\lim_{n \rightarrow + \infty} 
{\mathbb E} \bigl\{ \text{Tr} \bigl[ 
 \partial_{v} \bigl( \partial_{\mu} (U \star \varphi_{n})
(\law{X}) \bigr)(X) a_{s}\bigr] \bigr\}  = 
{\mathbb E} \bigl\{ \text{Tr} \bigl[
  \partial_{v} \bigl( \partial_{\mu} U
(\law{X}) \bigr)(X) a_{s}
\bigr] \bigr\}.
\end{split}
\end{equation*}
Recall that the above is true for any $\mu \in {\mathcal P}_{2}(\R^d)$
and any $X \in L^2(\Omega,{\mathcal A},\P;\R^d)$ with $\mu$ as distribution. 
In particular, we can choose $\mu=\mu_{s}$ and $X=X_{s}$ in the above limits.
As the bound ${\mathbb E}[ \vert b_{s} \vert^2 + 
\vert \sigma_{s} \vert^4] < \infty$ is satisfied for almost every $s \in [0,t]$, this permits 
to pass to the limit inside the integrals appearing in the chain rule applied to each of the $(U \star \varphi_{n})_{n \geq 1}$.  
In order to pass to the limit in the chain rule itself, we must exchange the pathwise limit that holds for almost every 
$s \in [0,t]$ and the integral with respect to  the time variable $s$. The argument is the same as in \eqref{eq:24:10:1}. 
Indeed, since the flow of measures 
$(\law{X_{s}})_{0 \leq s \leq t}$ is continuous for the $2$-Wasserstein distance,
{the family of measures 
$((\law{\varphi_{n}(X_{s})})_{0 \leq s \leq t})_{n \geq 1}$
is relatively compact and thus
\begin{equation*}
\sup_{n \geq 1}
\sup_{s \in [0,t]} {\mathbb E} \Bigl[ 
\bigl\vert \partial_{\mu}  U \bigl(\law{\varphi_{n}(X_{s})}\bigr)\bigl(\varphi_{n}(X_{s})
\bigr)
\bigr\vert^2 + 
\bigl\vert \partial_{v} \bigl[ \partial_{\mu} U \bigl(\law{\varphi_{n}(X_{s})} \bigr) \bigr]\bigl(\varphi_{n}(X_{s})
\bigr)
\bigr\vert^2 \Bigr] < \infty,
\end{equation*}
which is enough to prove that the functions
\begin{equation*}
\Bigl( 
[0,t] \ni s \mapsto {\mathbb E} \bigl[  \partial_{\mu} (U \star \varphi_{n})  ( \law{X_s} )(X_s) b_{s} \bigr] +
{\mathbb E} \bigl\{ \text{Tr} \bigl[ 
 \partial_{v} \bigl( \partial_{\mu} (U \star \varphi_{n})
(\law{X_s}) \bigr)(X_s) a_{s}\bigr] \bigr\}
\Bigr)_{n \geq 1}
\end{equation*}
are uniformly integrable on $[0,t]$.} \qed
\end{proof}
\vspace{5pt}

We now turn to the proof of Theorem \ref{thm:full}. We give just a sketch of the proof, as a refined
version of Theorem \ref{thm:full} is given later, see Theorem \ref{thm:partial:C2} in the next subsection. 
\vspace{5pt}

\begin{proof}[Proof of Theorem \ref{thm:full}.]
By Proposition \ref{prop:mollif}, 
we can replace $U$ by $U \star \varphi$, for some compactly supported smooth function $\varphi$. Equivalently, we can replace $(X_{t})_{t \geq 0}$ by $(\varphi(X_{t}))_{t \geq 0}$. 
In other words, we can assume that $U$ and its first and second order derivatives are bounded and uniformly 
continuous and that 
$(X_{t})_{t \geq 0}$ is a bounded It\^o process. 

Finally by the same argument as  in the proof of 
Proposition
\ref{prop:mollif}, we can also assume that $(b_{t})_{t \geq 0}$ and $(\sigma_{t})_{t \geq 0}$ are bounded. 
Indeed, it suffices to prove the chain rule when 
$(X_{t})_{t \geq 0}$ is driven by truncated processes and then to pass to the limit along a sequence of truncations that converges to $(X_{t})_{t \geq 0}$. 

Let  $((X^\ell_{t})_{t \geq 0})_{\ell \geq 1}$ a sequence of i.i.d. copies of 
$(X_{t})_{t \geq 0}$. That is, for any $\ell \geq 1$,
\begin{equation*}
\ud X^\ell_{t} = b^\ell_{t} \ud t + \sigma^\ell_{t} \ud W^\ell_{t}, \quad t \geq 0,
\end{equation*}
where $((b^\ell_{t},\sigma^\ell_{t},W^\ell_{t})_{t \geq 0},X_{0}^\ell)_{\ell \geq 1}$ are i.i.d copies of $((b_{t},\sigma_{t},W_{t})_{t \geq 0},X_{0})$. 

Recalling the definition of the flow of marginal empirical measures:
\begin{equation*}
\bar\mu^{N}_{t} = \frac{1}{N} \sum_{\ell=1}^N \delta_{X^\ell_{t}},
\end{equation*} 
the standard It\^o's formula yields together with Proposition \ref{prop:derivatives}, $\P$-a.s., for any $t \geq 0$
\begin{align}
&u^N\bigl({X}^1_{t},\dots,{X}^N_{t}\bigr) = 
u^N\bigl({X}^1_{0},\dots,{X}^N_{0}\bigr) \nonumber
\\
&\hspace{5pt} + 
\frac{1}{N} \sum_{\ell=1}^N \int_{0}^t 
\partial_{\mu} U\bigl(\bar{\mu}_{s}^N \bigr)({X}^\ell_{s}) b_{s}^\ell \ud s 
 +  \frac{1}{N} \sum_{\ell=1}^N \int_{0}^t 
\partial_{\mu} U\bigl(\bar{\mu}_{s}^N \bigr)({X}^\ell_{s}) \sigma_{s}^\ell \ud 
W_{s}^\ell  \label{eq:chain:rule:discrete}
\\
&\hspace{5pt} +
\frac{1}{2N} \sum_{\ell=1}^N \int_{0}^t 
\text{Tr} \bigl\{
\partial_{v} \bigl[ \partial_{\mu} U\bigl(\bar{\mu}_{s}^N \bigr) \bigr] 
({X}^\ell_{s}) a_{s}^\ell \bigr\} \ud s 
+ 
\frac{1}{2N^2} \sum_{\ell=1}^N \int_{0}^t
\text{Tr} \bigl\{
 \partial_{\mu}^2 U\bigl(\bar{\mu}_{s}^N \bigr)({X}^\ell_{s},{X}^\ell_{s}) 
 a_{s}^\ell \bigr\} \ud s, \nonumber
\end{align}
with $a_{s}^\ell:= \sigma_{s}^\ell (\sigma_{s}^\ell)^\dagger$.

We take expectation on both sides of the previous equality and obtain (the stochastic integral has zero expectation
due to the boundedness of the coefficients), recalling 
\eqref{eq:23:10:4}, 
\begin{align*}
\esp{U\bigl(\bar{\mu}_{t}^N \bigr)} &= 
\esp{U\bigl(\bar{\mu}_{0}^N \bigr)} + 
\frac{1}{N} \sum_{\ell=1}^N \esp{\int_{0}^t 
\partial_{\mu} U\bigl(\bar{\mu}_{s}^N \bigr)({X}^\ell_{s}) b_{s}^\ell \ud s }
\\
&  \hspace{15pt}+
\frac{1}{2N} \sum_{\ell=1}^N \esp{\int_{0}^t 
\text{Tr} \bigl\{
\partial_{v} \bigl[ \partial_{\mu} U\bigl(\bar{\mu}_{s}^N \bigr) \bigr] 
({X}^\ell_{s}) a_{s}^\ell \bigr\} \ud s} 
\\
& \hspace{15pt}
+ \frac{1}{2N^2} \sum_{\ell=1}^N \esp{\int_{0}^t
\text{Tr} \bigl\{
 \partial_{\mu}^2 U\bigl(\bar{\mu}_{s}^N \bigr)({X}^\ell_{s},{X}^\ell_{s}) 
 a_{s}^\ell \bigr\} \ud s}. 
\end{align*}
All the above expectations are finite, due to the boundedness of the coefficients. 
Using the fact that the processes $((a^\ell_{s},b^\ell_{s},X^\ell_{s})_{s \in [0,t]})_{\ell \in \{1,\dots,N\}}$ are i.i.d., we deduce that 
\begin{align}
\esp{U\bigl(\bar{\mu}_{t}^N \bigr)} &= 
\esp{U\bigl(\bar{\mu}_{0}^N \bigr)} + 
 \esp{\int_{0}^t 
\partial_{\mu} U\bigl(\bar{\mu}_{s}^N \bigr)({X}^1_{s}) b_{s}^1 \ud s } \label{eq:esp:term1}
\\
&\hspace{15pt} +
\frac{1}{2}  \esp{\int_{0}^t 
\text{Tr} \bigl\{
\partial_{v} \bigl[ \partial_{\mu} U\bigl(\bar{\mu}_{s}^N \bigr) \bigr] 
({X}^1_{s}) a_{s}^1 \bigr\} \ud s} \label{eq:esp:term2}
\\
&\hspace{15pt} + 
\frac{1}{2N} \esp{\int_{0}^t
\text{Tr} \bigl\{
 \partial_{\mu}^2 U\bigl(\bar{\mu}_{s}^N \bigr)({X}^1_{s},{X}^1_{s}) 
 a_{s}^1 \bigr\} \ud s}, \label{eq:esp:term3}
\end{align}
In particular, because of the additional $1/N$, the term in \eqref{eq:esp:term3} converges to $0$.
Moreover,
the coefficients $(a_{s})_{s \in [0,t]}$
and $(b_{s})_{s \in [0,t]}$ being bounded, 
we know from \cite[Theorem 10.2.7]{RachevRuschendorf}:
\begin{equation}
\label{eq:09:01:1}
\lim_{N \rightarrow + \infty} \E \bigl[ \sup_{0 \leq s \leq t} W_{2}^2(\bar{\mu}^N_{s},{\mu}_{s}) \bigr] = 0.
\end{equation}
This implies together with the uniform continuity of $U$ with respect to the distance $W_{2}$,  that
$\esp{U\bigl(\bar{\mu}_{t}^N \bigr)}$ (resp. 
$\esp{U\bigl(\bar{\mu}_{0}^N \bigr)}$) converges  to 
$U(\mu_{t})$ (resp. $U(\mu_{0})$). Combining the uniform continuity of $\partial_{\mu} U$ on ${\mathcal P}_{2}(\R^d) \times \R^d$
with \eqref{eq:09:01:1}, the second term in the right-hand side 
of \eqref{eq:esp:term1} converges. Similar arguments lead to the convergence of the term in \eqref{eq:esp:term2}.
\qed
\end{proof}

\vspace{5pt}
The notion of differentiation as defined by Lions plays an essential role 
in the chain rule formula. It is the right differentiation procedure to give the 
natural extension from the chain rule for empirical distribution processes 
to the chain rule for measure valued processes. 

\subsection{The chain rule for $U$ partially ${\mathcal C}^2$}
\label{subse:partial:C2}
We observe that, in the formula for chain rule \eqref{eq:thm:chain:rule}, the second order derivative $\partial^2_{\mu} U$ does not appear. 
It is thus a quite natural question to study its validity when 
$\partial^2_{\mu} U$ does not exist. This is what we refer to as `partial ${\mathcal C}^2$ regularity'.
More precisely, we will say that $U$ is partially ${\mathcal C}^2$ (in $v$) if 
{the lift 
${\mathcal U}$ is Fr\'echet differentiable
and, for any $\mu \in {\mathcal P}_{2}(\R^d)$, 
we can find a continuous version of the mapping 
$\R^d \ni v \mapsto 
\partial_{\mu} U(\mu)(v)$ such that}:

$\bullet$ the mapping
${\mathcal P}_{2}(\R^d) \times \R^d \ni (\mu,v) \mapsto \partial_{\mu} U(\mu)(v)$
{
is jointly continuous at any 
$(\mu,v)$ such that $v \in \textrm{Supp}(\mu)$,}

$\bullet$ for any $\mu \in {\mathcal P}_{2}(\R^d)$, the mapping 
$\R^d \ni v \mapsto \partial_{\mu} U(\mu)(v) \in \R^d$ is continuously differentiable and its derivative
is jointly continuous with respect to $\mu$ and $v$
{at any 
point $(\mu,v)$ such that $v \in \textrm{Supp}(\mu)$, the derivative 
being denoted by $\R^d \ni v \mapsto 
\partial_{v}[\partial_{\mu} U(\mu)](v) \in \R^{d \times d}$.} 

{Recall from the discussion in Subsection 
\ref{subse:ito:full} that, for each $\mu \in \R^d$, 
the mapping 
$\partial_{\mu} U(\mu) : v \mapsto \partial_{\mu} U(\mu)(v)$
is uniquely defined on the support of $\mu$.}
\vspace{5pt}

The following is the chain rule for is partially ${\mathcal C}^2$:
\begin{Theorem}
\label{thm:partial:C2}
Assume that $U$ is partially ${\mathcal C}^2$ and 
that, for any compact subset 
${\mathcal K} \subset {\mathcal P}_{2}(\R^d)$, 
\eqref{eq:compact:bound} holds true. Then, the chain rule holds
for an It\^o process satisfying \eqref{eq:23:10:2}.
\end{Theorem}
{Notice that,
in the chain rule, 
the mapping 
$\partial_{\mu} U : {\mathcal P}_{2}(\R^d)
\times \R^d \ni (\mu,v) \mapsto \partial_{\mu} U(\mu)(v)$
is always evaluated at points $(\mu,v)$ 
such that $v$ belongs
to the support of $\mu$
and thus for which $\partial_{\mu} U(\mu)(v)$ is uniquely defined.}
\vspace{5pt}

\begin{proof} 
\textit{First step.} We start with the same mollification procedure as in the proof 
of Theorem \ref{thm:full}, see 
\eqref{eq:composition}. 

Repeating the computations,  
$U \star \varphi$ and its first and partial second order derivatives are 
bounded. 
Nevertheless, 
contrary to the argument in the proof of
Theorem \ref{thm:full},
we cannot claim here
that
$\partial_{\mu} (U \star \varphi)$ and 
$\partial_{v} [\partial_{\mu} (U \star \varphi)]$
are continuous on the whole space
since
$\partial_{\mu} U$ and 
$\partial_{v} [ \partial_{\mu} U]$ are only continuous at points $(\mu,v)$ such that 
$v$ is in the support of $\mu$. 
In order to circumvent this difficulty, we first
notice, from 
\eqref{eq:composition},
 that 
$\partial_{\mu} (U \star \varphi)$ and 
$\partial_{v} [\partial_{\mu} (U \star \varphi)]$
are also 
continuous at points $(\mu,v)$ such that 
$v$ is in the support of $\mu$, the reason being that 
$v \in \textrm{Supp}(\mu)$ implies $\varphi(v) 
\in \textrm{Supp}(\varphi \sharp \mu)$.
We then
 change 
${\mathcal P}_{2}(\R^d) \ni \mu \mapsto 
(U \star \varphi)(\mu)$ 
into ${\mathcal P}_{2}(\R^d) \ni \mu \mapsto (U \star \varphi)
(\mu \star \rho)$
where $\rho$ is a smooth convolution kernel, with the entire 
$\R^d$ as support and with exponential decay at infinity,
and $\mu \star \rho$ stands for the probability measure with density
given by
\begin{equation*}
\R^d \ni x \mapsto \int_{\R^d} \rho(x-y) \ud\mu(y). 
\end{equation*} 
We then observe that 
\begin{equation*}
\begin{split}
&\partial_{\mu} \bigl[ \bigl( U \star \varphi\bigr) ( \mu \star \rho) \bigr](v )
= \int_{\R^d} \partial_{\mu} \bigl( U \star \varphi\bigr)(\mu \star \rho)(v-v') \rho(v') \ud v', 
\\
&\partial_{v} \bigl[ \partial_{\mu} \bigl[ \bigl(U \star \varphi \bigr) ( \mu \star \rho) \bigr] \bigr](v)
= \int_{\R^d} \partial_{v} \bigl[ \partial_{\mu} \bigl( U \star \varphi\bigr)(\mu \star \rho) \bigr](v-v') \rho(v') \ud v'.
\end{split}
\end{equation*}
Since the support of $\rho$ is the whole $\R^d$, the measure 
$\mu \star \rho$ also has $\R^d$ as support, so that, 
for any $v \in \R^d$, $(\mu \star \rho,v)$ 
is a continuity point of both 
$\partial_{\mu}( U \star \varphi)$ and 
$\partial_{v}[\partial_{\mu}( U \star \varphi)]$. 
Since 
$\partial_{\mu}( U \star \varphi)$ and 
$\partial_{v}[\partial_{\mu}( U \star \varphi)]$ are bounded, 
we deduce from Lebesgue's theorem that 
the maps 
$(\mu,v) \mapsto 
\partial_{\mu} ( U \star \varphi) ( \mu \star \rho)(v)$
and 
$(\mu,v) \mapsto 
\partial_{v} [\partial_{\mu} ( U \star \varphi) ( \mu \star \rho)](v)$
are continuous on the whole ${\mathcal P}_{2}(\R^d) \times \R^d$.

Moreover, whenever $\rho$ is chosen along a sequence that converges to the Dirac mass at $0$ (for the $W_{2}$ distance),
it is also easy to check that, for any $\mu \in {\mathcal P}_{2}(\R^d)$
and any $v \in \textrm{Supp}(\mu)$, $\partial_{\mu} 
(U \star \varphi)(\mu \star \rho)(v)$ and 
$\partial_{v}[ \partial_{\mu} 
(U \star \varphi)](\mu \star \rho)(v)$
converge to 
$\partial_{\mu} 
(U \star \varphi)(\mu)(v)$ and 
$\partial_{v}[ \partial_{\mu} 
(U \star \varphi)](\mu)(v)$. In particular, if It\^o's formula holds true for functionals
of the type ${\mathcal P}_{2}(\R^d) \ni \mu \mapsto 
(U \star \varphi)(\mu \star \rho)$, it also holds true for
functionals 
of the type ${\mathcal P}_{2}(\R^d) \ni \mu \mapsto 
(U \star \varphi)(\mu)$
and then for 
functionals 
of the type ${\mathcal P}_{2}(\R^d) \ni \mu \mapsto 
U(\mu)$ by the same approximation argument as in 
the proof of Theorem \ref{thm:full}.

Therefore, without any loss of generality, we can assume that $U$ and its 
first and partial second order derivatives are bounded 
 and uniformly continuous on the whole space. As in the proof of 
 Theorem \ref{thm:full},
 we can also assume
 that $(X_{t})_{t \geq 0}$ 
is a bounded It\^o process. 
\color{black}
\vspace{5pt}

\textit{Second step.} The proof requires another mollification argument. 
{Taking now $\rho$ as a smooth compactly supported density on $\R^d$ and
using the same notations as above}, we define the convolution 
$u_{n}^N$ of $u^N$: 
\begin{equation}
\label{eq:24:10:6:ter}
\begin{split}
u_{n}^N(x^1,\dots,x^N) &= n^{Nd}  \int_{(\R^d)^N} u^N(x^1-y^1,\dots,x^N-y^N) \prod_{\ell=1}^N \rho \bigl(n y^\ell \bigr) 
\prod_{\ell=1}^N \ud y^{\ell}
\\
&= {\mathbb E} \biggl[ U \biggl( \frac{1}{N} \sum_{i=1}^N \delta_{x^i - Y^i/n} \biggr) \biggr],
\end{split}
\end{equation}
where $Y^1,\dots,Y^N$ are $N$ i.i.d. random variables with density $\rho$. Recalling that 
\begin{equation*}
W_{2}^2
\biggl( \frac{1}{N} \sum_{i=1}^N \delta_{x^i - Y^i/n},  
\frac{1}{N} \sum_{i=1}^N \delta_{x^i}
\biggr) \leq \frac{1}{N} \sum_{i=1}^N \bigl( \frac{Y^i}{n} \bigr)^2,
\end{equation*}
we notice that 
\begin{equation}
\label{eq:24:10:6:bis}
\E \biggl[
W_{2}^2
\biggl( \frac{1}{N} \sum_{i=1}^N \delta_{x^i - Y^i/n},  
\frac{1}{N} \sum_{i=1}^N \delta_{x^i}
\biggr)
\biggr] \leq \frac{C}{n^2}, 
\end{equation}
as $\rho$ has compact support. Above and in the rest of the proof, the constant
$C$ is a general constant that is allowed to increase from line to line. Importantly,
it does not depend on $n$ nor $N$.

Observe now that, for two random variables $X,X' \in L^2(\Omega,\cA,\P;\R^d)$, 
we can find $t \in [0,1]$ such that 
\begin{equation*}
\begin{split}
\vert U([X]) - U([X']) \vert &= \bigl\vert {\mathbb E} \bigl[ \partial_{\mu} U \bigl( \bigl[t  X + (1-t) X' \bigr] \bigr)
\bigl( t  X + (1-t) X' \bigr) (X-X') \bigr] \bigr\vert 
\\
&\leq \big\| \partial_{\mu} U\bigl( \bigl[ tX + (1-t) X' \bigr] \bigr)\bigl( t  X + (1-t) X' \bigr) \big\|_{2} \|X-X' \|_{2}
\\
&\leq C \| X-X' \|_{2},
\end{split}
\end{equation*}
the last line following from the fact that the function $\cP_{2}(\R^d) \times \R^d \ni 
(\mu,v) \mapsto \partial_{\mu} U(\mu)(v)$ is bounded. Therefore, we deduce from 
\eqref{eq:24:10:6:ter} and
\eqref{eq:24:10:6:bis} that 
\begin{equation}
\label{eq:24:10:6}
\begin{split}
\bigl\vert u_{n}^N(x^1,\dots,x^N) - u^N(x^1,\dots,x^N) \bigr\vert 
&= 
\biggl\vert 
{\mathbb E} \biggl[ U \biggl( \frac{1}{N} \sum_{i=1}^N \delta_{x^i - Y^i/n} \biggr) 
- U \biggl( \frac{1}{N} \sum_{i=1}^N \delta_{x^i} \biggr)
\biggr]
\biggr\vert
\\
&\leq 
Cn^{-1}.
\end{split}
\end{equation}

Given a bounded random variable $X$ with law $\mu$, we know 
from \cite[Theorem 10.2.1]{RachevRuschendorf} 
that the quantity 
${\mathbb E}[W_{2}^2(\mu,\bar{\mu}^N)]$ tends to $0$ as $N$ tends to the infinity, 
$\bar{\mu}^N$ denoting the empirical measure of a sample of size $N$ of the same law as $X$. Moreover, the rate of convergence of 
$({\mathbb E}[W_{2}^2(\mu,\bar{\mu}^N)])_{N \geq 1}$ towards $0$ only depends upon the bounds for the 
moments of $X$.
 Together with \eqref{eq:24:10:6}, this says that we can find a sequence $(\varepsilon_{\ell})_{\ell \geq 1}$ converging to $0$ as 
$\ell$ tends to $\infty$ such that, for any $n,N \geq 1$ and for any $t \geq 0$,
\begin{equation}
\label{eq:24:10:7}
\begin{split}
&\E \bigl[ \bigl\vert u_{n}^N(X^1_{t},\dots,X^N_{t}) - U \bigl( \mu_{t} \bigr) \bigr\vert \bigr] 
\\
&\leq \E \bigl[ \bigl\vert u_{n}^N(X^1_{t},\dots,X^N_{t}) - u^N(X^1_{t},\dots,X^N_{t}) \bigr\vert \bigr]  
+ \E \bigl[ \bigl\vert U\bigl( \bar\mu_{t}^N \bigr) - U \bigl( \mu_{t} \bigr) \bigr\vert \bigr] 
\\
&\leq \varepsilon_{n} + \varepsilon_{N}. 
\end{split}
\end{equation}
(It is worth mentioning that the sequence $(\varepsilon_{\ell})_{\ell \geq 1}$ may be assumed to be independent of $t$.)
By boundedness of $U$, we deduce that, for any $p \geq 1$ and any $t  \geq 0$,
\begin{equation}
\label{eq:24:10:7:b}
\E \bigl[ \bigl\vert u_{n}^N(X^1_{t},\dots,X^N_{t}) - U \bigl( \mu_{t} \bigr) \bigr\vert^p \bigr]^{1/p} \leq \varepsilon_{n}^{(p)} + \varepsilon_{N}^{(p)}, 
\end{equation}
for a sequence $(\varepsilon_{\ell}^{(p)})_{\ell \geq 1}$ that tends to $0$ as $\ell$ tends to $\infty$ 
(and the terms of which are allowed to increase from line to line). 

Now, by the first part in Proposition \ref{prop:derivatives}, we compute
\begin{equation*}
\begin{split}
\partial_{x_{i}} u_{n}^{N}(x^1,\dots,x^N) 
&= n^{Nd} \int_{(\R^d)^N}
\partial_{x_{i}} u^N(x^1-y^1,\dots,x^N-y^N) \prod_{\ell=1}^N \rho(n y^\ell) 
\prod_{\ell=1}^N \ud y^\ell
\\
&= \frac{n^{Nd}}{N} \int_{(\R^{d})^N}
\partial_{\mu} U\biggl( \frac{1}{N} \sum_{\ell=1}^N \delta_{x^\ell - y^\ell}\biggr)(x^i - y^i) \prod_{\ell=1}^N 
\rho(n y^\ell) \prod_{\ell=1}^N \ud y^\ell
\\
&= \frac{1}{N} {\mathbb E} \biggl[ \partial_{\mu} U
\biggl( \frac{1}{N} \sum_{\ell=1}^N \delta_{x^\ell - Y^\ell/n}\biggr)(x^i - Y^i/n)
\biggr].
\end{split}
\end{equation*}
Using the uniform continuity of 
$\partial_{\mu} U$ on the whole space and 
following the proof of \eqref{eq:24:10:7}, we deduce that, for any $t \geq 0$,
\begin{equation}
\label{eq:24:10:8}
\E \bigl[ 
\bigl\vert N \partial_{x_{i}} u_{n}^N(X^1_{t},\dots,X^N_{t}) - \partial_{\mu} U ( \mu_{t})(X_{t}^i) \bigr\vert \bigr] \leq \varepsilon_{n} + \varepsilon_{N}. 
\end{equation}
Again, by boundedness of $\partial_{\mu} U$, we deduce that, for any $p \geq 1$ and any $t \geq 0$, 
\begin{equation}
\label{eq:24:10:8:b}
\E \bigl[ 
\bigl\vert N \partial_{x_{i}} u_{n}^N(X^1_{t},\dots,X^N_{t}) - \partial_{\mu} U ( \mu_{t})(X_{t}^i) \bigr\vert^p \bigr]^{1/p} \leq \varepsilon_{n}^{(p)} + \varepsilon_{N}^{(p)}. 
\end{equation}
Now, we differentiate once more in $x_{i}$:
\begin{equation*}
\begin{split}
&\partial_{x_{i}x_{i}}^2 u_{n}^N(x^1,\dots,x^N) 
\\
&= \frac{n^{Nd+1}}{N}\int_{(\R^d)^N}
\biggl\{
\partial_{\mu} U\biggl( \frac{1}{N} \sum_{\ell=1}^N \delta_{x^\ell - y^\ell}\biggr)(x^i - y^i) 
\biggr\}
\otimes
\nabla \rho(n y^i)
\prod_{\ell \not = i} \rho( n y^\ell) \prod_{\ell=1}^N \ud y^\ell,
\end{split}
\end{equation*}
the tensorial product operating on elements of $\R^d$.
We then split the derivative into two pieces:
\begin{equation*}
N \partial_{x_{i} x_{i}}^2 u_{n}^N(x^1,\dots,x^N) = T_{n,i}^{1,N}(x^1,\dots,x^N) + T_{n,i}^{2,N}(x^1,\dots,x^N),
\end{equation*}
with 
\begin{equation*}
\begin{split}
&T_{n,i}^{1,N}(x^1,\dots,x^N) 
\\
&\hspace{5pt}= n^{Nd+1} \int_{(\R^d)^N}
\biggl\{
\partial_{\mu} U\biggl( \frac{1}{N} \sum_{\ell \not = i} \delta_{x^\ell - y^\ell} + \frac{1}{N} \delta_{x^i} \biggr)
(x^i - y^i) \biggr\} 
\otimes \nabla \rho(n y^i)
\prod_{\ell \not = i} \rho(n y^\ell) \prod_{\ell=1}^N \ud y^\ell
\\
&T_{n,i}^{2,N}(x^1,\dots,x^N) 
\\
&\hspace{5pt}= n^{Nd+1} \int_{(\R^d)^N}
\biggl\{ \biggl[
\biggl( \partial_{\mu} U\biggl( \frac{1}{N} \sum_{\ell=1}^N \delta_{x^\ell - y^\ell}  \biggr)
\\
&\hspace{50pt}
- \partial_{\mu} U\biggl( \frac{1}{N} \sum_{\ell \not = i} \delta_{x^\ell - y^\ell} + \frac{1}{N} \delta_{x^i} \biggr)
\biggr] 
(x^i - y^i) \biggr\} 
\otimes \nabla \rho(n y^i)
\prod_{\ell \not = i} \rho(n y^\ell) \prod_{\ell=1}^N \ud y^\ell.
\end{split}
\end{equation*}
By integration by parts (recall that $\R^d \ni v \mapsto 
\partial_{\mu} U(\mu)(v)$ is differentiable), we can split $T^{1,N}_{n,i}$ into 
$$T^{1,N}_{n,i}(x^1,\dots,x^N)=T^{11,N}_{n,i}(x^1,\dots,x^N) + 
T^{12,N}_{n,i}(x^1,\dots,x^N),$$ 
with
\begin{equation*}
\begin{split}
&T^{11,N}_{n,i}(x^1,\dots,x^N) 
= n^{Nd} \int_{(\R^{d})^N}
\biggl\{ \partial_{v} \biggl[
\partial_{\mu} U\biggl( \frac{1}{N} \sum_{\ell=1}^N \delta_{x^\ell - y^\ell} \biggr) \biggr]  (x^i - y^i)\biggr\} 
\prod_{\ell = 1}^N \rho(n y^\ell) \prod_{\ell=1}^N \ud y^\ell
\\
&T^{12,N}_{n,i}(x^1,\dots,x^N) 
= n^{Nd}
\int_{(\R^d)^N}
\biggl\{
\partial_{v} \biggl[
\partial_{\mu} U\biggl( \frac{1}{N} \sum_{\ell \not = i} \delta_{x^\ell - y^\ell} + \frac{1}{N} \delta_{x^i} \biggr)
\\
&\hspace{100pt} - \partial_{\mu} U\biggl( \frac{1}{N} \sum_{\ell = 1}^N \delta_{x^\ell - y^\ell}  \biggr)
 \biggr] (x^i - y^i) \biggr\} 
\prod_{\ell = 1}^N \rho(n y^\ell) \prod_{\ell = 1}^N \ud y^\ell.
\end{split}
\end{equation*}
The first term is treated as per \eqref{eq:24:10:7} and \eqref{eq:24:10:8}. Namely, we have, for any $t \geq 0$, 
\begin{equation}
\label{eq:24:10:10}
\E \bigl[ 
\bigl\vert T^{11,N}_{n,i}(X^1_{t},\dots,X^N_{t}) - \partial_{v} \bigl[ \partial_{\mu} U ( \mu_{t}) \bigr] (X_{t}^i) \bigr\vert \bigr] \leq \varepsilon_{n} + \varepsilon_{N}. 
\end{equation}
Then, by boundedness of $\partial_v[\partial_\mu U]$ for any $p \geq 1$ and any $t \geq 0$, 
\begin{equation}
\label{eq:24:10:10:b}
\E \bigl[ 
\bigl\vert T^{11,N}_{n,i}(X^1_{t},\dots,X^N_{t}) - \partial_{v} \bigl[ \partial_{\mu} U ( \mu_{t}) \bigr] (X_{t}^i) \bigr\vert^p \bigr]^{1/p} \leq \varepsilon_{n}^{(p)} + \varepsilon_{N}^{(p)}. 
\end{equation}
To handle the second term, we use uniform  continuity of $\partial_{v} [\partial_{\mu} U]$. 
Indeed, we have
$\vert T^{12,N}_{n, i}(x^1,\dots,x^N)\vert \leq \varepsilon_{N}$ as 
\begin{equation*}
W_{2}^2 \biggl( \frac{1}{N} \sum_{\ell \not = i} \delta_{x^\ell - y^\ell} + \frac{1}{N} \delta_{x^i},
 \frac{1}{N} \sum_{\ell = 1}^N \delta_{x^\ell - y^\ell} \biggr) \leq \frac{1}{N} \vert y^i \vert^2 \leq \frac{C}{N},
\end{equation*}
since, in $T^{12,N}_{n,i}(x^1,\dots,x^N)$, $n y^i$ belongs to the (compact) support of $\rho$. This says that, for any $t \geq 0$,
\begin{equation}
\label{eq:24:10:11}
\E \bigl[ 
\bigl\vert T^{12,N}_{n,i}(X^1_{t},\dots,X^N_{t})  \bigr\vert \bigr] \leq  \varepsilon_{N}. 
\end{equation}
And, then, for any $p \geq 1$ and any $t \geq 0$, 
\begin{equation}
\label{eq:24:10:11:b}
\E \bigl[ 
\bigl\vert T^{12,N}_{n, i}(X^1_{t},\dots,X^N_{t})  \bigr\vert^p \bigr]^{1/p} \leq  \varepsilon_{N}^{(p)}.
\end{equation}

We finally handle $T^{2,N}_{n, i}$. Following the proof of \eqref{eq:24:10:11:b}, we have,
for any $p \geq 1$ and any $t \geq 0$,
\begin{equation}
\label{eq:24:10:12:b}
\E \bigl[ 
\bigl\vert T^{2,N}_{n,i}(X^1_{t},\dots,X^N_{t})  \bigr\vert^p \bigr]^{1/p} \leq n \varepsilon_{N}^{(p)}, 
\end{equation}
the additional $n$ coming from the differentiation of the regularization kernel.  
\vspace{5pt}

\textit{Third step.}
In order to complete the proof, we apply It\^o's formula to $(u^N_{n}(X^1_{t},\dots,X^N_{t}))_{t \geq 0}$ for given values of $n$ and $N$. We obtain
\begin{equation*}
\label{eq:24:10:15}
\begin{split}
0 &= u^N_{n}\bigl({X}^1_{t},\dots,{X}^N_{t}\bigr) -
u^N_{n}\bigl({X}^1_{0},\dots,{X}^N_{0}\bigr)
 -
\sum_{\ell=1}^N \int_{0}^t \partial_{x^{\ell}} u^N_{n}\bigl({X}^1_{s},\dots,{X}^N_{s}\bigr) b_{s}^\ell \ud s 
\\
&\hspace{5pt}  -
\sum_{\ell=1}^N \int_{0}^t \partial_{x^{\ell}} u^N_{n}\bigl({X}^1_{s},\dots,{X}^N_{s}\bigr) 
\sigma_{s}^\ell \ud W_{s}^\ell 
-
\frac{1}{2} \sum_{\ell=1}^N \int_{0}^t 
{\rm Tr} \bigl\{
\partial_{x_{\ell}}^2  u^N_{n}\bigl({X}^1_{s},\dots,{X}^N_{s}  \bigr) a_{s}^\ell
  \bigr\} \ud s, 
\end{split}
\end{equation*}
with $a_{s}^\ell := \sigma_{s}^\ell ( \sigma_{s}^\ell)^{\dagger}$.
To compare with the expected result, we take the difference with
\begin{equation}
\label{eq:24:10:16}
\begin{split}
\Delta_{t}^N  &= U(\mu_{t}) - U(\mu_{0})
- \frac{1}{N} \sum_{\ell=1}^N \int_{0}^t \partial_{\mu} U(\mu_{s})(X_{s}^{\ell}) b_{s}^{\ell} \ud s
\\
&\hspace{5pt} - \frac{1}{N} \sum_{\ell=1}^N \int_{0}^t \partial_{\mu} U(\mu_{s})(X_{s}^{\ell}) \sigma_{s}^{\ell} \ud
W_{s}^{\ell} 
- \frac{1}{2N} \sum_{\ell=1}^N \int_{0}^t 
{\rm Tr} \bigl\{
\partial_{v} \bigl[ \partial_{\mu} U(\mu_{s})\bigr](X_{s}^{\ell}) 
 a_{s}^{\ell} 
 \bigr\}
\ud s.
\end{split}
\end{equation}
From \eqref{eq:24:10:7:b}, 
\eqref{eq:24:10:8:b},
\eqref{eq:24:10:10:b}, 
\eqref{eq:24:10:11:b} and \eqref{eq:24:10:12:b},
we obtain, for any $T>0$,
\begin{equation*}
\sup_{0 \leq t \leq T} {\mathbb E} \bigl[ \vert \Delta_{t}^N \vert \bigr]
\leq \varepsilon_{n} + (1+n) \varepsilon_{N}, 
\end{equation*}
the sequence $(\varepsilon_{\ell})_{\ell \geq 1}$ now depending on $T$. 
Letting $N$ tend to $\infty$, we deduce from Fatou's lemma and the law of large numbers that 
\begin{equation}
\label{eq:24:10:17}
\sup_{0 \leq t \leq T} \vert \Delta_{t} \vert 
\leq \varepsilon_{n}, 
\end{equation}
where 
\begin{equation*}
\begin{split}
&\Delta_{t}  = U(\mu_{t}) - U(\mu_{0})
- \int_{0}^t \E \bigl[ \partial_{\mu} U(\mu_{s})(X_{s}) b_{s} \bigr] \ud s
 - \frac{1}{2} \int_{0}^t \E \Bigl[ {\rm Tr} \bigl\{
   \partial_{v} \bigl[ \partial_{\mu} U(\mu_{s})\bigr](X_{s})
   a_{s}   
\bigr\} \Bigr] \ud s.
\end{split}
\end{equation*}
Letting $n$ tend $\infty$ in \eqref{eq:24:10:17}, we deduce that $\Delta \equiv 0$, which completes the proof. 
\end{proof}

\subsection{A sufficient condition for partial 
${\mathcal C}^2$ regularity}
\label{sec:ito:frechet}

The following is a sufficient criterion for partial ${\mathcal C}^2$ regularity used in the next section:

\begin{Theorem}
\label{thm:ito:frechet}
Let $U : {\mathcal P}_{2}(\R^d) \rightarrow \R$
be a function such that its lifted version ${\mathcal U} : L^2(\Omega,\cA,\P;\R^d) \ni \xi \mapsto U([\xi]) \in \R$ is once continuously Fr\'echet differentiable. Assume also that for any continuously differentiable
map $\R \ni \lambda \mapsto X^{\lambda} \in L^2(\Omega,{\mathcal A},\P;\R^d)$,
with the property that all the $(X^{\lambda})_{\lambda \in \R}$
have the same distribution and that $\vert [\ud/\ud \lambda]X^{\lambda} \vert \leq 1$
(in $L^{\infty}$), 
  the mapping
\begin{equation}
\label{eq:ito:frechet:assumption}
\R \ni \lambda \mapsto D {\mathcal U}(X^{\lambda}) \cdot \chi =
{\mathbb E} \bigl[ \partial_{\mu} U([X^\lambda])(X^\lambda) \chi
\bigr]
\in \R
\end{equation}
is continuously differentiable 
for any $\chi \in L^2(\Omega,{\mathcal A},\P;\R^d)$. Moreover assume that the derivative of the mapping $\R \ni \lambda \mapsto D {\mathcal U}(X^{\lambda}) \cdot \chi$ at $\lambda =0$ depends on the
family $(X^\lambda)_{\lambda \in \R}$
only through the value of $X^0$ and of 
$[\ud / \ud \lambda]_{\vert \lambda =0} X^{\lambda}$ 
{(see footnote\footnote{
{This means that 
for two 
families 
$(X^\lambda)_{\lambda \in \R}$
and 
$(X^{\lambda,\prime})_{\lambda \in \R}$
with $X^{0}=X^{0,\prime}$ and 
$[\ud / \ud \lambda]_{\vert \lambda =0} X^{\lambda}
= [\ud / \ud \lambda]_{\vert \lambda =0} X^{\lambda,\prime}$, 
the derivatives
$[\ud/\ud \lambda]_{\vert \lambda =0}
[ D {\mathcal U}(X^{\lambda}) \cdot \chi ]$
and 
$[\ud/\ud \lambda]_{\vert \lambda =0}
[ D {\mathcal U}(X^{\lambda,\prime}) \cdot \chi ]$ are the same (the variable 
$\chi$ being given).}
\label{footnote:second:order:derivative}
} below for more details)}, so that 
we can denote
\begin{equation*}
\partial^2_{\zeta,\chi} {\mathcal U}(X) := \frac{\ud}{\ud \lambda}_{\vert \lambda =0}
\bigl[ D {\mathcal U}(X^{\lambda}) \cdot \chi \bigr], 
\end{equation*}
whenever $ X := X^0 $ and $\zeta := [\ud/\ud \lambda]_{\vert \lambda =0} X^{\lambda}$. Finally, assume that there exist a constant $C$ and 
an exponent $\alpha \geq 0$ 
such that, for any $X$, $\chi$ and 
$\zeta$ in $L^2(\Omega,\cA,\P;\R^d)$, with $\vert \zeta\vert \leq 1$
(in $L^{\infty}$), it holds (with $\Phi_{\alpha}$ as in 
\HYP{1} and in particular satisfying 
\eqref{phialpha}):
\begin{equation*}
\begin{split}
&(i) \quad \ \vert D {\mathcal U}(X) \cdot \chi \vert 
+
\vert \partial^2_{\zeta,\chi} U(X) \vert
\leq C  \| \chi \|_{2},
\\
&(ii) \quad \vert D {\mathcal U}(X) \cdot \chi - D {\mathcal U}(X') \cdot \chi \vert
+
\vert \partial^2_{\zeta,\chi} {\mathcal U}(X) 
- \partial_{\zeta,\chi}^2 {\mathcal U}(X')  \vert \leq  C \Phi_{\alpha}(X,X') \| \chi \|_{2}.
\end{split}
\end{equation*}
Then $U$ is partially ${\mathcal C}^2$ and 
satisfies for any compact subset ${\mathcal K} \subset {\mathcal P}_{2}(\R^d)$:
\begin{equation*}
\sup_{ \mu \in {\mathcal K}}
\biggl[
\int_{\R^d} \bigl\vert  \partial_{\mu} U(\mu)(v) \bigr\vert^2 \ud\mu(v) 
+
\int_{\R^d} \bigl\vert \partial_{v} \bigl[ \partial_{\mu} U(\mu) \bigr](v) \bigr\vert^2 \ud\mu(v) 
\biggr]
< \infty, 
\end{equation*}
so that the chain rule applies to an It\^o process satisfying \eqref{eq:23:10:2}. 
\end{Theorem}

\begin{Remark}
The thrust of Theorem \ref{thm:ito:frechet}
is to study the smoothness of the mapping $v \mapsto \partial_{\mu} U(\mu)(v)$ independently 
of the smoothness in the direction $\mu$
by restricting the `test' random variables $(X^\lambda)_{\lambda \in \R}$ to an identically distributed family. One of the issue in the proof is precisely to construct such a family of test random
variables. 
\end{Remark}
\begin{proof} 
In the proof, we use quite often the following result, which is a refinement 
of \cite[Lemma 3.3]{carmona:delarue:aop} (see the adaptation of the proof in 
Subsection \ref{subse:appendix:1} in Appendix):
\begin{Proposition}
\label{prop:lipschitz:lifted}
Consider a collection $(V(\mu) : \R^d \ni v \mapsto V(\mu)(v))_{\mu}$ of Borel functions from $\R^d$ into $\R^d$ 
indexed by elements $\mu \in {\mathcal P}_{2}(\R^d)$ such that, for any $\mu \in 
{\mathcal P}_{2}(\R^d)$, the mapping 
$\R^d \ni v \mapsto V(\mu)(v) \in \R^d$ belongs to $L^2(\mu,\R^d;\R^d)$. Assume also that
there exist a constant $C$ and an exponent $\alpha$ such that,
for any $\mu \in {\mathcal P}_{2}(\R^d)$ and 
any $\xi,\xi' \in L^2(\Omega,\cA,\P;\R^d)$, such that $\xi$ and $\xi'$ have distribution $\mu$, and
\begin{equation}
\label{eq:assumption:prop:lipschitz:lifted}
{\mathbb E} \bigl[ \vert V(\mu)(\xi) - V(\mu)(\xi') \vert^2 \bigr]^{1/2} 
\leq C 
 {\mathbb E} \bigl[ \bigl( 1+ \vert \xi \vert^{2\alpha} + \vert \xi' \vert^{2\alpha} + \| \xi \|_{2}^{2\alpha} \bigr)
\vert \xi - \xi' \vert^2  \bigr]^{1/2}.
\end{equation}
Then, for any $\mu \in {\mathcal P}_{2}(\R^d)$, the mapping $v \mapsto V(\mu)(v)$ admits a locally Lipschitz continuous version, that satisfies 
$$\vert V(\mu)(v) - V(\mu)(v') \vert \leq 
C \biggl[ 1+
2 \max \bigl( \vert v \vert^{2\alpha} , \vert v' \vert^{2\alpha} \bigr) + 
 \biggl( \int_{\R^d} \vert x \vert^2 \ud\mu(x) \biggr)^{\alpha}\biggr]^{1/2}  \vert v-v' \vert.$$   
\end{Proposition}
As a warm-up, we discuss what 
Proposition \ref{prop:lipschitz:lifted} says in the framework of Theorem 
\ref{thm:ito:frechet}. Representing $D {\mathcal U}(X) \cdot \chi$ as ${\mathbb E}[ \partial_{\mu} U([X])(X) \chi]$, 
we can write (choosing $X=\xi$ and $X'=\xi'$, with $[\xi]=[\xi']=\mu$, in  part
\textit{(ii)} of the statement of Theorem \ref{thm:ito:frechet})
\begin{equation*}
\begin{split}
&\bigl\vert {\mathbb E} \bigl[ \bigl( \partial_{\mu}U(\mu)(\xi') - \partial_{\mu} U(\mu)(\xi) \bigr) \chi \bigr]
\bigr\vert
\\
&\hspace{15pt}\leq C  {\mathbb E} \bigl[ 
\bigl( 1 + \vert \xi \vert^{2\alpha} + \vert \xi' \vert^{2\alpha} 
+\| \xi \|_{2}^{\alpha}\bigr)
 \vert \xi - \xi' \vert^2  \bigr]^{1/2}
{\mathbb E} \bigl[ \vert \chi \vert^2 \bigr]^{1/2}.
\end{split}
\end{equation*}
This says that, for any $\mu \in {\mathcal P}_{2}(\R^d)$, we can find a locally Lipschitz continuous version of the 
mapping $\R^d \ni v \mapsto \partial_{\mu} U(\mu)(v)$, the local Lipschitz constant 
being at most of $\alpha$-polynomial growth, uniformly with respect to $\mu$ in $W_{2}$-balls. In addition, Proposition \ref{prop:lipschitz:lifted} gives us a bit more. Consider a sequence
$(\mu_{n})_{n \geq 0}$ with values in ${\mathcal P}_{2}(\R^d)$ such that $\mu_{n} \rightarrow \mu$ in the
$2$-Wasserstein distance. Then, the functions $(\R^d \ni v \mapsto \partial_{\mu} U(\mu_{n})(v))_{n \geq 0}$ are uniformly  continuous on compact sets. 
Moreover, we notice, by Markov inequality that $\P(\vert \xi_{n} \vert \geq 2 \| \xi_{n} \|_{2}) 
\leq 1/4$, so that 
\begin{equation}
\label{eq:trick:markov}
\begin{split}
\frac{3}{4
}\inf_{\vert v \vert \leq 2 \| \xi_{n} \|_{2}}
\vert \partial_{\mu} U(\mu_{n})(v) \vert 
&\leq {\mathbb E} \bigl[ {\mathbf 1}_{\{ \vert \xi_{n} \vert \leq 2\| \xi_{n} \|_{2}\}} \vert \partial_{\mu} U(\mu_{n})(\xi_{n}) \vert^2 \bigr]^{1/2}
\\
&\leq 
{\mathbb E} \bigl[ \vert \partial_{\mu} U(\mu_{n})(\xi_{n}) \vert^2 \bigr]^{1/2}
\leq C,
\end{split}
\end{equation}
where $\xi_{n}$ has distribution $\mu_{n}$, the last inequality following from \textit{(i)} in the statement of Theorem 
\ref{thm:ito:frechet}. This says that the family $(
\inf_{\vert v \vert \leq 2 \| \xi_{n} \|_{2}}
\vert \partial_{\mu} U(\mu_{n})(v) \vert)_{n \geq 0}$ is bounded. 
As the sequence $(\| \xi_{n}\|_{2})_{n \geq 0}$ is bounded and the mappings 
$(\R^d \ni v \mapsto \partial_{\mu} U(\mu_{n})(v))_{n \geq 0}$
are uniformly locally Lipschitz continuous, the sequence 
$(\vert \partial_{\mu} U(\mu_{n})(0) \vert)_{n \geq 0}$ is also bounded. 
Therefore, the family $(\R^d \ni v \mapsto \partial_{\mu} U(\mu_{n})(v))_{n \geq 0}$ is relatively compact for the topology of uniform convergence on compact subsets. Passing to the limit \footnote{
 {From \cite[Theorem 6.9]{villani}, $\mu_n$ converges weakly to $\mu$. Using the Skorokhod representation
 theorem, we can find a sequence $(\xi_n)$ converging almost surely to $\xi$.
 The convergence holds also in $L^2$ since this sequence is uniformly square integrable, recall \cite[Definition 6.8(iii)]{villani}.
 }
 }
(up to a subsequence) into the relationship
\begin{equation*}
D{\mathcal U}(\xi_n) \cdot \chi = 
{\mathbb E} \bigl[ 
\partial_{\mu} U(\mu_{n})(\xi_{n}) \chi
\bigr],
\end{equation*}
{we deduce, by identification, that the limit of $\partial_{\mu} U(\mu_{n})$ must coincide with $\partial_{\mu} U(\mu)$ on the support of $\mu$.
This says that the function ${\mathcal P}_{2}(\R^d) \times \R^d \ni (\mu,v) \mapsto 
\partial_{\mu} U(\mu)(v)$ is (jointly) continuous
at any point $(\mu,v)$ such that $v \in \textrm{Supp}(\mu)$. 
Moreover,
by point \textit{(i)} in the statement of Theorem 
\ref{thm:ito:frechet}, 
we have 
$\int_{\R^d} \vert \partial_{\mu} U(\mu)(v) \vert^2 d\mu(v)
\leq C$, 
for a constant $C$ independent of $\mu$,
which is the first part in the condition \eqref{eq:compact:bound}
for applying the chain rule to partially $\cC^2$ functions.}
\vspace{5pt}

To complete the proof we have two main steps. The first one uses a new mollification argument. 
The second consists in a coupling lemma, which permits to choose relevant versions of the random variables along which the differentiation is performed. 
\vspace{5pt}

\textit{First step.} Given a distribution $\mu$ and a random variable $\xi$ with distribution $\mu$, we introduce the convoluted version 
$\mu^n$ of $\mu$:
\begin{equation*}
\mu^n = \mu \star {\mathcal N}_{d}(0,\tfrac1n I_{d}),
\end{equation*}
$n$ denoting an integer larger than $1$
and ${\mathcal N}_{d}(0,(1/n)I_{d})$ denoting the $d$-dimensional 
Gaussian distribution with covariance matrix $(1/n) I_{d}$, where 
$I_{d}$ is the identity matrix of dimension $d$. Then, we can define the mapping 
\begin{equation}
\label{eq:29:11:2}
{\mathcal V}^n(\mu,v) =  \int_{\R^d} \partial_{\mu} U \bigl( \mu^n \bigr) (v - u) n^{d/2} \rho \bigl( n^{1/2}u) \ud u, 
\end{equation}
where $\rho$ stands for the standard $d$-dimensional Gaussian kernel. The mapping ${\mathcal V}^n$
is the convolution of $\partial_{\mu} U(\mu^n)(\cdot)$ with the measure ${\mathcal N}_{d}(0,
(1/n)I_{d})$. 
{
By the warm-up, 
the sequence $(\partial_{\mu} U(\mu^n)(0))_{n \geq 1}$
is bounded and the functions $(\partial_{\mu} U(\mu^n) : \R^d 
\ni v \mapsto \partial_{\mu} U(\mu^n)(v) \in \R^d)_{n \geq 1}$
are locally Lipschitz, the Lipschitz constant being 
at most of $\alpha$-polynomial growth, 
uniformly in $n \geq 1$. 
In particular, the sequence of functions $({\mathcal V}^n(\mu,\cdot))_{n \geq 0}$ 
is relatively compact for the topology of uniform convergence on compact subsets. 
Any limit must coincide with $\partial_{\mu} U(\mu)$ at points
$v$ in the support of $\mu$ or, put it differently, 
any limit provides a version of $\partial_{\mu} U(\mu)$
which is locally Lipschitz continuous, 
the Lipschitz constant being at most of $\alpha$-polynomial growth, uniformly in 
$\mu$ in bounded subsets of ${\mathcal P}_{2}(\R^d)$. 
When $\mu$ has full support, 
the sequence $({\mathcal V}^n(\mu,\cdot))_{n \geq 0}$ 
converges to 
the unique continuous version of 
$\partial_{\mu} U(\mu)$, the convergence being uniform on compact subsets.} 

Letting $\xi^n = \xi + n^{-1/2} G$, where $G$ is an ${\mathcal N}_{d}(0,I_{d})$ Gaussian variable independent of $\xi$, we then observe that,
for any $\R^d$-valued square integrable random variable $\chi$ such that the pair $(\xi,\chi)$ is independent of $G$, 
\begin{equation}
\label{eq:28:11:1}
\begin{split}
D {\mathcal U}(\xi^n) \cdot \chi = {\mathbb E} \bigl[ \partial_{\mu} U \bigl( \mu^n \bigr) \bigl(\xi^n \bigr) \chi \bigr] 
&= {\mathbb E} \biggl[ \biggl(
\int_{\R^d}  \partial_{\mu} U \bigl( \mu^n)(\xi - u) n^{d/2} \rho(n^{1/2}u) \ud u \biggr)\chi \biggr] 
\\
&= {\mathbb E} \bigl[ {\mathcal V}^n ( \mu,\xi ) \chi \bigr],
\end{split}
\end{equation}
where $\cV^n(\mu,\xi)$ is viewed as a row vector. 
We note that the mapping $\R^d \ni v \mapsto 
{\mathcal V}^n(\mu,v)$ is differentiable with respect to $v$ (this was not the case for the original mapping  
$\R^d \ni v \mapsto \partial_{\mu} U(\mu)(v)$ at this stage of the proof). 
 \vspace{5pt}
 

\textit{Second step.}
We construct now, 
{independently of the measure $\mu$
considered above, a family $(Y^{\lambda})_{\lambda \in \R}$ 
that is differentiable with respect to $\lambda$ in $L^2(\Omega,\cA,\P;\R)$ but which is, at the same time, invariant in law, all the $Y^{\lambda}$, for $\lambda \in \R$, being uniformly 
distributed on $[-\pi/2,\pi/2]$}. 
The strategy consists in starting with the uniform distribution:

Given two independent  ${\mathcal N}(0,1)$ random variables $Z$ and $Z'$, we let, for any $\lambda \in \R$, 
\begin{equation*}
Z^{\lambda} = \cos(\lambda) Z + \sin(\lambda) Z', \quad Z^{\prime,\lambda}
= - \sin(\lambda) Z + \cos(\lambda) Z', 
\end{equation*}
so that $(Z^{\lambda},Z^{\prime,\lambda})$ has the same law as $(Z,Z')$ (because of the invariance of the Gaussian distribution by rotation). We then let
\begin{equation*}
Y^{\lambda} = \arcsin \bigl( \frac{Z^{\lambda}}{\sqrt{(Z^{\lambda})^2 + (Z^{\prime,\lambda})^2}} \bigr)
= \arcsin \bigl( \frac{Z^{\lambda}}{\sqrt{Z^2 + (Z')^2}} \bigr).
\end{equation*}
It is easy to check that $Y^{\lambda}$ has a uniform distribution on $[-\pi/2,\pi/2]$ for any 
$\lambda \in \R$. Pathwise,  the mapping $\R \ni \lambda \mapsto Y^{\lambda}$ is differentiable 
at any $\lambda$ such that $Z^{\prime,\lambda} \not = 0$. 
Noticing that $[\ud/\ud\lambda] Z^\lambda = Z^{\prime,\lambda}$ (pathwise), we get: 
\begin{equation*}
\frac{\ud}{\ud \lambda} Y^{\lambda} = \frac{Z^{\prime,\lambda}}{
\sqrt{Z^2 + (Z')^2}} \Bigl( 1 - \frac{(Z^{\lambda})^2}{(Z^{\lambda})^2+ (Z^{\prime,\lambda})^2}
\Bigr)^{-1/2} = \textrm{sign}\bigl(Z^{\prime,\lambda}\bigr). 
\end{equation*}
On the event $\{Z^{\prime,0} \not = 0\} = \{Z' \not = 0\}$, which is of probability 1, the set of $\lambda$'s such that 
$Z^{\prime,\lambda} = 0$ is locally finite. The above derivative being bounded by $1$, this says that, pathwise, the mapping $\R \ni \lambda \mapsto Y^{\lambda}$ is $1$-Lipschitz continuous. Therefore, the random variables 
$(Y^{\lambda}- Y^0)/\lambda$, $\lambda \not =0$, are bounded by $1$. Moreover, still 
on the event $\{Z' \not =0\}$, the above computation shows that 
\begin{equation}
\label{eq:derivee:en:0}
\lim_{\lambda \rightarrow 0} \frac{Y^{\lambda}- Y^0}{\lambda} = \textrm{sign}\bigl(Z^{\prime}\bigr).
\end{equation}
Therefore, by Lebesgue's dominated convergence theorem, 
the mapping $\R \ni \lambda \mapsto Y^{\lambda} \in L^2(\Omega,\cA,\P;\R)$ is differentiable 
at $\lambda = 0$ with $\textrm{sign}(Z^{\prime})$ as its derivative. 
%
In the sequel, we will denote $Y^0$ by $Y$. 

Actually, by a rotation argument, differentiability holds at any $\lambda \in \R$, 
with $[\ud/\ud \lambda] Y^\lambda = \textrm{sign}(Z^{\prime,\lambda})$. 
It is then clear that $\R \ni \lambda \mapsto \textrm{sign}(Z^{\prime,\lambda})
\in L^2(\Omega,{\mathcal A},\P;\R^d)$ is continuous. Indeed, the path 
$\R \ni \lambda \mapsto Z^{\prime,\lambda}$ is continuous. Composition by the function \textit{sign} preserves continuity since, for any 
$\lambda \in \R$, the set of zero points of $Z^{\prime,\lambda}$ is of zero probability. 
\vspace{5pt}

\textit{Third step.} Assume now that $\mu$ denotes a given distribution as in the first step. We then choose a random variable
$\xi$ with $\mu$ as distribution, $\xi$ being independent of the pair $(Z,Z')$. 
Given the same $(Y^{\lambda})_{\lambda \in \R}$ as above and some parameter $\delta >0$, we let  
\begin{equation*}
\forall \lambda \in \R, \quad \xi^{\lambda} = (\delta \times Y^{\lambda}) e + \xi, 
\end{equation*}
where $e$ is an arbitrary deterministic unitary vector in $\R^d$. (We omit the dependence upon $\delta$ in the notation $\xi^{\lambda}$.) 
Then, we know that the mapping $\R \ni \lambda \mapsto \xi^{\lambda}$ is continuously
differentiable in $L^2(\Omega,\cA,\P;\R^d)$, with  
\begin{equation*}
\frac{\ud}{\ud \lambda}_{\vert \lambda =0} \xi^{\lambda} = (\delta \times  \textrm{sign}(Z')) e. 
\end{equation*}
Going back to \eqref{eq:28:11:1}, we get for another random variable $\chi \in L^2(\Omega,\cA,\P;\R^d)$, 
with $(\xi,\chi,Z,Z')$ independent of $G$,
\begin{equation*}
D {\mathcal U}\bigl( \xi^{\lambda}+ %
{\tfrac{1}{\sqrt{n}}} G
\bigr) \cdot \chi = {\mathbb E} \bigl[ {\mathcal V}^n \bigl(\law{\xi^{\lambda}},\xi^{\lambda}\bigr) \chi \bigr],
\end{equation*}
where ${\mathcal V}^n(\mu,v)$ is seen as a row vector. 
As the mapping $\R \ni \lambda \mapsto \xi^{\lambda}$ is continuously differentiable in $L^2(\Omega,\cA,\P;\R^d)$ 
and since all the random variables $(\xi^{\lambda})_{\lambda \in \R}$ have the same distribution, we deduce that (for $(\xi,\chi,Z,Z')$ independent of $G$)  
\begin{equation*}
\begin{split}
\partial^2_{ \textrm{sign}(Z') e,\chi}
{\mathcal U}\bigl( \xi + \delta Y e +   {\tfrac{1}{\sqrt{n}}} G
\bigr) 
&=\frac{\ud}{\ud \lambda}_{\vert \lambda =0} \bigl[ D {\mathcal U}\bigl( \xi^{\lambda/\delta}+   {\tfrac{1}{\sqrt{n}}} G
\bigr) \cdot \chi \bigr]
\\
 &= \frac{1}{\delta} \frac{\ud}{\ud \lambda}_{\vert \lambda =0} \bigl[ D {\mathcal U}\bigl( \xi^{\lambda}+   {\tfrac{1}{\sqrt{n}}} G
\bigr) \cdot \chi \bigr]
\\
&= {\mathbb E} \bigl[ \text{Tr} \bigl\{ \partial_{v}{\mathcal V}^n \bigl( [ \xi + \delta Y e], \xi + \delta Y e\bigr)   
\bigl( ( \textrm{sign}(Z') \chi ) \otimes e \bigr) \bigr\} \bigr].
\end{split}
\end{equation*}
Noticing that the random variable $\vert \textrm{sign}(Z') \vert$ is equal to $1$ almost surely, 
we can replace $\chi$ by $\textrm{sign}(Z') \chi$ (recall that 
$\vert \chi\vert$ must be less than $1$) with $(\xi,\chi)$ independent of $(Z,Z')$, so that 
\begin{equation*}
\partial^2_{\textrm{sign}(Z') e,\textrm{sign}(Z')  \chi}
{\mathcal U}\bigl( \xi + \delta Y e +   {\tfrac{1}{\sqrt{n}}} G
\bigr)  = 
 {\mathbb E} \bigl[\text{Tr} \bigl\{
  \partial_{v}{\mathcal V}^n \bigl(
[ \xi + \delta Y e], \xi + \delta Y e\bigr)   
 \bigl( \chi \otimes e \bigr) \bigr\}
 \bigr].
\end{equation*}
Finally, we let 
\begin{equation}
\label{eq:1:12:1}
{\mathcal W}^{n,\delta}(\mu,v)  = \int_{\R} \partial_{v}{\mathcal V}^n \bigl(
\mu \star p^\delta, v + \delta u e \bigr) p(u ) \ud u,
\end{equation}
where $p$ is the uniform density on $[-\pi/2,\pi/2]$ and
 $p^{\delta}(\cdot)=p(\cdot/\delta)/\delta$ is the uniform density on 
 $[-\delta \pi/2,\delta \pi/2]$. Moreover
 $\mu \star p^{\delta}$ is an abbreviated notation for denoting 
the convolution of $\mu$ with the uniform distribution on the segment 
$[-(\delta \pi/2)e,(\delta \pi/2)e]$. 
Since the pair $(\xi,\chi)$ is independent of $(Z,Z')$, we end up with the duality formula:
\begin{equation}
\label{eq:29:11:5}
\partial^2_{\textrm{sign}(Z') e,\textrm{sign}(Z')  \chi}
{\mathcal U}\bigl( \xi + \delta Ye +   {\tfrac{1}{\sqrt{n}}} G
\bigr)
 = 
 {\mathbb E} \bigl[ \text{Tr} \bigl\{ 
 {\mathcal W}^{n,\delta}(\mu,\xi) \bigl( \chi \otimes e \bigr) \bigr\}\bigr].
\end{equation}
By the smoothness assumption on $\partial^2_{\zeta,\chi} {\mathcal U}$ (see \textit{(ii)} in 
the statement of Theorem \ref{thm:ito:frechet}), we deduce that, for another $\xi'$, 
with distribution $\mu$ as well, such that
the triple $(\xi,\xi',\chi)$ is independent of $(Z,Z')$ and the $5$-tuple
$(\xi,\xi',\chi,Z,Z')$ is independent of $G$,
\begin{align}
&\bigl\vert 
{\mathbb E} \bigl[ {\rm Tr} \bigl\{  \bigl( {\mathcal W}^{n,\delta}(\mu,\xi) -
 {\mathcal W}^{n,\delta}(\mu,\xi^{\prime}) \bigr) \bigl( \chi \otimes e \bigr) \bigr\}\bigr] \bigr\vert \nonumber
 \\
 &\leq 
C  {\mathbb E} 
 \Bigl[
\bigl( 1 + \vert \xi \vert^{2\alpha}+ \vert \xi' \vert^{2\alpha} 
+ \vert  \delta Y \vert^{2\alpha} + \vert   {\tfrac1{\sqrt{n}}} G \vert^{2\alpha} 
+ \| \xi \|_{2}^{2\alpha}
\bigr)
\vert \xi - \xi'
 \vert^2 
 \Bigr]^{1/2}
  {\mathbb E} \bigl[ \vert \chi \vert^2 \bigr]^{1/2} \nonumber
  \\
  &\leq C  {\mathbb E} 
 \bigl[
\bigl( 1+ \vert \xi \vert^{2\alpha} + \vert \xi' \vert^{2\alpha}  + \| \xi \|_{2}^{2\alpha}\bigr)
 \vert \xi - \xi'
 \vert^2
 \bigr]^{1/2}
  {\mathbb E} \bigl[ \vert \chi \vert^2 \bigr]^{1/2}, \label{eq:lip:regul}
\end{align}
where we used the independence of $(\xi,\xi')$ and $(Z,Z')$ to pass from the second to the third line, the value of $C$ 
varying from the second to the third line (but remaining 
independent of $\delta$ and $n$, when $\delta$ is taken in a bounded set). 
The above is true for any $\sigma(\xi,\xi')$-measurable $\chi \in L^2(\Omega,\cA,\P)$. We deduce that, 
for any other $e' \in \R^d$ with $\vert e' \vert =1$,  
\begin{equation*}
\begin{split}
&{\mathbb E} \bigl[ \bigl\vert 
{\rm Tr} \bigl\{ \bigl( {\mathcal W}^{n,\delta}(\mu,\xi) -
 {\mathcal W}^{n,\delta}(\mu,\xi^{\prime}) \bigr) \bigl( e' \otimes e \bigr)
 \bigr\}
  \bigr\vert^2 \bigr]  
 \\
&\hspace{15pt} \leq C    {\mathbb E} 
 \bigl[
\bigl( 1+ \vert \xi \vert^{2\alpha} + \vert \xi' \vert^{2\alpha} 
+ \| \xi \|_{2}^{2\alpha}
 \bigr) 
 \vert \xi - \xi'
 \vert^2
 \bigr]^{1/2}.
 \end{split}
\end{equation*}
By Proposition \ref{prop:lipschitz:lifted}, this says that $\R^d \ni v \mapsto 
{\rm Tr} \{ (
{\mathcal W}^{n,\delta} (\mu,v))(e' \otimes e)\}$ has a locally Lipschitz version, 
the local Lipschitz constant on a ball of center $0$ and radius $\gamma$
is less than $C(1+\gamma^\alpha)$, the constant $C$ being uniform with respect
to $\xi$ in $L^2$ balls.  
\vspace{5pt}

\textit{Fourth step.}
From \eqref{eq:29:11:2} and \eqref{eq:1:12:1}, we know that 
\begin{equation*}
\begin{split}
{\mathcal W}^{n,\delta}(\mu,v) 
&= \int_{\R}
\partial_{v}{\mathcal V}^n \bigl(
\mu \star p^\delta, v + \delta u e \bigr) p (u ) \ud u
\\
&= n^{(d+1)/2} \int_{\R \times \R^d}  \partial_{\mu} U \bigl( \mu \star p^{\delta} \star {\mathcal N}_{d}(0,
\tfrac1n I_{d}),
w + \delta u e \bigr) p (u ) \rho' (n^{1/2} (v-w)) \ud u \ud w.
\end{split}
\end{equation*}
  {Since $\mu \star {\mathcal N}_{d}(0,(1/n)I_{d}$ has full support, 
we know from} 
the warm-up that
 $\partial_{\mu} U( \mu \star p^{\delta} \star {\mathcal N}_{d}(0,(1/n)I_{d}),\cdot)$ 
converges towards 
$\partial_{\mu} U( \mu  \star {\mathcal N}_{d}(0,(1/n)I_{d}),\cdot)$ as $\delta$ tends to $0$, uniformly on compact subsets. We deduce that, as $\delta$ tends to $0$, ${\mathcal W}^{n,\delta}(\mu,v)$
converges to
\begin{equation*}
\begin{split}
{\mathcal W}^{n}(\mu,v) &= n^{(d+1)/2} 
\int_{\R^d} \partial_{\mu} U \bigl( \mu  \star {\mathcal N}(0,\tfrac1n I_{d}),
 w \bigr) \rho' \bigl(n^{1/2} (v-w)\bigr) \ud w
\\
&= \partial_{v} \biggl( n^{d/2} 
\int_{\R} \partial_{\mu} U \bigl( \mu  \star {\mathcal N}(0,\tfrac1n I_{d}),
 w \bigr) \rho (n^{1/2} (v-w)) \ud w \biggr)
 = \partial_{v} {\mathcal V}^n(\mu,v). 
\end{split}
\end{equation*} 
Therefore, we deduce that the mappings $(\R^d \ni v \mapsto 
\text{Tr} \{
(\partial_{v} {\mathcal V}^n(\mu,v)) (e' \otimes e)
\}
)_{n \geq 1}$
are locally Lipschitz continuous, uniformly in $\mu$ (the local Lipschitz constant 
being at most of $\alpha$-polynomial growth).
Since $\partial_{v} {\mathcal V}^n(\mu,v)$ is 
independent of $e$,
this says that  the mappings $(\R^d \ni v \mapsto 
\partial_{v} {\mathcal V}^n(\mu,v))_{n \geq 1}$
are locally Lipschitz continuous, uniformly with respect to $\mu$
in sets of probability measures with uniformly bounded second-order moments. 

By \eqref{eq:29:11:5}
and \new{\textit{(i)}} in the statement of Theorem \ref{thm:ito:frechet},
\begin{equation}
\label{eq:proof:C2partial:unif:integrability}
\sup_{n \geq 1, \delta \in [0,1]}
{\mathbb E} \bigl[ \bigl\vert
\text{Tr}\bigl\{
 \bigl( {\mathcal W}^{n,\delta}(\mu,\xi) \bigr) (e' \otimes e) 
 \bigr\} \bigr\vert^2
\bigr] \leq C,
\end{equation}
for a possibly new value of $C$. Letting $\delta$ tend to $0$, we deduce from Fatou's lemma that 
$\sup_{n \geq 1}
{\mathbb E} [ \vert \text{Tr}\{
( \partial_{v} {\mathcal V}^{n}(\mu,\xi) )(e' \otimes e) \}\vert^2] \leq C$
and thus that
$\sup_{n \geq 1}
{\mathbb E} [ \vert  \partial_{v} {\mathcal V}^{n}(\mu,\xi) \vert^2] \leq C$,
which implies that, by local Lipschitz property of $\partial_{v}{\mathcal V}^n(\mu,\cdot)$
(the local Lipschitz constant being at most of $\alpha$-polynomial growth),
\begin{equation}
\label{eq:29:11:10}
\forall n \geq 1, \quad 
\inf_{\vert v \vert \leq 2 \vert \xi \vert_{2}}
 \vert \partial_{v}{\mathcal V}^{n}(\mu,v) \vert \leq C,
\end{equation}
where we used the same argument as in
\eqref{eq:trick:markov}. 
This says that the sequence of mappings $(\R^d \ni v \mapsto \partial_{v} {\mathcal V}^n(\mu,v))_{n \geq 1}$
is relatively compact for the topology of uniform convergence.
  {
By the warm-up, the sequence of 
functions
$(\R^d \ni v \mapsto ({\mathcal V}^n(\mu,v),\partial_{v} {\mathcal V}^n(\mu,v)))_{n \geq 1}$
is relatively compact for the topology of uniform convergence.
As any limit of the sequence 
$(\R^d \ni v \mapsto {\mathcal V}^n(\mu,v))_{n \geq 1}$
provides a version of $\partial_{\mu} U(\mu)$, we deduce that there 
exists a version of $\partial_{\mu} U(\mu) : \R^d \ni v 
\mapsto \partial_{\mu} U(\mu) (v)$ which is continuously differentiable 
with respect to $v$.
By \eqref{eq:proof:C2partial:unif:integrability}, we 
deduce that, for any $\mu \in {\mathcal P}_{2}(\R^d)$
and any $\xi \in L^2(\Omega,\cA,\P;\R^d)$ with $\mu$ as distribution,
${\mathbb E}
[ \vert 
\partial_{v} [\partial_{\mu} U(\mu)](\xi) \vert^2] \leq C$,
for a constant $C$ independent of $\mu$. Moreover,
passing to the limit in \eqref{eq:29:11:5} (first on $\delta$ and then on $n$),
we get}
\begin{equation}
\label{eq:29:11:50}
\partial^2_{\text{sign}(Z')e,\text{sign}(Z')\chi} {\mathcal U}(\xi)
 = 
 {\mathbb E} \bigl[ \text{Tr} \bigl\{ \bigl(
  \partial_{v} [\partial_{\mu} U(\mu)](\xi) \bigr) \bigl( \chi \otimes e \bigr) \bigr\}\bigr].
\end{equation}
Combining the above identity and 
point \new{\textit{(i)}} in the statement of Theorem 
\ref{thm:ito:frechet}, 
we recover the fact that 
$\int_{\R^d} \vert \partial_{v} [\partial_{\mu} U(\mu)](v) \vert^2 d\mu(v)
\leq C$, for a constant $C$ independent of $\mu$, 
which is a required condition for applying the chain rule. 
\vspace{5pt}

\textit{Last step.} {We have just found, for any 
$\mu \in {\mathcal P}_{2}(\R^d)$, 
a version 
of the mapping $\partial_{\mu} U(\mu)$ 
that
is differentiable in the variable $v$, with  
$\partial_{v}[\partial_{\mu} U(\mu)]$ denoting its derivative. 
In order to complete the proof, it remains to prove that 
the resulting mapping ${\mathcal P}_{2}(\R^d) \times \R^d \ni (\mu,v) \mapsto 
\partial_{v} [\partial_{\mu} U(\mu)](v)
$ is continuous in the joint variable $(\mu,v)$ at any point 
$v \in \textrm{Supp}(\mu)$.} We already know that it is locally Lipschitz continuous with respect to $v$, the local Lipschitz constant being at most of $\alpha$-polynomial growth, uniformly in $\mu$ in sets of probability measures with uniformly bounded second-order moments. 
  {
For a sequence $(\mu^n)_{n \geq 1}$
in ${\mathcal P}_{2}(\R^d)$
converging for the $2$-Wasserstein distance to some $\mu$, 
we deduce from the local Lipschitz property and by the same argument 
as in \eqref{eq:29:11:10} that the sequence of functions
$(\R \ni v \mapsto \partial_{v}[\partial_{\mu} U(\mu^n)](v))_{n \geq 1}$ is relatively compact for the topology of uniform convergence on compact subsets.}  
By means of the bound  $\sup_{n \geq 1}
{\mathbb E} [ \vert 
\partial_{v} [\partial_{\mu} U(\mu^n)](\xi^n) \vert^2] \leq C$, with 
$\xi^n \sim \mu^n$, it is quite easy to pass to the limit in the right-hand side of 
\eqref{eq:29:11:50}. By \textit{(ii)}
in the statement of the theorem, we can also pass to the limit in the left-hand side. 
Equation \eqref{eq:29:11:50}
then permits to identify any limit with $\partial_{v}[\partial_{\mu} U(\mu)]$
  {on the support of $\mu$.
Since the mappings $(\partial_{v}[\partial_{\mu} U(\mu^n)])_{n \geq 1}$
are uniformly continuous on compact subsets, we deduce that,
for an additional sequence 
$(v^n)_{n \geq 1}$, with values in $\R^d$, 
that converges to some $v \in \textrm{Supp}(\mu)$, 
the sequence $(\partial_{v}[\partial_{\mu} U(\mu^n)](v^n))_{n \geq 1}$
converges, up to a subsequence, to 
$\partial_{v}[\partial_{\mu} U(\mu)](v)$. 
Now, by relative compactness 
of the sequence $(\R \ni v \mapsto \partial_{v}[\partial_{\mu} U(\mu^n)](v))_{n \geq 1}$,
 the sequence $(\partial_{v}[\partial_{\mu} U(\mu^n)](v^n))_{n \geq 1}$ is bounded. 
 By a standard compactness argument,  
the sequence $(\partial_{v}[\partial_{\mu} U(\mu^n)](v^n))_{n \geq 1}$
must be convergent with 
$\partial_{v}[\partial_{\mu} U(\mu)](v)$ as limit.}
\qed
\end{proof}

\subsection{Proof of Theorem \ref{main:thm:uniqueness}}
In order to prove Theorem \ref{main:thm:uniqueness},
we first need an extension of the chain rule to functions that depend on 
time, space and measure: 

\begin{Proposition}
\label{prop:ito:general}
Consider an It\^o process 
$(X_{t})_{t \in [0,T]}$
driven by $(b_{t},\sigma_{t})_{t \in [0,T]}$
satisfying \eqref{eq:23:10:2} and a function 
$V : [0,T] \times \R^d \times \cP_{2}(\R^d) \rightarrow \R^m$
belonging to  $\bigcup_{\beta \geq 0} \cD_{\beta}$, see Definition 
\ref{eq:sol:admissible}. Then, $\P$ almost surely, for any $t \in [0,T]$,
\begin{equation*}
\begin{split}
&V\bigl(t,X_{t},\law{X_t}\bigr) - V\bigl(0,X_{0},\law{X_{0}}\bigr) 
\\
&= 
\int_{0}^t \Bigl( \partial_{t} V(r,X_{r},\law{X_r})+ \partial_{x} V(r,X_{r},\law{X_r})
b_r + \hesp{\partial_\mu V\bigl(r,X_{r},\law{X_r}\bigr)
\bigl(
\cc{X_r}\bigr)  
\langle b_r \rangle} \Bigr) \ud r 
 \\
 &\hspace{15pt} + \frac12 \int_0^t \Bigl(
 {\rm Tr} \bigl[ \partial^2_{xx} V \bigl(r,X_{r},\law{X_{r}} \bigr) 
 \bigl( \sigma_{r} \sigma_{r}^\dagger \bigr)
 \bigr] 
 \\
&\hspace{60pt} 
+ \hesp{
 \textrm{\rm Tr} \bigl[
 \partial_v \bigl[\partial_\mu V\bigr]\bigl(r,X_{r},\law{X_r}\bigr)\bigl(\cc{X_r}\bigr) \langle \sigma_{r}
  (\sigma_{r})^{\dagger} \rangle \bigr)
\bigr]
} \Bigr) \ud r
\\
&\hspace{15pt} + \int_{0}^t \partial_{x} V\bigl(r,X_{r},[X_{r}] 
\bigr)
\sigma_{r} \ud W_{r}. 
\end{split}
\end{equation*}
\end{Proposition}

\begin{Remark}
  {
In comparison with 
Theorem 
\ref{thm:full}, the 
formula is stated here in terms of the 
expectation $\hat{\E}$ on the auxiliary probability space
$(\hat{\Omega},\hat{\mathcal A},\hat{\P})$. 
The goal is to distinguish the random variables 
$X_r$, $b_{r}$ and $\sigma_{r} \sigma_{r}^{\dagger}$, observed on the ``physical space''
$(\Omega,{\mathcal A},\P)$, from 
the random variables $\langle {X}_{r} \rangle$, $\langle 
{b}_{r} \rangle$ and $\langle \sigma_{r} \sigma_{r}^{\dagger} \rangle$
that are used to express the derivatives in the direction $\mu$.} 
\end{Remark}

\begin{proof}
The proof is similar to that outlined in 
Subsection
\ref{subse:solution:PDE}. 
As in the proof of Theorem \ref{thm:full}, we can assume that 
the processes $(b_{t})_{t \in [0,T]}$
and $(\sigma_{t})_{t \in [0,T]}$ are bounded. 
Given $s \in [0,T]$ and $h>0$ such that $s+h \in [0,T]$, we then expand 
\begin{equation}
\label{eq:ito:general:1}
\begin{split}
&V\bigl(s+h,X_{s+h},[X_{s+h}]\bigr) - V\bigl(s,X_{s},[X_{s}]\bigr)
\\
&\hspace{15pt}= 
V\bigl(s+h,X_{s+h},[X_{s+h}] \bigr) - V\bigl(s+h,X_{s+h},[X_{s}]\bigr) 
\\
&\hspace{30pt}
+ 
V\bigl(s+h,X_{s+h},[X_{s}] \bigr)
- V\bigl(s,X_{s},[X_{s}] \bigr). 
\end{split}
\end{equation}
  {
Thanks to the regularity assumptions in \HYP{1} and \HYP{2},
we notice that, almost surely, the map $\cP_{2}(\R^d) \ni \mu \mapsto V(s+h,X_{s+h},\mu)$
satisfies the assumption of Theorem \ref{thm:ito:frechet}.} Therefore, we can 
write
\begin{equation*}
\begin{split}
&V\bigl(s+h,X_{s+h},[X_{s+h}] \bigr) - V\bigl(s+h,X_{s+h},[X_{s}]\bigr) 
\\
&= \int_{s}^{s+h}
\hat{\mathbb E} \bigl[ \partial_{\mu} V  (s+h,X_{s+h}, [X_{r}] )(\langle X_{r} \rangle) \langle b_{r} \rangle  \bigr] \ud r 
\\
&\hspace{15pt}+ 
\frac{1}{2} \int_{s}^{s+h} \hat{\mathbb E} \bigl[ {\rm Tr} \bigl( 
\partial_{v} \bigl[ \partial_{\mu} V(s+h,X_{s+h},[X_{r}]) \bigr](\langle X_{r} \rangle)
 \langle \sigma_{r} \sigma_{r}^{\dagger} \rangle \bigr) 
\bigr] \ud r. 
\end{split}
\end{equation*}
  {Recall that any versions of 
$\R^d \ni v \mapsto 
\partial_{\mu} V(s+h,X_{s+h},\mu)(v)$ 
and 
$\R^d \ni v \mapsto 
\partial_{v}[\partial_{\mu} V(s+h,X_{s+h},\mu)](v)$
may be used in the writing of the above formula. 
In particular, we can choose the versions of 
$\partial_{\mu}V$ and $\partial_{v}[\partial_{\mu} V]$
that satisfy \HYP{1} and 
\HYP{2}.} By the assumption we have on the regularity of 
$\partial_{\mu}V$ and $\partial_{v}[\partial_{\mu} V]$ in the variable $x$, see \HYP{1} and 
\HYP{2}, and in the variable $t$, see \textit{(ii)} in Definition 
\eqref{eq:sol:admissible}, we deduce that there exists a sequence 
of non-negative random variables $(\varepsilon_{h})_{h >0}$ that tends to $0$ in probability
with $h$,  such that 
\begin{equation*}
\begin{split}
&\biggl\vert V\bigl(s+h,X_{s+h},[X_{s+h}] \bigr) - V\bigl(s+h,X_{s+h},[X_{s}]\bigr) 
\\
&- \int_{s}^{s+h}
\hat{\mathbb E} \bigl[ \partial_{\mu} V  (r,X_{r}, [X_{r}] )(\langle X_{r} \rangle) \langle b_{r} \rangle  \bigr] \ud r 
\\
&\hspace{15pt}- 
\frac{1}{2} \int_{s}^{s+h} \hat{\mathbb E} \bigl[ {\rm Tr} \bigl( 
\partial_{v} \bigl[ \partial_{\mu} V(r,X_{r},[X_{r}]) \bigr](\langle X_{r} \rangle)
 \langle \sigma_{r} \sigma_{r}^{\dagger} \rangle \bigr) 
\bigr] \ud r \biggr\vert \leq h \varepsilon_{h}.
\end{split}
\end{equation*}
  {It must be noticed that the family 
$(\varepsilon_{h})_{h >0}$ may be chosen independently of $s \in [0,T]$. The reason is that, 
thanks to
 \HYP{1} and 
\HYP{2}, for any continuous $\R^d$-valued 
path $(x_{t})_{t \in [0,T]}$,
\begin{equation*}
\lim_{\delta \rightarrow 0}
\sup_{r,s \in [0,T] : \vert r-s \vert \leq \delta}
\hat{\mathbb E}
\Bigl[ 
\bigl\vert 
\partial_{\mu} V\bigl(s,x_{s},[X_{r}] \bigr)
\bigl( \langle X_{r} \rangle \bigr) 
- 
\partial_{\mu} V\bigl(r,x_{r},[X_{r}] \bigr)
\bigl( \langle X_{r} \rangle \bigr) 
\bigr\vert
\Bigr] = 0,
\end{equation*}
with a similar 
result when 
$\partial_{\mu} V$ is replaced by 
$\partial_{v}[\partial_{\mu} V]$.}
By means of the standard It\^o formula, the second difference on the right hand side of   \eqref{eq:ito:general:1}
can be handled in a similar way, yielding a similar bound (for the relevant expansion)
on an interval of length $h$. We then easily complete the proof
by dividing any interval $[0,t] \subset [0,T]$ into pieces of length less than $h$, applying  
the above bound on each piece of the subdivision 
and then by letting $h$ tend to $0$.
 \qed
\end{proof}
\vspace{5pt}

\noindent We now turn to \\[2mm]
\begin{proof}[Proof of Theorem \ref{main:thm:uniqueness}.]
The proof is a variant of the four-step-scheme used in 
\cite{ma:protter:yong}. We divide it into two steps. 

\textit{First step.}
Given a solution $U$ 
to \eqref{eq:master:PDE} in the class $\bigcup_{\beta \geq 0} \cD_{\beta}$ and given  
$t \in [0,T]$ and 
$\xi \in L^2(\Omega,\cF_{t},\P;\R^d)$,
we build a solution to
\eqref{eq X-Y t,xi}.

Letting 
\[
\partial^\sigma_{x} U(t,x,\mu) 
= \partial_{x} U(t,x,\mu) \sigma\bigl(x,U(t,x,\mu),[(\xi,U(t,\xi,\mu))]\bigr),
\] with $\xi \sim \mu$, we indeed claim that the 
McKean-Vlasov SDE 
\begin{equation}
\label{eq:MKV:SDE}
\begin{split}
dX_{s} &= b\bigl(X_{s},U(s,X_{s},[X_{s}]),\partial^{\sigma}_{x} U(s,X_{s},[X_{s}]), \bigl[X_{s},U(s,X_{s},[X_{s}]) \bigr] \bigr)
\ud s 
\\
&\hspace{15pt}+ \sigma\bigl( X_{s},U(s,X_{s},[X_{s}]), \bigl[X_{s},U(s,X_{s},[X_{s}]) \bigr] \bigr) \ud W_{s}, 
\quad X_{t}= \xi, 
\end{split}
\end{equation}
has a solution (the idea that we shall exploit in the proof being that the triplet process $(X_{s},U(s,X_{s},[X_{s}]),\partial_{x}^\sigma U(s,X_{s},[X_{s}]))_{s \in [t,T]}$
solves the system \eqref{eq X-Y t,xi}).
The proof is not completely straightforward as $\partial_{x}^\sigma U$ is not 
Lipschitz continuous in the direction of the measure (see \HYP{1}). In particular, 
we cannot apply Sznitman's result in \cite{Sznitman}, which relies on a contraction argument. Instead, we make use of Schauder's theorem, applying the same 
strategy as in 
\cite{carmona:delarue:ecp}. 

The argument is as follows. Let $\cC([t,T],\cP_{2}(\R^d))$ be the family of marginal measures
$(\mu_{r})_{r \in [t,T]}$ with finite second-order moments such that the mapping $[t,T] \ni r \mapsto \mu_{r} \in \cP_{2}(\R^d)$
is continuous. For $(\mu_{r})_{r \in [t,T]}\in \cC([t,T],\cP_{2}(\R^d)), $ we may solve  
\begin{equation*}
\begin{split}
dX_{s} &= b\bigl(X_{s},U(s,X_{s},\mu_{s}),\partial^{\sigma}_{x} U(s,X_{s},\mu_{s}),\bigl[X_{s},U(s,X_{s},\mu_{s})
\bigr] 
\bigr)
\ud s 
\\
&\hspace{15pt}+ \sigma\bigl( X_{s},U(s,X_{s},\mu_{s}) , \bigl[X_{s},U(s,X_{s},\mu_{s}) \bigr] \bigr) \ud W_{s}, \quad X_{t}= \xi.
\end{split}
\end{equation*}
By Sznitman's result, 
the above equation admits a unique solution, which we will denote by $(X_{s}^{(\mu_{r})_{r \in [t,T]}})_{s \in [t,T]}$. 
We then consider the mapping 
$$\Phi :  
\bigl(\mu_{r}\bigr)_{r \in [t,T]} \mapsto \bigl(\bigl[X_{s}^{(\mu_{r})_{r \in [t,T]}}\bigr]\bigr)_{s \in [t,T]},$$ 
which maps 
$\cC([t,T],\cP_{2}(\R^d))$ into itself. By standard stability arguments, it is quite clear that the mapping 
$\Phi$ is continuous, $\cC([t,T],\cP_{2}(\R^d))$ being endowed with the supremum distance
$d((\mu_{r})_{r \in [t,T]},(\tilde{\mu}_{r})_{r \in [t,T]})
:= \sup_{r \in [t,T]} W_{2}(\mu_{r},\tilde{\mu}_{r})$. 
Moreover, by boundedness of $\partial_{x} U$ and $\sigma$ and by the Lipschitz property of 
$U$, 
we can find a constant $C$ (independent of the input $(\mu_{r})_{r \in [t,T]}$) such that, for any 
$S \in [t,T]$
\begin{equation*}
{\mathbb E}_{t} \bigl[
\sup_{s \in [t,S]} \vert X_{s}^{(\mu_{r})_{r \in [t,T]}} \vert^4 \bigr]^{1/2} \leq C \biggl( 1 + \vert \xi \vert^2
+ \int_{t}^S \int_{\R^d} \vert x \vert^2 d\mu_{s}(x) \ud s \biggr).
\end{equation*}
This proves that, when 
\begin{equation}
\label{eq:schauder:1}
\forall s \in [t,T], \quad 
\int_{\R^d} \vert x \vert^2 d\mu_{s}(x)  \leq C(1+ \| \xi \|_{2}^2) \exp\bigl(C(s-t)\bigr),
\end{equation}
the same holds for $\E[\vert X_{s}^{(\mu_{r})_{r \in [t,T]}}\vert^2]$ for all 
$s \in [t,T]$. In such a case, we also have
\begin{equation*}
{\mathbb E}_{t} \bigl[
\sup_{s \in [t,S]} \vert X_{s}^{(\mu_{r})_{r \in [t,T]}} \vert^4 \bigr]^{1/2} \leq C \bigl(1 + \vert \xi \vert^2
\bigr)
+ C \bigl(1+ \| \xi \|_{2}^2\bigr) \exp(CT). \label{eq:XmuisUI}
\end{equation*}
so that, for any event $A \in \cA$ and any real $R>0$, Cauchy-Schwarz inequality yields  
\begin{equation*}
\begin{split}
\E \bigl[ {\mathbf 1}_{A}
\sup_{s \in [t,S]} \vert X_{s}^{(\mu_{r})_{r \in [t,T]}} \vert^2 \bigr]
&\leq 
C \E \Bigl[ \E_{t}[{\mathbf 1}_{A}]^{1/2} 
\Bigl( \bigl(1 + \vert \xi \vert^2
\bigr)
+  \bigl(1+ \| \xi \|_{2}^2\bigr) \exp(CT)
\Bigr)
 \Bigr],
\\
&\leq C   
\Bigl( \bigl(1 + R^2
\bigr)
+  \bigl(1+ \| \xi \|_{2}^2\bigr) \exp(CT)
\Bigr)
\P(A)^{1/2} 
\\
&\hspace{15pt} + C \E \Bigl[ {\mathbf 1}_{\{ \vert \xi \vert \geq R\}} 
\Bigl( \bigl(1 + \vert \xi \vert^2
\bigr)
+  \bigl(1+ \| \xi \|_{2}^2\bigr) \exp(CT)
\Bigr)
 \Bigr].
\end{split}
\end{equation*}
In particular, choosing $A = \{\vert X_{s}^{(\mu_{r})_{r \in [t,T]}} \vert > R^4\}$
for some $s \in [t,T]$, 
applying Markov inequality and using the fact that
\eqref{eq:schauder:1} is also satisfied by $\E[\vert X_{s}^{(\mu_{r})_{r \in [t,T]}}\vert^2]$, 
we get that 
\begin{equation}
\label{eq:schauder:2}
\begin{split}
&\sup_{s \in [t,S]}\E \bigl[ {\mathbf 1}_{ \{\vert X_{s}^{(\mu_{r})_{r \in [t,T]}} \vert > R^4\}}
 \vert X_{s}^{(\mu_{r})_{r \in [t,T]}} \vert^2 \bigr]
\\
&\leq C^{3/2}  R^{-4}
\Bigl( \bigl(1 + R^2
\bigr)
+  \bigl(1+ \| \xi \|_{2}^2\bigr) \exp(CT)
\Bigr) \Bigl( \bigl(1+ \| \xi \|_{2}^2\bigr) \exp(CT) \Bigr)^{1/2}
\\
&\hspace{15pt} + C \E \Bigl[ {\mathbf 1}_{\{ \vert \xi \vert \geq R\}} 
\Bigl( \bigl(1 + \vert \xi \vert^2
\bigr)
+  \bigl(1+ \| \xi \|_{2}^2\bigr) \exp(CT)
\Bigr)
 \Bigr].
\end{split}
\end{equation}
For a given $s \in [t,T]$, 
we now denote by $\cK_{s}$ the subset of $\cP_{2}(\R^d)$
made of probability measures such that 
$\int_{\R^d} \vert x \vert^2 \ud \mu(x)$ is less than 
the right-hand side in \eqref{eq:schauder:1}
and $\int_{\{\vert x \vert > R^4\}} \vert x \vert^2 \ud \mu(x)$ is
less than the right-hand side in \eqref{eq:schauder:2} for any $R >0$. 
It is easy to checked that $\cK_{s}$ is a compact subset of $\cP_{2}(\R^d)$. 
Indeed, 
any sequence in $\cK_{s}$ is tight and admits a subsequence that 
converges in the weak sense. 
Using \eqref{eq:schauder:2}, the subsequence is square-uniformly integrable. 
Using Skorohod's representation theorem, we deduce that 
the sequence converges in the $W_{2}$-Wasserstein sense. 
By Fatou's lemma, $\cK_{s}$ is closed. Below, we let 
$\cK = \{ (\mu_{r})_{r \in [t,T]} 
\in \cC([t,T],\cP_{2}(\R^d)) : \forall r \in [t,T], \ \mu_{r} \in \cK_{r}\}$. 

Notice now that, under \eqref{eq:schauder:1}, we have, for all $s,s' \in [t,T]$, 
\begin{equation*}
\E \bigl[ \vert X_{s'}^{(\mu_{r})_{r \in [t,T]}}
- X_{s}^{(\mu_{r})_{r \in [t,T]}} \vert^2 
\bigr] \leq C' \vert s'-s\vert,
\end{equation*}
for a constant $C'$ depending upon $C$, $\| \xi \|_{2}$ and $T$.  
This says that
the map is $[t,T] \ni s \mapsto X_{s}^{(\mu_{r})_{r \in [t,T]}}
\in L^2(\Omega,{\mathcal A},\P;\R^d)$, 
is continuous, uniformly in $(\mu_{r})_{r \in [t,T]} \in \cK$. 
Using the Arzel\`a-Ascoli theorem, we deduce that the restriction of $\Phi$ 
to $\cK$ has a relatively compact range. Since $\cK$ is closed 
and convex, Schauder's theorem applies and \eqref{eq:MKV:SDE} has a solution. 

\textit{Second step.}
 We consider 
another solution $U'$ to \eqref{eq:master:PDE} in the class $\bigcup_{\beta \geq 0}
\cD_{\beta}$. 
With $X$ a solution of 
\eqref{eq:MKV:SDE}, we can apply the chain rule to 
both $(Y_{s}=U(s,X_{s},[X_{s}]))_{s \in [t,T]}$
and $(Y_{s}'=U'(s,X_{s},[X_{s}]))_{s \in [t,T]}$, the drift of $X$ being 
square-integrable and $\sigma$ being bounded.
Letting $(Z_{s} = \partial_{x}^\sigma U(s,X_s,[X_{s}]))_{s \in [t,T]}$,
$(Z_{s}' = \partial_{x} U'(s,X_{s},[X_{s}]) \sigma(X_{s},Y_{s}',[X_{s},Y_{s}']))_{s \in [t,T]}$, 
$(\theta_{s}=(X_{s},Y_{s},Z_{s}))_{s \in [t,T]}$, 
$(\theta_{s}'=(X_{s},Y_{s}',Z_{s}'))_{s \in [t,T]}$,
$(\theta_{s}^{(0)}=(X_{s},Y_{s}))_{s \in [t,T]}$
and $(\theta_{s}^{(0)\prime}=(X_{s},Y_{s}'))_{s \in [t,T]}$,
we deduce from the master PDE \eqref{eq:master:PDE} that
\begin{equation*}
\begin{split}
Y_{s}-Y_{s}' &= \int_{s}^T
\bigl( f(\theta_{r},[\theta_{r}^{0)}]) - f(\theta_{r},[\theta_r^{(0)\prime}]) \bigr) \ud r
\\
&\hspace{15pt} + \int_{s}^T \Bigl\{
\partial_{x} U'(r,X_{r},[X_{r}])
\Bigl( 
 b\bigl(\theta_{r},[\theta_{r}^{(0)}]\bigr)
 -
 b\bigl(\theta_{r}',[\theta_{r}^{(0)\prime}]\bigr)
\Bigr) 
\\
&\hspace{30pt} +  \hat{\E} \Bigl[
\partial_{\mu} U'(r,X_{r},[X_{r}])(\langle X_{r} \rangle)
\Bigl(
 b\bigl(\langle \theta_{r} \rangle,[\theta_{r}^{(0)}]\bigr)
 -
  b\bigl( \langle \theta_{r}'\rangle ,[\theta_{r}^{(0)\prime}]\bigr)
\Bigr) \Bigr] \Bigr\} \ud r
\\
 &\hspace{15pt} + \frac12 \int_s^T \Bigl\{
 {\rm Tr} \Bigl[ \partial^2_{xx} U' \bigl(r,X_{r},\law{X_{r}} \bigr) 
 \Bigl(
   \bigl( \sigma \sigma^\dagger \bigr)
 \bigl(\theta_{r},\law{\theta^{(0)}_{r}} \bigr)
 -
 \bigl( 
 \sigma \sigma^\dagger \bigr)
 \bigl( \theta_{r}',\law{\theta^{(0)\prime}_{r}} \bigr)
 \Bigr)
  \Bigr] 
 \\
&\hspace{30pt} + \hat{\E} \Bigl[
 \textrm{\rm Tr} \Bigl[
 \partial_v \bigl[\partial_\mu U'\bigr]\bigl(r,X_{r},\law{X_r}\bigr)\bigl(\cc{X_r}\bigr)
\\
&\hspace{100pt}
 \times 
 \Bigl(
\bigl( \sigma
  \sigma^{\dagger}\bigr)\bigl(  
\cc{\theta^{(0)}_r},\law{\theta^{(0)}_r}\bigr)
-
 \bigl( \sigma
  \sigma^{\dagger}\bigr)\bigl(  
\cc{\theta^{(0)\prime}_r},\law{\theta^{(0)\prime}_r}\bigr) 
\Bigr)
\Bigr]
\Bigr] \Bigr\} \ud r
\\
&\hspace{15pt} - \int_{s}^T \bigl(Z_{r} - \partial_{x} U'(r,X_{r},[X_{r}])
\sigma(\theta_{r}^{(0)\prime},[\theta_{r}^{(0)\prime}]) \bigr) \ud W_{r}.
\end{split}
\end{equation*}
By using Assumptions   {\HYP{0}(i)} and \HYP{\sigma} on the coefficients and 
Assumptions \HYP{1} and \HYP{2} that enter in the definition of $\cD_{\beta}$, we deduce 
from stability estimates for BSDEs, in the spirit of \cite{pardoux:peng},
that 
\begin{equation*}
\begin{split}
&\E \bigl[ \vert Y_{s} - Y_{s}' \vert^2 \bigr] + \E \int_{s}^T 
\vert Z_{r} - \partial_{x} U'(r,X_{r},[X_{r}])
\sigma(\theta_{r}^{(0)\prime},[\theta_{r}^{(0)\prime}])
 \vert^2 \ud r
\\
&\leq C \E \int_{s}^T  \vert Y_{r} - Y_{r}' \vert^2 \ud r + 
\frac12 \E \int_{s}^T 
\vert Z_{r} - Z_{r}' \vert^2  \ud r,
\end{split}
\end{equation*}
from which we get, by the boundedness of $\partial_{x} U'$, that 
\begin{equation*}
\begin{split}
&\E \bigl[ \vert Y_{s} - Y_{s}' \vert^2 \bigr] + \E \int_{s}^T 
\vert Z_{r} - Z_{r}'
 \vert^2 \ud r
\leq C \E \int_{s}^T \vert Y_{r} - Y_{r}' \vert^2 \ud r +
\frac12 \E \int_{s}^T 
\vert Z_{r} - Z_{r}' \vert^2  \ud r.
\end{split}
\end{equation*}
We deduce that $Y_{s}=Y_{s}'$ for any $s \in [t,T]$, that 
is $U(t,\xi,[\xi]) = U'(t,\xi,[\xi])$ almost surely. 
When $[\xi]$ has full support over $\R^d$, continuity of 
$U$ and $U$' yield $U(t,x,[\xi]) = U'(t,x,[\xi])$ for all $x \in \R^d$. 
When the support of $[\xi]$ is strictly included in $\R^d$, 
we can approximate $\xi$  
by a sequence $(\xi_{n})_{n \geq 1}$
that converges to $\xi$ in $L^2$ such that, for each $n \geq 1$, 
$\xi_{n}$ has full support over $\R^d$. Passing to the limit in 
the relationship $U(t,x,[\xi_{n}]) = U'(t,x,[\xi_{n}])$, we complete the proof.  
\qed
\end{proof}

\section{Smoothness for small time horizons -- proof of Theorem \ref{main:thm:short:time}}
\label{se smoothness}
The purpose of this section is to prove that the mapping $U$ given in
Definition \ref{le dec field} satisfies the required smoothness property for applying 
the chain rule. Generally speaking, this is proved by showing the smoothness  of the corresponding stochastic flows defined in \eqref{eq X-Y t,xi} and \eqref{eq X-Y t,x,mu}. More precisely, 
we prove that the stochastic flows defined in \eqref{eq X-Y t,xi} and \eqref{eq X-Y t,x,mu} are differentiable with respect to $\xi$, $x$ and $\mu$ in the sense discussed in Section 
\ref{se chain rule}. This is not a straightforward generalization of the method used by Pardoux and Peng 
in \cite{pardoux:peng} in order to prove the smoothness of the flow generated by the solution of a classical backward stochastic differential equation as we are facing here two additional difficulties: First, the initial conditions live in non-Euclidean spaces, which requires some special care; second, the backward equation is fully coupled to the forward equation. 
In order to handle the full coupling between the forward and backward components, we shall assume that $T$ is small. In particular, throughout the whole section, $T$ is less than $1$. In the following section, we shall give sufficient conditions for extending the results from small to arbitrary 
large time horizons. 

Below, Assumption \HYP{2} is in force. 
We shall use quite intensively the following lemma, which is an adaptation 
of the stability estimates in \cite{del02}: 
\begin{Lemma}
\label{le app cont}
For any $p \geq 1$, 
there exist two constants
$c_{p}:=c_p(L)>0$ and $C_{p} 
\geq 0$ such that, for $T \leq c_{p}$,
for any $t \in [0,T]$, $x \in \R^d$ and $\xi \in L^2(\Omega,\cF_{t},\P;\R^d)$,
\begin{equation} 
\label{eq app bound 1}
\begin{split}
 &\NSt{p}{X^{t,\xi}} + \NSt{p}{Y^{t,\xi}}
+ \NHt{p}{Z^{t,\xi}}
 \le
  C_p \bigl(1+|\xi|+\NL{2}{\xi'} \bigr),
 \\
 &\NS{p}{X^{t,x,\law{\xi}}} + \NS{p}{Y^{t,x,\law{\xi}}}
+ \NH{p}{Z^{t,x,\law{\xi}}}
  \le
  C_p \bigl(1+|x|+\NL{2}{\xi}\bigr),
 \end{split}
 \end{equation}
 and, for any $x' \in \R^d$ and $\xi' \in L^2(\Omega,{\mathcal F}_{t},\P;\R^d)$,
\begin{equation} 
\label{eq app bound 2}
\begin{split}
 &\NSt{p}{X^{t,\xi}-X^{t,\xi'}} + \NSt{p}{Y^{t,\xi} - Y^{t,\xi'}}
+ \NHt{p}{Z^{t,\xi} - Z^{t,\xi'}}
\\
 &\hspace{15pt} \le
  C_p \bigl[|\xi - \xi'|+W_{2}\bigl([\xi],[\xi'] \bigr)\bigr],
 \\
 &\NS{p}{X^{t,x,\law{\xi}}-X^{t,x',\law{\xi'}}} + \NS{p}{Y^{t,x,\law{\xi}} - Y^{t,x',\law{\xi'}}}
+ \NH{p}{Z^{t,x,\xi} - Z^{t,x',\law{\xi'}}}
  \\
 &\hspace{15pt} \le
  C_p \bigl[|x-x'|+W_{2}\bigl([\xi],[\xi']\bigr)\bigr].
 \end{split}
 \end{equation}
\end{Lemma}

In the statement above, the notation $c_p := c_p(L)$ emphasizes the fact that 
$c_p$ only depends on the Lipschitz constant $L$ introduced in 
$\HYP{0}-\HYP{1}$. The constant $C_p$  is allowed to depend on 
the other parameters appearing in  $\HYP{0}-\HYP{2}$, but there is no need to keep track of 
them for our purpose. 

\subsection{Stability estimate for McKean-Vlasov linear FBSDEs}
\label{subse:stab:1}
The strategy for investigating the derivatives of the solutions to 
\eqref{eq X-Y t,xi} and \eqref{eq X-Y t,x,mu} is standard. We  identify the derivatives with the solutions of  linearized systems, obtained by formal differentiation of the coefficients. For that reason, the analysis of the differentiability relies on some 
preliminary 
stability estimates for linear FBSDEs. Unfortunately, because of the 
McKean-Vlasov structure of the coefficients, we cannot borrow any estimate from the literature. We thus have to use a tailor-made version, which is the precise purpose of this subsection. 

\subsubsection{General set-up}
Generally speaking, we are dealing with a linear FBSDE of the form 
\begin{equation}
\label{eq:sys:linear:mkv}
\begin{split}
&{\mathcal X}_{s} = \eta + \int_{t}^s B\bigl(r,\theta_{r},\langle \hat\theta_{r} \rangle)
\bigl(\vartheta_{r},\langle 
\hat\vartheta_{r}^{(0)} \rangle\bigr) \ud r 
+ \int_{t}^s \Sigma\bigl(r,\theta_{r}^{(0)},\langle \hat\theta_{r}^{0)} \rangle \bigr) \bigl(\vartheta_{r}^{(0)},\langle \hat\vartheta_{r}^{0)} \rangle \bigr) \ud W_{r},
\\
&{\mathcal Y}_{s} = G\bigl(X_{T},\langle \hat X_{T} \rangle)\bigl({\mathcal X}_{T},\langle \new{ \hat {\mathcal X}_{T}} \rangle \bigr) + \int_{s}^T 
F\bigl(r,\theta_{r}, \langle \hat \theta_{r}^{(0)} \rangle \bigr)\bigl(\vartheta_{r},\langle \hat \vartheta_{r}^{(0)} \rangle \bigr) \ud r - \int_{s}^T {\mathcal Z}_{r} \ud W_{r},
\end{split}
\end{equation}
where $\eta$ is an initial condition in $L^2(\Omega,\cF_{t},\P;\R^d)$, 
$\theta = (X,Y,Z)$
and $\hat \theta = (\hat X,\hat Y,\hat Z)$
 are solutions of \eqref{eq X-Y t,xi} or \eqref{eq X-Y t,x,mu}, 
 $\vartheta = ({\mathcal X},{\mathcal Y},{\mathcal Z})$
 denotes the unknowns in the above equation
and $\hat \vartheta = (\hat {\mathcal X},\hat {\mathcal Y},\hat {\mathcal Z})$
is an 
auxiliary process, which may be $\vartheta$ itself (in which case it is unknown). 
The exponent $(0)$ denotes the restriction of the processes to the two first coordinates, 
as in \eqref{eq X-Y t,xi}
and \eqref{eq X-Y t,x,mu}. 
The processes $X$, $\hat{X}$, $\cX$ and 
$\hat{\cX}$ have the same dimension, the same being true
for the processes $Y$, $\hat{Y}$, $\cY$ and $\hat \cY$ and
for the processes
$Z$, $\hat{Z}$, $\cZ$ and $\hat \cZ$.
In particular, the mappings 
$B$, $\Sigma$, $F$ and $G$ take values in Euclidean spaces of according dimensions. 
The symbol $\langle \cdot \rangle$ is used to denote the copy of the underlying random variable onto the probability 
space $(\hat{\Omega},\hat{\cA},\hat{\P})$. Although the role of the copy is rather vague at this stage of the paper, it 
indicates that the coefficients may depend in a non-Markovian way of the various stochastic processes involved. Here is an example:

\begin{Example} \label{ex appli}
As a typical example for the coefficients 
$B$, $\Sigma$, $F$ and $G$, we 
may think of the derivatives, with respect to some parameter 
$\lambda$, of the original coefficients 
$b$, $f$, $\sigma$ and $g$ when computed
along some triplet $(\theta^{\lambda}=(X^{\lambda},Y^{\lambda},Z^{\lambda}))$
solving \eqref{eq intro XYZ}. 
As a typical example for the parametrization by $\lambda$, we may think of the parametrization with respect to the initial condition 
which is applied to the entire system. 

The shape of the coefficients $B$, $\Sigma$, $F$ and $G$ can then be derived by 
replacing $b$, $f$, $\sigma$ and $g$ by a generic continuously differentiable Lipschitz function 
$h : (\R^d \times \R^m \times \R^{m \times d}) \times {\mathcal P}_{2}(\R^d \times 
\R^m) \rightarrow \R$. Given such a generic $h$, we can indeed consider a process of the form 
\begin{equation*}
\Bigl( 
h\bigl(\theta_{r}^{\lambda},[\theta_r^{\lambda,(0)}] \bigr)\Bigr)_{r \in [t,T]} 
\end{equation*}
where $\R \ni \lambda \mapsto 
(\theta_{r}^{\lambda})_{r \in [t,T]}
\in {\mathcal S}^2([t,T];\R^d) \times {\mathcal S}^2([t,T];\R^m)
\times {\mathcal H}^2([t,T];\R^{m \times d})$
is differentiable 
with respect to $\lambda$, with (derivatives being taken in the aforementioned space)
\begin{equation*}
\theta_{r}^{\lambda}{}_{\vert\lambda=0} = \theta_{r}, \quad \frac{\ud}{\ud\lambda}_{\vert \lambda=0}
\theta_{r}^{\lambda} = \vartheta_{r}, \quad r \in [t,T],
\end{equation*}
the process $(\vartheta_{r})_{r \in [t,T]}$ taking its values in $\R^d \times \R^m \times \R^{m \times d}$ (and, for the moment, having nothing to do with the solution of
\eqref{eq:sys:linear:mkv}). 
Then, it is easy to check that the mapping $\R \ni \lambda \mapsto 
(h(\theta_{r}^{\lambda},\law{\theta_{r}^{\lambda,(0)}})_{r \in [t,T]} \in 
{\mathcal H}^2([t,T];\R)$ is differentiable and that the
derivative reads as follows
  \begin{equation}
\label{eq:ex:mkv:stability}
H^{(1)}(r,\theta_{r},\langle \theta_{r}^{(0)} \rangle)(\vartheta_{r},\langle \vartheta_{r}^{(0)} \rangle) := 
\partial_{w} h (\theta_{r},[\theta_{r}^{(0)}]) \vartheta_{r} +  \hat{\mathbb E}
\bigl[ \partial_{\mu} h(\theta_{r},[\theta_{r}^{(0)}])(\langle \theta_{r}^{(0)}\rangle) 
\langle \vartheta_{r}^{(0)} \rangle \bigr]. 
\end{equation}
Of course, if $h$ only acts on $((\theta^{(0)}_{r},[\theta_{r}^{(0)}]))_{r \in [t,T]}$
instead of $((\theta_{r},[\theta_{r}^{(0)}]))_{r \in [t,T]}$, then differentiability 
holds in ${\mathcal S}^2([t,T];\R)$. 
\end{Example}


In Example \ref{ex appli}, the coefficients $B$, $\Sigma$, $F$ and 
$G$ are obtained by replacing $h$ by $b$, $\sigma$, $f$ and $g$
and by computing $B^{(1)}$, $\Sigma^{(1)}$, $F^{(1)}$ and 
$G^{(1)}$ accordingly. Leaving Example \ref{ex appli} and going back to the general case, 
we apply the same procedure: In order to specify the shape of $B$, $\Sigma$, $F$ and $G$ (together with the assumptions they satisfy), it suffices to make explicit the generic form of a function $H$ that may be $B$, $\Sigma$, $F$ or $G$ and 
to detail the assumptions it satisfies. 
Given square-integrable 
processes $(V_{r})_{r \in 
[t,T]}$ and $(\hat V_{r})_{r \in [t,T]}$, 
$(V_{r})_{r \in 
[t,T]}$ 
possibly matching 
$(X_{r})_{r \in [t,T]}$,  
 $(\theta_{r}^{(0)})_{r \in [t,T]}$ or $(\theta_{r})_{r \in [t,T]}$, and similarly for 
 $(\hat V_{r})_{r \in [t,T]}$, together with
other square-integrable  
processes $({\mathcal V}_{r})_{r \in 
[t,T]}$ and $(\hat{\mathcal V}_{r})_{r \in 
[t,T]}$, $({\mathcal V}_{r})_{r \in 
[t,T]}$ possibly matching
 $(\cX_{r})_{r \in [t,T]}$, $(\vartheta_{r}^{(0)})_{r \in [t,T]}$ or $(\vartheta_{r})_{r \in [t,T]}$,
and similarly for $(\hat{\mathcal V}_{r})_{r \in 
[t,T]}$, 
we thus assume that $H(r,V_{r},\langle \hat V^{(0)}_{r} \rangle)$ acts on 
$({\mathcal V}_{r},\hat {\mathcal V}_{r}^{(0)})$ in the following
way:
\begin{equation}
\label{splitting:H}
H(r,V_{r},\langle \hat V^{(0)}_{r} \rangle)({\mathcal V}_{r},\hat {\mathcal V}_{r}^{(0)}) = H_{\ell}(V_{r},\langle \hat V^{(0)}_{r} \rangle)({\mathcal V}_{r},\hat {\mathcal V}_{r}^{(0)}) + 
H_{a}(r),
\end{equation}
where $H_{a}(r)$ is square-integrable
and $H_{\ell}(V_{r},\langle \hat V^{(0)}_{r} \rangle)$ acts linearly
on $({\mathcal V}_{r},\hat {\mathcal V}_{r}^{(0)})$
 in the following sense
\begin{equation}
\label{eq:mkv:stability:decomposition}
H_{\ell}(V_{r},\langle \hat V^{(0)}_{r} \rangle)({\mathcal V}_{r},\hat {\mathcal V}_{r}^{(0)}) = 
h_{\ell}(V_{r},\langle \hat V^{(0)}_{r} \rangle) {\mathcal V}_{r} + \hat{\mathbb E}
\bigl[ \hat{H}_{\ell}(V_{r},\langle \hat V^{(0)}_{r} \rangle) \langle \hat {\mathcal 
V}^{(0)}_{r} \rangle \bigr].
\end{equation}
Here $h_{\ell}$ and $\hat{H}_{\ell}$ are maps from $\R^k 
\times L^2(\hat\Omega,\hat{\mathcal A},\hat\P;\R^{l})$
into $\R^{l'}$ and
from $\R^{k} \times L^2(\hat{\Omega},\hat{{\mathcal A}},\hat{\P};
\R^{l})$ into $L^2(\hat{\Omega},\hat{{\mathcal A}},\hat{\P};\R^{l''})$ respectively, 
for suitable $k$, $l$, $l'$ and $l''$. Moreover,  there exist 
three constants $C,K,\alpha \geq 0$ and a function $\Phi_{\alpha} : 
[ L^2(\Omega,{{\mathcal A}},{\P};
\R^{l})]^2 \rightarrow \R_{+}$, continuous at any point of the diagonal, such that,
for $w,w' \in \R^k$ and $ \hat{V}^{(0)}, \hat{V}^{(0)\prime} \in L^2(\Omega,\cA,\P;\R^d)$,
\begin{align}
&\bigl\vert h_{\ell}(w,\langle \hat{V}^{(0)} \rangle) \bigr\vert + \hat{\mathbb E} \bigl[ \vert 
\hat{H}_{\ell}(w,\langle \hat V^{(0)} \rangle)
\vert^2 \bigr]^{1/2} \leq K,
\label{assumption:mkv:1}
\\
&\vert \hat{H}_{\ell}(w,\langle \hat V^{(0)} \rangle) \vert \leq C \bigl( 1+ 
\vert \langle 
\hat V^{(0)} \rangle \vert^{\alpha +1} 
+ \| \hat V^{(0)} \|_{2}^{\alpha +1} \bigr),
\label{assumption:mkv:2}
\\
&\bigl\vert h_{\ell}(w,\langle \hat V^{(0)} \rangle) - h_{\ell}
(w',\langle \hat V^{(0)\prime} \rangle) 
\bigr\vert^2
+
\hat{\mathbb E}
\bigl[ \vert \hat{H}_{\ell}(w,\langle \hat V^{(0)} \rangle) - 
\hat{H}_{\ell}(w',\langle 
\hat {{V}}^{(0)\prime} \rangle) 
\vert^2 \bigr] \nonumber
\\
&\hspace{5pt}
\leq C
\Bigl\{
\vert w - w' \vert^2 + 
\Phi^2_{\alpha}(\hat V^{(0)} ,\hat V^{(0) \prime}) \Bigr\}, 
\label{assumption:mkv:3}
\end{align}
with the condition that,  when $\hat V^{(0)} \sim \hat V^{(0) \prime}$, \begin{equation}
\Phi_{\alpha}(\hat V^{(0)} ,\hat V^{(0) \prime}) \leq 
C  
 {\mathbb E} \bigl[ \bigl(1+ \vert \hat V^{(0)}  \vert^{2\alpha} +
  \vert \hat V^{(0) \prime} \vert^{2\alpha} + \| \hat V^{(0)}  \|_{2}^{2\alpha} \bigr)
   \vert \hat V^{(0)}  - \hat V^{(0) \prime} \vert^2 
  \bigr]^{1/2}. 
  \label{assumption:mkv:4}
\end{equation}
We shall also require the additional assumption:
\begin{equation}
\label{assumption:mkv:5}
\textrm{For any} \ \hat V^{(0)}, \ 
\text{the \ family \ } \bigl( \vert \hat{H}_{\ell}(w,\langle \hat V^{(0)}
\rangle) \vert^2 \bigr)_{w \in \R^k}\text{is uniformly integrable.}
\end{equation}
Conditions \eqref{assumption:mkv:1}, \eqref{assumption:mkv:3}, \eqref{assumption:mkv:4}
and \eqref{assumption:mkv:5}
must be compared with \HYP{1}, the constant $K$ in \eqref{assumption:mkv:1} playing the 
role of $L$ in \HYP{1}. It is worth mentioning that the constant $K$ has a major role in 
the sequel as it dictates the size of the time interval on which all the
estimates derived in this section hold true.

The comparison between \eqref{assumption:mkv:1}--\eqref{assumption:mkv:2}--\eqref{assumption:mkv:3}--\eqref{assumption:mkv:4}--\eqref{assumption:mkv:5}
and \HYP{1} may be made more explicit within the framework of  Example \ref{ex appli}: 
%
\vspace{5pt}
\begin{Example}\label{ex appli suite}{(Continuing Example 
\ref{ex appli})}\\
Assumptions \eqref{assumption:mkv:1}, 
\eqref{assumption:mkv:2}, \eqref{assumption:mkv:3} and 
\eqref{assumption:mkv:5} read in the 
following way
  {
when, 
in the decomposition \eqref{eq:ex:mkv:stability}, $h_{\ell}(V_{r},\langle \hat{V}_{r}^{(0)} \rangle)
\equiv \partial_{w} h(\theta_{r},[\theta_{r}^{(0)}])$
and $\hat{H}_{\ell}(V_r,\langle \hat{V}_{r}^{(0)} \rangle)
\equiv \partial_{\mu}h(\theta_{r},[\theta_{r}^{(0)}])$}:
\begin{enumerate}
\item   
Equation \eqref{assumption:mkv:1} expresses the fact that $h$ is Lipschitz 
continuous with respect 
to $(w,\mu)$, so that the derivatives are bounded, in $L^{\infty}$ in the 
direction 
$w$ and in $L^2$ in the direction $\mu$. Importantly (and as already suggested), the 
constant $K$ corresponds to $L$ in \HYP{1}.
\item Equation \eqref{assumption:mkv:2} expresses the fact that, 
  {
for any $(w,\mu)$,
$v \mapsto \partial_{\mu}h(w,\mu)(v)$}
admits a version that is at most of polynomial growth (in $v$) under \HYP{1} (see the proof right below).
\item Equation \eqref{assumption:mkv:3} says that the
derivatives in the direction $w$ and in the direction of the measure are continuous
(in a suitable sense).
Except when $\alpha=0$, derivatives may not be Lipschitz continuous in the direction of the measure, which 
is a crucial relaxation for our purpose. 
\item Condition \eqref{assumption:mkv:5} expresses the fact that,
for $\chi \in L^2$, the family $(\partial_{\mu} h(w,\mu)(\chi))_{w \in \R^k}$
must be uniformly \new{square-}integrable. 
\end{enumerate}
The existence of a version 
of $v \mapsto \partial_{\mu}h(w,\mu)(v)$
that is at most of polynomial growth can be proved as follows. 
When $h$ is understood as one of the coefficients $b$, $\sigma$, $f$ or 
$g$, we know that, under \HYP{1}, $\partial_{\mu}h$ (which 
might be identified with a
Fr\'echet derivative) satisfies, 
for two random variables $\chi$ and $\chi'$, with the same distribution $\mu$,  
\begin{equation}
\label{eq:reg:dmuh}
\begin{split}
&{\mathbb E} \bigl[ \bigl\vert \partial_{\mu} h(w,\mu)(\chi) 
- \partial_{\mu} h(w,\mu)(\chi')
\bigr\vert^2 \bigr]^{1/2}
\\
&\hspace{15pt}
\leq C {\mathbb E}
\bigl[ \bigl( 1 + \vert \chi \vert^{2\alpha} + \vert \chi' \vert^{2\alpha}
+ \| \chi\|_{2}^{2\alpha}
\bigr) \vert \chi - \chi' \vert^2 \bigr]^{1/2},
\end{split}
\end{equation}
which implies that the mapping $v \mapsto \partial_{\mu} h(w,\mu)(v)$
is locally Lipschitz continuous, see Proposition \ref{prop:lipschitz:lifted}. More precisely, 
for a random variable $\chi$ with $\mu$ as distribution,
\begin{equation*}
\bigl\vert \partial_{\mu} h(w,\mu)(v) 
- \partial_{\mu} h(w,\mu)(v') 
 \bigr\vert \leq C 
  \bigl( 1+ \vert v \vert^{\alpha} + \vert v' \vert^\alpha
 +  \| \chi \|_{2}^\alpha
 \bigr) \vert v - v' \vert. 
\end{equation*}
Now, we know that, 
\begin{equation}
\label{eq:L2bound:dmuh}
{\mathbb E} \bigl[ \vert \partial_{\mu} h(w,\mu)(\chi) \vert^2 \bigr]^{1/2} \leq C. 
\end{equation}
Therefore, by the same method as in \eqref{eq:trick:markov},
we deduce that
\[
\inf_{\vert v \vert \leq 2 \| \chi \|_{2}}
\vert \partial_{\mu} h(w,\mu)(v) \vert \leq C,
\]
which, together with local Lipschitz property, says that, for any $w \in \R^k$,
\begin{equation}
\label{eq:lineargrowth}
\begin{split}
\bigl\vert \partial_{\mu} h(w,\mu)(v) 
\bigr\vert &\leq C + C \bigl( 1+ \vert v \vert^{\alpha} + \| \chi \|_{2}^\alpha
 \bigr) \bigl( \vert v  \vert +  \| \chi \|_{2} \bigr) 
 \\
 &\leq C \bigl( 1+ \vert v \vert^{\alpha+1}
  + \| \chi \|_{2}^{\alpha+1}
  \bigr),
 \end{split}
\end{equation}
which completes the proof of the polynomial growth property. 
\end{Example}

\begin{Remark}
The reader may wonder about the sharpness of the bound \eqref{eq:lineargrowth}. Indeed,  
when specialized to the case $\alpha=0$ and $h$ independent of $w$, \eqref{eq:lineargrowth} provides just a linear growth bound for the 
derivative $\R^l \ni v \mapsto \partial_{\mu} h(\mu)(v)$ of a Lipschitz-continuous function 
$h : \cP_{2}(\R^l) \ni \mu \mapsto h(\mu)$ (the Lipschitz continuity of $h$ follows from 
\eqref{eq:L2bound:dmuh}). This might seem rather weak and 
it might be tempting to expect an $L^\infty$ bound instead of a linear growth bound. 
As shown by the example in Remark \ref{rem:comparaison:buck},
there is no way of guaranteeing that the derivative 
$\partial_{\mu} h$ of the Lipschitz-continuous function $h$ is bounded in $L^\infty$, even when $\alpha=0$ (which is the strongest case) .
Boundedness of the derivative 
only holds in $L^2$, as is written in \eqref{eq:L2bound:dmuh}.

This important feature explains why the space of boundary conditions we consider
in the paper is not limited to functions with derivatives that are globally Lipschitz with respect to the measure argument. 
Because of the gap in the growth of the derivatives, 
we would fail to prove that the derivatives of the solution of the master equation (or equivalently of the decoupling 
field of the FBSDEs \eqref{eq X-Y t,xi}) are also globally Lipschitz with respect to the measure argument. 
Due to this lack of stability, we would not be able to extend the results from short to long time horizons. 
 
\end{Remark}
\subsubsection{Estimate of the solution}
\label{subsub:estimate}
Part of our analysis relies on stability estimates for systems of a more general form than 
\eqref{eq:sys:linear:mkv}, namely 
\begin{align}
&{\mathcal X}_{s} = \eta + \int_{t}^s 
B\bigl(r,\bar{\theta}_{r},\langle \check {\theta}_{r}^{(0)} \rangle \bigr)\bigl(\bar{\vartheta}_{r},\langle \check {\vartheta}_{r}^{(0)} \rangle\bigr) \ud r 
+ \int_{t}^s \Sigma\bigl(r,\bar{\theta}_{r}^{(0)},\langle \check{\theta}_{r}^{(0)} \rangle \bigr)
\bigl(\bar{\vartheta}_{r}^{(0)},\langle  \check{\vartheta}_{r}^{(0)} \rangle
\bigr) \ud W_{r}, \nonumber
\\
&{\mathcal Y}_{s} = G\bigl(X_{T},\langle \hat X_{T} \rangle\bigr)\bigl({\mathcal X}_{T},\langle \hat \cX_{T} \rangle \bigr) + \int_{s}^T 
F\bigl(r,\bar{\theta}_{r},\langle \check{\theta}_{r} \rangle \bigr)\bigl(\bar{\vartheta}_{r},\langle \check{\vartheta}_{r}^{(0)} \rangle \bigr) \ud r - \int_{s}^T {\mathcal Z}_{r} \ud W_{r}, \label{eq:sys:linear:mkv:2}
\end{align}
the difference between \eqref{eq:sys:linear:mkv:2} and \eqref{eq:sys:linear:mkv} being that the coefficients 
(except the terminal boundary condition) may depend on other 
triplets $\bar{\theta}$, $\check{\theta}$, $\bar{\vartheta}$ and $\check{\vartheta}$.
 We shall make use of the following definition, directly inspired from \eqref{assumption:mkv:1}:
 
\begin{Definition}
\label{def:admissible}
Given triplets $(\theta_{r}=(X_{r},Y_{r},Z_{r}))_{r \in [t,T]}$ 
and $(\hat \theta_{r}=(\hat X_{r},\hat Y_{r},\hat Z_{r}))_{r \in [t,T]}$
of the same form as above, we say that 
a subset $\cJ$ of $L^2(\Omega \times \hat{\Omega},\cA \otimes \hat{\cA},\P \otimes \hat{\P};\R_{+})$
is admissible for $(\theta,\hat{\theta})$ if 

(i)
for any $r \in [t,T]$, for $H$ matching $B$, $\Sigma$, $F$ or $G$
and $(V_{r},\hat V_{r}^{(0)})$ matching $(X_{r},\hat{X}_{r})$, 
$(\theta_{r},\hat{\theta}_{r}^{(0)})$
or $(\theta_{r}^{(0)},\hat\theta_{r}^{(0)})$, there exists $\Lambda \in \cJ$ such that 
$\hat{\mathbb E}[\vert \hat{H}_{\ell}(r,V_{r}, \langle \hat V_{r}^{(0)}
\rangle ) \vert^2]^{1/2} \leq \Lambda$; 

(ii) any $\Lambda$ in $\cJ$ satisfies $\P( \hat{\E}(\Lambda^2)^{1/2} \leq K)=1$. 
\end{Definition}

\noindent {\bf Notations.}
Throughout \S \ref{subsub:estimate}, $\cJ$
is an admissible class for both $(\theta, \hat \theta)$ and $(\bar{\theta},\check \theta)$. For 
a real $\gamma \geq T$,
an integer $p \geq 1$, a real $C>0$,  
a triplet $\vartheta=(\cX_{s},\cY_{s},\cZ_{s})_{s \in [t,T]}$
with values in ${\mathcal S}^2([t,T];\R^d) \times 
{\mathcal S}^2([t,T];\R^m) \times {\mathcal H}^2([t,T];\R^{m \times d})$ and a pair  of random variables $(X,\chi)$ with values in a Euclidean space,
we let
\begin{align} 
&{\mathcal M}_{{\mathbb M}}^{p}(\vartheta) := {\mathbb M} \biggl[ 
\sup_{s \in [t,T]} \bigl(
\vert \cX_{s} \vert^{p} + \gamma^{1/2} \vert \cY_{s} \vert^{p}\bigr)
+ \gamma^{1/2}\biggl( \int_{t}^T \vert \cZ_{s} \vert^2 \ud s \biggr)^{p/2} \biggr],
\nonumber
\\
&{\mathcal N}_{{\mathbb M}}^{p,C}(X,\chi) 
\label{stab:notations:1}
\\
&\hspace{15pt}:=
 \new{\esssup}_{\Lambda \in \cJ} {\mathbb M}
 \biggl[ 
\hat{\mathbb E} \biggl[ \Bigl\{
 \Lambda
 \wedge
 \Bigl[ C \bigl( 1+ \hat{\E}_{t} \bigl[|\langle X \rangle|^{2\alpha+2} \bigr]^{1/2} +
 \|X \|_{2}^{\alpha+1} 
 \bigr) \Bigr] \Bigr\}
\hat{\E}_{t} \bigl[ |\langle \chi \rangle|^2 \bigr]^{1/2} \biggr]^p \biggr],
\nonumber
\end{align} 
with the convention that ${\mathbb M}$ can be  ${\mathbb E}_{t}$ or ${\mathbb E}$ 
\new{(in the latter case $\esssup$ is just a $\sup$)}.
Note that ${\mathcal M}_{{\mathbb M}}^{p}(\vartheta)$ depends on $\gamma$, $t$ and $T$. 
We shall omit this dependence in the notation 
${\mathcal M}_{{\mathbb M}}^{p}(\vartheta).$
With these notations, we shall write ${\mathcal M}_{{\mathbb M}}^{p}(\vartheta^{(0)})$
for ${\mathcal M}_{{\mathbb M}}^{p}(\cX,\cY,0)$.
Similarly, we shall not specify the dependence upon $t$ 
in the notation 
${\mathcal N}_{{\mathbb M}}^{p,C}(X,\chi)$.
Regarding the structure of the coefficients, $B$, $\Sigma$, $F$ and $G$, we also let
\begin{equation}
\label{stab:notations:2}
{\mathcal R}_{a}^{p} :=
\E_{t} \biggl[ \gamma^{1/2} \big\vert G_{a}(T)  \big\vert^{p}
+ \biggl(
\int_{t}^T \bigl\vert  
(B_{a},F_{a})(s) \bigr\vert   \ud s \biggr)^{p}
+ \biggl(
\int_{t}^T \bigl\vert  
\Sigma_{a}(s) \bigr\vert^2   \ud s \biggr)^{p/2} \biggr].
\end{equation}
From Cauchy-Schwarz' inequality and \textit{(ii)} in Definition \ref{def:admissible}, we get that: 
\begin{Lemma}
\label{lem:cauchy}
For any pair $(X,\chi)$ and any $p \geq 1$, ${\mathcal N}^{p,C}_{{\mathbb M}}(X,\chi)
\leq K^p \| \chi \|_{2}^{p}$. 
\end{Lemma}
We deduce that 
\begin{Lemma}
\label{lem:apriori}
For any $p \geq 1$, there exist two constants $\Gamma_{p}:=\Gamma_{p}(K) \geq 1$ 
and $C  >0$ ($C$ independent of $p$), such that, 
for $T \leq \gamma \leq 1/\Gamma_{p}$ and for any solution $\vartheta$ to a system
of the same type as
 \eqref{eq:sys:linear:mkv:2}, it holds 
\begin{equation} 
\label{eq:lem:apriori:1}
\begin{split}
{\mathcal M}_{{\mathbb E}_{t}}^{2p}(\vartheta)
&\leq  \Gamma_{p}
 \biggl[ \vert \eta \vert^{2p}
+ \gamma^{1/2}
{\mathcal M}_{{\mathbb E}_{t}}^{2p}(\bar \vartheta)
+ 
{\mathcal R}_{a}^{2p} 
\\
&\hspace{15pt} 
+ \gamma^{1/2} \Bigl\{  {\mathcal N}_{\E_{t}}^{2p,C}\bigl(\hat X_{T},\hat \cX_{T}\bigr)  
+
 \sup_{s \in [t,T]}  {\mathcal N}_{\E_{t}}^{2p,C}\Bigl(\check {\theta}_{s}^{(0)},
 \bigl({\mathcal M}_{{\mathbb E}_{t}}^2(\check \vartheta^{(0)}) \bigr)^{1/2}
 \Bigr)  \Bigr\}
 \biggr].
\end{split}
\end{equation}
In particular (redefining the value of $\Gamma_{p}$ if necessary), 
\begin{equation}
\label{eq:lem:apriori:2}
\begin{split}
{\mathcal M}_{{\mathbb E}_{t}}^{2p}(\vartheta) 
&\leq  \Gamma_{p} \Bigl[
\bigl(\vert \eta \vert + \NL{2}{\eta}\bigr)^{2p}
+{\mathcal R}_{a}^{2p} 
+ \esp{{\mathcal R}_{a}^2}^{p} 
\\
&\hspace{15pt}
 +
  \gamma^{1/2}
\Bigl( {\mathcal M}_{{\mathbb E}_{t}}^{2p}(\bar\vartheta)
+ \bigl[{\mathcal M}_{{\mathbb E}}^2(\hat \vartheta^{(0)}) \bigr]^{p}
+ \bigl[{\mathcal M}_{{\mathbb E}}^2(\check \vartheta^{(0)}) \bigr]^{p}
\Bigr)
\Bigr].
\end{split}
\end{equation}
\end{Lemma}

\begin{proof}
We make use of standard results for solutions
of an FBSDE. 
We can indeed start with the trivial case when the coefficients 
$B_{\ell}$, $\Sigma_{\ell}$ and $F_{\ell}$ are null and 
$\hat{G}_{\ell}$ is also null (see 
\eqref{splitting:H}
 and \eqref{eq:mkv:stability:decomposition} for the notations).
Then, \eqref{eq:sys:linear:mkv:2} reads  
as a system driven by the linear part $g_{\ell}$ -- that
appears in the decomposition 
\eqref{eq:mkv:stability:decomposition} of $G$ -- plus a 
remainder involving $B_{a}$, $\Sigma_{a}$, $F_{a}$ and $G_{a}$.
Without any McKean-Vlasov interaction, 
\eqref{eq:lem:apriori:1} follows from stability estimates 
for standard linear FBSDEs.
For instance, following
Delarue \cite{del02}, we get that, for any 
$p \geq 1$, we can find $\Gamma_{p}:=\Gamma_{p}(K) >0$ (the value of which is allowed to increase from line to line), such that for $\gamma \leq 1/\Gamma_{p}$,
\eqref{eq:lem:apriori:1} holds, but with a simpler right-hand side
just consisting of $\Gamma_{p}[ \vert \eta \vert^{2p} + \cR_{a}^{2p}]$.  

In the case when 
$B_{\ell}$, $\Sigma_{\ell}$, $F_{\ell}$  are non-zero, we view them, when taken along 
the values of 
$(\bar{\theta},\check{\theta}^{(0)},\bar{\vartheta},\check{\vartheta}^{(0)})$
as parts of $B_{a}$, $\Sigma_{a}$ and $F_{a}$.
Similarly, we can see 
$\hat{G}_{\ell}$, when taken along the values of $(X_{T},\langle \hat{X}_{T} \rangle)$, 
as a part of $G_{a}$. We are thus led back to the previous case, but with
 a generalized version of the remainder term ${\mathcal R}_{a}$.
In order to complete the proof, it suffices to bound this remainder in $L^{2p}$. 
The analysis of the remainder may be split into three pieces: One first term involves
$b_{\ell}$, $\sigma_{\ell}$ and $f_{\ell}$; another one involves 
$\hat{B}_{\ell}$, $\hat{\Sigma}_{\ell}$, $\hat{F}_{\ell}$
and
$\hat{G}_{\ell}$; the last one involves $B_{a}$, $\Sigma_{a}$, $F_{a}$
and $G_{a}$ and corresponds to the original ${\mathcal R}_{a}$. As a final bound, we get
\begin{align}
&{\mathcal M}_{{\mathbb E}_{t}}^{2p}(\vartheta)
\leq  \Gamma_{p}
\vert \eta \vert^{2p}
\\
&\hspace{5pt}
+ \Gamma_{p} \E_{t}\biggl[ \biggl( 
\int_{t}^T
\bigl\vert  
(b_{\ell},f_{\ell})(\bar{\theta}_{s},\langle \check{\theta}_{s}^{(0)} \rangle)
 \bar{\vartheta}_{s}  \bigr\vert \ud s
\biggr)^{2p} + 
\biggl( 
\int_{t}^T
\bigl\vert  
\sigma_{\ell}(\bar{\theta}_{s}^{(0)},\langle \check{\theta}_{s}^{(0)}\rangle)
 \bar {\vartheta}_{s}^0  \bigr\vert^2 \ud s
\biggr)^{p} \biggr]
\label{eq:apriori:mkv:1}
\\
&\hspace{5pt} +
\Gamma_{p} \gamma^{1/2} \biggl[
{\mathbb E}_{t} \Bigl[ \bigl\vert 
\hat{\mathbb E} \bigl[
\hat{G}_{\ell}(X_{T},\langle \hat X_{T}\rangle)\langle \hat {\cX}_{T} \rangle
\bigr]
\bigr\vert^{2p} \Bigr] \nonumber
\\
&\hspace{30pt} 
 + \gamma^{p/2} \esssup_{s \in [t,T]}
 {\mathbb E}_{t} \Bigl[
\bigl\vert 
\hat{\mathbb E} \bigl[
(\hat{B}_{\ell},\hat{F}_{\ell},\hat{\Sigma}_{\ell})(\bar{\theta}_{s},\langle 
\check{\theta}_{s}^{(0)} \rangle)\langle 
\check {\vartheta}^{(0)}_{s} \rangle
\bigr] 
\bigr\vert^{2p} \Bigr]   \biggr]
\label{eq:apriori:mkv:2}
  \\
&\hspace{5pt} 
 + \Gamma_{p}
{\mathbb E}_{t} \biggl[
\gamma^{1/2} \big\vert G_{a}(T)  \big\vert^{2p}
+ \biggl(
\int_{t}^T \bigl\vert  
(B_{a},F_{a})(s) \bigr\vert   \ud s \biggr)^{2p} 
+ \biggl(
\int_{t}^T \bigl\vert  
\Sigma_{a}(s) \bigr\vert^2   \ud s
\biggr)^{p}
\biggr].
\nonumber
\end{align}
Observe that, in
\eqref{eq:apriori:mkv:2}, we used the supremum to get $T^{p}$, which we bounded by 
$\gamma^{1/2}$ times $\gamma^{p/2}$. 

Making use of \eqref{assumption:mkv:1}, we easily handle 
the term \eqref{eq:apriori:mkv:1}. In \eqref{eq:lem:apriori:1}, it gives the contribution of the form $\gamma^{1/2}
{\mathcal M}_{{\mathbb E}_{t}}^{2p}(\bar\vartheta)$, the $\gamma^{1/2}$
in front of ${\mathcal M}_{{\mathbb E}_{t}}$ and the $\gamma^{1/2}$ in the definition 
of ${\mathcal M}_{{\mathbb E}_{t}}^{2p}(\bar\vartheta)$ arising as follows. 
When handling $(b_{\ell},f_{\ell})$, we can let a power 2 enter 
inside the time integral. This introduces the ${\mathcal H}^2$-norm of $\bar\cZ$
times an additional $T$ less than $\gamma$, which can be split into $\gamma^{1/2}$ and $\gamma^{1/2}$. 

Next we discuss the second term in \eqref{eq:apriori:mkv:2}.
For this we use
\eqref{assumption:mkv:1} and
\eqref{assumption:mkv:2}.
With the shortened notation $H=(B,F,\Sigma)$, 
we can indeed either say that
$\hat{H}_{\ell}(\bar{\theta}_{s},\langle \check{\theta}_{s}^{(0)}\rangle)$
is bounded in $L^2$ or use the polynomial growth assumption. We get 
\begin{equation*}
\Bigl\vert 
\hat{\mathbb E} \Bigl[
\hat{H}_{\ell}(\bar{\theta}_{s},\langle \check{\theta}_{s}^{(0)}\rangle)\langle 
\check {\vartheta}_{s}^{(0)} \rangle
 \Bigr] \Bigr\vert
\leq  
\hat{\mathbb E} \Bigl[ \Bigl\{ {\Lambda}_{s}  
\wedge \Bigl( C+ C \vert \langle \check{\theta}_{s}^{(0)} \rangle \vert^{\alpha+1}
+ C \| \check{\theta}_{s}^{(0)} \|_{2}^{\alpha+1}  \Bigr) \Bigr\} 
 \bigl\vert 
\langle \check{\vartheta}_{s}^{(0)} 
\rangle
 \bigr\vert 
\Bigr],
\end{equation*}
where $\Lambda_{s}$ is a 
shortened notation for 
$\vert \hat{H}_{\ell}(\bar{\theta}_{s},\langle \check{\theta}_{s}^{(0)}\rangle)\vert$. 
Now, using the conditional Cauchy-Schwarz inequality
and the obvious bound ${\mathbb E}_{t} [ S_{1} \wedge S_{2}
] \leq \E_{t}[S_{1}] \wedge \E_{t}[ S_{2}]$ for two 
non-negative random variables $S_{1}$ and $S_{2}$, we obtain:
\begin{equation*}
\begin{split}
&\Bigl\vert 
\hat{\mathbb E} \Bigl[
\hat{H}_{\ell}(\bar{\theta}_{s},\langle \check{\theta}_{s}^{(0)}\rangle)\langle 
\check {\vartheta}_{s}^{(0)}
 \rangle
 \Bigr] \Bigr\vert
\\
&\hspace{15pt} \leq  
\hat{\mathbb E} \Bigl[ \Bigl\{ \hat{\mathbb E}_{t}\bigl[ {\Lambda}_{s}^2 
\bigr]^{1/2} 
\wedge \Bigl( C+ C \hat{\mathbb E}_{t } \bigl[ 
\vert \langle \check{\theta}_{s}^{(0)} \rangle \vert^{2\alpha+2} \bigr]^{1/2} 
+ C \| \check{\theta}_{s}^{(0)} \|_{2}^{\alpha+1} \Bigr) \Bigr\} 
\hat{\mathbb E}_{t}\Bigl[
 \bigl\vert 
\langle \check{\vartheta}_{s}^{(0)} \rangle 
 \bigr\vert^2 \Bigr]^{1/2} \Bigr].
\end{split}
\end{equation*}
Taking the power $2p$ and the conditional expectation $\E_{t}$, we get 
a term which is less than ${\mathcal N}^{2p,C}_{\E_{t}}(\check{\theta}_{s}^{(0)},
\check{\vartheta}_{s}^{(0)})$. Multiplying by 
$\gamma^{p/2}$ (see the prefactor in \eqref{eq:apriori:mkv:2}), we get that it is less than
${\mathcal N}^{2p,C}_{\E_{t}}(\check{\theta}_{s}^{(0)},
[ \vert 
\check{\cX}_{s} 
\vert^2 
+  \gamma^{1/2} \vert 
\check{\cY}_{s} 
\vert^2
]^{1/2} 
)$
which is less than 
${\mathcal N}^{2p,C}_{\E_{t}}(\check{\theta}_{s}^{(0)},
({\mathcal M}_{{\mathbb E}_{t}}^2(\check \vartheta^{(0)}) )^{1/2})$.
Of course, we can use the same kind of argument for the first term in 
\eqref{eq:apriori:mkv:2} and get 
${\mathcal N}^{2p,C}_{\E_{t}}(\hat X_{T},\hat \cX_{T})$
as resulting bound. 

The second claim follows from Lemma \ref{lem:cauchy}. \qed
\end{proof}
\vspace{5pt}

In particular, we have the following useful result for systems of the form 
\eqref{eq:sys:linear:mkv} obtained by considering 
$\vartheta \equiv \bar{\vartheta}$ and 
$\hat{\vartheta} \equiv \check{\vartheta}$
in \eqref{eq:lem:apriori:2} 
 and setting $\gamma$ 
small enough.

\begin{Corollary}
 \label{co:bound:sys:lin:mkv:easy}
For any $p \geq 1$, there exists a constant $\Gamma_{p}:=\Gamma_{p}(K) \geq 
1$ such that, 
for $T \leq \gamma \leq 1/\Gamma_{p}$ and for any solution $\vartheta$ to a system
of the same type as
 \eqref{eq:sys:linear:mkv}, it holds
 \begin{equation}
\label{eq:co:bound:sys:lin:mkv:easy}
\begin{split}
&{\mathcal M}_{{\mathbb E}_{t}}^{2p}( \vartheta )
\leq \Gamma_{p} \Bigl[ \bigl( \eta + \|\eta\|_{2} \bigr)^{2p}
+
{\mathcal R}_{a}^{2p} 
+ \esp{{\mathcal R}_{a}^2}^p + \gamma^{1/2} \bigl[{\mathcal M}_{{\mathbb E}}^2(\hat \vartheta^{(0)}) \bigr]^{p}
\Bigr].
\end{split}
\end{equation}
When $\vartheta \equiv \hat \vartheta$, we have (modifying the constant 
$\Gamma_{p}$ if necessary):
 \begin{equation}
\label{eq:co:bound:sys:lin:mkv:easy:2}
\begin{split}
&{\mathcal M}_{{\mathbb E}_{t}}^{2p}( \vartheta )
\leq \Gamma_{p} \Bigl[ \bigl( \eta + \|\eta\|_{2} \bigr)^{2p}
+
{\mathcal R}_{a}^{2p} 
+ \esp{{\mathcal R}_{a}^2}^p  
\Bigr].
\end{split}
\end{equation}
\end{Corollary}
\begin{proof}
Inequality \eqref{eq:co:bound:sys:lin:mkv:easy} directly follows from 
\eqref{eq:lem:apriori:2}. 
To get \eqref{eq:co:bound:sys:lin:mkv:easy:2}, we choose 
$p=1$ and then take the expectation. 
For $\gamma$ small enough, we obtain
${\mathcal M}_{{\mathbb E}}^{2}( \vartheta )
\leq \Gamma_{1} ( \|\eta\|_{2}^{2} + \E[\cR_{a}^2])$ (up to a new value for 
$\Gamma_{1}$). Plugging the bound into 
\eqref{eq:co:bound:sys:lin:mkv:easy}, we deduce that
\eqref{eq:co:bound:sys:lin:mkv:easy:2} holds. \qed
\end{proof}

\subsubsection{Stability estimates} 
\label{subsub:stability}
The next step is to compare two solutions 
of \eqref{eq:sys:linear:mkv:2}
$\vartheta$ and $\vartheta'$ driven by 
two different sets of inputs $(\hat{\theta},\bar{\theta},\check{\theta},\hat{\vartheta},\bar{\vartheta},\check{\vartheta})$ and
$(\hat{\theta}',\bar{\theta}',\check{\theta}',\hat{\vartheta}',\bar{\vartheta}',\check{\vartheta}')$
 but with the same starting point $\eta$.
Throughout \S \ref{subsub:stability}, $\cJ$
is an admissible class for $(\theta,\hat{\theta)}$ and $(\bar{\theta},\check{\theta})$.

Given an integer $p \geq 1$, 
define similar notations 
to \eqref{assumption:mkv:4} and
\eqref{stab:notations:1} (but without $\gamma^{1/2}$ in front of 
the terms in ${\mathcal Y}$):
\begin{equation}
\label{eq:mkv:stability:Psi}
\begin{split}
 &\Phi_{\alpha}\bigl(\hat \vartheta^{(0)},\hat \vartheta^{(0)\prime}\bigr) :=  \sup_{s \in [t,T]}\bigl\{
\Phi_{\alpha} \bigl(\hat \vartheta^{(0)}_{s},\hat \vartheta^{(0)\prime}_{s} \bigr) \bigr\} 
\\
&\bar{\mathcal M}^{2p}(\vartheta,\hat{\vartheta})
:=
\sup_{s \in [t,T]}
\Bigl\{
{\mathbb E}_{t} \bigl[ 
| \cX_{s}
|^{2p} + | \cY_{s}
|^{2p} \bigr]
+
\| \hat \cX_{s}
 \|_{2}^{2p} + \|\hat \cY_{s}\|_{2}^{2p}
 \Bigr\}
+  \NHt{2p}{\cZ}^{2p},
 \\
&\bar{\mathcal M}^{2p}\bigl((\vartheta,\hat{\vartheta}),(\vartheta',\hat{\vartheta}') \bigr)
:= \bar{\mathcal M}^{2p}\bigl(\vartheta - \vartheta',\hat{\vartheta}- \hat{\vartheta}' \bigr)
+ 
\Phi_{\alpha}^{2p}\bigl(\hat \vartheta^{(0)},\hat \vartheta^{(0)\prime}\bigr), 
\\[6pt]
&\bar{\mathcal M}^{2p}
\llbracket \vartheta \rrbracket :=
\bar{\mathcal M}^{2p}\bigl(\vartheta,{\vartheta}\bigr), \quad 
\bar{\mathcal M}^{2p}
\llbracket \vartheta,\vartheta' \rrbracket :=
\bar{\mathcal M}^{2p}\bigl((\vartheta,{\vartheta}),(\vartheta',{\vartheta}') \bigr),
\end{split}
\end{equation}
and denote by $\Delta {\mathcal R}_{a}^{2p}$ the quantity (recall 
\eqref{stab:notations:2} for
the definition of ${\mathcal R}_{a}^{2p}$):
\begin{equation}
\label{eq:mkv:stability:R}
\begin{split}
&{\Delta {\mathcal R}_{a}^{2p}} :=
{\mathbb E}_{t} \biggl[
\gamma^{1/2} \big\vert G_{a}(T) - G_{a}'(T) \big\vert^{2p}
\\
&\hspace{15pt} + \biggl(
\int_{t}^T \bigl\vert  
(B_{a}-B_{a}',F_{a}-F_{a}')(s) \bigr\vert   \ud s \biggr)^{2p} 
+ 
\biggl(
\int_{t}^T \bigl\vert (\Sigma_{a} - \Sigma_{a}')(s) \bigr\vert^2 \ud s
\biggr)^p
\biggr].
\end{split}
\end{equation}
(The notations $B_{a}'$, $F_{a}'$, $\Sigma_{a}'$
and $G_{a}'$ refer to the fact, along the processes labelled with 
a `prime', the remainders in the decomposition of the coefficients 
may be different.)
Then, we have

\begin{Lemma} 
\label{lem:stability}
For any $p \geq 1$,
there exist three constants $C$ (independent of $p$), 
$\Gamma_{p}:=\Gamma_{p}(K) \geq 1$ and $C_{p} >0$, such that, for $T \leq \gamma \leq 1/\Gamma_{p}$, 
\begin{align}
{\mathcal M}_{{\mathbb E}_{t}}^{2p}
\bigl( \vartheta - \vartheta'\bigr)
&\leq  \Gamma_{p} \gamma^{1/2} \Bigl\{
{\mathcal M}_{{\mathbb E}_{t}}^{2p}
\bigl( \bar\vartheta -\bar\vartheta' \bigr) +
 {\mathcal N}_{\E_{t}}^{2p,C}\bigl(\hat X_{T},\hat \cX_{T}- \hat \cX_{T}'\bigr) \nonumber
\\
&\hspace{30pt}+ \sup_{s \in [t,T]} 
 {\mathcal N}_{\E_{t}}^{2p,C}\Bigl(
 \check{\theta}_{s}^{(0)}
 , \bigl({\mathcal M}_{{\mathbb E}_{t}}^2(\check \vartheta^{(0)} - \check {\vartheta}^{(0)\prime}) \bigr)^{1/2} \Bigr)
\Bigr\} \label{eq:lem:stability:1}
\\
&\hspace{15pt} + C_{p}\Bigl[
\Bigl( \bar{\mathcal M}^{4p}(\vartheta',\hat{\vartheta}') + 
\bar{\mathcal M}^{4p}(\bar\vartheta',\check{\vartheta}') \Bigr)^{1/2} \nonumber
\\
&\hspace{30pt} \times \Bigl\{ 1 \wedge 
 \Bigl(
\bar{\mathcal M}^{4p}\bigl((\theta,\hat \theta),(\theta', \hat \theta')\bigr) + 
\bar{\mathcal M}^{4p}\bigl( (\bar\theta,\check \theta),(\bar \theta',\check \theta') \bigr)  
\Bigr) \Bigr\}^{1/2} 
+ \Delta {\mathcal R}_{a}^{2p}
 \Bigr]. \nonumber
\end{align}
 In particular, choosing $p=1$ and taking expectation, we have, for some constant $\Gamma':=\Gamma'(K)$ such that 
 $T \le \gamma \leq 1/\Gamma'$ and for some $C'>0$,
 \begin{align} 
{\mathcal M}_{{\mathbb E}}^{2}
\bigl( \vartheta -\vartheta'\bigr)
&\leq  \Gamma' \gamma^{1/2} 
\Bigl\{ 
{\mathcal M}_{{\mathbb E}}^{2}
\bigl( \bar\vartheta -\bar\vartheta' \bigr)
+ {\mathcal M}_{{\mathbb E}}^{2}
\bigl( \hat \vartheta^{(0)} -\hat \vartheta^{(0)\prime} \bigr)
+ {\mathcal M}_{{\mathbb E}}^{2}
\bigl( \check \vartheta^{(0)} -\check \vartheta^{(0)\prime} \bigr) \Bigr\} \nonumber
\\
&\hspace{15pt} + C' \E \Bigl[
\Bigl( \bar{\mathcal M}^{4}(\vartheta',\hat{\vartheta}') + 
\bar{\mathcal M}^{4}(\bar\vartheta',\check{\vartheta}') \Bigr)^{1/2}  \label{eq:lem:stability:3esp}
\\
&\hspace{30pt} \times \Bigl\{ 1 \wedge 
 \Bigl(
\bar{\mathcal M}^{4}\bigl((\theta,\hat \theta),(\theta', \hat \theta')\bigr) + 
\bar{\mathcal M}^{4}\bigl( (\bar\theta,\check \theta),(\bar \theta',\check \theta') \bigr)  
\Bigr) \Bigr\}^{1/2} 
+  \Delta {\mathcal R}_{a}^{2}  \nonumber
 \Bigr].
 \end{align}
\end{Lemma}

\begin{Remark}
\label{rem:uniqueness}
Specialized to the case when 
$\theta \equiv \bar \theta \equiv \theta' \equiv \bar \theta'$,
$\hat \theta \equiv \check \theta \equiv \hat \theta' \equiv \check \theta'$,
$\vartheta \equiv \bar\vartheta \equiv \hat \vartheta \equiv \check \vartheta$,
$\vartheta' \equiv \bar\vartheta' \equiv \hat \vartheta' \equiv \check \vartheta'$ and $\Delta \cR_{a}^2 \equiv 0$,
Lemma \ref{eq:lem:stability:3esp}
reads as a uniqueness result to
\eqref{eq:sys:linear:mkv} in short time when $\vartheta \equiv \hat{\vartheta}$ therein. 
\end{Remark}

\begin{proof}
We start with the proof of \eqref{eq:lem:stability:1}.
We take benefit of the linearity to make the difference
of the two systems of the form 
\eqref{eq:sys:linear:mkv:2}
 satisfied by 
$\vartheta$ and $\vartheta'$.
The resulting system is linear in $\Delta \vartheta := \vartheta-\vartheta'$,
 $\Delta \hat{\vartheta} := \hat{\vartheta}-\hat{\vartheta}'$,
 $\Delta \bar{\vartheta} := \bar{\vartheta}-\bar{\vartheta}'$
 and  $\Delta \check{\vartheta} := \check{\vartheta}-\check{\vartheta}'$, but contains some remainders. 
We denote these remainders by 
$\Delta B_{a}$, $\Delta F_{a}$, 
$\Delta \Sigma_{a}$ and $\Delta G_{a}$. 
Using the notations introduced in \eqref{splitting:H}
and \eqref{eq:mkv:stability:decomposition}, they may be expanded as:
\begin{equation}
\label{eq:delta:H:a}
\begin{split}
&\Delta H_{a}(s) = 
\bigl( h_{\ell}(\bar{\theta}_{s},\langle \check{\theta}_{s}^{(0)} \rangle)
- h_{\ell}(\bar{\theta}_{s}',\langle \check{\theta}_{s}^{(0)\prime} \rangle) \bigr) \bar{\vartheta}_{s}'
\\
&\hspace{60pt}
+
\hat{\E} \bigl[
\bigl( 
\hat{H}_{\ell}(\bar{\theta}_{s},\langle \check{\theta}_{s}^{(0)} \rangle)
-  
\hat{H}_{\ell}(\bar{\theta}_{s}^{\prime},\langle \check{\theta}_{s}^{(0)\prime} 
\rangle)
\bigr)
\langle 
\check{\vartheta}_{s}^{(0) \prime} \rangle \bigr]
 + H_{a}(s) - H_{a}'(s),
\\
&\Delta G_{a}(T) = 
\bigl( g_{\ell}(X_{T},\langle \hat X_{T} \rangle)-g_{\ell}(X_{T}',\langle \hat X_{T}'\rangle) \bigr) \cX_{T}' 
\\
&\hspace{60pt} 
+ 
\hat{\mathbb E} \bigl[ \bigl( \hat{G}_{\ell}(X_{T},\langle \hat X_{T}\rangle)
- \hat{G}_{\ell}(X_{T}',\langle \hat X_{T}'\rangle) \bigr) \langle 
\hat {\cX}_{T}'  \rangle \bigr]
+ G_{a}(T) - G_{a}'(T),
\end{split}
\end{equation}
where $H$ may stand for $B$, $F$ or $\Sigma$,
with a corresponding meaning for 
$h_{\ell}$, $\hat{H}_{\ell}$ and $H_{a}$:
$h_{\ell}$ may be $b_{\ell}$, $f_{\ell}$, $\sigma_{\ell}$;
$\hat{H}_{\ell}$ may be $\hat{B}_{\ell}$, $\hat{F}_{\ell}$, or 
$\hat{\Sigma}_{\ell}$; $H_{a}$ may be $B_{a}$, $F_{a}$
or $\Sigma_{a}$; and $H_{a}'$ may be $B_{a}'$, $F_{a}'$
or $\Sigma_{a}'$. 
With these notations in hand, the terms $\Delta H_{a}(s)$
and $\Delta G_{a}(T)$ come
from (recall \eqref{splitting:H}):
\begin{equation}
\label{eq:lem:description:coef:1}
\begin{split}
&H\bigl(r,\bar{\theta}_r,\langle \check{\theta}_{r}^{(0)} \rangle\bigr)\bigl(\bar{\vartheta}_r,\langle 
\check{\vartheta}_r^{(0)}
\rangle\bigr) - H\bigl(r,\bar{\theta}_r',\langle \check{\theta}_{r}^{(0)\prime} \rangle\bigr)
\bigl(\bar{\vartheta}_r',\langle 
\check{\vartheta}_r^{(0)\prime}
\rangle\bigr)
\\
&\hspace{15pt}=
H_\ell\bigl(\bar{\theta}_r,\langle \check{\theta}_{r}^{(0)} \rangle\bigr)
\bigl(\Delta \bar{\vartheta}_r,\langle \Delta \check{\vartheta}_r^{(0)} \rangle \bigr) + \Delta H_a(r)\,,
\\
&G\bigl(X_T,\langle \hat X_{T} \rangle\bigr)\bigl(\cX_T,\langle \hat \cX_{T} \rangle \bigr) - 
G\bigl(X'_T,\langle \hat X_{T}' \rangle\bigr) \bigl(\cX'_T,\langle \hat \cX_{T}'\rangle \bigr) 
\\
&\hspace{15pt}
= G_\ell\bigl(X_T,\langle \hat X_{T}\rangle \bigr)\bigl(\Delta  \cX_T,\langle \Delta \hat \cX_T \rangle \bigr) + \Delta G_a(T)\,. 
\end{split}
\end{equation}

We will apply 
Lemma \ref{lem:apriori}. 
In the statement of the Lemma, we see from \eqref{eq:lem:description:coef:1}
that 
$\vartheta$ must be understood 
as $\Delta \vartheta$,
$\bar \vartheta$
as $\Delta \bar{\vartheta}$
and similarly for the processes labelled with `hat' and `check'.  
Moreover, the remainder $(B_{a},F_{a},\Sigma_{a},G_{a})$
in the statement must be understood as
$(\Delta B_{a},\Delta F_{a},\Delta \Sigma_{a},\Delta G_{a})$. 

We estimate the remainder terms in 
\eqref{eq:lem:apriori:1}, recalling
\eqref{stab:notations:2}
for the meaning we give to the remainder in the stability estimate.  
By \eqref{eq:delta:H:a}, the remainder can be split into three pieces according to 
$h_{\ell}$, $\hat{H}_{\ell}$ and $H_{a}$. 
\vspace{5pt}

\textit{First step. Upper bound for the terms involving $(b_{\ell},f_{\ell})$, $\sigma_{\ell}$ and $g_{\ell}$.} 
We make use of the assumption 
\eqref{assumption:mkv:3} and of the 
conditional Cauchy-Schwarz inequality. 
Getting rid of the constant $\gamma^{1/2}$ in 
front of $\vert G_{a}(T)\vert^{2p}$ in \eqref{stab:notations:2}, we let
\begin{equation*}
\begin{split}
\Delta r_{\ell}^{2p} &:=
{\mathbb E}_{t} \biggl[  \bigl\vert \bigl( 
g_{\ell}(X_{T},\langle \hat X_{T}\rangle)-g_{\ell}(X_{T}',\langle \hat X_{T}'\rangle) \bigr) \cX_{T}' \bigr\vert^{2p} 
\\
&\hspace{15pt}
+
\biggl(
 \int_{t}^T
\bigl\vert \bigl( (b_{\ell},f_{\ell})(\bar{\theta}_{s},\langle \check{\theta}_{s}^{(0)} \rangle)
-(b_{\ell},f_{\ell})(\bar{\theta}_{s}',\langle \check{\theta}_{s}^{(0)\prime}\rangle) \bigr) \bar{\vartheta}_{s}' \bigr\vert \ud s
\biggr)^{2p} 
\\
&\hspace{15pt} + 
\biggl(
\int_{t}^T
\bigl\vert   \bigl(
\sigma_{\ell}(\bar{\theta}_{s}^{(0)},\langle \check{\theta}_{s}^{(0)} \rangle)
- \sigma_{\ell}(\bar{\theta}_{s}^{(0)\prime},\langle \check{\theta}_{s}^{(0)\prime} \rangle) \bigr) \bar{\vartheta}_{s}^{(0) \prime} \bigr\vert^2 \ud s \biggr)^p
\biggr].
\end{split}
\end{equation*}
Recalling the Lipschitz property \eqref{assumption:mkv:3}, we know that, for a generic function $h_{\ell}$, which may be $b_{\ell}$, $f_{\ell}$ or $\sigma_{\ell}$,
\begin{equation}
\label{eq:first_step:1}
\begin{split}
&\bigl\vert 
\bigl( 
h_{\ell}(\bar{\theta}_{s},\langle \check{\theta}_{s}^{(0)}\rangle)
- h_{\ell}(\bar{\theta}_{s}^{\prime},\langle 
\check{\theta}_{s}^{(0)\prime} \rangle)
\bigr)
\bar{\vartheta}_{s}'
\bigr\vert^2
\leq  C
\bigl( 
\vert \bar{\theta}_{s} - \bar{\theta}_{s}^{\prime}\vert^2
+ \Phi_{\alpha}^2( \check{\theta}^{(0)},  \check{\theta}^{(0)\prime})
 \bigr) 
 | \bar\vartheta_{s}' |^2.
 \end{split}
\end{equation}
Therefore, we get 
(for a constant $C'$ possibly depending on $p$ and varying from line to line)
\begin{equation*}
\begin{split}
&\biggl(
 \int_{t}^T
\bigl\vert \bigl( (b_{\ell},f_{\ell})(\bar{\theta}_{s},\langle \check{\theta}^{0)}_{s} \rangle)
-(b_{\ell},f_{\ell})(\bar{\theta}_{s}',\langle \check{\theta}_{s}^{(0)\prime}\rangle) \bigr) \bar{\vartheta}_{s}' \bigr\vert \ud s
\biggr)^{2p} 
\\
&\leq C' 
\biggl[
\biggl(
 \int_{t}^T \vert \bar{\theta}_{s} - \bar{\theta}_{s}^\prime \vert^2
 \ud s 
\biggr)^{p} 
+ \Phi_{\alpha}^{2p}\bigl(\check{\theta}^{(0)},\check{\theta}^{(0)\prime}\bigr)
\biggr]
\biggl(
 \int_{t}^T \vert \bar{\vartheta}_{s}'\vert^2
 \ud s 
\biggr)^{p},
\end{split}
\end{equation*}
and by conditional Cauchy-Schwarz inequality, we deduce that (with the notation introduced in 
\eqref{eq:mkv:stability:Psi}):
\begin{equation*}
\begin{split}
&{\mathbb E}_{t} \biggl[ \biggl(
 \int_{t}^T
\bigl\vert \bigl( (b_{\ell},f_{\ell})(\bar{\theta}_{s},\langle \check{\theta}_{s}^{(0)} \rangle)
-(b_{\ell},f_{\ell})(\bar{\theta}_{s}',\langle \check{\theta}_{s}^{(0)\prime}\rangle) \bigr) \bar{\vartheta}_{s}' \bigr\vert \ud s
\biggr)^{2p} \biggr]
\\
&\hspace{15pt} \leq C' 
\bigl\{ \bar{\mathcal M}^{4p}\bigl((\bar\theta,\check \theta),(\bar\theta',\check \theta') \bigr)
\bigr\}^{1/2} 
\bigl\{ \bar{\mathcal M}^{4p}\bigl( \bar{\vartheta}',\check{\vartheta}' \bigr) \bigr\}^{1/2}.
\end{split}
\end{equation*}
It is pretty clear that we can get a similar bound when replacing $(b_{\ell},f_{\ell})$
by $\sigma_{\ell}$ (using the supremum norm to handle the fact that there is already a 
square inside the integral). 

Finally, the term involving $g_{\ell}$ can be also handled in a similar way, paying attention that the  `bar' process has to be replaced by the `non-bar' process and the `check' process by the `hat' process. 
We thus get 
\begin{equation*}
\begin{split}
&\Delta r_{\ell}^{2p}
\\
&\hspace{5pt} \leq C' 
\bigl\{ \bar{\mathcal M}^{4p}\bigl((\theta,\hat{\theta}),(\theta',\hat \theta') \bigr) + 
\bar{\mathcal M}^{4p}\bigl((\bar  \theta,\check \theta),(\bar\theta',\check \theta') \bigr) \bigr\}^{1/2}
\bigl\{ {\mathcal M}^{4p}\bigl(\vartheta',\hat \vartheta' \bigr) + 
\bar{\mathcal M}^{4p}\bigl(\bar\vartheta' , \check \vartheta' \bigr) \bigr\}^{1/2}.
\end{split}
\end{equation*} 

Using \eqref{assumption:mkv:1}, we get another bound for the same quantity, just 
by taking advantage of the fact that $(b_{\ell},f_{\ell})$, $\sigma_{\ell}$ and $g_{\ell}$ are bounded:
\begin{equation*}
\begin{split}
\Delta r_{\ell}^{2p} \leq C' \biggl\{ {\mathbb E}_{t} \bigl[ \bigl\vert
\cX_{T}' \vert^{2p} \bigr] + 
\sup_{s \in [t,T]} 
{\mathbb E}_{t} \bigl[
\bigl\vert \bar{\vartheta}_{s}^{(0) \prime}
\bigr\vert^{2p} \bigr]
+
{\mathbb E}_{t} \biggl[  \biggl( \int_{t}^T 
\vert \bar{\vartheta}_{s}' \vert^2 \ud s 
\biggr)^p \biggr] \biggr\},
\end{split}
\end{equation*}
so that 
\begin{equation*}
\begin{split}
\Delta r_{\ell}^{2p}
 &\leq C' \Bigl[ 1 \wedge 
\bigl\{ \bar{\mathcal M}^{4p}\bigl((\theta,\hat{\theta}),(\theta',\hat \theta') \bigr) + 
\bar{\mathcal M}^{4p}\bigl((\bar  \theta,\check \theta),(\bar\theta',\check \theta') \bigr) \bigr\}^{1/2} \Bigr]
\\
&\hspace{15pt} \times
\bigl\{ {\mathcal M}^{4p}\bigl(\vartheta',\hat \vartheta' \bigr) + 
\bar{\mathcal M}^{4p}\bigl(\bar\vartheta',\check \vartheta' \bigr) \bigr\}^{1/2}.
\end{split}
\end{equation*} 
\textit{Second step. 
Upper bound for the terms involving $\hat{B}_{\ell}$, 
$\hat{F}_{\ell}$, $\hat{\Sigma}_{\ell}$ or $\hat{G}_{\ell}$.}
We can make use of the Lipschitz property  
\eqref{assumption:mkv:3} or of the $L^2$ bound
\eqref{assumption:mkv:1}. 
For a generic function $\hat{H}_{\ell}$, which may be 
$\hat{B}_{\ell}$, $\hat{F}_{\ell}$ or 
$\hat{\Sigma}_{\ell}$, we get
\begin{equation}
\label{eq:second_step:1}
\begin{split}
&\bigl\vert 
\hat{\mathbb E} \bigl[
\bigl( 
\hat{H}_{\ell}(\bar{\theta}_{s},\langle \check{\theta}_{s}^{(0)}\rangle)
- \hat{H}_{\ell}(\bar{\theta}_{s}^{\prime},\langle 
\check{\theta}_{s}^{(0)\prime}) \rangle)
\bigr)
\langle \check{\vartheta}_{s}^{(0) \prime}
\rangle \bigr] 
\bigr\vert^2
\\
&\hspace{15pt}\leq  C
\bigl[ 1 
\wedge \bigl( 
\vert \bar{\theta}_{s} - \bar{\theta}_{s}^{\prime}\vert^2
+ \Phi_{\alpha}^2( \check{\theta}^{(0)},  \check{\theta}^{(0)\prime}
 \bigr) \bigr]
 \| \check{\vartheta}_{s}^{(0) \prime} \|_{2}^2.
 \end{split}
\end{equation}
Therefore, recalling the bound $\int (1 \wedge h) \ud \nu
\leq 1 \wedge \int h \ud \nu$ that holds for a general measure $\nu$ with mass less than 1
and a general measurable nonnegative function $h$, we get
\begin{equation}
\label{eq:second_step:2}
\begin{split}
&{\mathbb E}_{t} \biggl[ 
\biggl( \int_{t}^T
\bigl\vert 
\hat{\mathbb E} \bigl[
\bigl(
\hat{H}_{\ell}(\bar{\theta}_{s},\langle \check{\theta}_{s}^{(0)} \rangle)
- \hat{H}_{\ell}(\bar{\theta}_{s}^{\prime},\langle \check{\theta}_{s}^{(0)\prime} \rangle) \bigr) 
\langle \check{\vartheta}_{s}^{(0) \prime} \rangle \bigr]
\bigr\vert   \ud s \biggr)^{2p} \biggr]
\\
&\leq {\mathbb E}_{t} \biggl[ \biggl( \int_{t}^T
\bigl\vert 
\hat{\mathbb E} \bigl[
\bigl(
\hat{H}_{\ell}(\bar{\theta}_{s},\langle \check{\theta}_{s}^{(0)} \rangle)
- \hat{H}_{\ell}(\bar{\theta}_{s}^{\prime},\langle \check{\theta}_{s}^{(0)\prime} \rangle) \bigr) 
\langle \check{\vartheta}_{s}^{(0) \prime} \rangle \bigr]
\bigr\vert^2   \ud s \biggr)^p \biggr]
\\
&\leq C' 
\sup_{s \in [t,T]} \| \check{\vartheta}_{s}^{(0) \prime} 
\|_{2}^{2p}  \biggl\{ 1 \wedge
\biggl( 
{\mathbb E}_{t} \biggl[
\biggl(
 \int_{0}^T 
  \vert \bar{\theta}_{s}
 - \bar{\theta}_{s}^{\prime}
 \vert^2  \ud s
  \biggr)^{p} \biggr]
  + \Phi_{\alpha}^{2p}\bigl(\check{\theta}^{(0)},\check{\theta}^{(0)\prime}
 \bigr) 
 \biggr) \biggr\},
\end{split}
\end{equation}
which satisfies the same bound as $\Delta r_{\ell}^{2p}$.
Above the passage from the first to the third line 
may be applied with $H$ equal to $F$ or $B$ and the passage from 
the second to the third line may be applied with $H$ equal to $\Sigma$.
We have a similar bound for the term involving $\hat{G}_{\ell}$:
\begin{equation}
\label{eq:second_step:3}
\begin{split}
&{\mathbb E}_{t} \Bigl[
\bigl\vert 
\hat{\mathbb E} \bigl[
\bigl(
\hat{G}_{\ell}(X_{T},\langle \hat X_{T}\rangle)
- \hat{G}_{\ell}(X_{T}',\langle \hat X_{T}'\rangle) \bigr) \langle 
\hat {\cX}_{T}'  \rangle
\bigr]
\bigr\vert^{2p}  \Bigr]
\\
&\leq  C' 
\sup_{s \in [t,T]} \| \hat {\vartheta}_{s}^{(0) \prime} 
\|_{2}^{2p}  \Bigl[ 1 \wedge 
\Bigl( \sup_{s \in [t,T]} {\mathbb E}_{t} \bigl[ \vert {\theta}_{s}^{(0)}
 - {\theta}_{s}^{(0),\prime}
 \vert^{2p} \bigr] 
+ \Phi_{\alpha}^{2p}\bigl(\hat{\theta}^{(0)},\hat{\theta}^{(0)\prime}
 \bigr) \Bigr)   \Bigr].
\end{split}
\end{equation}
\textit{Conclusion.} In order to complete the proof of the first part, notice that the terms labelled by $a$ directly give the remainder
$\Delta {\mathcal R}_{a}^{2p}$ in 
\eqref{eq:lem:stability:1}. The second part of the statement easily follows
from Lemma \ref{lem:cauchy}.
\end{proof} 
\vspace{5pt}

\begin{Remark}
\label{rem:comparaison}
As the reader may guess,  terms 
of the form $\bar\cM^{4p}(\vartheta,\hat{\vartheta})$ and $\bar\cM^{4p}(\bar \vartheta,\check{\vartheta})$
in \eqref{eq:lem:stability:1}
will be handled by means of Corollary \ref{co:bound:sys:lin:mkv:easy}. 
However, we note that, in comparison with $\bar{\cM}^{4p}$, the `conditional' norm
$\cM^{4p}$ that is used in Corollary \ref{co:bound:sys:lin:mkv:easy} incorporates an additional pre-factor 
$\gamma^{1/2}$, see \eqref{stab:notations:1}. 
Roughly speaking, 
$\bar{\cM}^{4p}(\vartheta,\hat \vartheta)$ and
$\cM_{\E_{t}}^{4p}(\vartheta) + (\cM_{\E}^2(\hat \vartheta^{(0)})^{2p}$ are `equivalent' provided $\gamma$ is not too small. 
In the sequel, we often choose $\gamma$ exactly equal to 
$1/\Gamma_{p}$, so that 
$\bar{\cM}^{4p}(\vartheta,\hat \vartheta)$ and
$\cM_{\E_{t}}^{4p}(\vartheta) + (\cM_{\E}^2(\hat \vartheta^{(0)}))^{2p}$ can be indeed compared. 
\end{Remark}

\begin{Corollary}\label{cor:stab:hard}
Consider a family of 
progressively-measurable random paths $((\theta^\xi,\hat{\theta}^{\xi}) :  [t,T]
\ni s \mapsto (\theta^\xi_{s},\hat{\theta}^\xi_{s}))_{\xi}$
parametrized by $\xi
 \in L^2(\Omega,\cF_{t},\P;\R^d)$. Assume that, for any $p \geq 1$, 
there exists a constant $C_{p}$ such that, for all $\xi$ and $\xi'$ (with the same notation 
as in \eqref{eq:mkv:stability:Psi} but with $\Phi_{\alpha}$
defined on $[L^2(\Omega,\cF_{t},\P;\R^d)]^2$
instead of 
$[L^2(\Omega,\cA,\P;\R^d)]^2$):
\begin{equation}
\label{eq:lem:stability:hyp}
\begin{split}
&\bigl(\bar{\mathcal M}^{2p}(\theta^\xi,\hat{\theta}^\xi) \bigr)^{1/2p}
\leq C_{p} 
\bigl[ 1 + \vert \xi  \vert + \| \xi  \|_{2} \bigr],
\\
&\bigl(\bar{\mathcal M}^{2p}\bigl((\theta^\xi ,\hat{\theta}^\xi),(\theta^{\xi'},\hat{\theta}^{\xi'})\bigr) \bigr)^{1/2p}
 \leq C_{p} 
\bigl[ \vert \xi - \xi' \vert + \Phi_{\alpha}\bigl(\xi,\xi' \bigr) \bigr],
\end{split}
\end{equation}
Assume also that we can find a 
Borel subset $\cO$ of a Euclidean space, a 
continuous functional $\Psi$ 
from $\cO \times 
L^2(\hat{\Omega},\hat{\cA},\hat{\P};\R^d)$ into 
$L^2(\Omega,\cA,\P;\R_+)$
and, for any $\xi \in L^2(\Omega,\cF_{t},\P;\R^d)$, an admissible class $\cJ^\xi$ for 
$(\theta^\xi,\hat{\theta}^\xi)$
such that, for any $\Lambda$ in 
$\cJ^\xi$, there exists a 
random variable $\lambda : (\Omega,\cA,\P) \rightarrow \cO$
satisfying
$\Lambda(\omega,\cdot) \leq \Psi( \lambda(\omega), \langle \xi \rangle)$,
where $\Lambda(\omega,\cdot)$ denotes the random
variable $\hat{\Omega} \ni \hat{\omega} \mapsto \Lambda(\omega,\hat{\omega})$
on $(\hat\Omega,\hat{\cA},\hat\P)$.

With $C$ as in Lemma \ref{lem:stability}, we then let, for 
$\varsigma \in \cO$ and $\xi \in L^2(\Omega,\cF_{t},\P;\R^d)$,
\begin{equation}
\label{eq:bar psi}
\bar{\Psi}(\varsigma,\xi)(\omega) = 
\bigl( \Psi(\varsigma,\xi)(\omega) \bigr) \wedge \Bigl\{ C \bigl( 1 + \vert \xi(\omega) \vert^{ \alpha+1} 
+ \| \xi \|_{2}^{\alpha+1} \bigr) \Bigr\}, \quad \omega \in \Omega, 
\end{equation}
where $\Psi(\varsigma,\xi)$ is an abuse of notation for denoting the 
copy of the variable $\Psi(\varsigma,\langle \xi \rangle)$
on the space $\Omega$
instead of $\hat{\Omega}$. (We may indeed assume that 
$L^2(\hat{\Omega},\hat{\cA},\hat\P;\R^d)$ is 
a copy of $L^2(\Omega,\cA,\P;\R^d)$, in which case we can transfer (canonically) $\Psi(\varsigma,\cdot)$
from one space to another.)


Then, for any $p \geq 1$, there exist two constants $\Gamma_{p} := \Gamma_{p}(K) \geq 1$
and $C_{p}' >0$, such that, for $T \leq \gamma \leq 1/\Gamma_{p}$, 
choosing  $(\bar{\theta},\bar{\vartheta}) \equiv 
(\theta,\vartheta)$, 
$(\check{\theta},\check{\vartheta}) \equiv 
(\hat \theta,\hat \vartheta)$,
$(\bar{\theta}',\bar{\vartheta}') \equiv 
(\theta',\vartheta')$
and 
$(\check{\theta}',\check{\vartheta}') \equiv 
(\hat \theta',\hat \vartheta')$
 in Lemma
\ref{lem:stability}, with
$(\theta,\hat \theta):\equiv(\theta^\xi,\hat \theta^\xi)$ and 
$(\theta',\hat \theta') :\equiv (\theta^{\xi'},\hat{\theta}^{\xi'})$, it holds that:
\begin{equation}
\label{eq:stab:1ere eq}
\begin{split}
&\bigl[ {\mathcal M}_{{\mathbb E}_{t}}^{2p}
\bigl( \vartheta - \vartheta' \bigr) \bigr]^{1/2p}
\\
&\leq C_{p}' 
\biggl\{ \bigl[1 \wedge
 \bigl( \vert \xi - \xi' \vert + \Phi_{\alpha}(\xi,\xi') \bigr) \bigr] 
 \\
 &\hspace{50pt} \times
\Bigl(
\vert \eta \vert
+ \| \eta \|_{2} + \bigl( {\mathcal R}_{a}^{4p} \bigr)^{1/4p} + 
\E \bigl( {\mathcal R}_{a}^2 \bigr)^{1/2} + \bigl( {\mathcal M}_{{\mathbb E}}^{2}( \hat \vartheta^{(0)\prime})
\bigr)^{1/2}
\Bigr) 
\\
&\hspace{30pt} +   \bigl( \Delta {\mathcal R}_{a}^{2p} \bigr)^{1/2p}
+
\sup_{\varsigma \in \cO}
  \sup_{\|\Lambda_{0}\|_{2} \leq K} \biggl\{ 
 {\mathbb E} \Bigl[ \big( \Lambda_{0} 
 \wedge 
\bar{\Psi}(\varsigma, \xi  ) 
\bigr) 
 \bigl( {\mathcal M}^2_{{\mathbb E}_{t}}(\hat \vartheta^{(0)} - \hat \vartheta^{(0)\prime})
 \bigr)^{1/2}
\Bigr] \biggr\}.  
\end{split}
\end{equation}
When $\vartheta \equiv \hat{\vartheta}$ and $\vartheta' \equiv \hat{\vartheta}'$, we have
(modifying the value of $\Gamma_{p}$ if necessary):
\begin{align}
&\bigl[ {\mathcal M}_{{\mathbb E}_{t}}^{2p}\bigl( \vartheta- \vartheta' \bigr)
\bigr]^{1/2p}
\nonumber
\\
&\leq C_{p}' \biggl\{
\bigl[
1 \wedge
 \bigl( \vert \xi - \xi' \vert + \Phi_{\alpha}(\xi,\xi') \bigr) \bigr] 
\Bigl(
\vert \eta \vert
+ \| \eta \|_{2} + \bigl( {\mathcal R}_{a}^{4p}\bigr)^{1/4p} + 
\E \bigl[ {\mathcal R}_{a}^{2}  \bigr]^{1/2}
\Bigr)   + \bigl( \Delta \cR_{a}^{2p} \bigr)^{1/2p} \biggr\}
 \nonumber
\\
&\hspace{15pt} +
C_{p}' 
\biggl\{ \sup_{\varsigma \in \cO}   \sup_{\|\Lambda_{0}\|_{2} \leq K} {\mathbb E} \Bigl[ 
\bigl( \Lambda_{0} \wedge \bar{\Psi}(\varsigma, \xi) \bigr)  \bigl[1 \wedge
 \bigl( \vert \xi - \xi' \vert + \Phi_{\alpha}(\xi,\xi') \bigr) \bigr] \label{eq:stab:main estimate}
\\
&\hspace{150pt} \times
\Bigl(
\vert \eta \vert
+ \| \eta \|_{2} + \bigl( {\mathcal R}_{a}^{4} \bigr)^{1/4} + 
\E \bigl( {\mathcal R}_{a}^2 \bigr)^{1/2}
\Bigr)  \Bigr] \biggr\} \nonumber
\\
&\hspace{15pt} +  C_{p}' 
\biggl\{ \sup_{\varsigma \in \cO}   \sup_{\|\Lambda_{0}\|_{2} \leq K} {\mathbb E} \Bigl[ 
\bigl( \Lambda_{0} \wedge \bar{\Psi}(\varsigma, \xi) \bigr)  
\bigl( \Delta {\mathcal R}_{a}^{2} \bigr)^{1/2} \Bigr] \biggr\}, \nonumber
 \end{align} 
 the variable $\Lambda_{0}$ in the supremum  being in $L^2(\Omega,\cA,\P;\R_{+})$
 and the function $\Phi_{\alpha}$ differing from the original 
 one in \eqref{assumption:mkv:3} and \eqref{assumption:mkv:4} but satisfying the same 
 properties
on 
 $[L^2(\Omega,\cF_{t},\P;\R^d)]^2$
instead of 
 $[L^2(\Omega,\cA,\P;\R^d)]^2$. 
\end{Corollary}
\begin{Remark}
\label{rem:apropos de N}
Before we proceed with the proof of Corollary \ref{cor:stab:hard}, we discuss
what the assumptions we made on the structure of $\cJ^\xi$ permit to 
say on the term 
${\mathcal N}^{2p,C}_{\E_{t}}(X,\chi)$
in \eqref{stab:notations:1}. Recall indeed that  
\begin{equation*}
\begin{split}
&{\mathcal N}^{p,C}_{\E_{t}}(X,\chi) 
= \sup_{\Lambda \in \cJ^\xi} {\mathbb E}_{t}
 \biggl[ 
\hat{\mathbb E} \biggl[ \Bigl\{
 \Lambda
 \wedge
 \Bigl[ C \bigl( 1+ \hat{\E}_{t} \bigl[|\langle X \rangle|^{2\alpha+2} \bigr]^{1/2} +
 \|X \|_{2}^{\alpha+1} 
 \bigr) \Bigr] \Bigr\}
\hat{\E}_{t} \bigl[ |\langle \cX \rangle|^2 \bigr]^{1/2} \biggr]^{p} \biggr].
\end{split}
\end{equation*}
Simplifying the notations, the term inside the conditional expectation 
may be rewritten as 
$\hat{\mathbb E} [ ( \Lambda \wedge \langle W \rangle ) \langle \cW \rangle ]$,
for some random variables $\langle W \rangle$ and $\langle
\cW \rangle$ in $L^2(\hat\Omega,\hat\cA,\hat\P;\R_{+})$
and for $\Lambda \in \cJ^\xi$. 
Allowing the constant $C_{p}$ in the assumption to increase from line to line,
the following bound is proved right below:
\begin{equation}
\label{eq:bound:west coast}
\begin{split}
\E_{t} \Bigl[
\hat{\mathbb E} \Bigl[ \bigl( \Lambda \wedge \langle W \rangle \bigr) \langle \cW \rangle \Bigr]^p
\Bigr]^{1/p} 
&\leq \sup_{\varsigma \in \cO} \sup_{\|\Lambda_{0}\|_{2} \leq K} \Bigl\{ 
 {\mathbb E} \Bigl[ \bigl(  \Lambda_{0} 
 \wedge 
 \Psi(\varsigma, \xi) 
 \wedge
W \bigr) \cW \Bigr] \Bigr\}.
\end{split}
\end{equation}
where, in the above expectation, $\Lambda_{0} \in L^2(\Omega,\cA,\P;\R_{+})$,
$W$ and $\cW$ are the copies of $\langle W\rangle$ and $\langle \cW \rangle$ on the space
$\Omega$ instead of $\Omega'$. 
\end{Remark}

We first prove the remark:

\proof[Remark \ref{rem:apropos de N}.] By assumption on the structure of $\cJ^\xi$,  we can find
$\lambda$ such that
\begin{equation}
\label{eq:westcoast}
\begin{split}
\hat{\mathbb E} \Bigl[ \bigl( \Lambda \wedge \langle W \rangle \bigr) \langle \cW \rangle \Bigr]
&=
\hat{\mathbb E} \Bigl[ \bigl( \Lambda \wedge \Psi(\lambda, \langle \xi \rangle) \wedge \langle W \rangle \bigr) \langle \cW \rangle \Bigr]
\\
&\leq \sup_{\varsigma \in \cO}
\hat{\mathbb E} \Bigl[ \bigl( \Lambda \wedge \Psi(\varsigma, \langle \xi \rangle) \wedge \langle W \rangle \bigr) \langle \cW \rangle \Bigr]. 
 \end{split} 
\end{equation}
%
Recalling that $\Lambda$ is a random 
 variable $\Lambda : 
 \Omega \times \hat{\Omega} \ni (\omega,\hat{\omega}) \mapsto \Lambda(\omega,\hat{\omega})$
 on the product space $(\Omega \times \hat{\Omega},\cA \otimes \hat{\cA},\P \otimes \hat{\P})$
 such that, for almost every $\omega \in \Omega$, $\Lambda(\omega,\cdot) \in L^2(\hat{\Omega},\hat{\cA},\hat{\P};\R_{+})$
 with $\hat{\E}[\Lambda^2(\omega,\cdot)] \leq K^2$, we can bound 
 the above right-hand side by 
 \begin{equation*}
 \begin{split}
&\hat{\mathbb E} \Bigl[ \bigl( \Lambda \wedge  \Psi(\varsigma, \langle \xi \rangle) 
 \wedge \langle W \rangle \bigr) \langle \cW \rangle \Bigr]
 \\
&\leq \sup \Bigl\{ 
 \hat{\mathbb E} \Bigl[ \bigl( \langle \Lambda_{0} \rangle
 \wedge 
 \Psi(\varsigma,\langle \xi \rangle ) 
 \wedge
 \langle W \rangle \bigr) \langle \cW \rangle \Bigr] ; \
\langle \Lambda_{0} \rangle \in L^2(\hat{\Omega},\hat{\cA},\hat{\P};\R_{+}) : 
\hat{\E} \bigl[ \langle \Lambda_{0} \rangle^2 \bigr]^{1/2} \leq K \Bigr\}.
 \end{split}
 \end{equation*}
Transferring the expectation appearing in the supremum into an expectation on $\Omega$, we get
\eqref{eq:bound:west coast}.
\eproof

We now turn to:

\begin{proof}[Corollary \ref{cor:stab:hard}.]
The strategy is to make use of Lemma \ref{lem:stability} and to estimate the various terms in 
\eqref{eq:lem:stability:1}.
We use two values for the parameter 
$\gamma$ in the definition \eqref{stab:notations:1}
of ${\mathcal M}_{{\mathbb M}}^{p}$.
As suggested in Remark \ref{rem:comparaison}, 
we first use $\gamma = 1/\Gamma_{p}$. 
Since we consider the case $(\bar{\theta}',\bar{\vartheta}') \equiv 
(\theta',\vartheta')$ and $(\check{\theta}',\check{\vartheta}') \equiv 
(\hat\theta',\hat\vartheta')$, 
we deduce 
from 
\eqref{eq:co:bound:sys:lin:mkv:easy} in
Corollary \ref{co:bound:sys:lin:mkv:easy} that there exists 
a constant $C_{p}'$ such that 
\begin{equation}
\label{eq:stab:demo:1}
\begin{split}
\bigl( {\mathcal M}_{{\mathbb E}_{t}}^{2p}( \vartheta')
\bigr)^{1/2p}
 \leq 
C_{p}' \Bigl[ \vert \eta \vert + \| \eta \|_{2} + 
\bigl({\mathcal R}_{a}^{2p}\bigr)^{1/2p} + \E \bigl({\mathcal R}_{a}^2 \bigr)^{1/2}
+ \bigl( {\mathcal M}_{{\mathbb E}}^{2}( \hat \vartheta^{(0)\prime})
\bigr)^{1/2}
\Bigr].
\end{split}
\end{equation}
Recalling again 
Remark \ref{rem:comparaison}
to compare $\cM_{\E_{t}}^{4p}$ and $\bar{\cM}^{4p}$
and using in addition
\eqref{eq:lem:stability:hyp}, 
 the last term in 
\eqref{eq:lem:stability:1}, when put to the power $1/2p$, gives the contribution: 
\begin{equation*}
\begin{split}
&C_{p}' 
\Bigl[ \bigl[1 \wedge
 \bigl( \vert \xi - \xi' \vert + \Phi_{\alpha}(\xi,\xi') \bigr) \bigr] 
\Bigl(
\vert \eta \vert
+ \| \eta \|_{2} + \bigl( {\mathcal R}_{a}^{4p} \bigr)^{1/4p} + 
\E \bigl( {\mathcal R}_{a}^2 \bigr)^{1/2}
+ \bigl( {\mathcal M}_{{\mathbb E}}^{2}( \hat \vartheta^{(0)\prime})
\bigr)^{1/2}
\Bigr) 
\\
&\hspace{15pt}+   \bigl( \Delta {\mathcal R}_{a}^{2p} \bigr)^{1/2p}\Bigr].
\end{split}
\end{equation*}

We now discuss the other terms in \eqref{eq:lem:stability:1}. 
In this perspective, we use another value for $\gamma$, namely $\gamma' \leq 1/\Gamma_{p}$. Note that there is no conflict with the previous choice
for $\gamma$, which just permitted to handle the terms of the form $\bar{\cM}$ in \eqref{eq:lem:stability:1}. 
We thus turn to the two terms ${\mathcal N}^{2p,C}_{\E_{t}}$
in \eqref{eq:lem:stability:1}. Taking them to the power $1/2p$ and making use 
of the first line in \eqref{eq:lem:stability:hyp}, 
this brings us with a term of the same form as in the left-hand side of 
\eqref{eq:bound:west coast}, 
with $W=C(1+ \vert \xi \vert^{2\alpha+1} + \| \xi \|_{2}^{2\alpha+1})$
and $\cW=[ {\mathcal M}^2_{{\mathbb E}_{t}}(\hat \vartheta^{(0)} - \hat \vartheta^{(0)\prime})
 ]^{1/2}$. By \eqref{eq:bound:west coast}, we get 
 the following contribution:
 \begin{equation*}
\sup_{\varsigma \in \cO}\sup_{\|\Lambda_{0}\|_{2} \leq K} \Bigl\{ 
 {\mathbb E} \Bigl[ \bigl(  \Lambda_{0} 
 \wedge 
 \bar{\Psi}(\varsigma, \xi  ) 
\bigr)  \bigl( {\mathcal M}^2_{{\mathbb E}_{t}}(\hat \vartheta^{(0)} - \hat \vartheta^{(0)\prime})
 \bigr)^{1/2} \Bigr]  \Bigr\}.
\end{equation*}
We obtain (modifying the constant $\Gamma_{p}$ in \eqref{eq:lem:stability:1}
in order to take into account the additional exponent $1/2p$):
\begin{equation}
\label{eq:1ere eq}
\begin{split}
&\bigl[ {\mathcal M}_{{\mathbb E}_{t}}^{2p}
\bigl( \vartheta - \vartheta' \bigr) \bigr]^{1/2p}
\\
&\leq C_{p}' 
\biggl\{ \bigl[1 \wedge
 \bigl( \vert \xi - \xi' \vert + \Phi_{\alpha}(\xi,\xi') \bigr) \bigr] 
 \\
 &\hspace{30pt} \times
\Bigl(
\vert \eta \vert
+ \| \eta \|_{2} + \bigl( {\mathcal R}_{a}^{4p} \bigr)^{1/4p} + 
\E \bigl( {\mathcal R}_{a}^2 \bigr)^{1/2} + \bigl( {\mathcal M}_{{\mathbb E}}^{2}( \hat \vartheta^{(0)\prime})
\bigr)^{1/2}
\Bigr)  +   \bigl( \Delta {\mathcal R}_{a}^{2p} \bigr)^{1/2p}\biggr\}
\\
&\hspace{15pt}+  \Gamma_{p} (\gamma')^{1/4p}
\sup_{\varsigma \in \cO}
  \sup_{\|\Lambda_{0}\|_{2} \leq K} \biggl\{ 
 {\mathbb E} \Bigl[ \big( \Lambda_{0} 
 \wedge 
\bar{\Psi}(\varsigma, \xi  ) 
\bigr) 
 \bigl( {\mathcal M}^2_{{\mathbb E}_{t}}(\hat \vartheta^{(0)} - \hat \vartheta^{(0)\prime})
 \bigr)^{1/2}
\Bigr] \biggr\},   
\end{split}
\end{equation}
which gives \eqref{eq:stab:1ere eq}.

We now prove \eqref{eq:stab:main estimate} when $\vartheta \equiv \hat{\vartheta}$
and $\vartheta' \equiv \hat{\vartheta}'$. 
We go back to \eqref{eq:stab:demo:1}.  
Applying 
\eqref{eq:co:bound:sys:lin:mkv:easy:2}
in
Corollary 
\ref{co:bound:sys:lin:mkv:easy}
with $p=1$ and taking expectation, we get, for $\gamma$ small
enough, 
\begin{equation*}
\bigl( {\mathcal M}_{{\mathbb E}_{t}}^{2p}( \vartheta')
\bigr)^{1/2p}
 \leq 
C_{p}' \Bigl[ \vert \eta \vert + \| \eta \|_{2} + 
\bigl({\mathcal R}_{a}^{2p}\bigr)^{1/2p} + \E \bigl({\mathcal R}_{a}^2 \bigr)^{1/2}
\Bigr],
\end{equation*}
which means that, in \eqref{eq:1ere eq}, we can get rid of the term 
${\mathcal M}_{{\mathbb E}_{t}}^{2p}( \hat\vartheta^{(0)\prime})$ in the right-hand side. 

Let now $p=1$ in \eqref{eq:1ere eq}. Multiply both sides by 
$\Lambda_{0} \wedge \bar{\Psi}(\varsigma,\xi)$ for an $\R_{+}$-valued 
random variable $\Lambda_{0}$
 such that $\|\Lambda_{0} \|_{2} \leq K$ and take the expectation and then the supremum
 over $\Lambda_{0}$ and $\varsigma$.  For $\gamma'$ small enough, we get that
\begin{align}
&\sup_{\varsigma \in \cO}
  \sup_{\|\Lambda_{0}\|_{2} \leq K} \Bigl\{ 
 {\mathbb E} \Bigl[ \big( \Lambda_{0} 
 \wedge 
\bar{\Psi}(\varsigma, \xi  ) 
\bigr) 
 \bigl( {\mathcal M}^2_{{\mathbb E}_{t}}(\vartheta^{(0)} - \vartheta^{(0)\prime})
 \bigr)^{1/2}
\Bigr] \Bigr\}
\nonumber
\\
&\leq C' 
\biggl\{ \sup_{\varsigma \in \cO}   \sup_{\|\Lambda_{0}\|_{2} \leq K} {\mathbb E} \Bigl[ 
\bigl( \Lambda_{0} \wedge \bar{\Psi}(\varsigma, \xi) \bigr)  \bigl[1 \wedge
 \bigl( \vert \xi - \xi' \vert + \Phi_{\alpha}(\xi,\xi') \bigr) \bigr] \nonumber
\\
&\hspace{150pt} \times
\Bigl(
\vert \eta \vert
+ \| \eta \|_{2} + \bigl( {\mathcal R}_{a}^{4} \bigr)^{1/4} + 
\E \bigl( {\mathcal R}_{a}^2 \bigr)^{1/2}
\Bigr)  \Bigr] \biggr\} \nonumber
\\
&\hspace{15pt} + C' 
\biggl\{ \sup_{\varsigma \in \cO}   \sup_{\|\Lambda_{0}\|_{2} \leq K} {\mathbb E} \Bigl[ 
\bigl( \Lambda_{0} \wedge \bar{\Psi}(\varsigma, \xi) \bigr)  
\bigl( \Delta {\mathcal R}_{a}^{2} \bigr)^{1/2} \Bigr] \biggr\} \nonumber. 
\end{align}
 Plugging the above estimate into \eqref{eq:1ere eq}, we complete the proof. \qed
\end{proof}
\vspace{5pt}

Here is a very useful condition to check \eqref{eq:lem:stability:hyp}:
\begin{Lemma}
\label{lem:function:phi:alpha}
Consider a family of 
progressively-measurable random paths $((\theta^{\xi},\hat\theta^{\xi}) :  [t,T]
\ni s \mapsto (\theta^{\xi}_{s},\hat \theta_{s}^\xi))_{\xi}$
parametrized by $\xi
 \in L^2(\Omega,\cF_{t},\P;\R^d)$, with the property that 
the paths $(\hat \theta^{\xi,(0)} :  [t,T]
\ni s \mapsto \hat \theta^{\xi,(0)}_{s})_{\xi}$
are continuous, and that
$(\hat \theta_{s}^{\xi,(0)})_{s \in [t,T]}$ and 
$(\hat \theta_{s}^{\xi',(0)})_{s \in [t,T]}$ have the same distribution when 
$\xi \sim \xi'$. 

Assume that, for any $p \geq 1$, 
there exists a constant $C_{p}$ such that, for all $\xi$ and $\xi'$,
\begin{equation}
\label{eq:lem:phialpha} 
\begin{split}
 &\NSt{p}{\theta^{\xi,(0)}}
+ \NSt{p}{\hat\theta^{\xi,(0)}} + \NHt{p}{\theta^{\xi}}
 \le
  C_p \bigl(1+|\xi|+\NL{2}{\xi} \bigr),
 \\
 &\NSt{p}{\theta^{\xi,(0)}-\theta^{\xi',(0)}} 
 + \NSt{p}{\hat\theta^{\xi,(0)}-\hat\theta^{\xi',(0)}} 
 + \NHt{p}{\theta^{\xi} - \theta^{\xi'}}
\\ 
&\hspace{160pt} \le
  C_p \bigl[|\xi - \xi'|+W_{2}\bigl([\xi],[\xi'] \bigr)\bigr],
\end{split}
\end{equation}
then, we can find constants $C_{p}'$ such that, for all $\xi$ and $\xi'$
 (with the notation  \eqref{eq:mkv:stability:Psi}),
\begin{equation*}
\begin{split}
&\bigl(\bar{\mathcal M}^{2p}(\theta^\xi,\hat \theta^\xi) \bigr)^{1/2p}
\leq C_{p}' 
\bigl[ 1 + \vert \xi  \vert + \| \xi  \|_{2} \bigr],
\\
&\bigl(\bar{\mathcal M}^{2p}\bigl((\theta^\xi ,\hat {\theta}^\xi), (\theta^{\xi'},\hat{\theta}^{\xi'}) \bigr) \bigr)^{1/2p}
 \leq C_{p}' 
\bigl[ \vert \xi - \xi' \vert + \tilde{\Phi}_{\alpha}(\xi,\xi') \bigr],
\end{split}
\end{equation*}
where
\begin{equation}
\label{eq:tilde:phi:alpha}
\tilde{\Phi}_{\alpha}(\xi,\xi')= 
\E \bigl[ \vert \xi - \xi' \vert^2 \bigr]^{1/2} +
\sup_{s \in [t,T]} \Phi_{\alpha}(\hat \theta^{\xi,(0)}_{s}, \hat \theta^{\xi',(0)}_{s}), 
\quad \xi,\xi' \in L^2(\Omega,\cF_{t},\P;\R^d).  
\end{equation}
The functional $\tilde{\Phi}_{\alpha}$ is continuous at any point of the diagonal of
$[L^2(\Omega,\cF_{t},\P;\R^d)]^2$ and satisfies \eqref{assumption:mkv:4}
(up to a modification of the constant $C$ therein).  
\end{Lemma}
\proof
The bound for 
$(\bar{\mathcal M}^{2p}(\theta^\xi, \hat \theta^\xi))^{1/2p}$
is a straightforward consequence of 
the first line in \eqref{eq:lem:phialpha}. 
The bound for 
$(\bar{\mathcal M}^{2p}((\theta^\xi,\hat \theta^\xi),(\theta^{\xi'},\hat \theta^{\xi'}))^{1/2p}$
follows from the second line 
in \eqref{eq:lem:phialpha} and from the definition of 
$\bar{\cM}^{2p}$ 
in \eqref{eq:mkv:stability:Psi}.

The main issue is to check that $\tilde{\Phi}_{\alpha}$
satisfies the same condition as $\Phi_{\alpha}$. 
By \eqref{eq:lem:phialpha}, the map 
$L^2(\Omega,\cA,\P;\R^d) \ni \xi \mapsto 
(\hat \theta_{s}^{\xi,(0)})_{s \in [t,T]} \in \cS^2([t,T];\R^{l})$
(with the appropriate $l$) is continuous
and, for any $\xi \in L^2(\Omega,\cA,\P;\R^d)$, the map 
$[t,T] \ni s \mapsto \hat \theta_{s}^{\xi,(0)} \in L^2(\Omega,\cA,\P;\R^l)$
is also continuous, 
proving that, for any sequence $(\xi_{n})_{n \geq 1}$ converging to $\xi$ in $L^2$, 
the family of random variables $(\hat \theta_{s}^{\xi_{n},(0)})_{s \in [t,T],n \geq 1}$
is relatively compact. Since, for any compact subset $\cK \subset L^2(\Omega,\cA,\P;\R^d)$,
$\sup \{ \Phi_{\alpha}(\chi,\chi'), \chi, \chi' \in \cK, \| \chi - \chi' \|_{2} \leq \delta \}$
tends to $0$ with $\delta$, continuity of 
$\tilde{\Phi}_{\alpha}$ at any point of the diagonal easily follows. 

Now, we check that $\tilde{\Phi}_{\alpha}$
satisfies \eqref{assumption:mkv:4} when $\xi$ and $\xi'$ have the same distribution.
Since $\hat\theta_{s}^{\xi,(0)}$ and $\hat\theta_{s}^{\xi',(0)}$ have the same distribution, 
we deduce from \eqref{eq:lem:phialpha} that
\begin{equation*}
\begin{split}
&\E \bigl[ \bigl( 1 + \vert \hat\theta_{s}^{\xi,(0)} \vert^{2\alpha} + \vert \hat\theta_{s}^{\xi',(0)}
\vert^{2\alpha}
+ \| \hat\theta_{s}^{\xi,(0)} \|_{2}^{2\alpha}
 \bigr) \vert \hat\theta_{s}^{\xi,(0)} - \hat\theta_{s}^{\xi',(0)} \vert^2 \bigr]^{1/2}
\\
&\leq \E \Bigl[ {\mathbb E}_{t}
\bigl[
\bigl( 1 + \vert \hat \theta_{s}^{\xi,(0)} \vert^{4\alpha} + \vert \hat \theta_{s}^{\xi',(0)}
\vert^{4\alpha}
+ \| \hat \theta_{s}^{\xi,(0)} \|_{2}^{4\alpha}
 \bigr) \bigr]^{1/2} 
\E_{t} \bigl[
\vert \hat \theta_{s}^{\xi,(0)} - \hat \theta_{s}^{\xi',(0)} \vert^4 \bigr]^{1/2} \Bigr]^{1/2}
\\
&\leq C \E \bigl[ \bigl( 1+ \vert \xi \vert^{2\alpha} + 
\vert \xi' \vert^{2\alpha}
+\| \xi \vert^{2\alpha}_{2}
\bigr)  \vert \xi - \xi' \vert^2  \bigr]^{1/2}.  
\end{split}
\end{equation*}
\eproof

\begin{Example}
\label{example:UI}
We illustrate the meaning of \eqref{eq:stab:main estimate} in the simplest (but crucial) case
when $\cR_{a}\equiv \Delta \cR_{a} \equiv 0$. Clearly, the most challenging term is 
\begin{equation*}
\sup_{\varsigma \in \cO}   \sup_{\|\Lambda_{0}\|_{2} \leq K} {\mathbb E} \Bigl[ 
\bigl( \Lambda_{0} \wedge \bar{\Psi}(\varsigma, \xi) \bigr)  \bigl[1 \wedge
 \bigl( \vert \xi - \xi' \vert + \Phi_{\alpha}(\xi,\xi') \bigr) \bigr] 
\bigl(
\vert \eta \vert
+ \| \eta \|_{2} \bigr)  \Bigr],
\end{equation*}
which is less than
\begin{equation*}
\sup_{\varsigma \in \cO}   \sup_{\|\Lambda_{0}\|_{2} \leq K} {\mathbb E} \Bigl[ 
\bigl( \Lambda_{0} \wedge \bar{\Psi}(\varsigma, \xi) \bigr)^2  \bigl[1 \wedge
 \bigl( \vert \xi - \xi' \vert + \Phi_{\alpha}(\xi,\xi') \bigr) \bigr]^2  \Bigr]^{1/2} \|\eta\|_{2}
\leq  \bar{\Phi}(\xi,\xi') \| \eta \|_{2},
\end{equation*}
with
\begin{equation}
\label{eq:bar phi}
\bar{\Phi}(\xi,\xi') = 
\sup_{\varsigma \in \cO}   \sup_{\|\Lambda_{0}\|_{2} \leq K} {\mathbb E} \Bigl[ 
\bigl( \Lambda_{0} \wedge \bar{\Psi}(\varsigma, \xi) \bigr)^2  \bigl[1 \wedge
 \vert \xi - \xi' \vert^2  \bigr] \Bigr]^{1/2} + K \Phi_{\alpha}(\xi,\xi').
 \end{equation}
Recalling the bound 
$0 \leq \bar{\Psi}(\varsigma,\xi) \leq C( 1 + \vert \xi(\omega) \vert^{ \alpha+1} 
+ \| \xi \|_{2}^{\alpha+1})$,
there exists a constant $C'$ such that, whenever $\xi$ and $\xi'$ 
have the same distribution,
\begin{equation*} 
\bar\Phi(\xi,\xi') \leq 
C' {\mathbb E} \bigl[ \bigl(1+ \vert \xi \vert^{2\alpha+2} +
  \vert \xi' \vert^{2\alpha+2} \bigr)
   \vert \xi - \xi' \vert^2 
  \bigr]^{1/2},
 \end{equation*} 
which fits \eqref{assumption:mkv:4},
with $\alpha+1$ instead of $\alpha$, up to another multiplicative constant. 
The functional $\bar{\Phi}$ is thus a candidate for being a function 
of the same type as $\Phi_{\alpha+1}$, according to the notation used in 
the assumptions \eqref{assumption:mkv:1}--\eqref{assumption:mkv:5}. 
Still, in order to guarantee that $\bar{\Phi}$ indeed
satisfies the same assumptions as $\Phi_{\alpha+1}$, it is necessary 
to prove that it is continuous at any point of the diagonal. We claim that it is the case under 
the two additional conditions (the proof is given right below):

(i) for each $\xi \in L^2(\Omega,\cF_{t},\P;\R^d)$, the family 
$(\Psi^2(\varsigma,\xi))_{\varsigma \in \cO}$ is uniformly integrable,

(ii)
the mappings $(L^2(\Omega,\cF_{t},\P;\R^d) \ni \xi \mapsto \Psi(\varsigma,\xi) \in L^2(\Omega,\cA,\P;\R^d))_{\varsigma \in \cO}$
are equicontinuous.

As an example of a
family $(\theta^{\xi},\hat{\theta}^\xi)_{\xi}$
and a
 functional $\Psi : \cO \times L^2(\Omega,\cF_{t},\P;\R^d) \ni (\varsigma,\xi) 
\mapsto \Psi(\varsigma,\xi)$ that satisfy the prescription in 
Corollary \ref{cor:stab:hard} together with
$(i)$ and $(ii)$, we can consider (again, the proof is given right below)
$(\theta^\xi :\equiv \theta^{t,\xi},\hat\theta^\xi :\equiv \hat \theta^{t,\xi})_{\xi}$ 
or $(\theta^\xi :\equiv \theta^{t,x,[\xi]},\hat\theta^\xi :\equiv \hat \theta^{t,\xi})_{\xi}$
and 
\begin{equation}
\label{eq:sncf:1}
\Psi\bigl(\varsigma=(w,s),\xi \bigr) = \sup_{H=B,\Sigma,F,G}  
{\E}_{t} \Bigl[ \bigl\vert \hat{H}_{\ell}\bigl(w,
\hat\theta_{s}^{t,\xi,(0)}  \bigr) \bigr\vert^2 \Bigr]^{1/2},
\end{equation} 
for 
$w \in \R^d \times \R^m \times \R^{m \times d}$ and $s \in [t,T]$.
(The definition of $\Psi(\varsigma,\xi)$ for 
a random variable 
$\xi$ that is not $\cF_{t}$-measurable is useless here, 
{since $\xi$ is exclusively thought as an initial condition 
of the system 
 \eqref{eq X-Y t,xi}
at time $t$}.)  
\end{Example}

\proof
\textit{First step.}
 We first check that, under \textit{(i)} and \textit{(ii)}, 
$\bar{\Phi}$ is continuous at any point of the diagonal. 
Given two sequences $(\xi_{n})_{n \geq 0}$
and $(\xi_{n}')_{n \geq 0}$ converging in $L^2(\Omega,\cF_{t},\P;\R^d)$ towards some $\xi$, we already know that 
$(\Phi_{\alpha}(\xi_{n},\xi_{n}'))_{n \geq 0}$
converges to $0$. Therefore, it suffices to focus on the first term 
in the right-hand side of \eqref{eq:bar phi}. We have
\begin{equation}
\label{eq:17:09:14:1}
\begin{split}
&\Bigl\vert \sup_{\varsigma \in \cO} \sup_{\| \Lambda_{0} \|_{2} \leq K}
\E \Bigl[ 
\bigl( \Lambda_{0} \wedge \bar{\Psi}(\varsigma,\xi_{n}) \bigr)^2
\bigl[
1 \wedge 
 \vert \xi_{n} - \xi_{n}' \vert^2 
\bigr]
\Bigr] 
\\
&\hspace{15pt}-
\sup_{\varsigma \in \cO} \sup_{\| \Lambda_{0} \|_{2} \leq K}
\E \Bigl[ 
\bigl( \Lambda_{0} \wedge \bar{\Psi}(\varsigma,\xi) \bigr)^2
\bigl[
1 \wedge 
 \vert \xi - \xi' \vert^2 
\bigr]
\Bigr] \Bigr\vert
\\
&\leq 
\sup_{\varsigma \in \cO} \sup_{\| \Lambda_{0} \|_{2} \leq K}
\E \Bigl[ 
\bigl( \Lambda_{0} \wedge \bar{\Psi}(\varsigma,\xi) \bigr)^2
\bigl\vert
1 \wedge 
 \vert \xi - \xi' \vert^2 
- 1 \wedge 
 \vert \xi_{n} - \xi_{n}' \vert^2 
\bigr\vert
\Bigr]
\\
&\hspace{15pt} + 
\sup_{\varsigma \in \cO} \sup_{\| \Lambda_{0} \|_{2} \leq K}
\E \Bigl[ \Bigl\vert
\bigl( \Lambda_{0} \wedge \bar{\Psi}(\varsigma,\xi) \bigr)^2
- \bigl( \Lambda_{0} \wedge \bar{\Psi}(\varsigma,\xi_{n}) \bigr)^2
\Bigr\vert \Bigr].
\end{split}
\end{equation}
Recalling the bound $\bar{\Psi}(\varsigma,\xi) \leq \Psi(\varsigma,\xi)$, the first term in the right-hand side 
is less than 
\begin{equation*}
\sup_{\varsigma \in \cO} 
\E \Bigl[ 
\Psi^2(\varsigma,\xi) 
\bigl\vert
1 \wedge 
 \vert \xi - \xi' \vert^2 
- 1 \wedge 
 \vert \xi_{n} - \xi_{n}' \vert^2 
\bigr\vert
\Bigr],
\end{equation*}
which tends to $0$ by uniform integrability of the family $(\Psi(\varsigma,\xi))_{\varsigma \in \cO}$.

Consider now the second term in the right-hand side of 
\eqref{eq:17:09:14:1}. We have
\begin{equation*}
\begin{split}
&\sup_{\varsigma \in \cO} \sup_{\| \Lambda_{0} \|_{2} \leq K}
\E \Bigl[ \Bigl\vert
\bigl( \Lambda_{0} \wedge \bar{\Psi}(\varsigma,\xi) \bigr)^2
- \bigl( \Lambda_{0} \wedge \bar{\Psi}(\varsigma,\xi_{n}) \bigr)^2
\Bigr\vert \Bigr]
\\
&\leq 2K \sup_{\varsigma \in \cO} \sup_{\| \Lambda_{0} \|_{2} \leq K}
\E \Bigl[ \Bigl\vert
 \Lambda_{0} \wedge \bar{\Psi}(\varsigma,\xi) 
- \Lambda_{0} \wedge \bar{\Psi}(\varsigma,\xi_{n}) \Bigr\vert^2 \Bigr]^{1/2}. 
\end{split}
\end{equation*}
Recalling from \eqref{eq:bar psi}
that $\bar{\Psi}(\varsigma,\xi) = \Psi(\varsigma,\xi) \wedge\varphi(\xi)$, 
with $\varphi(\xi) = [C(1+ \vert \xi \vert^{\alpha+1} + \| \xi \|_{2}^{\alpha+1}]$,
writing $\vert 
\Lambda_{0} \wedge \bar{\Psi}(\varsigma,\xi) 
- \Lambda_{0} \wedge \bar{\Psi}(\varsigma,\xi_{n}) 
\vert \leq 
\vert 
 \Lambda_{0} \wedge \Psi(\varsigma,\xi) \wedge \varphi(\xi)
- \Lambda_{0} \wedge \Psi(\varsigma,\xi) \wedge \varphi(\xi_{n}) 
\vert
+ 
\vert 
\Lambda_{0} \wedge \Psi(\varsigma,\xi) \wedge \varphi(\xi_{n}) 
- \Lambda_{0} \wedge \Psi(\varsigma,\xi_{n}) \wedge \varphi(\xi_{n}) \vert$
and using the Lipschitz property of the
map $\R \ni x \mapsto a \wedge x$, for any $a \in \R$, 
we deduce that
\begin{equation*}
\begin{split}
&\sup_{\varsigma \in \cO} \sup_{\| \Lambda_{0} \|_{2} \leq K}
\E \Bigl[ \Bigl\vert
\bigl( \Lambda_{0} \wedge \bar{\Psi}(\varsigma,\xi) \bigr)^2
- \bigl( \Lambda_{0} \wedge \bar{\Psi}(\varsigma,\xi_{n}) \bigr)^2
\Bigr\vert \Bigr]
\\
&\leq 2K \biggl[ \sup_{\varsigma \in \cO} 
\E \Bigl[ \Bigl\vert
{\Psi}(\varsigma,\xi) 
-  {\Psi}(\varsigma,\xi_{n}) \Bigr\vert^2 \Bigr]^{1/2} 
+ \sup_{\varsigma \in \cO} 
\E \Bigl[ \Bigl\vert \Psi(\varsigma,\xi) \wedge \varphi(\xi) - 
\Psi(\varsigma,\xi) \wedge \varphi(\xi_{n})
\Bigr\vert^2 \Bigr] \biggr]^{1/2}. 
\end{split}
\end{equation*}
By uniform continuity of the mappings $(\Psi(\varsigma,\cdot))_{\varsigma \in \cO}$, the first term in the right-hand side tends to $0$. By uniform integrability of the family $(\Psi^2(\varsigma,\xi))_{\varsigma \in \cO}$, 
the second one also tends to $0$. 

\textit{Second step.}
We now check 
the example.
By Lemmas \ref{le app cont} and 
\ref{lem:function:phi:alpha},
\eqref{eq:lem:stability:hyp}
is satisfied with $(\theta^\xi,\hat \theta^{\xi}) :\equiv (\theta^{t,\xi},\hat\theta^{t,\xi})$
or $(\theta^\xi,\hat \theta^{\xi}) :\equiv(\theta^{t,x,[\xi]},\hat\theta^{t,\xi})$.
We prove that  $\Psi$ in \eqref{eq:sncf:1}
satisfies \textit{(i)} and \textit{(ii)}. 
We check first the uniform integrability property \textit{(i)}.
It suffices to check it for $H$ equal to $B$, $\Sigma$, $F$ or $G$
(if uniform integrability holds for $H$ equal to $B$, $\Sigma$, $F$ or $G$, then 
the supremum over $H$ equal to $B$, $\Sigma$, $F$ or $G$
also satisfies \textit{(i)}).
 Given $\xi \in L^2(\Omega,\cF_{t},\P;\R^d)$, 
it suffices to prove that the family $(\sup_{r \in [t,T]} \vert \hat{H}_{\ell}(w, 
\hat\theta_{r}^{t,\xi,(0)}) \vert^2)_{w \in \R^k}$ (for the appropriate $k$) is uniformly integrable. 
Consider a positive constant  $\varepsilon >0$. 
Since the path $[t,T] \ni r \mapsto \hat\theta_{r}^{t,\xi,(0)}$ is continuous
and $\Phi_{\alpha}$ is continuous at any point of the diagonal, we can find
a constant $\delta >0$ such that 
\begin{equation}
\label{eq:19:09:14:2}
\sup_{(r,s) \in [t,T]^2 : \vert s-r \vert \leq \delta} \Phi_{\alpha}\bigl(\hat\theta_{r}^{t,\xi,(0)},\hat\theta_{s}^{t,\xi,(0)}\bigr)
\leq \varepsilon.
\end{equation}
Then, for $(r,s) \in [t,T]^2$, Cauchy Schwarz' inequality yields
\begin{equation*}
\begin{split}
\Bigl\vert \E \bigl[ \vert \hat{H}_{\ell}(w,\hat\theta_{s}^{t,\xi,(0)}) \vert^2
\bigr] -
 \E \bigl[ \vert \hat{H}_{\ell}(w,\hat\theta_{r}^{t,\xi,(0)}) \vert^2
\bigr]  \Bigr\vert &\leq 
2K 
 \E \Bigl[ \bigl\vert \hat{H}_{\ell}(w,\hat\theta_{s}^{t,\xi,(0)}) -  \hat{H}_{\ell}(w,\hat\theta_{r}^{t,\xi,(0)}) \bigr\vert^2
\Bigr]^{1/2} 
\\
&\leq 2K \varepsilon^{1/2}. 
\end{split}
\end{equation*}
Therefore, 
denoting by $(t=s_{0}<s_{1}<\dots<s_{N}=T)$ a subdivision of $[t,T]$ 
with stepsize less than $\delta$, we deduce that, 
for any event $A \in \cA$, 
\begin{equation*}
\begin{split}
\sup_{w \in \R^k}
\sup_{t \leq r \leq T}
\E \Bigl[
 \bigl\vert \hat{H}_{\ell}\bigl(w,
\hat\theta_{r}^{t,\xi,(0)}  \bigr) \bigr\vert^2 {\mathbf 1}_{A} \Bigr]
&\leq \sup_{w \in \R^k}\sup_{i=0,\dots,N}  \E \Bigl[
 \bigl\vert \hat{H}_{\ell}\bigl(w,
\hat\theta_{s_{i}}^{t,\xi,(0)}  \bigr) \bigr\vert^2 {\mathbf 1}_{A} \Bigr]
+ 2K \varepsilon^{1/2}. 
\end{split}
\end{equation*}
By the uniform integrability of each of the family 
$(\vert \hat{H}_{\ell}(w,\hat\theta_{s_{i}}^{t,\xi,(0)})  \vert^2)_{w \in \R^k}$, 
for $i=0,\dots,N$, see \HYP{1}, we deduce that the left-hand side is indeed less
than $4K \varepsilon^{1/2}$ for $\delta$ small enough. 

We check uniform continuity of the mappings
$(\xi \mapsto \hat{H}_{\ell}(w,\hat\theta_{s}^{t,\xi,(0)}))_{w \in \R^k,s \in [t,T]}$: 
\begin{equation*}
\sup_{w \in \R^k} \sup_{s \in [t,T]}
\E \Bigl[
\bigl\vert
\hat{H}_{\ell}\bigl(w,\hat\theta_{s}^{t,\xi,(0)}\bigr)
- \hat{H}_{\ell}\bigl(w,\hat\theta_{s}^{t,\xi',(0)}\bigr) \bigr\vert^2 \Bigr]^{1/2}
\leq C \sup_{s \in [t,T]} 
\Phi_{\alpha}\bigl( \hat\theta_{s}^{t,\xi,(0)},\hat\theta_{s}^{t,\xi',(0)} \bigr),
\end{equation*}
which tends to $0$ as $\xi'- \xi$ tends to $0$, by the same argument 
as in Lemma \ref{lem:function:phi:alpha}.
\eproof
\vspace{5pt}

We complete the subsection with a very important observation:

\begin{Remark}
\label{rem:unif:phi:alpha}
Example \ref{example:UI} ensures that
Corollary \ref{cor:stab:hard} may be applied 
with $(\theta^\xi,\hat{\theta}^\xi):\equiv (\theta^{t,\xi},\theta^{t,\xi})$
or $(\theta^\xi,\hat{\theta}^\xi):\equiv (\theta^{t,x,[\xi]},\theta^{t,\xi})$,
in which case \eqref{eq:lem:stability:hyp} holds for 
a suitable function $\Phi_{\alpha}$ (defined on $[L^2(\cA,\cF_{t},\P;\R^d)]^2$) 
and the 
the second term in the right-hand side of \eqref{eq:stab:main estimate}
may be bounded by a function of 
the type $\Phi_{\alpha+1}$ (also defined on $[L^2(\cA,\cF_{t},\P;\R^d)]^2$). 

It is worth mentioning that, with the construction that is suggested, 
both $\Phi_{\alpha}$ and $\Phi_{\alpha+1}$ may depend on 
$t$, which is clear from \eqref{eq:tilde:phi:alpha}
and \eqref{eq:sncf:1}. 

Below, we want to use versions of both that are independent of 
$t$. This requires first to restrict the domain of definition of both functionals to 
$[L^2(\cA,\cF_{0},\P;\R^d)]^2$. Second, this requires a suitable
adaptation of \eqref{eq:tilde:phi:alpha} and \eqref{eq:sncf:1}. 

When $\xi \in L^2(\Omega,\cF_{0},\P;\R^d)$, we may extend $(\theta^{t,\xi}_{s})_{s \in [t,T]}$
to the interval $[0,T]$ by letting 
$X^{t,\xi}_{s}=\xi$, 
$Y^{t,\xi}_{s} = Y^{t,\xi}_{t}$ 
and $Z^{t,\xi}_{s} = 0$
for $s \in [0,t]$. 
Then, 
for $\xi,\xi' \in L^2(\Omega,\cF_{0},\P;\R^d)$,
instead of \eqref{eq:tilde:phi:alpha}, we may let
\begin{equation*}
\tilde{\Phi}_{\alpha}(\xi,\xi')= 
\E \bigl[ \vert \xi - \xi' \vert^2 \bigr]^{1/2} +
\sup_{t \in [0,T]} \sup_{s \in [0,T]} \Phi_{\alpha}(\hat \theta^{t,\xi,(0)}_{s}, \hat \theta^{t,\xi',(0)}_{s}), 
\end{equation*}
(that is we also take the supremum in $t$),
and, instead of \eqref{eq:sncf:1}, we may let
\begin{equation*}
\Psi\bigl(\varsigma=(w,t,s),\xi \bigr) = \sup_{H=B,\Sigma,F,G}  
{\E}_{t} \Bigl[ \bigl\vert \hat{H}_{\ell}\bigl(w,
 \hat\theta_{s}^{t,\xi,(0)} \bigr) \bigr\vert^2 \Bigr]^{1/2},
\end{equation*} 
(that is we include $t$ in the variable $\varsigma$). 

Then, the resulting new functionals $\Phi_{\alpha}$ and ${\Phi}_{\alpha+1}$
are independent of $t$, are 
continuous at any point of the diagonal of 
$[L^2(\Omega,\cF_{0},\P;\R^d)]^2$
 and satisfy \eqref{assumption:mkv:4}
 with respect to $\alpha$ and $\alpha+1$. 
The 
proof works exactly as in Lemma
 \ref{lem:function:phi:alpha}
and in Example
\ref{example:UI}, noticing that 
that the mapping $L^2(\Omega,\cF_{0},\P;\R^d) \times [0,T] \times
[0,T] \ni (\xi,s,t) \mapsto \theta_{s}^{t,\xi} \in L^2(\Omega,\cA,\P;\R^d)$
is continuous (which is the main ingredient to make the argument work). 
\end{Remark}

 \subsection{Analysis of the first-order derivatives} 
  \subsubsection{First-order derivatives
 of the McKean-Vlasov system} 
As we already explained in Examples \ref{ex appli} and \ref{ex appli suite}, the 
shape of the system 
\eqref{eq:sys:linear:mkv} has been specifically designed in order to 
investigate the derivative of the system of the original FBSDE in the direction 
of the measure. Thus, we shall make use of the results from 
Subsection \ref{subse:stab:1}, the constant $L$
 in 
\HYP{0}(i)-\HYP{1} now playing the role of the constant $K$ in the above statements. 
In order to stress the fact that this subsection is devoted to the application of the general results 
proved above to the specific question of the differentiability of the flow, we shall use constants 
$c(L)$ or $c_{p}(L)$ instead of $1/\Gamma(K)$ or $1/\Gamma_{p}(K)$ for quantifying
small time constraints of the type 
$T \leq c(L)$ or $T \leq c_{p}(L)$.

To make things clear, we also recall the identification of $h_{\ell}$, 
$\hat{H}_{\ell}$ and $H_{a}$
in \eqref{eq:ex:mkv:stability}:
\begin{equation}
\label{eq:identification:derivatives}
h_{\ell}(w,\langle \hat V^{(0)} \rangle) = \partial_{w} h(w,\law{\hat V^{(0)}}), \ 
\hat{H}_{\ell}(w,\langle \hat V^{(0)} \rangle) = 
\partial_{\mu} h(w,\law{\hat V^{(0)}})(\langle \hat V^{(0)} \rangle), \ 
H_{a} \equiv 0.
\end{equation}

The next results state the first order differentiability of the McKean-Vlasov 
system.

\begin{Lemma}
\label{lem:1storder:diff}
Given a continuously differentiable path of initial conditions 
$\R \ni \lambda \mapsto \xi^{\lambda}
\in L^2(\Omega,{\mathcal F}_{t},\P;\R^d)$, $t$ standing for the initial time in 
$[0,T]$, we can find a constant $c:=c(L) >0$
such that, for $T \leq c$,  the path 
$\R \ni \lambda \mapsto \theta^{\lambda} = 
(X^{\lambda},Y^{\lambda},Z^{\lambda}) := \theta^{t,\xi^\lambda} \in {\mathcal S}^2([t,T];\R^d) 
\times 
{\mathcal S}^2([t,T],\R^m) \times {\mathcal H}^2([t,T];\R^{m \times d})$
is continuously differentiable. 
\end{Lemma}

\proof
Under 
\HYP{0}(i), existence and uniqueness of a solution 
to \eqref{eq intro XYZ} may be proved for a small time horizon $T$ by a contraction argument. 
As in \cite{del02}, for $T$ small enough, we can approximate 
$(X^{\lambda},Y^{\lambda},Z^{\lambda})$ as the limit of 
a Picard sequence $\theta^{n,\lambda} := (X^{n,\lambda},Y^{n,\lambda},Z^{n,\lambda})$,
defined by 
\begin{equation*}
\begin{split}
&X^{n+1,\lambda}_{s}
= \xi^\lambda + \int_{t}^s b(\theta^{n,\lambda}_{r},[
\theta^{n,\lambda,(0)}_{r}]) \ud r + \int_{t}^s 
\sigma(\theta^{n,\lambda,(0)}_{r},[
\theta^{n,\lambda,(0)}_{r}]) \ud W_{r}
\\
&Y^{n+1,\lambda}_{s} = 
g(X_{T}^{n+1,\lambda},[
X_{T}^{n+1,\lambda}]) + \int_{s}^T
f(\theta^{n,\lambda}_{r},[
\theta^{n,\lambda,(0)}_{r}]) \ud r - \int_{s}^T Z_{r}^{n+1,\lambda} \ud W_{r},
\end{split}
\end{equation*}
where we have used the notation $\theta^{n,\lambda,(0)}_{s} = 
(X_{s}^{n,\lambda},Y_{s}^{n,\lambda})$, with the initialization 
$\theta^{0,\lambda} \equiv 0$. By the standard theory of It\^o processes and 
backward equations 
(see in particular \cite{pardoux:peng}), we can prove by induction that, for any 
$n \geq 0$, the mapping  
$\R \ni \lambda \mapsto \theta^{\lambda} = 
(X^{n,\lambda},Y^{n,\lambda},Z^{n,\lambda}) \in {\mathcal S}^2([t,T];\R^d) 
\times 
{\mathcal S}^2([t,T],\R^m) \times {\mathcal H}^2([t,T];\R^{m \times d})$
is continuously differentiable. We give just a sketch of proof. For the forward component, 
this follows from the fact that given a continuously differentiable path $\R \ni \lambda \mapsto 
h^\lambda \in {\mathcal H}^2([t,T],\R)$, the paths 
$\R \ni \lambda \mapsto (\int_{t}^{s} h_{r}^{\lambda} \ud r)_{s \in [t,T]}$
and $\R \ni \lambda \mapsto (\int_{t}^s h_{r}^{\lambda} \ud W_{r})_{s \in [t,T]}$, 
with values in ${\mathcal S}^2([t,T],\R^l)$ for a suitable dimension $l$,
are continuously differentiable, which is obviously true. To handle the backward component, it suffices to prove first that the path $\R \ni \lambda \mapsto ({\mathbb E}_{s}[h_{T}^{\lambda}])_{s \in [t,T]}$, with values in ${\mathcal S}^2([t,T],\R)$, is continuously differentiable, which is straightforward by means of Doob's inequality.
This is enough to handle the terminal condition and also the driver since we can split the integral from $s$ to $T$ into an integral from $t$ to $s$ (to which we can apply the result used for the forward component) and an integral from $t$ to $T$ (which can be seen as a new $h_{T}$).
In this way, we can prove that 
$\R \ni \lambda \mapsto Y^{n+1,\lambda}$ is continuously differentiable
from $\R$ to ${\mathcal S}^2([t,T],\R^m)$. This shows that 
$\R \ni \lambda \mapsto (\int_{t}^s Z_{r}^{n+1,\lambda} \ud W_{r})_{s \in [t,T]}$
is also continuously differentiable from $\R$ to ${\mathcal S}^2([t,T],\R^m)$.
By It\^o's isometry, this finally proves that 
$\R \ni \lambda \mapsto (Z_{s}^{n+1,\lambda})_{s \in [t,T]}$
is continuously differentiable from $\R$ to ${\mathcal H}^2([t,T],\R^{m \times d})$, the derivative of 
$Z^{n+1,\lambda}$ writing as the martingale representation term of 
the derivative of $\int_{t}^T Z_{r}^{n+1,\lambda} \ud W_{r} $.

The derivatives, denoted by 
$(\cX^{n,\lambda},\cY^{n,\lambda},\cZ^{n,\lambda})$,
satisfy the system
\begin{equation}
\label{eq:syst:derive:picard}
\begin{split}
&\cX^{n+1,\lambda}_{s}
= \chi^\lambda + \int_{t}^s 
B^{(1)}\bigl(r,\theta^{n,\lambda}_{r},\langle \theta_{r}^{n,\lambda,(0)} \rangle\bigr)\bigl(\vartheta^{n,\lambda}_{r},
\langle \vartheta^{n,\lambda,(0)}_{r} \rangle
\bigr)
 \ud r 
 \\
 &\hspace{50pt}
 + \int_{t}^s 
\Sigma^{(1)}(r,\theta^{n,\lambda,(0)}_{r},\langle \theta_{r}^{n,\lambda,(0)} \rangle)
\bigl(\vartheta^{n,\lambda,(0)}_{r},\langle 
\vartheta^{n,\lambda,(0)}_{r}
\rangle
\bigr)
 \ud 
W_{r}
\\
&\cY^{n+1,\lambda}_{s} = 
G^{(1)}(X_{T}^{n+1,\lambda},\langle X_{T}^{n+1,\lambda} \rangle)\bigl(\cX_{T}^{n+1,\lambda},\langle 
\cX_{T}^{n+1,\lambda} \rangle \bigr)
\\
&\hspace{50pt} + \int_{s}^T
F^{(1)}\bigl(r,\theta^{n,\lambda}_{r},\langle \theta_{r}^{n,\lambda,(0)} \rangle\bigr)
\bigl(\vartheta^{n,\lambda}_{r},\langle \vartheta^{n,\lambda,(0)}_{r} \rangle \bigr)
 \ud r - \int_{s}^T \cZ_{r}^{n+1,\lambda} \ud W_{r},
\end{split}
\end{equation}
where we have used the notations 
$\chi^\lambda = [\ud/\ud \lambda]\xi^{\lambda}$, 
$\vartheta^{n,\lambda} = (\cX^{n,\lambda},\cY^{n,\lambda},\cZ^{n,\lambda})$
and $\vartheta^{n,\lambda,(0)} = (\cX^{n,\lambda},\cY^{n,\lambda})$
 and where
$B$, $\Sigma$, $F$ and $G$ 
are defined according to 
\eqref{eq:identification:derivatives}
and are denoted by $B^{(1)}$, 
$\Sigma^{(1)}$, $F^{(1)}$ and $G^{(1)}$
as in
\eqref{eq:ex:mkv:stability}, the superscript $(1)$ stressing the fact that we are dealing 
with \emph{first-order} derivatives. 
We thus obtain a system of the form 
\eqref{eq:sys:linear:mkv:2} with 
${\theta} \equiv \hat{\theta} \equiv \theta^{n+1,\lambda}$,
$\bar{\theta} \equiv \check \theta \equiv \theta^{n,\lambda}$, 
$\vartheta \equiv \hat{\vartheta} \equiv \vartheta^{n+1,\lambda}$
and
$\bar\vartheta \equiv \check \vartheta \equiv 
\vartheta^{n, \lambda}$ and $\chi^\lambda$ playing the role of $\eta$.
We now apply Lemma \ref{lem:apriori}, {noticing that the 
remainder ${\mathcal R}_{a}$ therein is zero, see \eqref{eq:identification:derivatives}.} 

First, we set  $p=1$ in \eqref{eq:lem:apriori:2} and choose $\gamma = 1/\Gamma_{1}(L)$
in \eqref{stab:notations:1},
in agreement with Remark
\ref{rem:comparaison}. We then take expectation on both sides. 
We get that, for $T$ small, the sequence 
$({\mathcal M}_{{\mathbb E}}^{2}(\vartheta^{n,\lambda}))_{n \geq 1}$ is at most of arithmetico-geometric type, with a geometric rate strictly less than 1.
By induction,
we deduce that there exist two constants $c:=c(L)>0$ and $C \geq 0$ (the values of which are allowed to increase from one line to another), such that, for $T \leq c$,
$\sup_{n \geq 0} 
{\mathcal M}_{{\mathbb E}}^2( \vartheta^{n,\lambda}) \leq C \| \chi^\lambda \|_2^2.$
Inserting this estimate into \eqref{eq:lem:apriori:2}
(with $\gamma=1/\Gamma_{2}(L)$ therein), we can prove, in the same way, that, for possibly new values of $c$ and $C$, 
\begin{equation}
\label{eq:Picard:bound:1st:order}
\sup_{n \geq 1}
\Bigl[ {\mathcal M}_{{\mathbb E}_{t}}^4 \bigl( \vartheta^{n,\lambda}
\bigr) \Bigr]^{1/2}
\leq C \bigl[ \vert \chi^{\lambda} \vert^2 + \|\chi^\lambda\|_{2}^2 \bigr]. 
\end{equation}
Exploiting Remark \ref{rem:comparaison},
we deduce that 
$[\bar{\mathcal M}^4\llbracket \vartheta^{n,\lambda}\rrbracket]^{1/2}$
and $[\bar{\mathcal M}^4\llbracket \vartheta^{n+1,\lambda}\rrbracket ]^{1/2}$
 in \eqref{eq:mkv:stability:Psi}
are less than $C (\vert \chi^{\lambda}\vert^2 +\| \chi^{\lambda} \|_{2}^2)$.

We now make use of  \eqref{eq:lem:stability:3esp} in Lemma \ref{lem:stability}, with $p=1$, in order to compare 
$\vartheta^{n,\lambda}$ and $\vartheta^{n+1,\lambda}$.
Clearly, the remainder $\Delta {\mathcal R}_{a}^2$ in \eqref{eq:mkv:stability:R} is zero
since the $\cR_a$ terms are here equal to zero, recall 
\eqref{eq:identification:derivatives}.
By the above argument, 
$[\bar{\mathcal M}^4\llbracket \vartheta^{n,\lambda}\rrbracket]^{1/2}$
and $[\bar{\mathcal M}^4\llbracket \vartheta^{n+1,\lambda} \rrbracket]^{1/2}$
 in \eqref{eq:mkv:stability:Psi}
are less than $C (\vert \chi^{\lambda}\vert^2 +\| \chi^{\lambda} \|_{2}^2)$. 
In order to apply Lemma \ref{lem:stability}, we also have to estimate 
$[\bar{\mathcal M}^4\llbracket \theta^{n+1,\lambda},\theta^{n,\lambda}\rrbracket]^{1/2}$.
Since $T$ is small enough, the Picard scheme for solving \eqref{eq X-Y t,xi}
is geometrically convergent
in $L^2$ and in any $L^p$, $p \geq 2$, conditional on ${\mathcal F}_{t}$, the geometric rate being 
independent of the initial conditions.
To be precise, 
there exist $\rho \in (0,1)$ and $C' \geq 0$ such that, almost surely, 
$[\bar{\mathcal M}^4\llbracket \theta^{n+1,\lambda} - \theta^{n,\lambda} \rrbracket
]^{1/2} \leq C' ( 1 + \vert \xi^\lambda \vert^2 + \| \xi^\lambda \|_{2}^2) \rho^n$. 
By continuity of the map $\R \ni \lambda \mapsto \xi^{\lambda}
\in L^2(\Omega,{\mathcal F}_{t},\P;\R^d)$,
this shows that  $\E( [\bar{\mathcal M}^4\llbracket \theta^{n+1,\lambda} - \theta^{n,\lambda}\rrbracket]^{1/2})$
converges to $0$, uniformly
in $\lambda$ in compact subsets.
Now, following the proof of Lemma 
\ref{lem:function:phi:alpha}
and using the fact that the map $\R \ni \lambda \mapsto \xi^{\lambda}
\in L^2(\Omega,{\mathcal F}_{t},\P;\R^d)$ is continuous, the family of random 
variables $(\theta^{n,(0),\lambda}_{s})_{n \geq 0, s \in [t,T],\lambda \in \cK}$
is relatively compact in $L^2(\Omega,\cA,\P;\R^d \times \R^m)$ for any compact subset $\cK \subset \R$.
Therefore, $\Phi_{\alpha}(\theta^{n+1,(0),\lambda},\theta^{n,(0),\lambda})$
converges to $0$, uniformly in $\lambda$ in compact subsets. 
We deduce that  $\E( [\bar{\mathcal M}^4\llbracket \theta^{n+1,\lambda},\theta^{n,\lambda}\rrbracket]^{1/2})$
converges to $0$, uniformly
in $\lambda$ in compact subsets.

By \eqref{eq:lem:stability:3esp}
with $\gamma^{1/2} \Gamma'(L)=1/4$, 
we deduce that, for $T \leq c$ (allowing the value of $c$ to decrease from line to line), 
\begin{equation}
\label{eq:derivatives:1st:order:cv}
\begin{split}
&{\mathcal M}_{\mathbb E}^2 \bigl( \vartheta^{n+1,\lambda} - \vartheta^{n,\lambda} \bigr)
\leq \tfrac12  {\mathcal M}_{\mathbb E}^2 \bigl( \vartheta^{n,\lambda} - \vartheta^{n-1,\lambda} \bigr)
 + C 
\esp{(\vert \chi^{\lambda}\vert^2 +\| \chi^{\lambda} \|_{2}^2)
 \psi_{n}(\lambda)
 \bigr]
\bigr)
},
\end{split}
\end{equation}
where $(\psi_{n}(\lambda))_{n \geq 0}$ is a sequence of random variables
that are bounded by $1$ and 
that converges in probability to $0$  as $n$ tends to $\infty$, uniformly in 
$\lambda$ in compact subsets. 
By a standard uniform integrability argument, we deduce from
the bound $\psi_{n}(\lambda) \leq 1$
and from the 
continuity property of the mapping 
$\R \ni \lambda \mapsto \chi^{\lambda} \in L^2$
that 
${\mathbb E}[ \bigl(\vert \chi^{\lambda}\vert^2 +\| \chi^{\lambda} \|_{2}^2
\bigr) \psi_{n}(\lambda)]$
tends to $0$ as $n$ tends to $\infty$, uniformly
in $\lambda$ in compact subsets.
Therefore, the left-hand side 
in \eqref{eq:derivatives:1st:order:cv} converges to $0$, 
the convergence being geometric, uniformly in $\lambda$ in compact subsets. 
By a Cauchy argument, 
the proof is completed. 
\eproof
\vspace{5pt}

We emphasize 
that the derivative process
$[\ud /\ud \lambda]_{\vert \lambda =0}
\theta^{\lambda}$
given by Lemma 
\ref{lem:1storder:diff}
satisfies \eqref{eq:sys:linear:mkv} with 
$\eta:=\chi$, with $\theta \equiv \hat{\theta} :\equiv \theta^0$ and 
$\vartheta \equiv \hat{\vartheta} :\equiv [\ud /\ud \lambda]_{\vert \lambda=0}\theta^\lambda$
and with the coefficients given in 
\eqref{eq:identification:derivatives}.
In particular, for $T$ small enough, the uniqueness of the solution to 
\eqref{eq:sys:linear:mkv} (see Remark \ref{rem:uniqueness}) ensures that 
the derivative process at $\lambda =0$
depends only on the family $(\xi^\lambda)_{\lambda \in \R}$
through $\xi^0$ and
$[\ud /\ud \lambda]_{\vert \lambda =0} \xi^\lambda$. Thus,
when $\xi^0:=\xi$
and $[\ud /\ud \lambda]_{\vert \lambda =0} \xi^\lambda := \chi$, we may denote
by $\partial_{\chi}\theta^{t,\xi}=
(\partial_{\chi}X^{t,\xi},\partial_{\chi}Y^{t,\xi},\partial_{\chi}Z^{t,\xi})$ the tangent process at $\xi$
in the direction $\chi$. By linearity of 
\eqref{eq:sys:linear:mkv},  
$\partial_{\chi} \theta^{t,\xi}$ is linear in $\chi$. 
By a direct application of Corollary \ref{co:bound:sys:lin:mkv:easy} 
-- recall $H_a \equiv 0$ in the current case --, we have 
\begin{Lemma}
\label{lem:apriori:derivatives}
For any $p \geq 1$, 
there exist two constants
$c_{p}:=c_{p}({L})>0$ and $C_{p}$, such that, for $T \leq c_{p}$
and with $\gamma=c_{p}$ in \eqref{stab:notations:1}, 
\begin{equation*}
\begin{split}
\bigl[{\mathcal M}^{2p}_{{\mathbb E}_{t}}
\bigl( \partial_{\chi} \theta^{t,\xi} \bigr) \bigr]^{1/(2p)} 
 \leq C_{p} 
\bigl( \vert \chi \vert +
\| \chi\|_{2} \bigr).
\end{split}
\end{equation*}
\end{Lemma}
Choosing $p=1$ and taking the expectation, we get that the 
mapping $L^2(\Omega,\cF_{t},\P;\R^d) \ni \chi \mapsto 
\partial_{\chi} \theta^{t,\xi} \in \cS^2([t,T];\R^d)  \times
 \cS^2([t,T];\R^m) \times \cH^2([t,T];\R^{m \times d})$ is continuous, which 
 proves that 
$L^2(\Omega,{\mathcal F}_{t},\P;\R^d) \ni 
\xi \mapsto \theta^{t,\xi}
\in {\mathcal S}^2([t,T];\R^d) \times 
{\mathcal S}^2([t,T];\R^m) \times {\mathcal H}^2([t,T];\R^{m \times d})$
is G\^ateaux differentiable. 
 The next lemma shows that the G\^ateaux derivative is continuous:
\begin{Lemma}
\label{lem:cont:derivatives}
For any $p \geq 1$, 
there exist two constants $c_{p}:=c_{p}(L)>0$ and $C_{p}$, 
such 
that, for $T \leq c_{p}$ and with $\gamma=c_{p}$ in \eqref{stab:notations:1},  
\begin{equation*}
\begin{split}
&\Bigl[ {\mathcal M}^{2p}_{{\mathbb E}_{t}} 
\bigl(  \partial_{\chi} \theta^{t,\xi} 
- \partial_{\chi} \theta^{t,\xi'} \bigr)
\Bigr]^{1/2p} 
 \leq C_{p} \Bigl( 
1 \wedge \bigl\{ \vert \xi - \xi' \vert + 
 \Phi_{\alpha+1}(t,\xi,\xi') \bigr\}
\Bigr)
\bigl( \vert \chi \vert +
\| \chi\|_{2} \bigr),
\end{split}
\end{equation*}
where $\Phi_{\alpha+1}(t,\cdot) : [ L^2(\Omega,\cF_{t},\P;
\R^{d})]^2 \rightarrow \R_{+}$ 
is continuous at any point of the diagonal, does not depend on $p$ and
satisfies \eqref{assumption:mkv:4} with $\alpha$ replaced by 
$\alpha+1$. 
The restriction of $\Phi_{\alpha+1}(t,\cdot,\cdot)$ to $[L^2(\Omega,\cF_{0},\P;\R^d)]^2$ may be assumed to be independent of $t \in [0,T]$. 
\end{Lemma}

\proof The proof is a consequence of Corollary \ref{cor:stab:hard},
with ${\mathcal R}_{a}^{2p} \equiv \Delta {\mathcal R}_{a}^{2p} \equiv 0$. 
Example \ref{example:UI} (see in particular \eqref{eq:sncf:1}) guarantees 
that the conditions of 
Corollary \ref{cor:stab:hard} are satisfied. 
We then deduce that \eqref{eq:stab:main estimate} holds true.
Existence of a function $\Phi_{\alpha+1}(t,\cdot,\cdot)$ satisfying the prescription described in the statement then follows from 
Example \ref{example:UI}.
By Remark \ref{rem:unif:phi:alpha}, we can assume that the restriction to 
$[L^2(\Omega,\cF_{0},\P;\R^d)]^2$ is independent of $t$. 
\eproof

\begin{Remark}
\label{rem:frechet}
It is easy to derive from Lemma \ref{lem:cont:derivatives} that
\begin{equation*}
\begin{split}
&\NS{1}{\partial_{\chi} X^{t,\xi}
- \partial_{\chi} X_{s}^{t,\xi'}} + 
 \NS{1}{\partial_{\chi} Y^{t,\xi}
- \partial_{\chi} Y_{s}^{t,\xi'}}
+ \NH{1}{
\partial_{\chi} Z_{s}^{t,\xi}
- \partial_{\chi} Z_{s}^{t,\xi'}} 
\\
&\hspace{15pt}\leq C \Bigl( {\mathbb E} \bigl[ \bigl( 
1 \wedge \vert \xi - \xi' \vert^2 \bigr) \bigr]^{1/2} + 
 \Phi_{\alpha+1}(\xi,\xi') 
\Bigr)
\| \chi\|_{2}.
\end{split}
\end{equation*}
By \cite[Proposition A.3]{andrews:hopper}, 
the map $L^2(\Omega,\cF_{t},\P;\R^d) \ni 
\xi \mapsto (X^{t,\xi},Y^{t,\xi},Z^{t,\xi}) \in
{\mathcal S}^1([t,T];\R^d) \times 
{\mathcal S}^1([t,T];\R^m) \times {\mathcal H}^1([t,T];\R^{m \times d})$ 
 is continuously Fr\'echet differentiable. 
\end{Remark}

 \subsubsection{First-order derivatives
 of the non McKean-Vlasov system with respect to the measure argument}
We reproduce the same analysis as above, but with the process 
$\theta^{t,x,[\xi]}$ instead of $\theta^{t,\xi}$ by taking advantage of the fact that the dependence of the coefficients of the system  \eqref{eq X-Y t,x,mu} upon 
the law is already known to be smooth. This permits to discuss the 
differentiability of $\theta^{t,x,[\xi]}$ in a straightforward manner. 

We mimic the strategy of the previous subsection. Considering a 
continuously differentiable 
mapping $\lambda \mapsto \xi^{\lambda} \in L^2(\Omega,{\mathcal F}_{t},\P)$, we are to prove that 
$\lambda \mapsto \theta^{t,x,[\xi^{\lambda}]}$ is continuously differentiable. 
{
The specific feature 
is that, for any $\lambda$, the coefficients 
of the FBSDE \eqref{eq X-Y t,x,mu}
satisfied by 
$\theta^{t,x,[\xi^{\lambda}]}$ 
depend in a smooth way upon the solution $\theta^{t,\xi^\lambda}$
of the FBSDE \eqref{eq X-Y t,xi}. Since we have
already established the 
continuous differentiability of the mapping 
$\lambda \mapsto \theta^{t,\xi\lambda}$, 
it suffices now to prove that the solution of a standard FBSDE depending on 
a parameter $\lambda$ in a continuously differentiable way is also 
continuously differentiable with respect to $\lambda$.} 
We shall not perform the proof, as it consists of a simple adaptation of the 
proof used to prove the differentiability of the flow of a standard FBSDE, see \cite{del02}.
When 
$\xi^0 = \xi$ and
$[\ud/\ud \lambda]_{\lambda =0} \xi^\lambda = \chi$,
we shall denote the directional derivative at $\xi$ along $\chi$ by 
\begin{equation*}
\bigl( \partial_{\chi} X_{s}^{t,x,[\xi]},\partial_{\chi} Y_{s}^{t,x,[\xi]},
\partial_{\chi} Z_{s}^{t,x,[\xi]}\bigr)_{s \in [t,T]},
\end{equation*}
seen as an element of 
the space 
${\mathcal S}^2([t,T];\R^d) 
 \times {\mathcal S}^2([t,T];\R^m) \times
 {\mathcal H}^2([t,T];\R^{m \times d})$.
  By the same argument as above, 
 it only depends on the family $(\xi^\lambda)_{\lambda \in \R}$
 through the values of 
 $\xi$ and $\chi$,
$( \partial_{\chi} X^{t,x,[\xi]},\partial_{\chi} Y^{t,x,[\xi]},
\partial_{\chi} Z^{t,x,[\xi]})$
 satisfying a `differentiated' system, 
of the type \eqref{eq:sys:linear:mkv}, 
for which uniqueness holds. 
In \eqref{eq:sys:linear:mkv},
$\eta = 0$ (since $[\ud/\ud \lambda] X^{t,x,[\xi]} =0$), $\theta \equiv \theta^{t,x,[\xi]}$, 
$\hat{\theta} \equiv \theta^{t,[\xi]}$, 
$\vartheta \equiv \partial_{\chi} \theta^{t,x,[\xi]}$
and  $\hat \vartheta \equiv \partial_{\chi} \theta^{t,\xi}$,
the tangent process $\partial_{\chi} \theta^{t,\xi}$ being 
given by Lemma \ref{lem:1storder:diff}. The coefficients 
are of the general shape \eqref{splitting:H}
and \eqref{eq:mkv:stability:decomposition}. 
When $h$ stands for one of the functions 
$b$, $f$, $\sigma$ or $g$
and $V$ for $\theta^{t,x,[\xi]}$, 
$\theta^{t,x,[\xi],(0)}$ or 
$X^{t,x,[\xi]}$
and $\hat{V}$ for
$\theta^{t,\xi}$, 
$\theta^{t,\xi,(0)}$ or 
$X^{t,\xi}$, according to the cases, it holds, as in 
\eqref{eq:identification:derivatives},
\begin{equation}
\label{eq:x,xi:H1}
\begin{split}
h_{\ell}(V,\langle \hat V^{(0)} \rangle) = \partial_{x} 
h(V, \law{\hat V^{(0)}}), 
\ \hat{H}_{\ell}(V,\langle \hat V^{(0)}\rangle)=
\partial_{\mu} h(V,\law{\hat V^{(0)}})(\langle 
\hat V^{(0)} \rangle),
\ H_{a} \equiv  0.
\end{split}
\end{equation}

\begin{Lemma}
\label{lem:cont:derivatives:x,xi}
For any $p \geq 1$, there exist two constants $c_{p}:=c_{p}(L)>0$
and $C_{p}$, 
such 
that,
for $T \leq c_{p}$ and with $\gamma=c_{p}$ in \eqref{stab:notations:1},
\begin{equation}
\label{eq:bound:first:order}
\begin{split}
\bigl[ {\mathcal M}_{{\mathbb E}}^{2p}\bigl( 
\partial_{\chi} \theta^{t,x,[\xi]} \bigr) \bigr]^{1/2p} 
\leq C_{p} 
\| \chi \|_{2},
\end{split}
\end{equation}
and
\begin{align}
\label{eq:bound:first:order:bis}
&\bigl[ 
{\mathcal M}_{\mathbb E}^{2p}\bigl( 
\partial_{\chi} \theta^{t,x,\law{\xi}}
- \partial_{\chi} \theta^{t,x,\law{\xi'}} \bigr)
 \bigr]^{1/2p} \leq 
C_p \Phi_{\alpha+1}(t,\xi,\xi')
\| \chi\|_{2}, 
\end{align}
where $\Phi_{\alpha+1}(t,\cdot,\cdot) : [ L^2(\Omega,\cF_{t},\P;
\R^{d})]^2 \rightarrow \R_{+}$ 
is continuous at any point of the diagonal, does not 
depend on $p$ and
satisfies \eqref{assumption:mkv:4}, with $\alpha$ replaced by 
$\alpha+1$. The restriction of $\Phi_{\alpha+1}(t,\cdot,\cdot)$ to $[L^2(\Omega,\cF_{0},\P;\R^d)]^2$ may be assumed to be independent of $t \in [0,T]$. 
\end{Lemma}

\begin{Remark}
\label{rem:CI}
Note that there is no 
conditional expectation 
on 
${\mathcal F}_{t}$ in the above bounds as the initial condition of $\partial_{\chi}
X^{t,x,[\xi]}$ is zero, which means that the filtration that is used for solving the linear 
equation can be assumed to be almost-surely trivial at time $t$. For that reason, the 
right hand side reduces to $\| \chi \|_{2}$. We stress the fact that it is not 
$\| \chi \|_{2p}$ but $\| \chi \|_{2}$, as the dependence upon $\chi$ comes 
through the McKean-Vlasov interaction terms, which is estimated in $L^2$ 
norm. 
\end{Remark}

\proof 
Equation 
\eqref{eq:bound:first:order}
is a direct consequence of 
\eqref{eq:co:bound:sys:lin:mkv:easy}
in
Corollary  \ref{co:bound:sys:lin:mkv:easy}, 
with $\eta =0$, $\cR_{a} \equiv 0$ and $\hat{\vartheta} = \partial_{\chi} \theta^{t,\xi}$,
combined with Lemma 
\ref{lem:apriori:derivatives} (to control the term $\hat{\vartheta}^{(0)} \equiv 
\partial_{\chi} \theta^{t,\xi,(0)}$). 

We now turn to \eqref{eq:bound:first:order:bis}.
It follows from \eqref{eq:stab:1ere eq}
 in Corollary 
\ref{cor:stab:hard}, with
$\eta=0$,
 $\cR_{a} \equiv \Delta \cR_{a} \equiv 0$,
 $\theta^\xi \equiv \theta^{t,x,[\xi]}$,
 $\vartheta^\xi \equiv \theta^{t,x,[\xi]}$, 
$\hat{\theta}^{\xi} \equiv \theta^{t,\xi}$
and  $\hat{\vartheta}^{\xi} \equiv \vartheta^{t,\xi}$
(and the same with $\xi'$ instead of $\xi$). By Lemma
\ref{lem:cont:derivatives}, we can indeed bound
the last term in \eqref{eq:stab:1ere eq} by 
\begin{equation*}
\sup_{\varsigma \in \cO}
  \sup_{\|\Lambda_{0}\|_{2} \leq K} \biggl\{ 
 {\mathbb E} \Bigl[ \big( \Lambda_{0} 
 \wedge 
\bar{\Psi}(\varsigma, \xi  ) 
\bigr) 
\Bigl( 
1 \wedge \bigl\{ \vert \xi - \xi' \vert + 
 \Phi_{\alpha+1}(t,\xi,\xi') \bigr\}
\Bigr)
\bigl( \vert \chi \vert +
\| \chi\|_{2} \Bigr) 
\Bigr] \biggr\}.  
\end{equation*}
with\new{ $\bar \Psi$ defined in 
\eqref{eq:bar psi}
and $\Psi$ given in this definition by \eqref{eq:sncf:1}}. 
Following Example \ref{example:UI} (see in particular 
\eqref{eq:bar phi}), we deduce that 
\eqref{eq:bound:first:order:bis}
holds true (modifying $\Phi_{\alpha+1}(t,\cdot,\cdot)$ if necessary).
 By Remark \ref{rem:unif:phi:alpha}, we can assume that the restriction
 of $\Phi_{\alpha+1}(t,\cdot,\cdot)$
  to 
$[L^2(\Omega,\cF_{0},\P;\R^d)]^2$ is independent of $t$. 
\eproof
\vspace{5pt}

Following Remark \ref{rem:frechet}
to pass from G\^ateaux to Fr\'echet, we deduce:
\begin{Lemma}
\label{lem:cont:derivatives:x,xi:b}
For $T \leq c$ with $c:=c(L) >0$, 
$t \in [0,T]$ and $x \in \R^d$, the function $L^2(\Omega,{\mathcal F}_{t},\P;\R^d) \ni \xi 
\mapsto {\mathcal U}(t,x,\xi) = Y_{t}^{t,x,[\xi]}$ is Fr\'echet continuously differentiable. In particular, the function ${\mathcal P}^2(\R^d) \ni \mu \mapsto 
U(t,x,\mu)$ is differentiable in Lions' sense. Moreover,
for all $x \in \R^d$, for all $\xi,\xi' \in
L^{2}(\Omega,\cA,\P;\R^d)$, we have,
with $\mu = [\xi]$ and $\mu' = [\xi']$, 
\begin{equation}
\label{eq:cont:derivatives:x,xi:b}
\| \partial_{\mu} U(t,x,\mu)(\xi)  \|_{2} \leq C,
\quad 
\| \partial_{\mu} U(t,x,\mu)(\xi) -
\partial_{\mu} U(t,x,\mu')(\xi') \|_{2} \leq 
C \Phi_{\alpha+1}(\xi,\xi'),
\end{equation}
where
 $\Phi_{\alpha+1} : [ L^2(\Omega,\cA,\P;
\R^{d})]^2 \rightarrow \R_{+}$ is continuous at any point of the diagonal and
satisfies \eqref{assumption:mkv:4}, with $\alpha$ replaced by 
$\alpha+1$.
\end{Lemma}

\proof
Fr\'echet differentiability is a consequence of 
the continuity 
Lemma \ref{lem:cont:derivatives:x,xi}
that permits to pass from G\^ateaux to Fr\'echet 
on the model of 
Remark \ref{rem:frechet}.
We then have $\partial_{\chi} Y_{t}^{t,x,[\xi]}
= \E[\partial_{\mu} U(t,x,[\xi])(\xi) \chi]$. Combined with Lemma 
\ref{lem:cont:derivatives:x,xi}, this gives \eqref{eq:cont:derivatives:x,xi:b}, 
but with 
$\Phi_{\alpha+1}$
defined on $[L^2(\Omega,\cF_{0},\P;\R^d)]^2$,
which requires that $\xi$ and $\xi'$ belong to 
$L^2(\Omega,\cF_{0},\P;\R^d)$.
The main issue is to prove that $\Phi_{\alpha+1}$
may be defined on the whole $[L^2(\Omega,\cA,\P;\R^d)]^2$. 
It is then worth mentioning 
that
$\| \partial_{\mu} U(t,x,\mu)(\xi) -
\partial_{\mu} U(t,x,\mu')(\xi') \|_{2}$ only depends on the law of $(\xi,\xi')$. 
Given $(\xi,\xi') \in 
[L^2(\Omega,\cA,\P;\R^d)]^2$, we 
can always find a pair $(\tilde{\xi},\tilde{\xi'}) \in 
[L^2(\Omega,\cF_{0},\P;\R^d)]^2$ with the same distribution
(provided that $(\Omega,\cF_{0},\P)$ is rich enough). 
This says that, with the $\Phi_{\alpha+1}$ given by 
Lemma \ref{lem:cont:derivatives:x,xi}, for all $\xi,\xi' \in L^2(\Omega,\cA,\P;\R^d)$,
\begin{equation*}
\begin{split}
&\bigl\| \partial_{\mu} U(t,x,\mu)(\xi) -
\partial_{\mu} U(t,x,\mu')(\xi') \bigr\|_{2} \leq \tilde{\Phi}_{\alpha+1}(\xi,\xi'),
\\
&\hspace{15pt} \text{with} \ \tilde{\Phi}_{\alpha+1}(\xi,\xi') := \inf 
\bigl\{
\Phi_{\alpha+1}(\tilde{\xi},\tilde{\xi}'), \quad (\tilde{\xi},\tilde{\xi}')
\in L^2(\Omega,\cF_{0},\P;\R^d), \quad 
(\tilde{\xi},\tilde{\xi'}) \sim (\xi,\xi') \bigr\}.
\end{split}
\end{equation*}
Clearly, $\tilde{\Phi}_{\alpha+1}$ 
is defined on the whole $[L^2(\Omega,\cA,\P;\R^d)]^2$.
It 
satisfies \eqref{assumption:mkv:4}.
Continuity at any point of the diagonal may be proved as follows. Given 
a sequence $(\xi_{n},\xi_{n}')_{n \geq 1}$ that converges to some $(\xi,\xi)$
in $[L^2(\Omega,\cA,\P;\R^d)]^2$, we may find
a pair $(\tilde \xi,\tilde \xi)
\in [L^2(\Omega,\cF_{0},\P;\R^d)]^2$ with the same law as $(\xi,\xi)$.
Now, for any $n \geq 1$, we can construct $(\tilde{\xi}_n,\tilde{\xi}_{n}')$
in $[L^2(\Omega,\cF_{0},\P;\R^d)]^2$
such that the 4-tuple $(\tilde{\xi}_{n},\tilde{\xi}_{n}',\tilde{\xi},\tilde{\xi})$
has the same law as 
$(\xi_{n},\xi_{n}',\xi,\xi)$ (it suffices to use the conditional law of 
$(\xi_{n},\xi_{n}')$ given $(\xi,\xi)$).
Then,
$(\tilde \xi_{n},\tilde \xi_{n}')_{n \geq 1}$ converges  
to $(\tilde \xi,\tilde \xi)$ in $L^2$. 
From the inequality $\tilde{\Phi}_{\alpha+1}(\xi_{n},\xi_{n}') 
\leq  {\Phi}_{\alpha+1}(\tilde \xi_{n},\tilde \xi_{n}')$,  
$\tilde{\Phi}_{\alpha+1}(\xi_{n},\xi_{n}')$ tends to $0$. 
\eproof
\vspace{5pt}

We now discuss the Lipschitz property in $x$ of $\partial_{\mu} U(t,x,\mu)$:
\begin{Lemma}
For $T \leq c$, with $c:=c(L)>0$, we can find a constant 
$C$ such that, 
for $\xi$ with $\mu$ as distribution, 
\begin{equation*}
\forall x,x' \in \R^d, \quad \| \partial_{\mu} U(t,x,\mu)(\xi) - 
\partial_{\mu} U(t,x',\mu)(\xi) \|_{2} \leq C | x - x' |.
\end{equation*}
\end{Lemma}
\begin{proof}
Thanks to the relationship 
$\partial_{\chi} Y_{t}^{t,x,[\xi]} = 
{\mathbb E}[ \partial_{\mu} U(t,x,[\xi])(\xi) \chi ]$,
it suffices to discuss the Lipschitz property (in $x$)  
of the tangent process $(\partial_{\chi} X_{s}^{t,x,\xi},\partial_{\chi} Y_{s}^{t,x,\xi},\partial_{\chi} Z_{s}^{t,x,\xi})_{s \in [t,T]}$, seen as an element of 
${\mathcal S}^2([t,T];\R^d) \times {\mathcal S}^2([t,T];\R^m) \times
{\mathcal H}^2([t,T];\R^{m \times d})$, $\xi$ and $\chi$ denoting elements of 
$L^2(\Omega,{\mathcal F}_{t},\P;\R^d)$. 

Basically, the strategy is the same as in the proofs of Lemmas 
\ref{lem:cont:derivatives} and \ref{lem:cont:derivatives:x,xi}.
It is based on a tailored-made version of Corollary 
\ref{cor:stab:hard},
obtained by applying 
Lemma \ref{le app cont} and
Lemma 
\ref{lem:stability}
with $\theta \equiv \bar\theta :\equiv \theta^{t,x,[\xi]}$,
$\theta' \equiv \bar\theta' :\equiv \theta^{t,x',[\xi]}$,
$\hat \theta \equiv \hat{\theta}' \equiv \check \theta \equiv \check
\theta' : \equiv \theta^{t,\xi}$
and $\hat \vartheta  \equiv \hat{\vartheta}' \equiv \check \vartheta \equiv \check
\vartheta' : \equiv \partial_{\chi} \theta^{t,\xi}$.
Informally, it consists in choosing $\eta = 0$ and in replacing $\vert \xi - \xi'\vert$ by $\vert x-x' \vert$ 
and $\Phi_{\alpha}(\xi,\xi')$ by $0$
in the statement of 
Corollary \ref{cor:stab:hard}.
%
%
 We end up with
$\vert \partial_{\chi} Y_{t}^{t,x,[\xi]}
- \partial_{\chi} Y_{t}^{t,x',[\xi]}
\vert \leq C \vert x'-x \vert \| \chi \|_{2}.  
$\qed

\end{proof}

\subsubsection{Derivatives with respect to the space argument}
 We now discuss the derivatives of $U$ with respect to the variable $x$. 
 Since the 
process  
$\theta^{t,x,[\xi]} = (X^{t,x,\law{\xi}},Y^{t,x,\law{\xi}},Z^{t,x,\law{\xi}})$ may be seen
as the solution of a standard FBSDE parametrized by the law of $\xi$, 
we can apply the results in \cite{del02} on the smoothness of the flow of a classical FBSDE in short time. 
Given $t \in [0,T]$ and $\xi \in L^2(\Omega,\cF_{t},\P;\R^d)$,
we deduce that the function $\R^d \ni x \mapsto \theta^{t,x,[\xi]}=
(X^{t,x,\law{\xi}},Y^{t,x,\law{\xi}},Z^{t,x,\law{\xi}}) \in 
{\mathcal S}^2([t,T];\R^d) \times 
{\mathcal S}^2([t,T];\R^m) \times {\mathcal H}^2([t,T];\R^{m \times d})$
is continuously differentiable, the derivative process 
at point $x \in \R^d$
being denoted by $\partial_{x}\theta^{t,x,[\xi]}=
(\partial_{x} X^{t,x,\law{\xi}},\partial_{x}
Y^{t,x,\law{\xi}},\partial_{x} Z^{t,x,\law{\xi}})$. 
 To be self-contained, notice that the same result could be obtained by 
 applying the results of Subsection \ref{subse:stab:1}, with the following version of $H(r,\cdot)$:
%
\begin{equation}
\label{eq:dx:x,t,mu:Hcoef}
\begin{split}
H(r,V_{r}^{t,x,\law{\xi}})(\cV^{t,x,\law{\xi}}_r) = \partial_{x} 
h(V_{r}^{t,x,[\xi]},[ V_{r}^{t,\xi}])\cV^{t,x,\law{\xi}}_r.
\end{split}
\end{equation}
%
As a consequence, we easily get, for $T \leq c_{p}$, $c_p := 
c_p(L)$ and with $\gamma=c_{p}$ in \eqref{stab:notations:1},
$[ {\mathcal M}_{{\mathbb E}}^{2p}( \partial_{x} \theta^{t,x,\law{\xi}}) 
]^{1/2p}
\le
C_p$.
{
Recalling the identity $U(t,x,[\xi])=\theta_{t}^{t,x,[\xi]}$, 
we recover the fact 
that $\R^d \ni x \mapsto U(t,x,[\xi])$ is continuously 
differentiable and
that $\| \partial_x U \|_\infty \le C$, see also 
\cite{del02}.  
On the same model (for instance by 
adapting 
Lemmas 
\ref{lem:cont:derivatives}
or 
\ref{lem:cont:derivatives:x,xi}
to investigate
the difference $\partial_{x} \theta^{t,x,[\xi]} - \partial_{x} \theta^{t,x',[\xi]}$
for two different $x,x' \in \R^d$
or by
taking benefit from the results proved in
\cite{del02}), it can be checked that, for any $t \in [0,T]$, any 
$\xi \in L^2(\Omega,{\mathcal F}_{t},\P;\R^d)$, the mapping 
$\R^d \ni x \mapsto \partial_{x} U(t,x,\law{\xi})$ is $C$-Lipschitz continuous.} 
Intuitively, such a bound is much simpler to get than the bound for the continuity 
of $\partial_{\mu} U$ because 
of the very simple structure of 
$H(r,\cdot)$ in
\eqref{eq:dx:x,t,mu:Hcoef}, the function $\partial_{x} h$
being Lipschitz-continuous with respect to the first argument. 

To get
the smoothness of $\partial_{x} U$ in the direction $\mu$,
we may investigate the difference 
$\partial_{x} \theta^{t,x,[\xi]} - \partial_{x} \theta^{t,x,[\xi']}$
for two different $\xi,\xi' \in L^2(\Omega,\cF_{t},\P;\R^d)$. 
Reapplying
 Corollary 
\ref{cor:stab:hard}, exactly in the same way as in the proof of 
Lemma \ref{lem:cont:derivatives:x,xi},
we deduce
\begin{equation} 
\label{eq:duhatmu:lip}
\begin{split}
&\forall x \in \R^d, \ \xi,\xi' \in L^2(\Omega,\cA,\P;\R^d),\ 
\bigl|\partial_x U(t,x,\law{\xi}) 
- \partial_{x} U(t,x,\law{\xi'})\bigr| \le  \Phi_{\alpha+1}(\xi,\xi').
\end{split}
\end{equation}
{Actually,
the above bound 
could be improved. Indeed, it
also holds with
$\Phi_{\alpha+1}(\xi,\xi')$ replaced by 
$\Phi_{\alpha}(\xi,\xi')$. 
The reason is that, in the analysis of 
$\partial_{x} \theta^{t,x,[\xi]} - \partial_{x} \theta^{t,x,[\xi']}$, 
there are no derivatives in the direction of the measure, 
whereas these are precisely these terms that make 
$\Phi_{\alpha+1}(\xi,\xi')$ appear in the proof of
Lemma 
\ref{lem:cont:derivatives:x,xi:b}
(or equivalently of Lemma 
\ref{lem:cont:derivatives:x,xi}).
In order
to keep some homogeneity between the various estimates 
we have on the derivatives of $U$, we feel it is more convenient to keep 
$\Phi_{\alpha+1}(\xi,\xi')$ in the above right-hand side.}

\subsubsection{Final statement}
The following is the complete statement about the first-order differentiability:
\begin{Theorem}
\label{thm:4:1}
For $T \leq c$, with $c:=c(L)>0$ and
$t \in [0,T]$, the function $\R^d \times L^2(\Omega,\cA,\P;\R^d) \ni \xi 
\mapsto U(t,x,[\xi]) = {\mathcal U}(t,x,\xi)$ is continuously differentiable and there exists a constant $C \geq 0$, such that
for all $x,x' \in \R^d$, for all $\xi,\xi' \in L^2
(\Omega,\cA,\P;\R^d)$,
\begin{equation}
\label{eq:thm:4:1:statement}
\begin{split}
&\vert U(t,x,\mu) - U(t,x',\mu') \vert \leq 
C \bigl( \vert x- x' \vert + W_{2}(\mu,\mu') \bigr)
\\
&\vert \partial_{x} U(t,x,\mu) - \partial_{x} U(t,x',\mu') \vert 
+
\| \partial_{\mu} U(t,x,\mu)(\xi) -
\partial_{\mu} U(t,x',\mu')(\xi') \|_{2}
\\
&\hspace{15pt} \leq 
C \bigl( \vert x - x' \vert + \Phi_{\alpha+1}(\xi,\xi') \bigr),
\end{split}
\end{equation}
where
 $\Phi_{\alpha+1} : [ L^2(\Omega,\cA,\P;
\R^{d})]^2 \rightarrow \R_{+}$ is continuous at any point of the diagonal and
satisfies \eqref{assumption:mkv:4}, with $\alpha$ replaced by 
$\alpha+1$.
In particular, for any $x \in \R^d$ and $\mu \in {\mathcal P}_{2}(\R^d)$, 
we can find a locally Lipschitz continuous version of the 
mapping $\R^d \ni v \mapsto \partial_{\mu} U(t,x,\mu)(v)$.

{
Moreover, 
the functions
$[0,T]
\times \R^d \times L^2(\Omega,\cA,\P;\R^d)
 \ni (t,x,\xi) 
 \mapsto 
\partial_{x} U(t,x,[\xi]) \in \R^d$
and $[0,T] 
\times \R^d \times L^2(\Omega,\cA,\P;\R^d)
 \ni (t,x,\xi) \mapsto
\partial_{\mu} U(t,x,[\xi])(\xi) \in L^2(\Omega,\cA,\P;\R^d)$ are continuous.}

Finally, for any $t \in [0,T]$, $\xi \in L^2(\Omega,\cF_{t},\P;\R^d)$ and $C' \geq 0$,
\begin{equation}
\label{eq:UI:U}
\lim_{\P(A) \rightarrow 0, A \in \cA}
\sup_{\new{(t,x)\in [0,T]\times\R^d}} \sup_{\Lambda \in L^2(\Omega,\cA,\P;\R^d) : \| \Lambda \|_{2} \leq C'}
\bigl|\E \bigl[ 
\partial_{\mu} U(t,x,[\xi])(\xi)  \Lambda {\mathbf 1}_{A} \bigr] \bigr\vert
= 0,
\end{equation}
which is the analogue of the uniform integrability property described in \HYP{1} for the original coefficients
$b$, $\sigma$, $f$ and  $g$. 
\end{Theorem}
\proof
The Lipschitz property of $U$ is a direct consequence of
the bounds we have for $\partial_{x} U$ and $\partial_{\mu} U$ (or equivalently of 
Lemma \ref{le app cont}). 
The joint continuous differentiability is a consequence of 
the partial continuous differentiability and of the joint continuity properties
of the derivatives. The extension of $\Phi_{\alpha+1}$
to the whole $[L^2(\Omega,\cA,\P;\R^d)]^2$ is achieved as in the proof 
of Lemma \ref{lem:cont:derivatives:x,xi:b}. 

The local Lipschitz property of $\R^d \ni v \mapsto \partial_{\mu} U(t,x,\mu)(v)$ follows from 
Proposition \ref{prop:lipschitz:lifted}. 

We now discuss the continuity of 
$[0,T] \ni t \mapsto 
\partial_{\mu} U(t,x,[\xi])(\xi) \in L^2(\Omega,\cA,\P;\R^d)$.
Clearly, there is no loss of generality in assuming that 
$\xi \in L^2(\Omega,\cF_{0},\P;\R^d)$. 
Given $\xi,\chi \in L^2(\Omega,{\mathcal F}_{0},\P;\R^d)$
and $0 \leq t \leq s \leq T$, it suffices to 
bound the time increment 
$\E [ \bigl( \partial_{\mu} U(t,x,[\xi])(\xi) - \partial_{\mu} U(s,x,[\xi])(\xi) \bigr)
 \chi ]$ by $C(t,s) \| \chi\|_{2}$, the constant $C(t,s)$ being independent of 
 $\chi$ and converging to $0$ as $s-t$ tends to $0$. 
We have
\begin{equation}
\label{eq:time reg}
\begin{split}
&\E \bigl[ \bigl( \partial_{\mu} U(t,x,[\xi])(\xi) - \partial_{\mu} U(s,x,[\xi])(\xi) \bigr)
 \chi \bigr]
\\
&= \hat{\E} \bigl[ \bigl( \partial_{\mu} U(t,x,[\xi])(\langle \xi \rangle) - \partial_{\mu} U(s,x,[\xi])(\langle \xi\rangle) \bigr)
\langle \chi \rangle \bigr]
\\
&= \E \hat\E \bigl[ \bigl( \partial_{\mu} U\bigl(s,X_{s}^{t,x,[\xi]},\bigl[ X_{s}^{t,\xi}\bigr] \bigr)\bigl(\big\langle 
X_{s}^{t,\xi} \big\rangle \bigr) - \partial_{\mu} U(s,x,[\xi])\bigl( \big\langle\xi \big\rangle \bigr) \bigr)
 \langle \chi \rangle \bigr]
\\
&\hspace{15pt} + \E \hat\E \bigl[ \bigl( 
 \partial_{\mu} U(t,x,[\xi])\bigl(\langle \xi \rangle \bigr)
-
\partial_{\mu} U\bigl(s,X_{s}^{t,x,[\xi]},\bigl[ X_{s}^{t,\xi}\bigr] \bigr)\bigl( \langle X_{s}^{t,\xi} \rangle \bigr) 
\bigr)
\langle \chi \rangle \bigr].
\end{split}
\end{equation}

By the smoothness property of $\partial_{\mu} U(s,\cdot)$, the first term in the right-hand side is bounded by
$C (\E [ \vert X_{s}^{t,x,[\xi]} - x \vert^2]^{1/2} +\Phi_{\alpha+1}(X_{s}^{t,\xi},\xi)) \| \chi \|_{2} $, the constant 
$C$ being allowed to increase from line to line.  
The coefficients  of 
 \eqref{eq X-Y t,xi}
and 
\eqref{eq X-Y t,x,mu}
being at most of linear growth, we deduce from \eqref{eq app bound 1} that 
$\E [ \vert X_{s}^{t,\xi} - \xi \vert^2]^{1/2}$
and 
$\E [ \vert X_{s}^{t,x,[\xi]} - x \vert^2]^{1/2}$
are less than 
$C(1 + \| \xi \|_{2})(s-t)^{1/2}$ and
$C(1+ \vert x \vert + \| \xi \|_{2})(s-t)^{1/2}$ respectively.
Therefore, the first term in the last line of \eqref{eq:time reg} is bounded by
\begin{equation}
\label{eq:time reg 2}
C \Bigl[ \bigl(1+ \vert x \vert + \| \xi \|_{2} \bigr) (s-t)^{1/2} + \sup_{ \xi ' : \| \xi' - \xi \|_{2} \leq C(1 + \| \xi \|_{2})
(s-t)^{1/2}}\Phi_{\alpha+1}(\xi',\xi) \Bigr] \| \chi \|_{2}.
\end{equation}
Clearly, the term in brackets goes to $0$ with $s-t$.

We now handle the second term in the last line of \eqref{eq:time reg}. 
Differentiating (with respect to $\xi$ in the direction $\chi$) 
the relationships \new{$U(t,x,[\xi]) = Y_{t}^{t,x,\law{\xi}}$ and 
$\E[U(t,X_{s}^{t,x,\law{\xi}},[X_{s}^{t,\xi}])] = \E[Y_{s}^{t,x,\law{\xi}}]$}, 
we obtain 
\begin{equation*}
\begin{split}
&\E \hat\E \bigl[ \bigl( 
 \partial_{\mu} U(t,x,[\xi])\bigl(\langle \xi \rangle \bigr)
-
\partial_{\mu} U\bigl(s,X_{s}^{t,x,[\xi]},\bigl[ X_{s}^{t,\xi}\bigr] \bigr)\bigl( \langle X_{s}^{t,\xi} \rangle \bigr) 
\bigr)
\langle \chi \rangle \bigr]
\\
&=\E \bigl[\partial_{\chi} Y_{t}^{t,x,[\xi]} - \partial_{\chi} Y_{s}^{t,x,[\xi]} \bigr]
+ \E \bigl[ \partial_{x} U \bigl(s,X_{s}^{t,x,[\xi]},[X_{s}^{t,\xi}] \bigr) \partial_{\chi} X_{s}^{t,x,[\xi]} \bigr]
\\
&\new{+\esp{\hesp{\partial_{\mu} U\bigl(s,X_{s}^{t,x,[\xi]},\bigl[ X_{s}^{t,\xi}\bigr] \bigr)\bigl( \langle X_{s}^{t,\xi} \rangle \bigr) 
\langle \partial_{\chi} X_{s}^{t,\xi}  - \chi\rangle}} }.
\end{split}
\end{equation*}
The first term
is equal to 
$\E \int_{t}^s F^{(1)}(r,\theta_{r}^{t,x,[\xi]},\langle \theta_{r}^{t,\xi,(0)} \rangle)
(\partial_{\chi}
\theta_{r}^{t,x,[\xi]},
\langle \partial_{\chi}
\theta_{r}^{t,\xi,(0)} \rangle) \ud r$
(with the notations of 
\eqref{eq:ex:mkv:stability}).
By \HYP{1} and Lemmas 
\ref{lem:apriori:derivatives}
and \ref{lem:cont:derivatives:x,xi} (with $p=1$),
it is bounded by $C (s-t)^{1/2} \| \chi \|_{2}$. 
Since $\partial_{x} U$ is bounded, 
the second term is less than $C \E [ \vert \partial_{\chi} X_{s}^{t,x,[\xi]} \vert]
= C \E [ \vert \partial_{\chi} X_{s}^{t,x,[\xi]} - \partial_{\chi} X_{t}^{t,x,[\xi]}\vert]$.
By \HYP{1} and Lemmas 
\ref{lem:apriori:derivatives}
and \ref{lem:cont:derivatives:x,xi} again,
it is less than $C(s-t)^{1/2} \| \chi\|_{2}$. 
\new{For the third term, we first apply Cauchy-Schwarz inequality to
get that it is less than $C\E [ |\partial_\chi X^{t,\xi}_s-\chi|^2]^{1/2}$, 
recall \eqref{eq:cont:derivatives:x,xi:b}. Then, by \HYP{1} and Lemma 
\ref{lem:apriori:derivatives},
it is bounded by $C (s-t)^{1/2} \| \chi \|_{2}$. }
Continuity of $[0,T] \ni t \mapsto 
\partial_{\mu} U(t,x,[\xi])(\xi) \in L^2(\Omega,\cA,\P;\R^d)$ easily follows. 
Continuity of $[0,T] \ni t \mapsto 
\partial_{x} U(t,x,[\xi]) \in \R^d$ may be proved in the same way. 
{
Together with the uniform continuity estimates
\eqref{eq:thm:4:1:statement}, 
we deduce that 
the functions
$[0,T]
\times \R^d \times L^2(\Omega,\cA,\P;\R^d)
 \ni (t,x,\xi) 
 \mapsto 
\partial_{x} U(t,x,[\xi]) \in \R^d$
and $[0,T] 
\times \R^d \times L^2(\Omega,\cA,\P;\R^d)
 \ni (t,x,\xi) 
 \mapsto
\partial_{\mu} U(t,x,[\xi])(\xi) \in L^2(\Omega,\cA,\P;\R^d)$ are continuous. 
}

We now prove \eqref{eq:UI:U}. For $A \in \cA$ and 
$\Lambda \in L^2(\Omega,\cA,\P;\R^d)$ with $\| \Lambda \|_{2} \leq C'$, 
we have
\begin{equation}
\label{eq:UI:proof:2}
\bigl|\E \bigl[ 
\partial_{\mu} U(t,x,[\xi])(\xi) \Lambda {\mathbf 1}_{A}  \bigr] \bigr\vert
= \partial_{\chi} Y_{t}^{t,x,[\xi]}, \quad \text{with} \ \chi = \Lambda {\mathbf 1}_{A}.
\end{equation}
We now apply 
 \eqref{eq:lem:apriori:1}
in Lemma
\ref{lem:apriori} with $\theta \equiv \bar \theta
   :\equiv
\theta^{t,x,[\xi]}$, $\vartheta \equiv \bar \vartheta :\equiv \partial_{\chi} \theta^{t,x,[\xi]}$, 
$\hat{\theta} \equiv \check{\theta} :\equiv \theta^{t,\xi}$, $\hat{\vartheta} \equiv \check{\vartheta} :\equiv \partial_{\chi}
\theta^{t,\xi}$
and $\eta = 0$.
The coefficients driving 
\eqref{eq:sys:linear:mkv}
are given by 
\eqref{eq:x,xi:H1}. 
By \eqref{eq:lem:apriori:1}, 
we get that, for $T \leq \gamma$ with  
$\gamma$ in \eqref{stab:notations:1} given by
$\gamma^{1/2} \Gamma_{2}(L) = \min[(1/8L^2) ,1/2]$, 
\begin{equation} 
\label{eq:UI:proof:3}
[\partial_{\chi} Y_{t}^{t,x,[\xi]}]^2 \leq \frac1{4L^2} \sup_{s \in [t,T]}  {\mathcal N}_{\E_{t}}^{2,C}\Bigl(
{\theta}_{s}^{t,\xi,(0)},
 \bigl({\mathcal M}_{{\mathbb E}_{t}}^2(\partial_{\chi} \theta^{t,\xi,(0)}) \bigr)^{1/2} \Bigr).
\end{equation}
We use  \eqref{eq:bound:west coast}
(with $\Psi$ given by 
\eqref{eq:sncf:1})
to bound the above term. For any $\varepsilon >0$, 
\begin{align}
  & \sup_{x \in \R^d, s \in [t,T]}  {\mathcal N}_{\E_{t}}^{2,C}\Bigl({\theta}_{s}^{t,\xi,(0)},
 \bigl({\mathcal M}_{{\mathbb E}_{t}}^2(\partial_{\chi} \theta^{t,\xi,(0)}) \bigr)^{1/2} \Bigr)
 \nonumber
  \\
  &\leq \sup_{(w,s) \in \R^k \times [t,T]} \sup_{\|\Lambda_{0}\|_{2} \leq L} \Bigl\{ 
 {\mathbb E} \Bigl[ \bigl(  \Lambda_{0} 
 \wedge 
 \Psi((w,s), \xi) 
 \bigr) 
\bigl( {\mathcal M}_{{\mathbb E}_{t}}^2(\partial_{\chi} \theta^{t,\xi,(0)}) \bigr)^{1/2} \Bigr]^2 \Bigr\}
\label{eq:UI:proof:4}
\\
  &\leq L^2 \varepsilon \nonumber
  \\
  &\hspace{15pt}+ \sup_{(w,s) \in \R^k \times [t,T]} \sup_{\|\Lambda_{0}\|_{2} \leq L} \Bigl\{ 
 {\mathbb E} \Bigl[ \bigl(  \Lambda_{0} 
 \wedge 
 \Psi((w,s), \xi) 
 \bigr)^2 {\mathbf 1}_{\{
 {\mathcal M}_{{\mathbb E}_{t}}^2(\partial_{\chi} \theta^{t,\xi,(0)})  \geq \varepsilon\}}
 \Bigr]^2
{\mathcal M}_{{\mathbb E}}^2(\partial_{\chi} \theta^{t,\xi,(0)})  \Bigr\}, \nonumber
  \end{align}
where we denoted $\R^d \times \R^m \times \R^{m \times d}$ by $\R^k$ and we 
used Cauchy-Schwarz' inequality in the last line. 
Recall that in the above suprema, $\Lambda_{0}$ takes values in $\R_{+}$.

By Lemma \ref{lem:apriori:derivatives}, 
$[{\mathcal M}_{{\mathbb E}}^2(\partial_{\chi} \theta^{t,\xi,(0)})]^{1/2} \leq C \| \chi\|_{2} \leq C C'$.
By uniform integrability of the
family $(\Psi^2((w,s),\xi))_{w \in \R^k, s \in [t,T]}$,
it thus suffices to prove that 
\begin{equation}
\label{eq:cv:proba}
\lim_{\P(A) \rightarrow 0} \P \bigl(
{\mathcal M}_{{\mathbb E}_{t}}^2(\partial_{\chi} \theta^{t,\xi,(0)})  \geq \varepsilon \bigr) = 0
\end{equation}
in order to prove \eqref{eq:UI:U} (recall that 
the above probability depends on $A$ through $\chi = \Lambda {\mathbf 1}_{A}$). 
We reapply 
 \eqref{eq:lem:apriori:1}
in Lemma
\ref{lem:apriori}, but with $\theta \equiv \bar{\theta} \equiv \hat{\theta} \equiv \check{\theta} : \equiv \theta^{t,\xi}$, $\vartheta \equiv \bar{\vartheta} \equiv \hat{\vartheta} \equiv \check{\vartheta} :\equiv \partial_{\chi} \theta^{t,\xi}$ 
and $\eta = \Lambda {\mathbf 1}_{A}$.
Following \eqref{eq:UI:proof:3} and \eqref{eq:UI:proof:4}, we get 
\begin{align}
\label{eq:UI:proof:5}
&{\mathcal M}_{{\mathbb E}_{t}}^2(\partial_{\chi} \theta^{t,\xi,(0)}) 
\\
\nonumber
&\hspace{5pt} \leq C \Lambda^2 {\mathbf 1}_{A}
+ \frac1{4L^2} \sup_{(w,s) \in \R^k \times [t,T]} \sup_{\|\Lambda_{0}\|_{2} \leq L} \Bigl\{ 
 {\mathbb E} \Bigl[ \bigl(  \Lambda_{0} 
 \wedge 
 \Psi((w,s), \xi) 
 \bigr) 
\bigl( {\mathcal M}_{{\mathbb E}_{t}}^2(\partial_{\chi} \theta^{t,\xi,(0)}) \bigr)^{1/2} \Bigr]^2 \Bigr\}.
\end{align}
Multiplying by ${\mathbf 1}_{A^{\complement}}$ and 
taking the expectation, we deduce that
\begin{equation*}
\begin{split}
&{\mathbb E} \bigl[
{\mathbf 1}_{A^{\complement}}
 {\mathcal M}_{{\mathbb E}_{t}}^2(\partial_{\chi} \theta^{t,\xi,(0)}) \bigr]
\\
&\leq 
\frac{1}{2L^2}\sup_{(w,s) \in \R^k \times [t,T]} \sup_{\|\Lambda_{0}\|_{2} \leq L} \Bigl\{ 
 {\mathbb E} \Bigl[ \bigl(  \Lambda_{0} 
 \wedge 
 \Psi((w,s), \xi) 
 \bigr) 
\bigl( {\mathcal M}_{{\mathbb E}_{t}}^2(\partial_{\chi} \theta^{t,\xi,(0)}) \bigr)^{1/2}
{\mathbf 1}_{A^\complement}
 \Bigr]^2 \Bigr\}
\\
&\hspace{15pt}
+
\frac{1}{2L^2}\sup_{(w,s) \in \R^k \times [t,T]} \sup_{\|\Lambda_{0}\|_{2} \leq L} \Bigl\{ 
 {\mathbb E} \Bigl[ \bigl(  \Lambda_{0} 
 \wedge 
 \Psi((w,s), \xi) 
 \bigr) 
\bigl( {\mathcal M}_{{\mathbb E}_{t}}^2(\partial_{\chi} \theta^{t,\xi,(0)}) \bigr)^{1/2}
{\mathbf 1}_{A}
 \Bigr]^2 \Bigr\}
 \\ 
&\leq \frac12 {\mathbb E} \bigl[
{\mathbf 1}_{A^{\complement}}
 {\mathcal M}_{{\mathbb E}_{t}}^2(\partial_{\chi} \theta^{t,\xi,(0)}) \bigr]
 + C \sup_{(w,s) \in \R^d \times [t,T]} \sup_{\|\Lambda_{0}\|_{2} \leq L} \Bigl\{ 
 {\mathbb E} \Bigl[ \bigl(  \Lambda_{0} 
 \wedge 
 \Psi((w,s), \xi) 
 \bigr)^2 
{\mathbf 1}_{A}
 \Bigr] \Bigr\},
\end{split}
\end{equation*}
where we used Cauchy-Schwarz' inequality twice to get the last line. 
By uniform integrability of the family $(\Psi^2((w,s),\xi))_{x \in \R^k, s \in [t,T]}$,
the second term in the last line tends to $0$ with $\P(A)$. Therefore, 
${\mathbb E}[
{\mathbf 1}_{A^{\complement}}
 {\mathcal M}_{{\mathbb E}_{t}}^2(\partial_{\chi} \theta^{t,\xi,(0)})]$
 also tends to $0$ with $\P(A)$. 
 
 Going back to \eqref{eq:UI:proof:5}, taking the root and then the expectation
 and splitting the expectation in the right-hand side according to 
 the indicator functions of $A^\complement$ and $A$, we get 
 in the same way
 \begin{equation*}
\begin{split}
\E \bigl[ \bigl( {\mathcal M}_{{\mathbb E}_{t}}^2(\partial_{\chi} \theta^{t,\xi,(0)}) \bigr)^{1/2 }\bigr] 
&\leq C \bigl( \P(A) \bigr)^{1/2}
+
C{\mathbb E} \bigl[
{\mathbf 1}_{A^{\complement}}
 {\mathcal M}_{{\mathbb E}_{t}}^2(\partial_{\chi} \theta^{t,\xi,(0)}) \bigr]^{1/2}
 \\
 &\hspace{15pt}+ C \sup_{(x,s) \in \R^d \times [t,T]} \sup_{\|\Lambda_{0}\|_{2} \leq L} \Bigl\{ 
 {\mathbb E} \Bigl[ \bigl(  \Lambda_{0} 
 \wedge 
 \Psi((x,s), \xi) 
 \bigr)^2 
{\mathbf 1}_{A}
 \Bigr]^{1/2} \Bigr\}.
 \end{split}
\end{equation*}
The right-hand side tends to $0$ with $\P(A)$, which proves 
\eqref{eq:cv:proba}.
\eproof

\subsection{Study of the second-order differentiability.}

The goal is now to discuss the second-order differentiability of $U$.

%

\subsubsection{Path property of $Z^{t,x,\law{\xi}}$ in 
${\mathcal S}^2([t,T];\R^{m \times d})$}

We start with the following remark. In the previous subsection, we proved that the function $\partial_{x} U$ was Lipschitz continuous with respect to the variables $x$ and $\mu$. Recalling the standard representation formula
\begin{equation}
\label{eq representation formula Z t,x,xi}
Z_{s}^{t,x,\law{\xi}} = \partial_{x} U\bigl(s,X_{s}^{t,x,\law{\xi}},\law{X_{s}^{t,\xi}}\bigr)
\sigma\bigl(X_{s}^{t,x,\law{\xi}},Y_{s}^{t,x,\law{\xi}},\law{X_{s}^{t,\xi},Y_{s}^{t,\xi}}\bigr),
\quad s \in [t,T],
\end{equation} 
see \eqref{eq de hvmu}, we may derive
bounds for $Z^{t,x,\law{\xi}}$ in the space ${\mathcal S}^2([t,T],\R^{m \times d})$
instead of ${\mathcal H}^2([t,T],\R^{m \times d})$ (and similarly for 
$Z^{t,\xi}$ by replacing $x$ by $\xi$ in the above formula). Under assumption \HYP{2}, which contains
\HYP{\sigma}, $\sigma$ is known to be bounded, so that 
$Z^{t,x,\law{\xi}}$ and $Z^{t,\xi,\law{\xi}}$ are indeed bounded 
(in $L^{\infty}$), independently of $\xi$. Moreover, for any $p \geq 1$, 
for $T \leq c_{p}$ with $c_p:=c_{p}(L)$, we can find $C_{p} 
\geq 0$ such that, for $\xi,\xi' \in L^2(\Omega,\cF_{t},\P;\R^d)$,
\begin{equation}
\label{eq sup Z xi}
\begin{split}
&{\mathbb E}_{t} \Bigl[ 
\sup_{r \in [t,T]} \bigl|Z^{t,x,\law{\xi}}_r - Z^{t,x',\law{\xi'}}_r
\bigr|^{2p} 
\Bigr]^{1/2p} 
 \le C_{p} \Bigl( 
1 \wedge \bigl\{ \vert x - x' \vert  + 
\Phi_{\alpha+1}(t,\xi,\xi') \bigr\}
\Bigr),
\\
&{\mathbb E}_{t} \Bigl[ 
 \sup_{r \in [t,T]} \bigl\vert Z^{t,\xi}_{r} - Z^{t,\xi'}_{r} \bigr\vert^{2p}
\Bigr]^{1/2p} 
\le C_{p} \Bigl( 
1 \wedge \bigl\{  \vert \xi - \xi' \vert + 
\Phi_{\alpha+1}(t,\xi,\xi') \bigr\}
\Bigr).
\end{split}
\end{equation}
Note that the term $\Phi_{\alpha+1}(t,\xi,\xi')$ comes from the fact that, 
when handling the difference $\partial_{x} U(s,X_{s}^{t,x,\law{\xi}},[X_{s}^{t,\xi}])
- \partial_{x} U(s,X_{s}^{t,x',\law{\xi'}},[X_{s}^{t,\xi'}])$, we get 
$C [ \vert X_{s}^{t,x,\law{\xi}} - X_{s}^{t,x',\law{\xi'}} \vert 
+ \Phi_{\alpha+1}(X_{s}^{t,\xi},X_{s}^{t,\xi'}) ]$ as bound. 
We then apply 
\eqref{eq:tilde:phi:alpha}
in 
Lemma \ref{lem:function:phi:alpha}
(with $\alpha+1$ instead of $\alpha$)
to handle $\Phi_{\alpha+1}(X_{s}^{t,\xi},X_{s}^{t,\xi'})$. 
{The 
part involving $\sigma$ in the definition of $Z_{s}^{t,x,[\xi]}$
can be treated by means of Lemma \ref{le app cont} using the fact that 
$\sigma$ is Lipschitz continuous.
Following Remark \ref{rem:unif:phi:alpha}}
and as in the statement of Lemma \ref{lem:cont:derivatives},
the restriction of $\Phi_{\alpha+1}(t,\cdot,\cdot)$ to the space $[L^2(\Omega,\cF_{0},\P;\R^d)]^2$ may be assumed to be independent of $t \in [0,T]$.

\subsubsection{Partial smoothness of $\partial_\mu U$. Overview.} 
\label{subsubse:overview}
By making use of \eqref{eq sup Z xi}, we first discuss the existence and the smoothness of the second-order derivatives
of $U$ in the measure argument.
The first remark is that we only need to discuss partial ${\mathcal C}^2$ 
differentiability in order to prove the chain rule. This says that, when 
investigating the second-order derivatives, there is no need to prove that the 
function $U$ has a twice Fr\'echet differentiable lifted version. Roughly 
speaking, the only thing we need is the 
differentiability of the mapping $\R^d \times 
\R^d \ni (x,v) \mapsto \partial_{\mu} U(t,x,\mu)(v)$ 
{(at least when $v$ is restricted 
to the support of $\mu$)}, together
with the continuity (in $(t,x,\mu,v)$) of the derivatives
{(again, at least when $v$ is restricted 
to the support of $\mu$)}.  
In order to differentiate in the direction $v$ without differentiating in the direction $\mu$, 
we shall make use of Theorem \ref{thm:ito:frechet}, which has been specifically designed for that purpose. 
Basically, we are to differentiate 
the lifted version
of $\partial_{\mu} U(t,x,\mu)$ along trajectories 
$(\xi^{\lambda})_{\lambda \in \R}$ that are continuously differentiable in $L^2(\Omega,{\mathcal F}_{t},\P;\R^d)$,
with the constraint that all the $(\xi^{\lambda})_{\lambda \in \R}$ have 
the same distribution and the assumption that 
\begin{equation}
\label{eq:bound:linfini:derivative:xi}
\forall \lambda \in \R, \quad
\bigl\| \frac{\ud}{\ud \lambda} \xi^{\lambda} \bigr\|_{\infty} \leq 1
,\quad \textrm{with the additional notation \ }
\zeta := \frac{\ud}{\ud \lambda}_{\vert \lambda = 0}
\xi^{\lambda}.
\end{equation}
\color{black}
In this framework, we will make use of the following technical lemma:
\begin{Lemma}
\label{lem:lemme:final???}
Consider a function 
$h : \R^k \times \cP_{2}(\R^l) \rightarrow \R$ as in \HYP{2},
a 
continuously differentiable 
mapping $\R \ni \lambda \mapsto  \chi^{\lambda} \in 
L^2(\Omega,\cA,\P;\R^l)$ with the property that all
the $\chi^{\lambda}$ have the same distribution,
and a random variable $\varpi \in L^2(\Omega,\cA,\P;\R^l)$
such that, for any bounded interval $[a,b] \subset \R$, the family
\begin{equation*}
\Bigl( \frac{\ud \chi^{\lambda}}{\ud \lambda} \otimes \varpi \Bigr)_{\lambda \in [a,b]}
\end{equation*}
is uniformly square integrable (the tensorial product acting on $\R^l$).
Then, the function
\begin{equation*}
\label{eq:lemme:final???:4} 
\R^k \times \R \ni (w,\lambda) \mapsto 
{\mathbb E}
\bigl[
\partial_{\mu} h (w,[\chi^{\lambda}])
(\chi^{\lambda}) \varpi
\bigr]
= \E \bigl[ \partial_{\mu} h (w,[\chi^{0}])
(\chi^{\lambda}) \varpi
\bigr]
\end{equation*}
is continuously differentiable, with 
\begin{equation*}
\begin{split}
&\R^k \times \R \ni (w,\lambda) \mapsto 
\E \Bigl[
 \partial_{w} \bigl[ \partial_{\mu} h (w,[\chi^{0}])
\bigr] (\chi^\lambda)
 \varpi
\Bigr]
\\
&\R^k \times \R \ni (w,\lambda) \mapsto 
\E \Bigl[ 
\partial_{v} \bigl[ \partial_{\mu} h (w,[\chi^{0}])\bigr] 
(\chi^{\lambda})
 \frac{\ud \chi^{\lambda}}{\ud \lambda}
 \otimes  \varpi\Bigr]
 \end{split}
\end{equation*}
as respective partial derivatives in $w$ and $\lambda$.
\end{Lemma}

\proof 
For $w,w' \in
\R^k$ and $\lambda,\lambda' \in \R$, we
write (thanks to \HYP{2}):
\begin{align}
&\partial_{\mu} h (w',[\chi^{0}])(\chi^{\lambda'})
-
\partial_{\mu} h (w,[\chi^{0}])(\chi^{\lambda})
\nonumber
\\
&=
\partial_{\mu} h (w',[\chi^{0}])(\chi^{\lambda'})
-
\partial_{\mu} h (w',[\chi^{0}])(\chi^{\lambda})
+
\partial_{\mu} h (w',[\chi^{0}])(\chi^{\lambda})
-
\partial_{\mu} h (w,[\chi^{0}])(\chi^{\lambda}) \nonumber
\\
&= \biggl( \int_{0}^1 
\partial_{v}
\bigl[ \partial_{\mu} h (w',[\chi^{0}]) \bigr]\bigl(
s
\chi^{\lambda'}
+ (1-s) \chi^{\lambda}
\bigr) 
\ud s 
\biggr)
 \bigl( \chi^{\lambda'} - \chi^{\lambda} \bigr)
 \label{eq:lemme:final???} 
\\
&\hspace{15pt} + 
\biggl( 
\int_{0}^1 
\partial_{w}
\bigl[ \partial_{\mu} h \bigl(s w' + (1-s) w,[\chi^{0}] \bigr) \bigr]
\bigl(
\chi^{\lambda}
\bigr)\ud s
\biggr) (w'-w). 
\nonumber
\end{align}
Thanks to the $L^2$ bounds on $\partial_{w}[\partial_{\mu} h]$
and $\partial_{v}[\partial_{\mu} h]$ in \HYP{2}, we deduce that, 
as $(w',\lambda') \rightarrow (w,\lambda)$,
\begin{equation}
\label{eq:lemme:final???:1} 
\begin{split}
&\E \biggl[
\biggl\vert \int_{0}^1 
\partial_{v}
\bigl[ \partial_{\mu} h (w',[\chi^{0}]) \bigr]\bigl(
s
\chi^{\lambda'}
+ (1-s) \chi^{\lambda}
\bigr) 
\ud s 
-
\partial_{v}
\bigl[ \partial_{\mu} h (w,[\chi^{0}]) \bigr]\bigl(
\chi^{\lambda}
\bigr) 
\biggr\vert^2 
\biggr] \rightarrow 0,
\\
&\E \biggl[
\biggl\vert
\int_{0}^1 
\partial_{w}
\bigl[ \partial_{\mu} h \bigl(s w' + (1-s) w,[\chi^{0}] \bigr) \bigr]
\bigl(
\chi^{\lambda}
\bigr)\ud s
-
\partial_{w}
\bigl[ \partial_{\mu} h (w,[\chi^{0}]) \bigr]\bigl(
\chi^{\lambda}
\bigr) 
\biggr\vert^2 
\biggr] \rightarrow 0.
\end{split}
\end{equation}
Notice now from the uniform integrability property in the assumption that,
as $\lambda' \rightarrow \lambda$ (with $\lambda' \not = \lambda$), 
\begin{equation}
\label{eq:lemme:final???:2} 
\E \Bigl[ \bigl\vert 
\bigl( \frac{\chi^{\lambda'} - \chi^{\lambda}}{\lambda'-\lambda}
- \frac{\ud \chi^{\lambda}}{\ud \lambda}
\bigr) \otimes \varpi
\bigr\vert^2 \Bigr] \rightarrow 0.
\end{equation}
Plugging 
\eqref{eq:lemme:final???:1} 
and
\eqref{eq:lemme:final???:2}
into 
\eqref{eq:lemme:final???}, we easily deduce that 
the mapping 
\eqref{eq:lemme:final???:4}
is differentiable. Continuity of the partial derivatives is proved 
in the same way.  
\eproof
\color{black}

\subsubsection{Partial smoothness of $\partial_\mu U$. Strategy.} 
Generally speaking, the strategy is the same as for proving the 
first-order continuous differentiability and consists in discussing 
the continuous differentiability of the derivative processes 
$\partial_{\chi} \theta^{t,\xi^{\lambda}} = (\partial_{\chi} X^{t,\xi^\lambda},
\partial_{\chi} Y^{t,\xi^\lambda},\partial_{\chi} Z^{t,\xi^\lambda})$
and $\partial_{\chi} \theta^{t,x,\law{\xi^{\lambda}}} = (\partial_{\chi} X^{t,x,\law{\xi^\lambda}},
\partial_{\chi} Y^{t,x,\law{\xi^\lambda}},\partial_{\chi} Z^{t,x,\law{\xi^\lambda}})$ 
with respect to $\lambda$ when the family $(\xi^{\lambda})_{\lambda \in \R}$
satisfies the aforementioned prescriptions and $\chi$ is in $L^2(\Omega,{\mathcal F}_{t},\P;\R^d)$. 
Together with the relationship $\partial_{\chi} Y_{t}^{t,x,[\xi^\lambda]} = D {\mathcal U}(t,x,\xi^\lambda)
\cdot \chi$, this will permit  
to apply Theorem
\ref{thm:ito:frechet} (compare in particular with 
\eqref{eq:ito:frechet:assumption}).

Intuitively, one has in mind to consider first the partial second order tangent 
process of the McKean-Vlasov FBSDE \eqref{eq X-Y t,xi} in the direction $\chi$ and $\zeta$, which we shall denote by
$\dzeki \theta^{t,\xi} = (\dzeki X^{t,\xi}, \dzeki Y^{t,\xi}, \dzeki Z^{t,\xi}) := [\ud/\ud \lambda]_{\vert \lambda=0} 
\partial_\chi \theta^{t,\xi^\lambda}$. Informally,  
this process should satisfy a 
system of the form \eqref{eq:sys:linear:mkv}, 
with coefficients of the generic form 
\eqref{splitting:H}. 
Precisely, 
the coefficients $H_{\ell}$ should have the same decomposition as in the first order case, see 
\eqref{eq:identification:derivatives}, $V$ and $\hat{V}$ also standing for $\theta$, 
$\theta^{(0)}$ or $X$ but $\cV$ and $\hat{\cV}$ now standing for $\dzeki \theta$, $\dzeki 
\theta^{(0)}$ or $\dzeki X$ (with the usual convention 
that the symbol $(0)$ in
$V^{(0)}$ and 
$\cV^{(0)}$ indicates the restriction to the two first coordinates).
Terms $B_{a}$, $\Sigma_{a}$, $F_{a}$ and $G_{a}$
in \eqref{splitting:H}
 should not be zero anymore 
and should be defined as follows
for a generic coefficient $h$ that may be $b$, $\sigma$, $f$ or $g$: 
\begin{equation}
\label{eq:coeff:H2}
\begin{split}
H_a(r) 
&= \partial^2_{ww} h 
\bigl(\theta^{t,\xi}_{r},[\theta^{t,\xi,(0)}_{r}] \bigr) 
\partial_{\chi} \theta^{t,\xi}_{r} \otimes 
\partial_{\zeta} \theta^{t,\xi}_{r} 
\\
&\hspace{15pt} + \hat{\E}
\bigl[ \partial_{w} \bigl[ \partial_{\mu} 
h 
\bigl(\theta^{t,\xi}_{r},[\theta^{t,\xi,(0)}_{r}] \bigr) \bigr] 
\bigl(\langle \theta^{t,\xi,(0)}_{r} \rangle \bigr)
\langle 
\partial_{\chi} \theta^{t,\xi,(0)}_{r} \rangle
\otimes \partial_{\zeta} \theta^{t,\xi}_{r} \bigr] 
\\
&\hspace{15pt}
+  \hat{\E}
\bigl[ \partial_{v} \bigl[ \partial_{\mu} 
h 
\bigl(\theta^{t,\xi}_{r},[\theta^{t,\xi,(0)}_{r}] \bigr) 
\bigr] \bigl(\langle \theta^{t,\xi,(0)}_{r} \rangle \bigr)
\langle 
\partial_{\chi} \theta^{t,\xi,(0)}_{r} \rangle
\otimes \langle \partial_{\zeta} \theta^{t,\xi,(0)}_{r} \rangle 
\bigr]
\\
&=: H_{a}^{(2)}\bigl(\theta^{t,\xi}_{r},\langle \theta^{t,\xi,(0)}_{r} \rangle,
\partial_{\chi} \theta^{t,\xi}_{r},\partial_{\zeta} \theta^{t,\xi}_{r},
\langle \partial_{\chi} \theta^{t,\xi,(0)}_{r} \rangle, \langle \partial_{\zeta} \theta^{t,\xi,(0)}_{r}
\rangle \bigr)
\\
&=:H_{a}^{ww}(r)
+H_{a}^{w\mu}(r)
+H_{a}^{v\mu}(r),
\end{split}
\end{equation}
where $H_{a}^{(2)}$ could be expressed (in an obvious way) as a function of general arguments 
$\theta_{r}$, $\langle \hat{\theta}_{r}^{(0)} \rangle$, $\vartheta_{r}^{1}$, $\vartheta_{r}^{2}$, 
$\langle \hat \vartheta_{r}^{1,(0)} \rangle$ and 
$\langle \hat \vartheta_{r}^{2,(0)} \rangle$ instead of 
$\theta^{t,\xi}_{r}$, $\langle \theta^{t,\xi,(0)}_{r} \rangle$, 
$\partial_{\chi} \theta^{t,\xi}_{r}$, $\partial_{\zeta} \theta^{t,\xi}_{r}$,
$\langle \partial_{\chi} \theta^{t,\xi,(0)}_{r} \rangle$ and $\langle \partial_{\zeta} \theta^{t,\xi,(0)}_{r}
\rangle$. By analogy with \eqref{eq:ex:mkv:stability}, we can let
\begin{equation}
\label{coeff:H2}
\begin{split}
&H^{(2)}\bigl(r,\theta_{r},
\langle \hat \theta_{r}^{(0)} \rangle,
\vartheta_r^{1},\vartheta_{r}^2,
\langle \hat \vartheta_{r}^{1,(0)} \rangle,
\langle \hat \vartheta_{r}^{2,(0)} \rangle
\bigr)\bigl(\vartheta_{r},\langle \hat \vartheta_{r}^{(0)} \rangle \bigr) 
\\
&:= 
\partial_{w} h (\theta_{r},[\hat \theta_{r}^{(0)}]) \vartheta_{r} +  \hat{\mathbb E}
\bigl[ \partial_{\mu} h(\theta_{r},[\hat \theta_{r}^{(0)}])(\langle \hat \theta_{r}^{(0)}\rangle) 
\langle \hat \vartheta_{r}^{(0)} \rangle \bigr]
\\
&\hspace{15pt} + 
 H_{a}^{(2)}\bigl(\theta_{r},\langle \hat \theta^{(0)}_{r} \rangle,
\vartheta_{r}^{1},\vartheta_{r}^{2},
\langle \hat \vartheta^{1,(0)}_{r} \rangle, \langle \hat \vartheta^{2,(0)}_{r}
\rangle \bigr), \quad r \in [t,T]. 
\end{split}
\end{equation}
Pay attention that there is no `second-order 
derivatives' in the direction of the measure (\textit{i.e.} `$\partial_{\mu\mu}^2 h$')  in 
\eqref{eq:coeff:H2}. Indeed, the fact that  
the initial conditions $(\xi^{\lambda})_{\lambda}$
have the same distribution forces the 
solutions $(\theta^\lambda)_{\lambda}$ to be identically distributed as well.
For the same reason, there is no crossed derivative of the form `$\partial_{\mu} [\partial_{w}h]$'. 
On the opposite, notice that $(\partial_\chi \theta^\lambda)$ (resp. $(\partial_\zeta 
\theta^\lambda)$) are not identically distributed since the coupling between
$\chi$ (resp. $\zeta$) and $\xi^\lambda$ may vary.
In particular, when differentiating with respect to $\lambda$
an expression of the form $\hat{\E}[\partial_{\mu} h(\theta^{\lambda},[\theta^{\lambda,(0)}])(\langle
\theta^{\lambda,(0)} \rangle) \langle \partial_{\chi} \theta^{\lambda,(0)} \rangle]$
for a function $h$ as above, the input $[\theta^{\lambda,(0)}]$ has a zero derivative
as it is constant in $\lambda$, but the two last inputs, namely 
$\langle
\theta^{\lambda,(0)} \rangle$ and $\langle \partial_{\chi} \theta^{\lambda,(0)} \rangle$,
may
give a non-trivial contribution.

On the model of \eqref{assumption:mkv:1}, \eqref{assumption:mkv:2}
and \eqref{assumption:mkv:3}, we shall use the following assumptions on the 
coefficients (compare also with \HYP{2}):
\begin{equation}
 \label{eq:cond:coef:second:1}
 \begin{split}
 & |\partial_{ww}^2h(w,\law{\hat V^{(0)}})| 
+  \hesp{ \bigl|\partial_w  \bigl[\partial_\mu h(w,\law{\hat V^{(0)}}) \bigr](\cc{\hat V^{(0)}})\bigr|^2}^{1/2} 
\\
&\hspace{15pt}+  \hesp{\bigl|\partial_v \bigl[ \partial_\mu h(w,\law{\hat V^{(0)}})\bigr](\cc{\hat V^{(0)}})\bigr|^2}^{1/2}
\le C, 
\end{split}
 \end{equation}
and
\begin{equation}
\label{eq:cond:coef:second:2}
\begin{split}
&|\partial_{ww}^2h(w,\law{\hat V^{(0)}})-\partial_{ww}^2h(w',\law{\hat V^{(0)\prime}})|
\\
&\hspace{15pt}
 +\hesp{ \bigl|\partial_w \bigl[ \partial_\mu h(w,\law{\hat V^{(0)}}) \bigr](\cc{\hat V^{(0)}})
-\partial_w \bigl[ \partial_\mu h(w',\law{\hat V^{(0)\prime}}) \bigr] (\cc{\hat V^{(0)\prime}})\bigr|^2}^\frac12 
\\
&\hspace{15pt}
 +\hesp{\bigl|\partial_v \bigl[ \partial_\mu h(w,\law{\hat V^{(0)}}) \bigr](\cc{\hat V^{(0)}})
-\partial_v \bigl[ \partial_\mu h(w',\law{\hat V^{(0)\prime}}) \bigr] (\cc{\hat V^{(0)\prime}}) \bigr|^2}^\frac12  \\
&\le C\bigl(|w-w'|+ \Phi_{\alpha}(\hat V^{(0)},\hat V^{(0)\prime} )\bigr)\;.
\end{split}
\end{equation}
\subsubsection{Preliminary estimates}
We start with the following bound of the remainder term
$H^{(2)}_{a}$ in \eqref{coeff:H2}:
\begin{Lemma}
\label{lem:bound:H2:1}
Given generic processes $\theta$, $\hat \theta$, 
$\vartheta^1$, $\vartheta^2$, $\hat \vartheta^{1}$
and $\hat  \vartheta^{2}$, denote 
by $H_{a}^{(2)}(r)$ the term
$H_{a}^{(2)}(\theta_{r},\langle \hat \theta^{(0)}_{r} \rangle,
\vartheta_{r}^{1}, \vartheta^{2}_{r},
\langle \hat \vartheta^{1,(0)}_{r} \rangle, \langle \hat \vartheta^{2,(0)}_{r} \rangle
)$
in \eqref{coeff:H2}, $H$ matching $B$, $\Sigma$, $F$ or $G$. For any $p \geq 1$, we can find
a constant $C_{p}$ (independent of 
the processes) such that (using the notation $\bar{\cM}$
from
\eqref{eq:mkv:stability:Psi})
\begin{align}
 &\E_{t} \biggl[
  \vert G_{a}^{(2)}(T) \vert^{2p} +
  \biggl(
 \int_t^T
 \bigl(
 \vert B_{a}^{(2)}(s) \vert + \vert F_{a}^{(2)}(s) \vert \bigr) \ud s \biggr)^{2p}  
  + \biggl(
 \int_t^T
 \vert \Sigma_{a}^{(2)}(s)|^2\ud s \biggr)^{p} \biggr]^{1/2p}
 \label{eq:sncf:retour:5}
 \\
 &\leq C_{p} 
 \biggl[
 \Bigl( \bar\cM^{4p}\bigl(\vartheta^1,\hat{\vartheta}^{1}\bigr) \Bigr)^{1/4p}
 \Bigl( \bar\cM^{4p}\bigl(\vartheta^2,\hat{\vartheta}^{2}\bigr) \Bigr)^{1/4p}
+
\E \Bigl[ 
 \NSt{4}{\hat\vartheta^{1,(0)}}^2\NSt{4}{\hat 
\vartheta^{2,(0)}}^2 \Bigr]^{1/2}
\biggr]. \nonumber
\end{align}
\end{Lemma}
\proof
We start with the case $H=B$ (resp. $F$), or equivalently $h=b$ (resp. $f$).  
We use a decomposition of the same type as \eqref{eq:coeff:H2} (with the same notations).
By conditional Cauchy-Schwartz 
inequality and \eqref{eq:cond:coef:second:1}, we can find a constant $C_{p}$
such that 
\begin{align}
\label{eq:sncf:retour:1}
 \E_{t} \biggl[ \biggl(
 \int_t^T
 \vert H_{a}^{ww}(s)|\ud s \biggr)^{2p} \biggr]
&\leq C_{p} 
 \NHt{4p}{\vartheta^1 }^{2p}
 \NHt{4p}{\vartheta^2 }^{2p}. 
 \end{align}
We now aim to obtain similar upper bound for the other terms in \eqref{eq:coeff:H2}.
We therefore observe, using \eqref{eq:cond:coef:second:1}, 
$ |H^{w\mu}_a(s)|
 \le C\NL{2}{\hat \vartheta_s^{1,(0)}}| \vartheta_s^2|$,
so that 
\begin{align}
\label{eq:sncf:retour:2}
 \E_{t} \biggl[ \biggl(
 \int_t^T
 \vert H_{a}^{w\mu}(s)|\ud s \biggr)^{2p} \biggr]
&\leq C_{p} 
 \NHt{2p}{\vartheta^2 }^{2p}
 \NS{2}{\hat \vartheta^{1,(0)} }^{2p}. 
 \end{align}
We now handle $H^{v\mu}_a$. By conditional H\"older inequality, we observe that, under 
condition \eqref{eq:cond:coef:second:1},
$ |H^{v\mu}_a(s)| \le C
\E [ \NSt{4}{\hat \vartheta^{1,(0)}}^2 
\NSt{4}{\hat 
\vartheta^{2,(0)}}^2 ]^{1/2}$,
from which we get
\begin{align}
\label{eq:sncf:retour:3}
 \E_{t} \biggl[ \biggl(
 \int_t^T
 \vert H_{a}^{v\mu}(s)|\ud s \biggr)^{2p} \biggr]
 \le C_{p}
\E \Bigl[ 
 \NSt{4}{\hat \vartheta^{1,(0)}}^2 \NSt{4}{
 \hat
\vartheta^{2,(0)}}^2 \Bigr]^{p}.
\end{align}
By \eqref{eq:sncf:retour:1}, \eqref{eq:sncf:retour:2} and \eqref{eq:sncf:retour:3}
and with the notation 
\eqref{eq:mkv:stability:Psi}, we 
get \eqref{eq:sncf:retour:5}.
The cases when $H=\Sigma$ or $G$ may be handled in the same way. 
\eproof

\begin{Lemma}
\label{lem:bound:H2:2}
Given processes $\theta$, $\theta'$, $\hat\theta$, $\hat \theta^{\prime}$, 
$\vartheta^1$, $\vartheta^{1\prime}$, $\vartheta^2$, 
$\vartheta^{2\prime}$, 
$\hat \vartheta^{1}$, $\hat \vartheta^{1\prime}$,
$\hat \vartheta^{2}$ and $\hat \vartheta^{2\prime}$, denote 
the terms
$H_{a}^{(2)}(\theta_{r},\langle \hat\theta^{(0)}_{r} \rangle,
\vartheta^1_{r},\vartheta^2_{r},
\langle \hat \vartheta^{1,(0)}_{r} \rangle, \langle \hat \vartheta^{2,(0)}_{r}
\rangle)$ in \eqref{coeff:H2}
by $H_{a}^{(2)}(r)$
and use a similar definition for $H_{a}^{(2)\prime}(r)$. For any $p \geq 1$, we can find
a constant $C_{p}$ (independent of 
the processes) such that, for any random variable $\varepsilon$
with values in $\R_{+}$ (with the notation 
$\bar{\cM}$ from 
\eqref{eq:mkv:stability:Psi}),
\begin{align}
&\E_{t} \biggl[ 
 \vert G_{a}^{(2)}(T) - G_{a}^{(2)\prime}(T) \vert^{2p}  +
\biggl(
 \int_t^T
 \bigl(
 \vert B_{a}^{(2)}(s) - B_a^{(2)\prime}(s) \vert + \vert F_{a}^{(2)}(s) - F_{a}^{(2)\prime}(s) \vert \bigr) \ud s \biggr)^{2p}  \nonumber
\\
&\hspace{50pt} + \biggl(
 \int_t^T
 \vert \Sigma_{a}^{(2)}(s) - \Sigma_{a}^{(2)\prime}(s)|^2\ud s \biggr)^{p} \biggr]^{1/2p}  \nonumber
\\
&\leq C_{p}\Bigl\{ \Bigl(1\wedge \Bigl[ \E_{t}\bigl(\varepsilon^{4p} \bigr)^{1/4p} +
\Bigl( \bar{\mathcal M}^{4p}\bigl(\theta - \theta',\hat\theta - \hat\theta'\bigr) \Bigr)^{1/4p}
 + \Phi_{\alpha}\bigl(\hat\theta^{(0)},\hat\theta^{(0) \prime}\bigr)
 \Bigr] 
\Bigr) \nonumber
\\
&\hspace{50pt} \times \Bigl[ 
\Bigl( \bar{\mathcal M}^{8p}\bigl(\vartheta^1,\hat \vartheta^1\bigr) \Bigr)^{1/8p}
\Bigl( \bar{\mathcal M}^{8p}\bigl(\vartheta^2,\hat \vartheta^2\bigr) \Bigr)^{1/8p}
+ \E \Bigl[
\NSt{4}{\hat \vartheta^{1,(0)}}^2
\NSt{4}{\hat \vartheta^{2,(0)}}^2 \Bigr]^{1/2}
\Bigr] \Bigr\}  \nonumber
 \\
 &\hspace{15pt} + 
 C_{p} \Bigl\{ \Bigl( \bar\cM^{4p}
\bigl(\vartheta^1 -\vartheta^{1\prime},\hat\vartheta^1 -\hat\vartheta^{1\prime}
  \bigr) \Bigr)^{1/4p}
\Bigl( \bar{\mathcal M}^{4p}\bigl(\vartheta^2,\hat \vartheta^2) \Bigr)^{1/4p} \label{eq:sncf:retour:5:b}
\\
&\hspace{50pt} + 
 \Bigl( \bar\cM^{4p}
\bigl(\vartheta^{1\prime},\hat \vartheta^{1\prime} \bigr) \Bigr)^{1/4p}
\Bigl( \bar{\mathcal M}^{4p}\bigl(\vartheta^2 - \vartheta^{2\prime},
\hat\vartheta^2 - \hat\vartheta^{2\prime}
\bigr) \Bigr)^{1/4p} \Bigr\}  \nonumber
\\
&\hspace{15pt}+ C_{p}  \Bigl\{
\E \Bigl[
\NSt{4}{\hat \vartheta^{1,(0)\prime} }^2
 \NSt{4}{\hat \vartheta^{2,(0)}
- \hat \vartheta^{2,(0) \prime}
}^2
 \Bigr]^{1/2} + 
\E \Bigl[
 \NSt{4}{\hat \vartheta^{1,(0)}
- \hat \vartheta^{1,(0) \prime}
}^2
\NSt{4}{\hat \vartheta^{2,(0)} }^2
\Bigr]^{1/2} \Bigr\} \nonumber
\\
&\hspace{15pt} + \E_{t} \biggl[
\left( \int_t^T 
{\mathbf 1}_{\{
\vert \theta_{s} - \theta_{s}' \vert > \varepsilon
 \}}
\vert \vartheta_{s}^1 \vert \, \vert \vartheta^2_{s}
\vert
\ud s \right)^{2p} 
\biggr]^{1/2p}. \nonumber
\end{align}
\end{Lemma}
\proof
We start with the case when $H=B,F$. 
As in the proof of Lemma \ref{lem:bound:H2:1},   
we make use of the decomposition \eqref{eq:coeff:H2}.
Denoting by $H_2^{ww}$ and $H_{2}^{ww\prime}$ the related terms in 
\eqref{eq:coeff:H2}, we compute:
\begin{equation}
\label{eq:sncf:retour:10}
\begin{split}
& \E_{t} \biggl[ \biggl(
 \int_t^T
 \vert H_{a}^{ww}(s) - {H_{a}^{ww}}'(s) |\ud s \biggr)^{2p} \biggr]^{1/2p}
\\
&\leq  \E_{t} \biggl[ \left( \int_t^T 
\left| \left\{
\partial^2_{ww} h \bigl(\theta_{s},[\hat \theta^{(0)}_{s}] \bigr) 
- \partial^2_{ww} h \bigl(\theta_{s}',[\hat \theta^{(0)\prime}_{s}] \bigr)
\right\}
\vartheta^1_{s} \otimes \vartheta^2_{s}
\right|
\ud s \right)^{2p} \biggr]^{1/2p} 
\\
&\hspace{15pt} +  \E_{t} \biggl[ \left( \int_t^T
\left| \partial^2_{ww} h \bigl(\theta_{s}',[\hat\theta^{(0) \prime}_{s}] \bigr)
\left\{
\vartheta^1_{s} - \vartheta_{s}^{1\prime}
\right\}
\otimes \vartheta^2_{s}
\right|
\ud s \right)^{2p} \biggr]^{1/2p} 
\\
&\hspace{15pt} + \E_{t} \biggl[ \biggl( \int_t^T
\left|
 \partial^2_{ww} h \bigl(\theta_{s}',[\hat \theta^{(0)\prime}_{s}] \bigr)
\vartheta_{s}^{1\prime}
\otimes 
\left\{
\vartheta^{2}_{s}-\vartheta_{s}^{2\prime}
\right\}
\right|
\ud s \biggr)^{2p} \biggr]^{1/2p}
\\
&:= A_{1} + A_{2} + A_{3}. 
\end{split}
\end{equation}
Bounding the difference of the terms in 
$\partial^2_{ww} h$ by a constant or by the increment of the underlying 
variables, 
we thus obtain, for any  
random variable 
$\varepsilon$ with values in $\R_{+}$,
\begin{align*}
 A_1 & \le C \E_{t} \biggl[
\Bigl\{ 1 \wedge \Bigl( \varepsilon + \sup_{s \in [t,T]} 
\Phi_{\alpha}\bigl(\hat \theta^{(0)}_s,\hat \theta^{(0) \prime}_s\bigr)
\Bigr)^{2p} \Bigr\}
\biggl( \int_t^T 
\vert \vartheta^1_{s} \vert \, \vert \vartheta^2_{s}
\vert
\ud s \biggr)^{2p} 
\biggr]^{1/2p}
\\
&\hspace{15pt} + \E_{t} \biggl[
\left( \int_t^T 
{\mathbf 1}_{\{
\vert \theta_{s} - \theta_{s}' \vert > \varepsilon
 \}}
\vert \vartheta^1_{s} \vert \, \vert \vartheta^2_{s}
\vert
\ud s \right)^{2p} 
\biggr]^{1/2p}
\\
& \le C_{p} \Bigl(1\wedge \Bigl[ \E_{t}\bigl(\varepsilon^{4p} \bigr)^{1/4p}+ 
 \sup_{s\in[t,T]}\Phi_{\alpha}\bigl(\hat \theta^{(0)}_s,\hat \theta^{(0) \prime}_s\bigr)
 \Bigr] 
\Bigr)
 \NHt{8p}{\vartheta^1}
\NHt{8p}{\vartheta^2}
\\
&\hspace{15pt} + \E_{t} \biggl[
\left( \int_t^T 
{\mathbf 1}_{\{
\vert \theta_{s} - \theta_{s}' \vert > \varepsilon
 \}}
\vert \vartheta_{s}^1 \vert \, \vert \vartheta^2_{s}
\vert
\ud s \right)^{2p} 
\biggr].
 \end{align*}
We also have
\begin{align*}
 A_2 + A_{3}&\le C_p \Bigl( 
 \NHt{4p}{\vartheta^1 -\vartheta^{1 \prime}}
 \NHt{4p}{\vartheta^2}
 + 
 \NHt{4p}{\vartheta^{1\prime}}
 \NHt{4p}{\vartheta^2 -\vartheta^{2\prime}}
 \Bigr).
\end{align*}
 Next, 
 using a similar decomposition to \eqref{eq:sncf:retour:10},
 we compute
\begin{align*}
 & \E \biggl[ \biggl(
 \int_t^T
 \vert H_{a}^{w\mu}(s) - {H_{a}^{w\mu}}'(s) |\ud s \biggr)^{2p} \biggr]^{1/2p}
 \\
 &\le C_{p} \Bigl\{ \Bigl(1\wedge \bigl\{
 \NHt{4p}{\theta -\theta'}+ 
\sup_{s\in [t,T]}\Phi_{\alpha}(\hat \theta^{(0)}_s,\hat \theta_s^{(0)\prime})
\bigr\} \Bigr)
\NS{2}{\hat \vartheta^{1,(0)}} \NHt{4p}{\vartheta^2}
 \\
 &\hspace{15pt} + 
\NS{2}{\hat \vartheta^{1,(0)}-\hat \vartheta^{1,(0)\prime}}
\NHt{2p}{\vartheta^2}
 + 
\NS{2}{\hat \vartheta^{1,(0)\prime}}
\NHt{2p}{ \vartheta^{2} - \vartheta^{2\prime}}\Bigr\}.
\end{align*}
We also get
\begin{align*}
&\E_{t} \biggl[ \biggl(
 \int_t^T
 \vert H_{a}^{v\mu}(s) - {H_{a}^{v\mu}}'(s) |\ud s \biggr)^{2p} \biggr]^{1/2p}
\\
 &\le C_{p} \Bigl\{ \Bigl(1\wedge \bigl\{ 
 \NHt{2p}{\theta -\theta'}+ 
\sup_{s\in[t,T]}\Phi_{\alpha}(\hat \theta^{(0)}_s,\hat \theta^{(0)\prime}_s) 
\bigr\}
\Bigr) \E \Bigl[
\NSt{4}{\hat \vartheta^{1,(0)}}^2 
\NSt{4}{\hat \vartheta^{2,(0)}}^2
\Bigr]^{1/2}
\\
&\hspace{5pt}+ 
\E \Bigl[ 
\NSt{4}{\hat 
\vartheta^{1,(0)\prime}}^2
\NSt{4}{\hat 
\vartheta^{2,(0)}-\hat 
\vartheta^{2,(0) \prime}}^2
 \Bigr]^{1/2}
+ 
\E \Bigl[
\NSt{4}{\hat 
\vartheta^{1,(0)}-\hat 
\vartheta^{1,(0) \prime}}^2
\NSt{4}{\hat 
\vartheta^{2,(0)}}^2
\Bigr]^{1/2} \Bigr\}.
\end{align*}
Collecting the various inequalities, we get \eqref{eq:sncf:retour:5:b}.

The proof is quite similar when $H=\Sigma$ or $G$,
{but there are two main differences. 
The first one is that, in the analysis of 
$\Sigma_{a}^{(2)}$ and $G_{a}^{(2)}$, processes are estimated 
with ${\mathcal S}$ instead of ${\mathcal H}$ norms. Obviously, this does not affect
\eqref{eq:sncf:retour:5:b} since $\Sigma_{a}^{(2)}$ and $G_{a}^{(2)}$ only 
involve the two first coordinates of $\theta$, $\vartheta^1$
and $\vartheta^2$.}  
 The second main difference comes 
from $A_{1}$. Since neither $\sigma$ nor $g$ depend on 
the component $Z$, we can replace $\vert \theta_{s} - \theta_{s}' \vert$
by $\vert \theta_{s}^{(0)} - \theta_{s}^{(0)\prime} \vert$ in the analysis of the term corresponding to $A_{1}$. Choosing
$\varepsilon = \sup_{s \in [t,T]}
\vert \theta_{s}^{(0)} - \theta_{s}^{(0)\prime}\vert$, we get rid of the remaining 
term containing the indicator function of the event $\{
\vert \theta_{s}^{(0)} - \theta_{s}^{(0)\prime} \vert > \varepsilon
 \}$. Then, $\E_{t}[ \varepsilon^{4p}]^{1/4p}$ is exactly 
 equal to $\NSt{4p}{\theta^{(0)} - \theta^{(0)\prime}}$, 
 which is less than $(\bar{\cM}^{4p}(\theta - \theta',\hat \theta - \hat \theta'))^{1/4p}$. 
\eproof

\subsubsection{Proof of the differentiability of the McKean-Vlasov system}
We claim:
\begin{Lemma}
\label{lem:2ndorder:diff}
There exists $c:=c(L)>0$ such that, for $T \leq c$, for 
$\chi$ and $(\xi^{\lambda})_{\lambda}$ as in 
\S \ref{subsubse:overview}, the mapping
\begin{equation*}
\R \ni \lambda \mapsto \partial_{\chi} \theta^{t,\xi^\lambda}
= \bigl( \partial_{\chi} X^{t,\xi^{\lambda}},\partial_{\chi}
Y^{t,\xi^{\lambda}},\partial_{\chi} Z^{t,\xi^{\lambda}} \bigr),
\end{equation*}
with values in ${\mathcal S}^{2}([t,T];\R^d) \times {\mathcal S}^2([t,T];\R^m)
\times {\mathcal H}^2([t,T];\R^{m \times d})$,
is continuously differentiable.
The derivative at $\lambda = 0$
only depends upon the family $(\xi^\lambda)_{\lambda \in \R}$
through the value of $\xi := \xi^0$ and 
$\zeta := [\ud/\ud \lambda]_{\lambda =0}
\xi^\lambda$ (see footnote \ref{footnote:second:order:derivative}
on page \pageref{footnote:second:order:derivative} for a precise meaning). It is denoted by $\partial^2_{\zeta,\chi} \theta^{t,\xi}$.
\end{Lemma}

\begin{proof}
We adapt the proof of 
Lemma \ref{lem:1storder:diff}. To do so, we use the Picard sequence $((\theta^{n,\lambda},\partial_{\chi} \theta^{n,\lambda}))_{n \geq 1}$
solving
\eqref{eq:syst:derive:picard},
with $X^{n,\lambda}_{t}=\xi^\lambda$
and
$\chi^\lambda = \chi$ for any $\lambda \in \R$. 
The sequence
converges 
in $[{\mathcal S}^2([t,T];\R^d) 
\times 
{\mathcal S}^2([t,T],\R^m) \times {\mathcal H}^2([t,T];\R^{m \times d})]^2$
towards $(\theta^{t,\xi^\lambda},\partial_{\chi} \theta^{t,\xi^{\lambda}})$, uniformly in $\lambda$ in compact subsets.
{Pay attention that, in 
\eqref{eq:syst:derive:picard},
 the choice 
$\chi^\lambda=\chi$, for all $\lambda \in \R$, 
fits the framework of Theorem 
\ref{thm:ito:frechet}
in which $\chi$ is kept frozen, independently of $\lambda$.}

{Similarly, we denote by $((\theta^{n,\lambda},\partial_{\zeta} \theta^{n,\lambda}))_{n \geq 1}$ 
the Picard sequence
solving 
\eqref{eq:syst:derive:picard},
with 
$X^{n,\lambda}_{t}=\xi^\lambda$ and
$\chi^\lambda = [\ud/\ud\lambda] \xi^\lambda$ for any $\lambda \in \R$. 
The sequence
converges 
in $[{\mathcal S}^2([t,T];\R^d) 
\times 
{\mathcal S}^2([t,T],\R^m) \times {\mathcal H}^2([t,T];\R^{m \times d})]^2$
towards $(\theta^{t,\xi^\lambda},\partial_{\zeta} \theta^{t,\xi^{\lambda}})$, uniformly in $\lambda$ in compact subsets.
In \eqref{eq:syst:derive:picard},
the choice 
$\chi^\lambda=[\ud/\ud\lambda] \xi^\lambda$, for any $\lambda \in \R$, 
fits the framework of Theorem 
\ref{thm:ito:frechet}
in which $\zeta = [\ud/\ud\lambda]_{\vert \lambda = 0} X^\lambda$, 
with $X^\lambda$ therein playing  the role of $\xi^\lambda$.}

{
Notice that 
$(\theta^{n,\lambda})_{n \geq 1}$, 
which appears in each of the two Picard sequences,
denotes the same process. 
The difference between 
$(\partial_{\chi} \theta^{n,\lambda})_{n \geq 1}$ 
and
$(\partial_{\zeta} \theta^{n,\lambda})_{n \geq 1}$ 
is that 
$[\ud/\ud\lambda] \theta^{n,\lambda} = 
\partial_{\zeta} \theta^{n,\lambda}$ 
but $[\ud/\ud\lambda] \theta^{n,\lambda} \not = 
\partial_{\chi} \theta^{n,\lambda}$. The motivation 
for considering 
$\partial_{\chi} \theta^{n,\lambda}$ is that 
$\partial_{\chi} Y^{n,\lambda}_{t}$
converges to ${\mathbb E}[\partial_{\mu} U(t,\xi^{\lambda},[\xi^{\lambda}]) \chi]$, 
which is precisely the quantity that we aim at differentiating with respect to 
$\lambda$. 
}

\textit{First step.}
The first point is to prove that, for any $n \geq 0$, 
the map $\lambda \mapsto 
(\partial_{\chi} \theta^{n,\lambda}_s)_{s \in [t,T]}$ is continuously differentiable from $\R$
to ${\mathcal S}^2([t,T];\R^d) \times {\mathcal S}^2([t,T];\R^m)
\times {\mathcal H}^2([t,T];\R^{m \times d})$. 
To do so, we recall the system 
\eqref{eq:syst:derive:picard}:
\begin{equation}
\label{eq:syst:derive:2}
\begin{split}
&\partial_{\chi} X^{n+1,\lambda}_{s}
= \chi + \int_{t}^s 
B^{(1)}\bigl(r,\theta^{n,\lambda}_{r},\langle \theta_{r}^{n,\lambda,(0)} \rangle\bigr)\bigl(
\partial_{\chi}\theta^{n,\lambda}_{r},
\langle \partial_{\chi}\theta^{n,\lambda,(0)}_{r} \rangle
\bigr)
 \ud r 
 \\
 &\hspace{40pt}
 + \int_{t}^s 
\Sigma^{(1)}(r,\theta^{n,\lambda,(0)}_{r},\langle \theta_{r}^{n,\lambda,(0)} \rangle)
\bigl(\partial_{\chi}\theta^{n,\lambda,(0)}_{r},\langle 
\partial_{\chi} \theta^{n,\lambda,(0)}_{r}
\rangle
\bigr)
 \ud 
W_{r}
\\
&\partial_{\chi} Y^{n+1,\lambda}_{s} = 
G^{(1)}(X_{T}^{n+1,\lambda},\langle X_{T}^{n+1,\lambda} \rangle)\bigl(\partial_{\chi} X_{T}^{n+1,\lambda},\langle 
\partial_{\chi} X_{T}^{n+1,\lambda} \rangle \bigr)
\\
&\hspace{40pt} + \int_{s}^T
F^{(1)}\bigl(r,\theta^{n,\lambda}_{r},\langle \theta_{r}^{n,\lambda,(0)} \rangle\bigr)
\bigl(\partial_{\chi} \theta^{n,\lambda}_{r},\langle \partial_{\chi} 
\theta^{n,\lambda,(0)}_{r} \rangle \bigr)
 \ud r - \int_{s}^T \partial_{\chi} \cZ_{r}^{n+1,\lambda} \ud W_{r},
\end{split}
\end{equation}
with $\partial_{\chi} \theta^{0,\lambda} \equiv (0,0,0)$ as initialization, 
with a  similar system for $\partial_{\zeta} \theta^{n,\lambda}$, 
replacing $\chi$
by $[\ud/\ud \lambda] X^\lambda$.

Generally speaking, the proof is the same as that of Lemma \ref{lem:1storder:diff}: 
We argue by induction, assuming at each step $n \geq 1$ that 
$\lambda \mapsto 
(\partial_{\chi} \theta^{n,\lambda}_s)_{s \in [t,T]}$
is continuously differentiable (the derivative being denoted by 
$(\partial_{\zeta,\chi}^2 \theta^{n,\lambda}_s)_{s \in [t,T]}$); 
we prove first the differentiability of the forward component and then the differentiability 
of the backward one in \eqref{eq:syst:derive:2}.

{In comparison with the proof 
of Lemma \ref{lem:1storder:diff}, we must pay attention to the two following points. 
The first question is to justify the differentiability under 
the various expectation symbols that appear in the definitions of
$B^{(1)}$, 
$\Sigma^{(1)}$, $F^{(1)}$ and $G^{(1)}$.
Thanks to
\eqref{eq:Picard:bound:1st:order}
and
from the fact that
the sequence 
 $[\ud/\ud \lambda]( \xi^{\lambda})$ is bounded
 in $L^\infty$
 (see \ref{eq:bound:linfini:derivative:xi}), we know that
 \begin{equation}
 \label{eq:moments:derives:1}
\sup_{n \geq 1}
\Bigl[ {\mathcal M}_{{\mathbb E}_{t}}^{4} \bigl( 
\partial_{\chi} \theta^{n,\lambda}
\bigr) \Bigr]^{1/4}
\leq C \bigl[ \vert \chi \vert + \|\chi \|_{2} \bigr],
\quad 
\sup_{n \geq 1}
\Bigl[ {\mathcal M}_{{\mathbb E}_{t}}^{2p} \bigl( 
\partial_{\zeta} \theta^{n,\lambda}
\bigr) \Bigr]^{1/4}
\leq C, 
\end{equation}
so that Lemma 
\ref{lem:lemme:final???} applies
with $[\ud/\ud \lambda]\chi^{\lambda} = 
\partial_{\zeta} \theta^{n,\lambda}$
and $\varpi = \partial_{\chi} \theta^{n,\lambda}$
and permits to guarantee the differentiability 
of the terms driven by an expectation.}

Another problem is that the coefficients now 
involve the product of two terms that are differentiable in 
$\cH^2([t,T];\R^k)$ (or $\cS^2([t,T];\R^k)$ in some cases), for a suitable $k \geq 1$, so that the product is differentiable in 
$\cH^1([t,T];\R^k)$ (or $\cS^1([t,T];\R^k)$)
only (for another $k$). We make this clear 
for $(\int_{t}^s 
B^{(1)}(r,\theta^{n,\lambda}_{r},\langle \theta_{r}^{n,\lambda,(0)} \rangle)(
\partial_{\chi}\theta^{n,\lambda}_{r},
\langle \partial_{\chi}\theta^{n,\lambda,(0)}_{r} \rangle)
 \ud r)_{t \leq s \leq T}$, the other terms being handled in a similar fashion. 
Repeating the analysis of Lemma \ref{lem:1storder:diff}, it is differentiable 
from $\R$ to ${\mathcal S}^1([t,T];\R^d)$, the derivative process writing
\begin{equation}
\label{eq:derive:2:proof} 
\begin{split}
&\int_{t}^s
 B^{(2)} \bigl(r,\Theta^{n,\lambda}_{r}\bigr)
\bigl(\partial^2_{\zeta,\chi} \theta^{n,{\lambda}}_{r},\langle
\partial^2_{\zeta,\chi} \theta^{n,{\lambda},(0)}_{r} \rangle
  \bigr)
 \ud r,
 \\
&\hspace{50pt} \text{with} \
\Theta^{n,\lambda}_{r} = \bigl( 
\theta^{n,\lambda}_{r},\langle \theta^{n,\lambda,(0)}_{r}\rangle,
 \partial_{\chi} \theta^{n,\lambda}_{r},
\partial_{\zeta} \theta^{n,\lambda}_{r},
\langle \partial_{\chi} \theta^{n,\lambda,(0)}_{r} \rangle,
\langle \partial_{\zeta} \theta^{n,\lambda,(0)}_{r} \rangle \bigr)
\end{split}
\end{equation}
with $s \in [t,T]$.
As explained 
in \eqref{coeff:H2}
and {on the model of 
Lemma 
\ref{lem:lemme:final???}}, 
we here used the crucial assumption that all the $(\xi^{\lambda})_{\lambda}$
are identically distributed to get the shape of $B^{(2)}$. 

In order to prove differentiability  in ${\mathcal S}^2([t,T];\R^d)$, a uniform integrability argument is needed. 
Assume indeed that a path $\R \ni \lambda \mapsto \vartheta^\lambda =(\vartheta_{s}^\lambda)_{s \in [t,T]} \in {\mathcal S}^1([t,T],\R^k)$, for some $k \geq 1$,
is continuously differentiable and that, for any finite interval $I$, 
the family $(\sup_{s \in [t,T]} \vert [\ud /\ud \lambda] \vartheta_{s}^\lambda \vert^2)_{\lambda \in I}$ is uniformly integrable. 
Then, $\R \ni \lambda \mapsto \vartheta^\lambda \in {\mathcal S}^2([t,T];\R^k)$ is continuously differentiable.

In our framework, the form of $([\ud/\ud \lambda]\vartheta_{s}^{\lambda})_{s \in [t,T]}$
is explicitly given by \eqref{eq:derive:2:proof}. The coefficient $B^{(2)}$
may be expanded by means of \eqref{coeff:H2}. Clearly, it involves a linear term
in $(\partial_{\zeta,\chi}^2 \theta^{n,\lambda}_{r},\langle 
\partial_{\zeta,\chi}^2 \theta^{n,\lambda,(0)}_{r} \rangle)$
and the remainder $B^{(2)}_{a}(\Theta^{n,\lambda}_{r})$. 
By \HYP{1} and \HYP{2}, we get that
\begin{equation}
\label{eq:derive:3:proof}
\begin{split}
\sup_{s \in [t,T]}
\bigl\vert
\frac{\ud}{\ud \lambda}\vartheta_{s}^{\lambda}
\bigr\vert &\leq C \biggl[
\biggl( \int_{t}^T 
\bigl[
\vert \partial_{\zeta,\chi}^2
\theta^{n,\lambda}_{r} \vert^2 
+ \|  \partial_{\zeta,\chi}^2
\theta^{n,\lambda}_{r} \|_{2}^2 \bigr] \ud r
\biggr)^{1/2}
+ \E \int_{t}^T \vert B^{(2)}_{a}(\Theta^{n,\lambda}_{r}) \vert \ud r
\biggr]. 
\end{split}
\end{equation}
By continuity of $\R \ni \lambda \mapsto \partial_{\zeta,\chi}^2 \theta^{n,\lambda} \in \cH^2([t,T];\R^k)$,
the first term in the right-hand side is  
uniformly square integrable. 
We thus discuss the term in $B^{(2)}_{a}$. 
Recalling Lemma 
\ref{lem:apriori:derivatives}
and the similar version \eqref{eq:Picard:bound:1st:order} 
for the Picard scheme in Lemma \ref{lem:1storder:diff} 
(which is given for $p=1$
only but which could be generalized), 
we have the more general version 
of \eqref{eq:moments:derives:1}:
 \color{black}
 \begin{equation}
\sup_{n \geq 1}
\Bigl[ {\mathcal M}_{{\mathbb E}_{t}}^{2p} \bigl( \partial_{\chi} \theta^{n,\lambda}
\bigr) \Bigr]^{1/2p}
\leq C_{p} \bigl[ \vert \chi \vert + \|\chi \|_{2} \bigr],
\quad 
\sup_{n \geq 1}
\Bigl[ {\mathcal M}_{{\mathbb E}_{t}}^{2p} \bigl( \partial_{\zeta} \theta^{n,\lambda}
\bigr) \Bigr]^{1/2p}
\leq C_{p}, 
\end{equation}
 so that,
by Lemma \ref{lem:bound:H2:1}
 (with $\theta \equiv \hat\theta :\equiv \theta^{n,\lambda}$, 
 $\vartheta^1 \equiv \hat\vartheta^1 :\equiv \partial_{\chi}
 \theta^{n,\lambda}$
 and $\vartheta^2 \equiv \hat\vartheta^2 :\equiv \partial_{\zeta}
 \theta^{n,\lambda}$,  and, as usual, for $T \leq c$) 
 \color{black}
\begin{equation}
\label{eq:bound:H2:1}
\begin{split}
&{\mathbb E}_{t} \biggl[ \biggl( \int_{t}^T \vert B^{(2)}_{a}(\Theta^{n,\lambda}_r) \ud r \biggr)^{2p}
\biggr]^{1/2p} \leq C_{p} 
\bigl(
| \chi| +  \| \chi\|_{2} \bigr).
\end{split}
\end{equation}
Now, choosing $p=2$, we get that, for any event $A \in \cA$,
\new{
\begin{equation*}
\begin{split}
&{\mathbb E} \biggl[ {\mathbf 1}_{A}
\biggl(
\int_{t}^T \vert B^{(2)}_{a}(\Theta^{n,\lambda}_r) \ud r 
 \biggr)^2
  \biggr]
  \\
&\hspace{15pt}\leq C {\mathbb E} \bigl[ \EFp{t}{\1_A}^{\frac12} 
\bigl( \vert \chi \vert + \| \chi \|_{2} \bigr)^2 \bigr]
=C\esp{\EFp{t}{\1_A( \vert \chi \vert + \| \chi \|_{2} )^2}^\frac12
( \vert \chi \vert + \| \chi \|_{2} )}
\\
&\hspace{15pt}\leq  C \| \chi \|_{2} 
{\mathbb E} \bigl[ {\mathbf 1}_{A}
\bigl( \vert \chi \vert + \| \chi \|_{2} \bigr)^2 \bigr]^\frac12,
\end{split}
\end{equation*}
where we have used the fact that $\chi$ is $\cF_{t}$ measurable.}
The above bound permits to establish the required uniform integrability argument,
a similar argument holding true for the terms driven by $F_{a}^{(2)}$,
$\Sigma_{a}^{(2)}$ and $G^{(2)}_{a}$.  
Inductively, this permits to prove that the map 
$\lambda \mapsto \partial_{\chi} \theta^{n,\lambda}$
is continuously differentiable
from $\R$ to ${\mathcal S}^2([t,T],\R^d) \times 
{\mathcal S}^2([t,T],\R^m) \times 
{\mathcal H}^2([t,T],\R^{m \times d})$. 
%
With the same notation as in 
\eqref{eq:derive:2:proof}, we have,  
for any $n \geq 0$,
\begin{equation}
\label{eq:sys:1:2}
\begin{split}
&\partial_{\zeta,\chi}^2 X^{n+1,\lambda}_{s}
= \int_{t}^s
 B^{(2)}\bigl(r,\Theta^{n,\lambda}_{r}\bigr)
 \bigl(
 \partial^2_{\zeta,\chi} \theta^{n,\lambda}_{r}, 
 \langle
 \partial^2_{\zeta,\chi} \theta^{n,\lambda,(0)}_{r} \rangle
 \bigr)
 \ud r
 \\
 &\hspace{60pt}
 +\int_{t}^s
 \Sigma^{(2)}\bigl(r,\Theta^{n,\lambda,(0)}_{r}\bigr)
 \bigl(
 \partial^2_{\zeta,\chi} \theta^{n,\lambda,(0)}_{r}, 
 \langle
 \partial^2_{\zeta,\chi} \theta^{n,\lambda,(0)}_{r} \rangle
 \bigr)
\ud W_{r},
\end{split}
\end{equation}
and
\begin{equation}
\label{eq:sys:2:2}
\begin{split}
&\partial_{\zeta,\chi}^2 Y^{n+1,\lambda}_{s} = 
G^{(2)}(\Xi^{n+1,\lambda}_{T})
\bigl(\partial^2_{\zeta,\chi} X_{T}^{n+1,\lambda},\langle
\partial^2_{\zeta,\chi} X_{T}^{n+1,\lambda}
\rangle\bigr) 
\\
&\hspace{60pt} 
+ \int_{s}^T
 F^{(2)}\bigl(r,\Theta^{n,\lambda}_{r}\bigr)
 \bigl(
 \partial^2_{\zeta,\chi} \theta^{n,\lambda}_{r}, 
 \langle
 \partial^2_{\zeta,\chi} \theta^{n,\lambda,(0)}_{r} \rangle
 \bigr)
 \ud r
 - \int_{s}^T \partial^2_{\zeta,\chi} Z_{r}^{n+1,\lambda} \ud W_{r},
\end{split}
\end{equation}
where we have let:
\begin{equation*}
\begin{split}
&\Theta^{n,\lambda,(0)}_{r} = \bigl( 
\theta^{n,\lambda,(0)}_{r},\langle \theta^{n,\lambda,(0)}_{r}\rangle,
 \partial_{\chi} \theta^{n,\lambda,(0)}_{r},
\partial_{\zeta} \theta^{n,\lambda,(0)}_{r},
\langle \partial_{\chi} \theta^{n,\lambda,(0)}_{r} \rangle,
\langle \partial_{\zeta} \theta^{n,\lambda,(0)}_{r} \rangle \bigr),
\\
&\Xi^{n,\lambda}_{T} = \bigl( X^{n,\lambda}_{T},\langle X^{n,\lambda}_{T}\rangle,
 \partial_{\chi} X^{n,\lambda}_{T},
\partial_{\zeta} X^{n,\lambda}_{T},
\langle \partial_{\chi} X^{n,\lambda}_{T} \rangle,
\langle \partial_{\zeta} X^{n,\lambda}_{T} \rangle \bigr).
\end{split}
\end{equation*}

\textit{Second step.}
Convergence of the sequence $(\partial^2_{\zeta,\chi} \theta^{n,\lambda})_{n \geq 0}$
in the space ${\mathcal S}^2([t,T],\R^d) \times 
{\mathcal S}^2([t,T],\R^m) \times 
{\mathcal H}^2([t,T],\R^{m \times d})$
is then shown as in the proof of Lemma \ref{lem:1storder:diff}. 
Generally speaking, the point is to compare approximations at steps 
$n$ and $n+1$ and then to prove that the norm of the difference decays 
geometrically fast as $n$ tends to $\infty$.  
As in the first step,
some precaution is needed as
the system differs 
from the one involved in the proof of Lemma \ref{lem:1storder:diff}, the difference coming from the remainder term $H_{a}^{(2)}$ in
\eqref{eq:coeff:H2}. 
Precisely, 
the proof of Lemma \ref{lem:1storder:diff} relies on 
Lemmas \ref{lem:apriori} and \ref{lem:stability}, with  
$0$ as remainder term ${\mathcal R}_{a}$, but, 
in the current framework, the remainder term is equal to 
$(H_{a}^{(2)}(\Theta^{n,\lambda}_{s}))_{s \in [t,T]}$ 
when $H=B$, $\Sigma$ or $F$
and 
$G_{a}^{(2)}(\Xi^{n+1,\lambda}_{T})$
when $H=G$, 
and is thus non-zero. 
The analysis thus imitates the proof of Lemma \ref{lem:1storder:diff}, but with 
 a non-zero remainder term $\Delta \cR_{a}$ in  
 \eqref{eq:lem:stability:3esp} that corresponds to the difference of the remainders 
 $\cR_{a}$ at steps $n$ and $n+1$. In short, it is enough to 
 prove that $\E[ \Delta \cR_{a}^2 ]$ tends to $0$ 
 as $n$ tends to $\infty$ (to simplify, we omit to specify the index $n$ in
 $\Delta \cR_{a}^2$).
 By convergence of $(\theta^{n,\lambda}_{s},\partial_{\chi} \theta^{n,\lambda}_{s},\partial_{\zeta} \theta^{n,\lambda}_{s})_{s \in [t,T]}$
to $(\theta^{t,\xi^{\lambda}}_{s},\partial_{\chi} \theta^{\xi\lambda}_{s},\partial_{\zeta}\theta^{t,\xi^{\lambda}}_{s})_{s \in [t,T]}$, 
we can deduce from Lemma \ref{lem:bound:H2:2}
(with $\theta \equiv \hat\theta :\equiv \theta^{n,\lambda}$, 
 $\vartheta^1 \equiv \hat\vartheta^1 :\equiv \partial_{\chi}
 \theta^{n,\lambda}$, $\vartheta^2 \equiv \hat\vartheta^2 :\equiv \partial_{\zeta}
 \theta^{n,\lambda}$ and
 $\theta' \equiv \hat\theta' :\equiv \theta^{n+1,\lambda}$, 
 $\vartheta^{1,\prime} \equiv \hat\vartheta^{1,\prime} :\equiv \partial_{\chi}
 \theta^{n+1,\lambda}$, $\vartheta^{2,\prime} \equiv \hat\vartheta^{2,\prime} :\equiv \partial_{\zeta}
 \theta^{n+1,\lambda}$)
that $\Delta \cR_{a}^2$
tends to $0$ in probability as $n$ tends to $\infty$, the convergence being uniform with respect to 
$\lambda$ in compact subsets: In 
\eqref{eq:sncf:retour:5:b}, we can 
check that all the terms not containing the variable 
$\varepsilon$ tend $0$; 
choosing $\varepsilon$ as a small deterministic real, 
it is standard to prove that the expectation of the last term in \eqref{eq:sncf:retour:5:b} tends 
to $0$. The latter property follows from the following fact: 
For any compact $I \subset \R$, 
the sequence 
$(\partial_{\chi} \theta^{n,\lambda})_{n \geq 1}$
and $(\partial_{\zeta} \theta^{n,\lambda})_{n \geq 1}$
are convergent in the $L^2$ sense on $\Omega \times [t,T]$, 
so that 
the families $(\partial_{\chi} \theta^{n,\lambda})_{n \geq 1,\lambda \in I}$
and $(\partial_{\zeta} \theta^{n,\lambda})_{n \geq 1,\lambda \in I}$
are uniformly square integrable on $\Omega \times [t,T]$. 

The convergence 
of $\Delta \cR_{a}^2$ to $0$
actually holds in the $L^1$ sense on $\Omega$, since the bound 
\eqref{eq:bound:H2:1} 
(with similar bounds for $F$, $\Sigma$ and $G$)
allows to apply another argument of uniform integrability. The convergence is uniform with respect 
to $\lambda$ in compact sets.
This proves the continuous differentiability of 
$\R \ni \lambda \mapsto \partial_{\chi} \theta^{t,\xi^\lambda}
\in \cS^2([t,T];\R^d) \times \cS^2([t,T];\R^m)
\times \cH^2([t,T];\R^{m\times d})$. The derivative at $\lambda=0$
satisfies a system of the form \eqref{eq:sys:linear:mkv}
(obtained by an obvious adaptation of \eqref{eq:sys:1:2}
and \eqref{eq:sys:2:2}), which is uniquely solvable in short time. 
This proves that the derivative at $\lambda=0$ only depends on the family $(X^\lambda)_{\lambda \in \R}$
through $X^0$ and $\zeta$.  
\vspace{5pt}

We complete the analysis as in the proof 
of Lemma \ref{lem:1storder:diff}. 
 \qed
\end{proof}

\subsubsection{Estimates of the directional derivatives of the McKean-Vlasov 
system}

We claim:
\begin{Lemma}
\label{lem:apriori:bound:2nd:derivatives:t,xi}
Recall the notations 
\eqref{stab:notations:1}.
For any $p \geq 1$, 
there exist two constants
$c:=c_{p}(L)>0$ and $C_{p}$, such that, for $T \leq c_{p}$
(and with $\gamma=c_{p}$ in 
\eqref{stab:notations:1}), 
\begin{equation*}
\begin{split}
&\bigl[ {\mathcal M}_{{\mathbb E}_{t}}^{2p}
\bigl(  \dzeki \theta^{t,\xi} \bigr) \bigr]^{1/(2p)} 
 \leq C_{p} 
\bigl( \vert \chi \vert +
\| \chi\|_{2} \bigr).
\end{split}
\end{equation*}
\end{Lemma}
%
%
%
\proof
The result follows from 
Corollary \ref{co:bound:sys:lin:mkv:easy}
with $\eta =0$, $\theta \equiv \hat\theta := \theta^{t,\xi}$, $\vartheta \equiv \hat\vartheta := \partial^2_{\zeta,\chi} \theta^{t,\xi}$, 
$H$ given by \eqref{eq:coeff:H2}, 
and, in particular, with remainders ${\mathcal R}_{a}^{2p}$
and ${\mathcal R}_{a}^{2}$ coming from $H^{(2)}_{a}$ in \eqref{eq:coeff:H2}.
Recalling Lemma \ref{lem:apriori:derivatives}
and the assumption $\|\zeta\|_\infty \le 1$, the remainders may be estimated by means of 
Lemma \ref{lem:bound:H2:1}, 
with $\theta \equiv \hat{\theta} := \theta^{t,\xi}$, 
$\vartheta^1 \equiv \hat{\vartheta}^1 := \partial_{\chi} \theta^{t,\xi}$
and 
$\vartheta^2 \equiv \hat{\vartheta}^2 := \partial_{\zeta} \theta^{t,\xi}$. 
 \eproof
\vspace{5pt}

We now discuss the continuity with respect to $\xi$. 
We claim:

\begin{Lemma}
\label{lem:apriori:cont:2nd:derivatives}
For any $p \geq 1$, 
there exist two constants
$c_{p}:=c_{p}(L)>0$ and $C_{p}$
 such that, for $T \leq 
c_{p}$ (and with $\gamma=c_{p}$ in 
\eqref{stab:notations:1}), 
\begin{equation}
\label{eq:phi2}
\begin{split}
&\Bigl[ {\mathcal M}^{2p}_{{\mathbb E}_{t}} \bigl(
  \dzeki \theta^{t,\xi} 
- \dzeki \theta^{t,\xi'} \bigr)
\Bigr]^{1/2p}
\leq C_{p} \Bigl( 
1 \wedge \bigl\{ \vert \xi - \xi' \vert + 
\Phi_{\alpha+1}(t,\xi,\xi') \bigr\}
\Bigr)
\bigl( \vert \chi \vert +
\| \chi\|_{2} \bigr),
\end{split}
\end{equation}
where $\Phi_{\alpha+1}(t,\cdot) : [ L^2(\Omega,\cF_{t},\P;
\R^{d})]^2 \rightarrow \R_{+}$ 
is continuous at any point of the diagonal, does not depend on $p$ and
satisfies \eqref{assumption:mkv:4} with $\alpha$ replaced by 
$\alpha+1$. 
The restriction of $\Phi_{\alpha+1}(t,\cdot,\cdot)$ to $[L^2(\Omega,\cF_{0},\P;\R^d)]^2$ may be assumed to be independent of $t \in [0,T]$. 
\end{Lemma}

\proof
Generally speaking, the strategy is to apply 
Corollary \ref{cor:stab:hard},
with $\eta =0$, $\theta^\xi \equiv \hat\theta^\xi :\equiv \theta^{t,\xi}$, $\vartheta \equiv \hat\vartheta := \partial^2_{\zeta,\chi} \theta^{t,\xi}$ (and the same for $\xi'$),
$H$ given by \eqref{eq:coeff:H2}
and, in particular, with
 remainders ${\mathcal R}_{a}^{2p}$
and ${\mathcal R}_{a}^{2}$ coming from $H^{(2)}_{a}$ in \eqref{eq:coeff:H2}
(and the same for the remainders labelled with `prime').
As in the proof of the previous 
Lemma
\ref{lem:apriori:bound:2nd:derivatives:t,xi}, we 
can bound the remainders $(\cR_{a}^{2p} )^{1/2p}$
by $C_{p}(\vert \chi \vert + \| \chi \|_{2})$. 

In order to estimate
$(\Delta \cR_{a}^{2p} )^{1/2p}$, we apply 
Lemma \ref{lem:bound:H2:2}.
A crucial fact is that we have 
\eqref{eq sup Z xi}. This says that, instead of working 
in conditional norm $[\bar{\cM}^{2p}\llbracket \cdot \rrbracket]^{1/2p}$ for estimating the distance 
between $\theta^{t,\xi}$ and $\theta^{t,\xi'}$, 
we can directly work with the conditional norm $\NSt{2p}{\cdot} + \NS{2}{\cdot}$. As a byproduct, we can choose $\varepsilon = \sup_{s \in [t,T]} \vert \theta_{s}^{t,\xi}
- \theta_{s}^{t,\xi'} \vert$ in \eqref{eq:sncf:retour:5:b}. 
By \eqref{eq sup Z xi}, we thus get $C_{p} ( 
1 \wedge \{ \vert \xi - \xi' \vert + 
\Phi_{\alpha+1}(t,\xi,\xi') \})
( \vert \chi \vert +
\| \chi\|_{2})$
as a bound 
for the terms containing the symbol $\varepsilon$ in 
 \eqref{eq:sncf:retour:5:b} 
 ($\Phi_{\alpha+1}$ being independent of $t$ when 
$\xi$ and $\xi'$ are $\cF_{0}$-measurable).
By Lemmas 
\ref{lem:apriori:derivatives}
and \ref{lem:cont:derivatives}, 
all the terms involving an $\bar{\cM}$ may be bounded in the same way. 
By \eqref{eq:tilde:phi:alpha} in Lemma \ref{lem:function:phi:alpha}, the same is true for the term involving 
$\Phi_{\alpha}$. 
By Cauchy-Schwarz inequality and once again 
by Lemmas 
\ref{lem:apriori:derivatives}
and \ref{lem:cont:derivatives}, the
same bound holds for the
 terms integrated under $\E$. 
In the end, the whole right-hand side 
in \eqref{eq:sncf:retour:5:b} may be bounded by  $C_{p} ( 
1 \wedge \{ \vert \xi - \xi' \vert + 
\Phi_{\alpha+1}(t,\xi,\xi') \})
( \vert \chi \vert +
\| \chi\|_{2})$ (without the $t$ when 
$\xi,\xi' \in L^2(\Omega,\cF_{0},\P;\R^d)$). 
In \eqref{eq:stab:main estimate}, this brings us to the case when 
the remainders are zero, but $\Phi_{\alpha}$ is replaced by $\Phi_{\alpha+1}$. 
{
Applying \eqref{eq:bar phi} in Example \ref{example:UI}, we complete the proof
of \eqref{eq:phi2}.
The last part of the statement (choice of a version of $\Phi_{\alpha+1}$
which is independent of $t$) follows 
from 
Remark \ref{rem:unif:phi:alpha}.}
\eproof
\vspace{5pt}

\subsubsection{Study of the Non McKean-Vlasov system}
We now repeat the same analysis but for the process
$(\theta^{t,x,[\xi]},\partial_{\chi} \theta^{t,x,[\xi]})$
(instead of $(\theta^{t,\xi},\partial_{\chi} \theta^{t,\xi})$).
Considering a continuously differentiable 
path $\lambda \mapsto \xi^{\lambda}$ from 
$\R$ into $L^2(\Omega,{\mathcal F}_{t},\P;\R^d)$
such that $\vert [\ud/\ud \lambda] \xi^{\lambda} \vert \leq 1$, we are first to 
prove that the mapping 
$\R \ni \lambda \mapsto \partial_{\chi} \theta^{t,x,[\xi^{\lambda}]}
\in {\mathcal S}^2([t,T];\R^d) \times {\mathcal S}^2([t,T];\R^m)
\times {\mathcal H}^{2}([t,T];\R^{m \times d})$
is continuously differentiable. 
Before we discuss the proof, we must say a word about the notation 
itself, which is slightly ambiguous. Since the law of 
$\xi^\lambda$ is independent of $\lambda$, we could be indeed tempted to say that 
$\partial_{\chi} \theta^{t,x,[\xi^{\lambda}]}$ is independent of $\lambda$, which is obviously false. 
The reason is that, in the coefficients driving the system 
satisfied by $\partial_{\chi} \theta^{t,x,[\xi^{\lambda}]}$, there are terms of 
the form $\hat{\mathbb E}[\partial_{\mu} H(\theta^{t,x,[\xi^{\lambda}]},
[\theta^{t,\xi^{\lambda},(0)}]
)(\langle \theta^{t,\xi^\lambda,(0)} \rangle) \langle 
\partial_{\chi} \theta^{t,\xi^{\lambda},(0)} \rangle ]$, see 
\eqref{eq:x,xi:H1}, which explicitly depend upon the joint law 
of 
$\chi$ and $\xi^{\lambda}$. Clearly, there is no reason for the 
joint law to be independent of $\lambda$.  

Recalling \eqref{eq:x,xi:H1}, we know that 
$\partial_{\chi} \theta^{t,x,[\xi^{\lambda}]}$ satisfies a standard linear FBSDE
with $\hat{\mathbb E}[\partial_{\mu} H(\theta^{t,x,[\xi^{\lambda}]},
[\theta^{t,\xi^{\lambda},(0)}]
)(\langle \theta^{t,\xi^\lambda,(0)} \rangle) \langle 
\partial_{\chi} \theta^{t,\xi^{\lambda},(0)} \rangle ]$
as affine part. The coefficients of the FBSDE 
read as coefficients parametrized by $\lambda$ through 
the values of
$(
\theta^{t,x,[\xi^\lambda]},
\theta^{t,\xi^\lambda},
\partial_{\chi} \theta^{t,\xi^{\lambda}})$. Now that the 
continuous differentiability of $\R \ni \lambda \mapsto 
(
\theta^{t,x,[\xi^\lambda]},
\theta^{t,\xi^\lambda},
\partial_{\chi} \theta^{t,\xi^{\lambda}})$ has been proved, 
we can repeat the arguments used in the proof of Lemma \ref{lem:2ndorder:diff}
to show that 
$\R \ni \lambda \mapsto \partial_{\chi} \theta^{t,x,[\xi^\lambda]}
\in \cS^2([t,T];\R^d) \times \cS^2([t,T];\R^m) \times \cH^2([t,T];\R^{m \times d})$
is also continuously differentiable. (The complete proof is left to the reader.)

With the notation $\zeta := [\ud/\ud \lambda]_{\vert \lambda=0} \xi^{\lambda}$, we denote
the second-order tangent process by
$\dzeki \theta^{t,x,\law{\xi}} := [\ud/\ud \lambda]_{\vert \lambda =0}
\partial_\chi \theta^{t,x,\law{\xi^\lambda}}$. It satisfies a system of the form
\eqref{eq:sys:linear:mkv}
with $\theta \equiv \theta^{t,x,\law{\xi}}$, $\hat \theta \equiv \theta^{t,\xi}$, 
$\vartheta \equiv \dzeki \theta^{t,x,\law{\xi}}$ and $\hat \vartheta 
\equiv \dzeki \theta^{t,\xi}$ and  with generic coefficients $H$ given by
(compare if needed with \eqref{eq:x,xi:H1}):
 \begin{equation*}
h_{\ell}(V,\langle \hat V^{(0)} \rangle) = \partial_{x} 
h(V, \law{\hat V^{(0)}}), 
\ \hat{H}_{\ell}(V,\langle \hat V^{(0)}\rangle)=
\partial_{\mu} h(V,\law{\hat V^{(0)}})(\langle 
\hat V^{(0)} \rangle),
\ H_{a} \equiv
\tilde{H}_a^{(2)},
\end{equation*}
where $\tilde{H}^{(2)}_{a}(r)$ is a variant 
of $H^{(2)}_a$ in \eqref{eq:coeff:H2}
and reads:
\begin{equation}
\label{eq:t,x,mu:H:second:1}
\begin{split}
\tilde{H}_a^{(2)}(r) &:=
H_{a}^{(2)}\bigl(\theta^{t,x,[\xi]}_{r},\langle \theta^{t,\xi,(0)}_{r} \rangle,
\partial_{\chi} \theta^{t,x,[\xi]}_{r},\partial_{\zeta} \theta^{t,x,[\xi]}_{r},
\langle \partial_{\chi} \theta^{t,\xi,(0)}_{r} \rangle, \langle \partial_{\zeta} \theta^{t,\xi,(0)}_{r}
\rangle \bigr).
\end{split}
\end{equation}

On the model of Lemmas 
\ref{lem:apriori:bound:2nd:derivatives:t,xi}
and \ref{lem:apriori:cont:2nd:derivatives}, we claim
(compare with Lemma \ref{lem:cont:derivatives:x,xi}):
\begin{Lemma}
\label{lem:second:t,x,xi}
For any $p \geq 1$, 
there exist two constants
$c_{p}:=c_{p}(L)>0$ and $C_{p}$ 
such 
that,
for $T \leq c_{p}$ and with $\gamma=c_{p}$ in \eqref{stab:notations:1},
\begin{equation*}
\bigl[ {\mathcal M}_{{\mathbb E}}^{2p}
 \bigl( \dzeki \theta^{t,x,\law{\xi}}\bigr) \bigr]^{1/2p}
\leq C_{p} \NL{2}{\chi},
\end{equation*}
and
\begin{equation*}
\label{eq:second:deriv:cont:t,x,[xi]}
\begin{split}
&\Bigl[ {\mathcal M}_{{\mathbb E}}^{2p}
 \bigl( \dzeki \theta^{t,x,\law{\xi}}
 - \dzeki \theta^{t,x',\law{\xi'}}
\bigr)  \Bigr]^{1/2p}
\le 
C_{p} \bigl( \vert x-x' \vert + \Phi_{\alpha+1}(t,\xi,\xi') \bigr) \| \chi \|_{2},
\end{split}
\end{equation*} 
where $\Phi_{\alpha+1}(t,\cdot) : [ L^2(\Omega,\cF_{t},\P;
\R^{d})]^2 \rightarrow \R_{+}$ 
is continuous at any point of the diagonal, does not depend on $p$ and
satisfies \eqref{assumption:mkv:4} with $\alpha$ replaced by 
$\alpha+1$. 
The restriction of $\Phi_{\alpha+1}(t,\cdot,\cdot)$ to $[L^2(\Omega,\cF_{0},\P;\R^d)]^2$ may be assumed to be independent of $t \in [0,T]$. 
\end{Lemma}

\proof
Loosely speaking, 
the result 
is similar to Lemmas
\ref{lem:apriori:bound:2nd:derivatives:t,xi}
and \ref{lem:apriori:cont:2nd:derivatives}, but with 
the realizations of $\xi$ and $\xi'$ therein replaced by $x$ and $x'$. 
Actually,   
the main difference with the computations made for the McKean-Vlasov system 
comes from the shape of the remainder $\cR_{a}$
that is implemented in 
the stability Corollary 
\ref{cor:stab:hard}.
In the proofs of Lemmas
\ref{lem:apriori:bound:2nd:derivatives:t,xi}
and
 \ref{lem:apriori:cont:2nd:derivatives}, 
 the definition of the remainder $\cR_{a}$ is based on 
the formula
\eqref{eq:coeff:H2}. In the current framework, it is based on the formula
\eqref{eq:t,x,mu:H:second:1},
which is slightly different. 
It can be estimated by means of Lemma \ref{lem:bound:H2:2}. 
The proof is then completed as that one of Lemma 
\ref{lem:cont:derivatives:x,xi}.
\eproof

\subsubsection{Final statement}
We finally claim:
\color{black}
\begin{Theorem}
\label{prop:partial:C^2:u}
There exists a constant $c:=c(L) >0$ such that, for $T \leq 
c$:

$\bullet$ for any $t \in [0,T]$ and $\mu \in {\mathcal P}_{2}(\R^d)$, 
the function $\R^d \ni x \mapsto U(t,x,\mu)$ is ${\mathcal C}^2$ 
and the functions 
$[0,T] \times \R^d \times \cP_{2}(\R^d) \ni (t,x,\mu)
\mapsto U(t,x,\mu)$, 
$[0,T] \times \R^d \times \cP_{2}(\R^d) \ni (t,x,\mu)
\mapsto \partial_{x}
U(t,x,\mu)$ 
and
$[0,T] \times \R^d \times \cP_{2}(\R^d) \ni (t,x,\mu)
\mapsto \partial^2_{xx}
U(t,x,\mu)$
are continuous, 

$\bullet$ for any $(t,x) \in [0,T] \times \R^d$, the function 
${\mathcal P}^2(\R) \ni \mu \mapsto U(t,x,\mu)$
is partially ${\mathcal C}^2$;
for any $(t,\mu) \in [0,T] \times \R^d 
\times \cP_{2}(\R^d)$, there exists a version
of $\R^d \ni v \mapsto 
\partial_{\mu} U(t,x,\mu)(v) \in \R^d$
such that 
$\R^d \times \R^d \ni (x,v) \mapsto 
\partial_{\mu} U(t,x,\mu)(v) \in \R^d$
is differentiable at any $(x,v)$ such that 
$v \in \textrm{\rm Supp}(\mu)$, 
the partial derivative
$\R^d \times \R^d \ni (x,v)
\mapsto 
\partial_{v} [\partial_{\mu} U(t,x,\mu)](v)$
being continuous at any $(w,v)$
such that $v \in \textrm{\rm Supp}(\mu)$
and 
the partial derivative
$\R^d \times \textrm{\rm Supp}(\mu) \ni (x,v)
\mapsto 
\partial_{x} [\partial_{\mu} U(t,x,,\mu)](v)$
being continuous in $(x,v)$.

 Moreover,
we can find a constant $C$ such that, for 
all $x \in \R^d$, for all $\xi \in L^2(\Omega,\cA,\P;\R^d)$,   
\begin{equation*}
\bigl\vert \partial_{xx}^2 U (t,x,[\xi])
\bigr\vert
+
{\mathbb E} \bigl[ \bigl\vert \partial_{x} \bigl[ \partial_{\mu} U (t,x,[\xi]) \bigr](\xi)
\bigr\vert^2 \bigr]^{1/2} 
+
{\mathbb E} \bigl[ \bigl\vert \partial_{v} \bigl[ \partial_{\mu} U (t,x,[\xi]) \bigr] (\xi)
\bigr\vert^2 \bigr]^{1/2} \leq C,
\end{equation*}
and, 
for 
all $x,x' \in \R^d$, for all $\xi,\xi' \in L^2(\Omega,\cA,\P;\R^d)$,  
\begin{equation*}
\begin{split}
&\bigl\vert \partial_{xx}^2 U (t,x,[\xi]) - \partial_{xx}^2 
U (t,x',[\xi']) \bigr\vert
\\
&\hspace{15pt}  
 + {\mathbb E} \bigl[ \bigl\vert \partial_{x} \bigl[ \partial_{\mu} U (t,x,[\xi]) \bigr](\xi)
- \partial_{x} \bigl[ \partial_{\mu} U (t,x',[\xi']) \bigr](\xi')
\bigr\vert^2 \bigr]^{1/2} 
\\
&\hspace{15pt} 
+ {\mathbb E} \bigl[ \bigl\vert \partial_{v} \bigl[\partial_{\mu} U (t,x,[\xi]) \bigr](\xi)
- \partial_{v} \bigl[ \partial_{\mu} U (t,x',[\xi']) \bigr] (\xi')
\bigr\vert^2 \bigr]^{1/2} 
\\
&\leq C \bigl\{ \vert x-x' \vert + \Phi_{\alpha+1}(\xi,\xi') \bigr\}, 
\end{split}
\end{equation*}
where
 $\Phi_{\alpha+1} : [ L^2(\Omega,\cA,\P;
\R^{d})]^2 \rightarrow \R_{+}$ 
satisfies \eqref{assumption:mkv:4}, with $\alpha$ replaced by 
$\alpha+1$.
In particular, for any $x \in \R^d$ and $\mu \in {\mathcal P}_{2}(\R^d)$, 
we can find a locally Lipschitz continuous version of the 
mappings $\R^d \ni v \mapsto \partial_{x} [\partial_{\mu} U(t,x,\mu)](v)$
and $\R^d \ni v \mapsto \partial_{v} [\partial_{\mu} U(t,x,\mu)](v)$.

The functions
$[0,T] \times \R^d \times L^2(\Omega,\cA,\P;\R^d) 
\ni (t,x,\xi)
\mapsto 
\partial_{xx}^2 U(t,x,[\xi])
\in \R^d$,
$[0,T] \times \R^d \times L^2(\Omega,\cA,\P;\R^d) 
\ni (t,x,\xi)
\mapsto 
\partial_{x} [ \partial_{\mu} U(t,x,[\xi])](\xi)\in L^2(\Omega,\cA,\P;\R^d)$
and 
$[0,T] \times \R^d \times L^2(\Omega,\cA,\P;\R^d) 
\ni (t,x,\xi)
\mapsto 
\partial_{v} [ \partial_{\mu} U(t,x,[\xi])](\xi)\in L^2(\Omega,\cA,\P;\R^d)$
are continuous. 
\end{Theorem}
\color{black}

\proof
We first apply Theorem \ref{thm:ito:frechet}
in order to prove the $\cC^2$-partial property of $\mu \mapsto U(t,x,\mu)$. 
By 
Theorem
\ref{thm:4:1}, we already know that the lifted version $L^2(\Omega,\cA,\P;\R^d) \ni 
\xi \mapsto \cU(t,x,\xi) = U(t,x,[\xi])$ is continuously differentiable in the sense of Fr\'echet. 
Recalling the identity 
$$\partial_{\chi} Y_{t}^{t,x,[\xi]} = \E\bigl[ D \cU(t,x,[\xi])(\xi) \chi
\bigr],$$ we deduce from Lemmas \ref{lem:apriori:derivatives}
and \ref{lem:cont:derivatives}
that the gradient $D \cU(t,x,\cdot)$ satisfies \textit{(i)} and \textit{(ii)}
in the statement of Theorem \ref{thm:ito:frechet}. 
Now, using the same sequence $(\xi^\lambda)_{\lambda \in \R}$ as 
in \S \ref{subsubse:overview}, we notice that 
$$\frac{\ud}{ \ud \lambda}_{\vert \lambda = 0} 
\E\bigl[ D \cU(t,x,\xi^\lambda) \chi \bigr] = \partial^2_{\zeta,\chi} Y_{t}^{t,x,[\xi]},$$ 
which satisfies \textit{(i)} and \textit{(ii)} in the statement of 
Theorem \ref{thm:ito:frechet} thanks to Lemmas 
\ref{lem:apriori:bound:2nd:derivatives:t,xi} and 
\ref{lem:apriori:cont:2nd:derivatives} (with $\xi^\lambda$ playing the role of $X^\lambda$ in 
the statement of Theorem \ref{thm:ito:frechet}).
We deduce that, for any $(t,x) \in [0,T] \times \R^d$, the 
map $\cP_{2}(\R^d) \ni \mu \mapsto U(t,x,\mu)$ 
is partially $\cC^{2}$. 
{In particular, 
for any $(t,x,\mu) \in [0,T] \times \R^d \times 
\cP_{2}(\R^d)$, we can find a version 
of 
$\R^d \ni v 
\mapsto \partial_{\mu}U(t,x,\mu)(v)$ that
is continuously differentiable, 
such a version being uniquely 
defined on the support of $\mu$.
}
Moreover, by 
\eqref{eq:29:11:50}, we have the relationship
\begin{equation*}
\partial^2_{\text{sign}(Z')e,\text{sign}(Z')\chi} 
Y_{t}^{t,x,[\xi]}
 = 
 {\mathbb E} \bigl[ \text{Tr} \bigl\{ \bigl(
  \partial_{v} [\partial_{\mu} U(t,x,\mu)](\xi) \bigr) \bigl( \chi \otimes e \bigr) \bigr\}\bigr],
\end{equation*}
which holds true for any $e \in \R^d$ and any $\xi,\chi \in L^2(\Omega,\cF_{t},\P;\R^d)$, 
with $\xi \sim \mu$,
and for a prescribed random variable $Z'$ independent of $(\xi,\chi)$. 
From Lemma 
\ref{lem:second:t,x,xi}, we deduce that  
\begin{equation*}
\begin{split}
&{\mathbb E}[ \vert \partial_{v} [\partial_{\mu} U (t,x,[\xi])](\xi)
\vert^2]^{1/2} \leq C,
\\
&{\mathbb E}[ \vert \partial_{v} [\partial_{\mu} U (t,x,[\xi])](\xi)
- \partial_{v}[ \partial_{\mu} U (t,x',[\xi'])] (\xi')\vert^2 ]^{1/2}
\leq C [ \vert x-x' \vert + \Phi_{\alpha+1}(\xi,\xi')],  
\end{split}
\end{equation*}
the extension of $\Phi_{\alpha+1}$
to the whole $[L^2(\Omega,\cA,\P;\R^d)]^2$ being achieved as in the proof 
of Lemma \ref{lem:cont:derivatives:x,xi:b}. 

By means of Proposition \ref{prop:lipschitz:lifted}, we deduce that, 
{
for given $t \in [0,T]$ and $\mu \in \cP_{2}(\R^d)$, 
we can choose, for any $x \in \R^d$, a version of 
$\R^d \ni v \mapsto \partial_{\mu} U(t,x,\mu)(v)$
such that 
the derivative mapping $\R^d \ni v \mapsto \partial_{v}[\partial_{\mu} U(t,x,\mu)](v)$ 
is continuous
on compact subsets of $\R^d$, uniformly in 
$x \in \R^d$}. 
Using the same trick as in \eqref{eq:trick:markov}, we deduce that the family
 $(\R^d \ni v \mapsto \partial_{v}[\partial_{\mu} U(t,x,\mu)](v))_{x \in \R^d}$
 is relatively compact for the topology of uniform convergence on compact subsets. Considering a 
 sequence $(x_{n})_{n \geq 1}$ that converges to $x$, 
we already know that 
the sequence of functions $(\R^d \ni v \mapsto \partial_{v}[\partial_{\mu} U(t,x_{n},\mu)](v) \in \R^{d \times d})_{n \geq 1}$
converges in $L^2(\R^d,\mu;\R^{d \times d})$ to 
$\R^d \ni v \mapsto \partial_{v}[\partial_{\mu} U(t,x,\mu)](v) \in \R^{d \times d}$. 
{Since $\partial_{v} [\partial_{\mu} U(t,x,\mu)]$ is uniquely defined on 
the support of $\mu$,
the limit of any converging subsequence (for the topology of uniform convergence 
on compact subsets of $\R^d$) of
$(\partial_{v}[\partial_{\mu} U(t,x_{n},\mu)](\cdot))_{n \geq 1}$ 
coincides with $\partial_{v}[\partial_{\mu} U(t,x,\mu)](\cdot)$
on the support of $\mu$. We deduce that the function $\R^d \times \R^d \ni 
(x,v) \mapsto \partial_{v} [\partial_{\mu} U(t,x,\mu)](v) \in \R^{d \times d}$ 
is continuous at any $(x,v)$ such that $v \in \textrm{Supp}(\mu)$}.
 
Proving a similar version of Lemma \ref{lem:second:t,x,xi}, but 
for $\partial^2_{xx} \theta^{t,x,[\xi]}$, we 
can show in the same way that $U$ is twice differentiable in $x$ and satisfies 
\begin{equation*}
\bigl\vert \partial_{xx}^2 U (t,x,[\xi]) \bigr\vert, \quad 
\bigl\vert \partial_{xx}^2 U (t,x,[\xi]) - \partial_{xx}^2 
U (t,x',[\xi']) \bigr\vert
\leq C \bigl[ \vert x-x' \vert + \Phi_{\alpha+1}(\xi,\xi') \bigr],
\end{equation*}
We notice indeed 
that, for $\xi \sim \mu$, $\partial^2_{xx} Y_{t}^{t,x,[\xi]}$
coincides with $\partial^2_{xx} U(t,x,\mu)$. 

Similarly, we can investigate $\partial_{x} [\partial_{\chi} \theta^{t,x,[\xi]}]$. 
By means of Lemma \ref{lem:appendix:2} in Appendix, we can prove that, 
{once 
a continuous version of $\partial_{\mu} U(t,x,\mu)$
has be chosen
for any $(t,x,\mu) \in [0,T] \times \R^d \times \cP_{2}(\R^d)$,
the
function $\R^d \ni x \mapsto \partial_{\mu} U(t,x,\mu)(v)$
is differentiable at any point $(x,v)$ such that 
$v \in \textrm{Supp}(\mu)$, the derivative
function $\R^d \times \textrm{Supp}(\mu) \ni (x,v) 
\mapsto \partial_{x}[\partial_{\mu} U(t,x,\mu)](v)$
being continuous. Combining with the continuous differentiability property in $v$, 
we deduce that
the mapping 
$\R^d \times \R^d\ni (x,v) \mapsto \partial_{\mu} U(t,x,\mu)(v)$
is differentiable at any point $(x,v)$ such that $v \in \textrm{Supp}(\mu)$, 
with the aforementioned prescribed continuity properties of the 
partial derivatives.}

Then $\partial_{x} [\partial_{\chi} Y_{t}^{t,x,[\xi]}]$ identifies with 
$\E[ \partial_{x} [\partial_{\mu} U(t,x,\mu)](\xi) \chi]$.  Moreover,
\begin{equation*} 
\begin{split}
&{\mathbb E} \bigl[ \vert \partial_{x} [ \partial_{\mu} U (t,x,[\xi]) ](\xi)
\vert^2 \bigr]^{1/2} \leq C,
\\
&{\mathbb E}\bigl[ \vert \partial_{x} [ \partial_{\mu} U (t,x,[\xi])](\xi)
- \partial_{x} [ \partial_{\mu} U (t,x',[\xi']) ](\xi')
\vert^2 \bigr]^{1/2} 
\leq C \bigl[ \vert x-x' \vert + \Phi_{\alpha+1}(\xi,\xi')\bigr].
\end{split}
\end{equation*}

{
Generally speaking, 
time continuity of the derivatives can be proved as in Theorem 
\ref{thm:4:1}.  
Anyhow, some precaution is needed since the drivers of the backward equations 
that represent all the second-order derivatives involve 
quadratic terms in 
$\partial_{\chi} Z^{t,\xi}$
and
$\partial_{\chi} Z^{t,x,[\xi]}$,
see for instance 
\eqref{eq:coeff:H2}.
The \textit{a priori} difficulty is that, so far, 
we have exhibited
bounds for
$\partial_{\chi} Z^{t,\xi}$
and
$\partial_{\chi} Z^{t,x,[\xi]}$
in 
${\mathcal H}$ norm only, which 
might not suffice for investigating the time regularity. 
The key point is then to notice that all these terms may be estimated 
in ${\mathcal S}$ instead of ${\mathcal H}$ norm. 
The trick is to invoke the representation formula 
\eqref{eq representation formula Z t,x,xi}
for the process $Z^{t,x,[\xi]}$, 
to differentiate it and then to make use of the bounds 
we just proved for $\partial_{xx}^2 U$ and  
$\partial_{x} [\partial_{\mu} U]$. 
}
\eproof 
\vspace{5pt}

We now turn to
\vspace{5pt}

\begin{proof}[Proof of Theorem \ref{main:thm:short:time}]
We first prove that $U$ is  a classical solution of the PDE
\eqref{eq:master:PDE}.
The main argument follows from 
\eqref{eq:expansion:master:PDE}, 
the idea being to apply the chain rule
to $U(t+h,x,\cdot)$, which is licit thanks to Theorem 
\ref{prop:partial:C^2:u}.
Following \eqref{eq:chain:rule:PDE}, we get
\begin{align}
&U\bigl(t+h,x,\law{X_{t+h}^{t,\xi}}\bigr) - U\bigl(t+h,x,\law{\xi}\bigr) \nonumber
\\
&= \int_{t}^{t+h} \hesp{\partial_\mu U\bigl(t+h,x,\law{X^{t,\xi}_r}\bigr)
\bigl(
\cc{X^{t,\xi}_r}\bigr)  
b\bigl(\cc{\theta^{t,\xi}_r},\law{\theta^{t,\xi,(0)}_r}\bigr)} \ud r \nonumber
\\
 &\hspace{15pt} + \frac12 \int_t^{t+h} \hesp{
 \text{Trace} \bigl[
 \partial_v \bigl[\partial_\mu U\bigr]\bigl(t+h,x,\law{X^{t,\xi}_r}\bigr)\bigl(\cc{X^{t,\xi}_r}\bigr)
 \bigl( \sigma
  \sigma^{\dagger}\bigr)\bigl(  
\cc{\theta^{t,\xi,(0)}_r},\law{\theta^{t,\xi,(0)}_r}\bigr)
\bigr]
} \ud r. \nonumber
 \end{align}
Assumption \HYP{0}
and Theorems 
\ref{thm:4:1}
and \ref{prop:partial:C^2:u}
provide estimates on the smoothness of $b$, $\sigma \sigma^{\dagger}$, 
$\partial_{\mu} U$ and $\partial_{v}[\partial_{\mu} U]$. 
We deduce that we can find a non-negative functional $\Phi$ on 
$[L^2(\Omega,\cA,\P;\R^{d} \times \R^m \times \R^{m \times d})]^2$, 
continuous at any point of the diagonal, 
matching $0$ on the diagonal, such that 
\begin{align}
&\biggl\vert U\bigl(t+h,x,\law{X_{t+h}^{t,\xi}}\bigr) - U\bigl(t+h,x,\law{\xi}\bigr) \nonumber
\\
&\hspace{15pt}- h \hesp{\partial_\mu U\bigl(t+h,x,\law{\xi}\bigr)
\bigl(
\cc{\xi}\bigr)  
b\bigl(\cc{\theta^{t,\xi}_t},\law{\theta^{t,\xi,(0)}_t}\bigr)}  \nonumber 
 \\
 &\hspace{15pt} - \frac{h}2 \hesp{
 \text{Trace} \bigl[
 \partial_v \bigl[\partial_\mu U\bigr]\bigl(t+h,x,\law{\xi}\bigr)\bigl(\cc{\xi}\bigr)
 \bigl( \sigma
  \sigma^{\dagger}\bigr)\bigl(  
\cc{\theta^{t,\xi,(0)}_t},\law{\theta^{t,\xi,(0)}_t}\bigr)
\bigr]
}   \biggr\vert \nonumber
\\
&\leq h \sup_{r \in [t,t+h]} \Phi \bigl( \theta^{t,\xi}_{r},\theta^{t,\xi}_{t} \bigr). \nonumber
 \end{align}
Recalling that $\theta_{r}^{t,\xi}=(X_{r}^{t,\xi},Y_{r}^{t,\xi},\partial_{x} U(r,X_{r}^{t,\xi},
[X_{r}^{t,\xi}]) \sigma(X_{r}^{t,\xi},Y_{r}^{t,\xi}))$, we deduce from 
Theorem 
\ref{thm:4:1}
(smoothness of $\partial_{x} U$ both in time and in space) that it converges (in $L^2$) 
to $\theta^{t,\xi}_{t}$
as $r$ tends to $t$, 
proving that the supremum above tends to $0$ as $h$ tends to $0$.
Now, using the time continuity of the derivatives $\partial_{\mu} U$ and 
$\partial_{v}[\partial_{\mu} U]$ (see Theorem \ref{prop:partial:C^2:u}), we deduce that 
there exists a function $\varepsilon : \R \ni u \mapsto \varepsilon_{u} \in \R_{+}$, 
with $\lim_{u \rightarrow 0}\varepsilon_{u}=0$, such that 
\begin{align}
&\biggl\vert U\bigl(t+h,x,\law{X_{t+h}^{t,\xi}}\bigr) - U\bigl(t+h,x,\law{\xi}\bigr) \nonumber
\\
&\hspace{15pt}- h \biggl[ \hesp{\partial_\mu U\bigl(t,x,\law{\xi}\bigr)
\bigl(
\cc{\xi}\bigr)  
b\bigl(\cc{\theta^{t,\xi}_t},\law{\theta^{t,\xi,(0)}_t}\bigr)} \label{eq:chain:rule:PDE:proof}
 \\
 &\hspace{15pt} - \frac{h}2  \hesp{
 \text{Trace} \bigl[
 \partial_v \bigl[\partial_\mu U\bigr]\bigl(t,x,\law{\xi}\bigr)\bigl(\cc{\xi}\bigr)
 \bigl( \sigma
  \sigma^{\dagger}\bigr)\bigl(  
\cc{\theta^{t,\xi,(0)}_t},\law{\theta^{t,\xi,(0)}_t}\bigr)
\bigr]
}  \biggr] \biggr\vert \nonumber \leq h \varepsilon_{h}.\nonumber
 \end{align}
Now, we can plug \eqref{eq:chain:rule:PDE:proof}
into \eqref{eq:expansion:master:PDE}. Following 
\eqref{eq:time:derivative:U}, we get that 
the time increment $[U(t+h,x,[\xi]) - U(t,x,[\xi])]/h$
has a limit as $h$ tends to $0$. 
As in Subsection \ref{subse:solution:PDE},
the right derivative in time satisfies 
\eqref{eq:master:PDE} and is thus continuous in
time. Since $U$ is obviously continuous in time, 
 we deduce that the mapping
$[0,T] \ni t \mapsto U(t,x,[\xi])$ is differentiable and 
that the PDE \eqref{eq:master:PDE} holds true. \qed
\end{proof}

\section{Large population stochastic control -- proof of Theorem \ref{main:thm:3}
}
\label{se app}
In this section, we discuss two applications of our previous results to large population stochastic control. 
The first application is related to mean-field games, whilst the second one is related to the optimal control of McKean-Vlasov equations. 

\subsection{The global smoothness of the decoupling field }
So far, smoothness of the decoupling field $U$ has been discussed for small time intervals $[0,T]$; namely for $T \leq \delta_{0}$ where $\delta_{0}>0$ only depends upon the Lipschitz constants of the coefficients 
$b$, $f$, $\sigma$ and $g$, denoted by the common letter 
$L$ in condition \HYP{0}(i). A natural, though quite challenging, question concerns the possible extension of such a result to the case when $T$ is arbitrarily large. 

The principle for extending the result to an arbitrarily large time horizon is discussed 
in the earlier paper \cite{del02}. It consists of a backward recursion 
starting from the terminal time $T$. Thanks to the short time result proved in 
the previous section, the mapping $[T-\delta_{0},T] \times \R^d \times {\mathcal P}_{2}(\R^d)
\ni (t,x,\mu)\mapsto U(t,x,\mu) \in \R^m$ is rigorously defined as the 
initial value $Y_{t}^{t,x,\mu}$ of the backward component of the system \eqref{eq X-Y 
t,x,mu}, existence and uniqueness of the solution of the forward-backward system 
following from the condition $T-t \leq \delta_{0}$. By Lemma \ref{le app 
cont}, $U$ is Lipschitz continuous in $(x,\mu)$, uniformly in $t \in [T-\delta_{0},T]$. Up to a 
modification of the choice of the constant  $\delta_{0}$, $\delta_{0}$ still depending on the Lipschitz 
constants of the coefficients only, the results established in 
Section \ref{se smoothness} show that, under the assumptions detailed in 
Subsection \ref{subse:assumption},  
$U$ belongs to the class $\bigcup_{\beta \geq 0} \cD_{\beta}([T-\delta_{0},T])$. 
As in  \cite{del02}, we proceed by reapplying 
the short time existence, uniqueness and differentiability result to a new 
interval of the form $[T-(\delta_{0}+\delta_{1}),T-\delta_{0}]$, with the new 
terminal condition $U(T-\delta_{0},\cdot,\cdot)$ at time $T-\delta_{0}$ 
replacing the terminal condition   
$g$ at time $T$. 
A preliminary condition for iterating the short time solvability property is that 
$U(T-\delta_{0},\cdot,\cdot)$ is an admissible boundary condition. 
Under \HYP{2}, Theorems \ref{thm:4:1} and
 \ref{prop:partial:C^2:u} say that it is indeed the case, up to a deterioration of 
$\alpha$ into $\alpha+1$, the exponent $\alpha$
driving the local Lipschitz regularity of the derivatives of 
the coefficients
in \HYP{1} and \HYP{2}. This makes possible to reapply the existence and uniqueness result for short time
horizons with $\alpha$ be replaced by $\alpha+1$. 
Fortunately, the length  
$\delta_{1}$ of the new interval of existence and uniqueness only depends on the Lipschitz constant 
of $b$, $f$, $\sigma$ and $U(T-\delta_{0},\cdot,\cdot)$. In particular, it does not suffer from the deterioration of the exponent $\alpha$ into $\alpha+1$, which is a crucial fact.
 As a result we
are able to extend the definition of $U$ to $[T-(\delta_{0}+\delta_{1}),T-\delta_0] 
\times \R^d \times {\mathcal P}_{2}(\R^d)$. Since the new terminal condition 
$U(T-\delta_{0},\cdot,\cdot)$ has the same properties as  $g$ (but possibly with a different Lipschitz constant and a different $\alpha$), 
the extended version of $U$ is in the class 
$\cD_{\alpha+1}([T-(\delta_{0}+\delta_{1}),T]) \subset \bigcup_{\beta \geq 0} \cD_{\beta}([T-(\delta_{0}+\delta_{1}),T])$. 
The argument can be applied recursively on a sequence of small intervals of the form 
$[T-(\delta_{0}+\dots+\delta_{n+1}),T-(\delta_{0}+\dots+\delta_{n})]$, $n \geq 0$. Of course, the issue is that the lengths $(\delta_{n})_{n \geq 0}$ may be smaller and smaller
so that the sum $\sum_{n \geq 0} \delta_{n}$ may not exceed $T$. This happens if the Lipschitz constant of $U$ at times $(T-(\delta_{0}+\dots+\delta_{n}))_{n \geq 1}$ blows up
before that the sequence $(\delta_{0}+\dots+\delta_{n})_{n \geq 1}$ exceeds $T$. Put it differently, the construction of the smooth decoupling field $U$ on $[0,T] 
\times \R^d \times {\mathcal P}_{2}(\R^d)$ can be achieved by means of a backward recursion provided that the Lipschitz constant of $U(t,\cdot,\cdot)$
remain bounded as $t$ runs backward along the induction. 

The crux of the matter is thus to get such a Lipschitz estimate. 
In the following, we present two examples, derived from large population
stochastic control, for which the following assumption holds true:

\begin{Assumption}[\HYP{3}]
For any $t \in [0,T]$ and any square integrable 
${\mathcal F}_{t}$-measurable random variable $\xi$, the system \eqref{eq 
X-Y t,xi} has a unique solution $(X_{s}^{t,\xi},Y_{s}^{t,\xi},Z_{s}^{t,\xi})_{s \in [t,T]}$ and it satisfies, for all $\xi,\xi' \in L^2(\Omega,\cF_{t},\P;\R^d)$,
\begin{align}\label{eq:weak lip}
\E \bigl[ |Y^{t,\xi}_t - Y^{t,\xi'}_t|^2 \bigr]^{1/2} \le \Lambda \esp{|\xi-\xi'|^2}^{1/2}, 
\end{align}
with $\Lambda$ a positive constant that does not depend on $\xi$, $\xi'$ nor on $t$.
\end{Assumption}

We will show below that, under \HYP{3}, the decoupling field $U$ constructed along the induction must 
satisfy at any time $t$ at which it has been defined
\begin{align} \label{eq control recu U}
\forall \ \xi,\xi' \in L^2(\Omega,\cA,\P;\R^d), \quad \esp{|U(t,\xi,\law{\xi})-U(t,\xi',\law{\xi'})|^2}^{1/2} \le \Lambda 
\esp{|\xi-\xi'|^2}^{1/2}.
\end{align}
Although it is a first step in the control of the Lipschitz constant for $U$, it remains insufficient for our purposes. The reason is that 
the control is here stated along the diagonal only. 
Fortunately, the next Lemma permits to fill the gap and to 
bound the Lipschitz constant of $U$, in $x$ and $\mu$, on the entire domain:{\color{green}}

\begin{Lemma} \label{le control lip U}
 Under
\HYP{2}, assume that $U$ has been constructed on some interval  $[T_{0},T]$, for $T_{0} \in [0,T]$. Assume moreover that 
 it satisfies \eqref{eq control recu U} for any $t \in [T_{0},T]$ and that it is continuously 
differentiable in the directions $x$ and $\mu$ at any time $t \in [T_{0},T]$. 
Then, we can find a constant $\tilde{\Lambda}$, independent of $T_{0}$, 
such that for $t \in [T_{0},T]$, $x,x' \in \R^d$ and $\mu,\mu' \in \cP_{2}(\R^d)$:
\begin{align*}
 |U(t,x,\mu)- U(t,x',\mu')| \le \tilde{\Lambda}\bigl(|x-x'|+W_2(\mu,\mu')\bigr).
\end{align*}
\end{Lemma}

\proof

\emph{Step 1.} Applying Proposition \ref{prop:lipschitz:lifted} (with $\alpha=0$) we get that 
$U$ is $\Lambda$-Lipschitz continuous in $x$, or equivalently that
$\| \partial_x U(t,\cdot,\cdot)\|_\infty 
\le \Lambda$ for $t \in [T_{0},T]$.

\emph{Step 2a.} Now, for  $t \in [T_{0},T]$, $x \in \R^d$ and $\xi,\xi' \in L^2(\Omega,\cA,\P;\R^d)$, we have
\begin{equation*}
\begin{split}
&\vert U(t,x,[\xi]) - U(t,x,[\xi']) \vert 
\\
&\hspace{15pt} = \biggl\vert \int_{0}^1 \E \bigl[ 
 \partial_{\mu} U(t,x,[(1-\lambda)\xi + \lambda \xi'])\bigl( (1-\lambda) \xi +\lambda \xi' \bigr)
 \bigl( \xi - \xi' \bigr) \bigr] \ud \lambda 
\biggr\vert
\\
&\hspace{15pt} \leq  {\mathbb E} \bigl[ \vert \xi' - \xi \vert^2 \bigr]^{1/2} 
\int_{0}^1 
{\mathbb E} \bigl[ \bigl\vert 
 \partial_{\mu} U(t,x,[(1-\lambda)\xi + \lambda \xi'])
\bigl((1-\lambda) \xi + \lambda \xi' \bigr) \bigr\vert^2 \bigr]^{1/2}
 \ud\lambda.  
 \end{split}
\end{equation*}
In particular, in order to complete the proof, it suffices to find a constant $C$,  independent of $T_{0}$,
such that, 
for all $(t,x,\mu) \in [T_{0},T] \times \R^d \times \cP_{2}(\R^d)$,
\begin{align} \label{eq control lip 0}
 \esp{|\partial_\mu U(t,x,\mu)(\xi)|^2}^{1/2} \le C.
 \end{align}

\emph{Step 2b.}
Combining Step 1 and \eqref{eq control recu U}, we obtain
\begin{align*}
 \esp{|U(t,\xi,\law{\xi})-U(t,\xi,\law{\xi'})|^2}^{1/2} \le 2\Lambda\esp{|\xi-\xi'|^2}^{1/2},
\end{align*}
which at the level of the gradient says (choosing $\xi'-\xi = h \chi$, letting $h$ tend to $0$ and applying Fatou's lemma)
\begin{equation}
\label{eq control lip 1}
\forall \chi \in L^2(\Omega,\cA,\P;\R^d), \quad
\E \Bigl[ \hat{\E} \bigl[ \partial_{\mu} U(t,\xi,\law{\xi})(\langle \xi \rangle) \langle \chi \rangle \bigr]^2 \Bigr]^{1/2}
\le 2 \Lambda \E [ \vert \chi \vert^2]^{1/2}.
\end{equation}
This control is weaker than \eqref{eq control lip 0}. In order to get 
\eqref{eq control lip 0},
the strategy is to apply, on some small interval $[t,S]$, the results proved in 
Section \ref{se smoothness} on the 
first-order differentiability of $U$ with respect to the measure. Assuming that $\xi$ is $\cF_{t}$ measurable, we make use of Lemma 
\ref{lem:1storder:diff}
 but on the interval 
$[t,S]$ and with $g$ replaced by $U(S,\cdot,\cdot)$, the value of $S$ being specified next. In the backward component of 
the system of the type
\eqref{eq:sys:linear:mkv} satisfied by the derivative process $(\partial_{\chi} X^{t,\xi}_{s},
\partial_{\chi} Y^{t,\xi}_{s},\partial_{\chi} Z^{t,\xi}_{s})_{s \in [t,S]}$,
 the 
boundary condition 
reads as
\begin{equation*}
\partial_{\chi} Y_{S}^{t,\xi} 
= \partial_{x} U
\bigl(S,X_{S}^{t,\xi},[X_{S}^{t,\xi}]\bigr) \partial_{\chi} X_{S}^{t,\xi}+ 
\hat{\E} \bigl[ 
\partial_{\mu} U \bigl(S,X_{S}^{t,\xi},[X_{S}^{t,\xi}]\bigr)
\bigl( \langle  X_{S}^{t,\xi} \rangle
\bigr)
\langle \partial_{\chi} X_{S}^{t,\xi}\rangle \bigr]. 
\end{equation*}
Now, by the a priori bound \eqref{eq control lip 1} and Step 1, we get that
\begin{equation}
\label{eq:lip L2}
{\mathbb E} \bigl[ \vert \partial_{\chi} Y_{S}^{t,\xi} \vert^2 \bigr]^{1/2}
\leq C{\mathbb E} \bigl[ \vert \partial_{\chi} X_{S}^{t,\xi} \vert^2 
\bigr]^{1/2}  \,.
\end{equation}
Above and in the computations below, the constant  $C$ may change from 
line to line, it depends on the parameters in assumptions and, importantly, is uniform with 
respect to $0 \leq T_{0}  \leq t \leq S \leq  T$. The bound \eqref{eq:lip L2} reads as a 
Lipschitz bound (in $L^2$), with a constant $C$.

We can make use of 
\eqref{eq:co:bound:sys:lin:mkv:easy:2}
in Corollary 
 \ref{co:bound:sys:lin:mkv:easy}, 
 with $p=1$, $\gamma\leq 1/\Gamma_{1}$, $g_{\ell} \equiv 0$, $\hat{G}_{\ell} \equiv 0$, $G_{a}(S) =  \partial_{\chi} Y_{S}^{t,\xi}$ (which is to say, in rough 
terms, that 
 we put the whole terminal condition in the remainder) and $[0,T]$ replaced by 
 $[t,S]$. 
The remainder term $\E[{\mathcal R}_{a}^2]$
is thus equal to $\gamma^{1/2}
\E[ \vert  \partial_{\chi} Y_{S}^{t,\xi} \vert^2]$, 
which is less than $C \gamma^{1/2} {\mathbb E} [ \vert \partial_{\chi} X_{S}^{t,\xi} 
\vert^2]$.
Therefore, choosing $\Gamma_{1} C \gamma^{1/2} =1/2$, we have, for 
$S-t \le c:=c(L)$,
\begin{equation}
\label{eq:lip L2:2}
\E \biggl[ \sup_{s \in [t,S]} \bigl( 
\vert 
\partial_{\chi} X_{s}^{t,\xi} \vert^2 +
\vert 
\partial_{\chi} Y_{s}^{t,\xi} \vert^2
\bigr) + \int_{t}^S \vert \partial_{\chi} Z_{s}^{t,\xi} \vert^2 \ud s
\biggr]^{1/2} \leq C \|\chi \|_{2}.
\end{equation}

Now, consider the derivative process of the non McKean-Vlasov system  \eqref{eq X-Y t,x,mu}. 
It satisfies a forward-backward system of the type \eqref{eq:sys:linear:mkv}.  The boundary 
condition in the backward component may be expressed as
\begin{equation*}
\begin{split}
\partial_{\chi} Y_{S}^{t,x,[\xi]} 
&= \partial_{x} U
\bigl(S,X_{S}^{t,x,[\xi]},[X_{S}^{t,\xi}]\bigr) \partial_{\chi} X_{S}^{t,x,\xi}
+ 
\hat{\mathbb E} \bigl[ \partial_{\mu} U\bigl( 
S,X_{S}^{t,x,[\xi]},[X_{S}^{t,\xi}]\bigr)
(\langle X_{S}^{t,\xi} \rangle)
\langle \partial_{\chi} X_{S}^{t,\xi}\rangle \bigr]. 
\end{split}
\end{equation*}
Under the notations 
\eqref{splitting:H}
and
\eqref{eq:mkv:stability:decomposition},
the above writing reads as the decomposition of the terminal condition in 
the form
$g_{\ell}(X_{S}^{t,x,[\xi]},[X_{S}^{t,\xi}]) = 
\partial_{x} U
(S,X_{S}^{t,x,[\xi]},[X_{S}^{t,\xi}])$, 
$\hat{G}_{\ell} \equiv 0$ and $G_{a}(S)=\hat{\E}[
\partial_{\mu} U( S,X_{S}^{t,x,[\xi]},[X_{S}^{t,\xi}]) \langle \partial_{\chi} X_{S}^{t,\xi} \rangle]$. 
We can apply once again 
 Corollary \ref{co:bound:sys:lin:mkv:easy}, with $p=1$,
 $\vartheta = \partial_{\chi} \theta^{t,x,[\xi]}$,
 $\hat \vartheta = \partial_{\chi} \theta^{t,\xi}$
 and $B$, $\Sigma$ and $F$ given by 
 \eqref{eq:x,xi:H1}. 
 Recalling that $\partial_{x} U$ is bounded by $\Lambda$,
we get for $S-t \le \tilde{c} := 
\tilde{c}(\Lambda \vee L)$,
\begin{equation*}
\begin{split}
\vert \partial_{\chi} Y_{t}^{t,x,[\xi]} \vert  &\leq C 
\Bigl( {\mathbb E} \bigl[
\bigl\vert  
\hat{\mathbb E} [ \partial_{\mu} U( S,X_{S}^{t,x,[\xi]},[X_{S}^{t,\xi}])
(\langle X_{S}^{t,\xi} \rangle)
\langle \partial_{\chi} X_{S}^{t,\xi}\rangle]
\bigr\vert^2 \bigr]^{1/2}
\\
&\hspace{15pt}
+ \E \bigl[ \sup_{s \in [t,S]} \bigl( 
\vert 
\partial_{\chi} X_{s}^{t,\xi} \vert^2 +
\vert 
\partial_{\chi} Y_{s}^{t,\xi} \vert^2
\bigr) 
\bigr]^{1/2 }\Bigr), 
\end{split}
\end{equation*}
the second part coming from the remainder term $H_{a}$
in \eqref{eq:x,xi:H1} when $H=B,\Sigma,F$. 

Therefore, 
from the relationship $\partial_{\chi} Y_{t}^{t,x,[\xi]}
= \hat{\E}[\partial_{\mu} U(t,x,[\xi])(\cc{\xi}) \chi]$
and
from \eqref{eq:lip L2:2}, we get
\begin{equation*}
\hat{\mathbb E} \bigl[ 
\vert \partial_{\mu} U(t,x,[\xi])(\cc{\xi}) \vert^2 \bigr]^{1/2} 
\leq 
C \Bigl( 1 + {\mathbb E} \hat{\mathbb E} \bigl[ 
\vert \partial_{\mu} U\bigl(S,X_{S}^{t,x,[\xi]},
[X_{S}^{t,\xi}]\bigr)(\cc{X_{S}^{t,\xi}}) \vert^2 \bigr]^{1/2}
\Bigr). 
\end{equation*}
We deduce 
\begin{equation*}
\begin{split}
&\sup_{x \in \R^d,\xi \in L^2(\Omega,{\mathcal F}_{t},\P;\R^d)}\hat{\mathbb E} 
\bigl[ 
\vert \partial_{\mu} U(t,x,[\xi])(\cc{\xi}) \vert^2 \bigr]^{1/2} 
\\
&\hspace{15pt}\leq 
C \Bigl( 1 + \sup_{x \in \R^d,\xi \in L^2(\Omega,{\mathcal 
A},\P;\R^d)}  \hat{\mathbb E} \bigl[ 
\vert \partial_{\mu} U\bigl(S,x,
[\xi]\bigr)(\cc{\xi}) \vert^2 \bigr]^{1/2} \Bigr).
\end{split}
\end{equation*}
Since the terms in the suprema only depend on the law of $\xi$, we can assume
that the supremum in the left-hand side is taken over $\xi \in L^2(\Omega,\cA,\P;\R^d)$. 
Assuming without any loss of generality that $C \geq 1$ and
iterating the inequality, we get 
\begin{equation*}
\begin{split}
&1 + \sup_{x \in \R^d,\xi \in L^2(\Omega,\cA,\P;\R^d)}\hat{\mathbb 
E} \bigl[ 
\vert \partial_{\mu} U(t,x,[\xi])(\cc{\xi}) \vert^2 \bigr]^{1/2} 
\\
&\hspace{4cm}\leq 
2C  \Bigl( 1 + \sup_{x \in \R^d,\xi \in L^2(\Omega,\cA,\P;\R^d)}  \hat{\mathbb E} \bigl[ 
\vert \partial_{\mu} U\bigl(S,x,
[\xi]\bigr)(\cc{\xi}) \vert^2 \bigr]^{1/2} \Bigr) 
 \\
&\hspace{4cm} \leq (2C)^n \Bigl( 1 + 
 \sup_{x \in \R^d,\xi \in L^2(\Omega,\cA,\P;\R^d)}  \hat{\mathbb E} 
\bigl[ 
\vert \partial_{\mu} g\bigl(x,
[\xi]\bigr)(\cc{\xi}) \vert^2 \bigr]^{1/2} \Bigr), 
\end{split}
\end{equation*}
with $n = \lceil (T-t)/\tilde{c} \rceil$. Recalling the notation 
$L$ in \HYP{0}(i), we deduce that  
\begin{equation*}
\sup_{x \in \R^d,\xi \in L^2(\Omega,{\mathcal A},\P;\R^d)}\hat{\mathbb E} 
\bigl[ 
\vert \partial_{\mu} U(t,x,[\xi])(\cc{\xi}) \vert^2 \bigr]^{1/2} 
\leq L C^{T/\tilde{c}+1}, 
\end{equation*}
which proves \eqref{eq control lip 0} and thus completes the proof. 
\eproof

\begin{Proposition}
\label{prop:iteration weak tentative}
Assume that $b$, $f$, $\sigma$ and $g$ satisfy 
\HYP{2} 
and that the statement \HYP{3} 
holds true. Then
 there exists a mapping $U : [0,T] \times \R^d \times {\mathcal P}_{2}(\R^d)
\ni (t,x,\mu) \mapsto U(t,x,\mu) \in \R^m
$, Lipschitz continuous in $(x,\mu)$, uniformly in $t \in [0,T]$, such that, for all 
$t \in [0,T]$ and $\xi \in L^2(\Omega,\cF_{t},\P;\R^d)$, 
\begin{equation*}
Y_{s}^{t,\xi} = U\bigl(s,X_{s}^{t,\xi},[X_{s}^{t,\xi}]\bigr).
\end{equation*}
Moreover, $U$ belongs to $\bigcup_{\beta \geq 0} \cD_{\beta}$
and satisfies the master equation \eqref{eq:master:PDE}.
\end{Proposition}

\proof The proposition is proved by induction. 
Given a large integer $N \geq 1$ (the value of which is fixed below), 
let $\delta = T/N$. 
The induction hypothesis reads, for $n \in \{1,\dots,N\}$:

$(\mathcal{I}_n):$ 
There exists a mapping $U : [T-n \delta,T] \times \R^d \times \cP_{2}(\R^d)
\ni (t,x,\mu) \mapsto U(t,x,\mu) \in \R^m$
that 
belongs to $\bigcup_{\beta \geq 0} \cD_{\beta}([T-n\delta,T])$ 
such that

\textit{(i)}
 for any $t \in [T-n \delta,T]$, the function 
 $U(t,\cdot,\cdot)$ satisfies
 the same assumption as $g$ in 
 \HYP{0}(i), \HYP{1} \HYP{2}, 
 but with the constant $L$ replaced by $\tilde{\Lambda}$ coming
 from Lemma \ref{le control lip U} 
($\tilde{L}$ and $\alpha$ being replaced by some 
$\tilde{L}_{n}$ and $\tilde{\alpha}_{n}$);

\textit{(ii)} $U$ satisfies the master PDE \eqref{eq:master:PDE} on $[T-n\delta,T] \times \R^d \times 
\cP_{2}(\R^d)$

\textit{(iii)} 
for all $t \in [T-n \delta,T]$ and $\xi \in L^2(\Omega,\cF_{t},\P;\R^d)$,
$Y^{t,\xi}_t = U(t,\xi,[\xi])$. 
\vspace{5pt}

 \emph{Step 1.} In this step, we first specify the value of $N$ and we prove 
that $(\cI_1)$ is satisfied.

First, notice that $\tilde{\Lambda}$ in Lemma \ref{le control lip U}
may be assumed to be larger than $L$ in \HYP{0}(i), \HYP{1}
and \HYP{2}.
We then choose $N$ as the smallest integer such that 
$\delta := T/N \leq c(\tilde{\Lambda})$, where $c$ is given by 
Theorems \ref{thm:4:1} and \ref{prop:partial:C^2:u} (or more precisely by the minimum of
the $c$'s in these two statements). 
For $T-t \le \delta$, we know that, for any 
$x \in \R^d$ and $\mu \in {\mathcal P}_{2}(\R^d)$, the system 
\eqref{eq X-Y t,x,mu} has a unique solution 
$(X^{t,x,\mu}_{s},Y^{t,x,\mu}_{s},Z^{t,x,\mu}_{s})_{s \in [t,T]}$
and, by Theorems \ref{thm:4:1} and \ref{prop:partial:C^2:u}, 
$U$ belongs to $\bigcup_{\beta \geq 0} \cD_{\beta}([T-\delta,T])$
and satisfies the master equation on $[T-\delta,T] \times \R^d \times \cP_{2}(\R^d)$. 
 
Now, by Corollary 1.5 in 
\cite{del02} (which holds true for small time horizons), we can replace $x$ by a square-integrable 
${\mathcal F}_{t}$-measurable random initial condition $\xi$ in \eqref{eq X-Y 
t,x,mu}.
With obvious notations, it must satisfy $Y_{t}^{t,\xi,\mu} = U(t,\xi,\mu)$. Choosing $\xi$ with 
distribution $\mu$, we deduce from uniqueness in small time
to the system
\eqref{eq X-Y t,xi} that  
$(X^{t,\xi,\mu}_{s},Y^{t,\xi,\mu}_{s},Z^{t,\xi,\mu}_{s})_{s \in [t,T]}$ coincides with
$(X^{t,\xi}_{s},Y^{t,\xi}_{s},Z^{t,\xi}_{s})_{s \in [t,T]}$. Indeed, 
$(X^{t,\xi}_{s},Y^{t,\xi}_{s},Z^{t,\xi}_{s})_{s \in [t,T]}$ solves 
\eqref{eq X-Y t,x,mu} with $x$ replaced by $\xi$ and the system \eqref{eq X-Y 
t,x,mu} has a unique solution.  Therefore, we deduce that, with probability 1, 
\begin{equation*}
Y^{t,\xi}_t = U(t,\xi,[\xi]), \; \text{for all} \quad t \in [T-\delta,T].
\end{equation*} 
By \HYP{3}, $U$ satisfies \eqref{eq control recu U} so that, by Lemma 
\ref{le control lip U}, $(\cI_1)$ is indeed satisfied.
\vspace{5pt}

\emph{Step 2} Assume that, for some $n \in \{1,\dots, 
N-1\}$, $(\cI_n)$ holds true. 

For any  $t \in [T-(n+1)\delta,T]$ and $\xi \in L^2(\Omega,\cF_t,\P;\R^d)$, we 
consider again the forward-backward system \eqref{eq X-Y t,xi}. 
By \HYP{3}, it admits a unique solution.
In particular, by the uniqueness property guaranteed by \HYP{3}, it must hold that 
\begin{equation}
\label{eq:induction:1}
Y_{T-n\delta}^{t,\xi}=Y_{T-n\delta}^{T-n 
\delta, X^{t,\xi}_{T-n \delta}}.
\end{equation} 
By the induction hypothesis,  
$Y_{T-n \delta}^{t,\xi}$ must have the form 
$$Y_{T-n \delta}^{t,\xi}=U\bigl(T-n\delta,X^{t,\xi}_{T-n \delta}, \law{X^{t,\xi}_{T-n 
\delta}}\bigr).$$
Therefore, we now consider \eqref{eq X-Y t,xi}
but on $[T-(n+1)\delta,T-n\delta]$, 
with 
$U(T-n\delta,\cdot,\cdot)$ as terminal boundary condition. 
By the induction hypothesis, we know that $U(T-n\delta,\cdot,\cdot)$
is $\tilde{\Lambda}$-Lipschitz continuous, so that
existence and uniqueness to \eqref{eq X-Y t,xi} with 
$U(T-n\delta,\cdot,\cdot)$ as terminal boundary condition
hold true. 
This permits to extend the definition of $U$
to $[T-(n+1)\delta,T-n \delta]$. By Theorems
\ref{thm:4:1} and \ref{prop:partial:C^2:u}, 
the extension of $U$ belongs to $\bigcup_{\beta \geq 0}
\cD_{\beta}([T-(n+1)\delta,T-n\delta])$ and 
thus to $\bigcup_{\beta \geq 0}
\cD_{\beta}([T-(n+1)\delta,T])$. Moreover, it satisfies
the master equation on $[T-(n+1)\delta,T] \times \R^d \times \cP_{2}(\R^d)$. 

Consider now 
the restriction of the global 
solution $(X^{t,\xi}_{s},Y_{s}^{t,\xi},Z^{t,\xi}_{s})_{s \in [t,T]}$
to the  small interval
$[T-(n+1)\delta,T-n \delta]$. 
By \eqref{eq:induction:1}, it must coincide with the short time solution constructed
on $[t,T-n\delta]$ with $U(T-n\delta,\cdot,\cdot)$ as terminal boundary conditions. 
By the same 
arguments as in Step 1, we thus get that 
\begin{align*}
 Y^{t,\xi}_t = U(t,\xi,\law{\xi})
\end{align*}
with probability one. 
This shows that $U$ satisfies \eqref{eq control recu U} and applying Lemma 
\ref{le control lip U}, we get that 
$(\cI_{n+1})$ is satisfied. 
\eproof

\subsection{Mean-field games}
\label{subse mfg}

\subsubsection{General set-up}
\label{subse:MFG:general:set:up}

Mean-field games were introduced simultaneously by Lasry and Lions 
\cite{MFG1,MFG2,MFG3} and by Huang, Caines and Malham\'e \cite{HuangCainesMalhame2}. Their purpose is to describe
asymptotic Nash equilibria within large population of controlled agents interacting with one another through
the empirical distribution of the system. When players are driven by similar dynamics and subject to similar cost functionals,
asymptotic equilibria are expected to obey some propagation of chaos, limiting the analysis of the whole population to the analysis of one single player and thus reducing the complexity 
in a drastic way.  

The dynamics of one single player read as
\begin{equation}
\label{eq:controlled SDE}
\ud X_{t} = b(X_{t},\mu_{t},\alpha_{t}) \ud t  + {\sigma(X_{t},\mu_{t})} \ud W_{t}, \quad t \in [0,T],
\end{equation}
for some possibly random initial condition $X_{0}$, where $(W_{t})_{t \in [0,T]}$
is an $\R^{d}$-valued Brownian motion and $b: \R^d \times {\mathcal P}_{2}(\R^d)
\times \R^k \rightarrow \R^d$ and $\sigma : \R^d \times \cP_{2}(\R^d)$
are
 Lipschitz-continuous on the model of \HYP{0}(i). 
Above, $(\alpha_{t})_{t \in [0,T]}$ 
denotes the control process. It takes values in $\R^k$ and is assumed to be progressively-measurable and to satisfy:
\begin{equation*}
{\mathbb E} \int_{0}^T \vert \alpha_{t} \vert^2 \ud t < + \infty.  
\end{equation*}
The family $(\mu_{t})_{t \in [0,T]}$ denotes an arbitrary flow of probability measures in 
${\mathcal P}_{2}(\R^d)$. It is intended to describe the statistical equilibrium of the game,
the notion of equilibrium being defined according to some cost functional
\begin{equation*}
J\bigl( (\alpha_{t})_{t \in [0,T]} \bigr) 
= {\mathbb E} \biggl[ G(X_{T},\mu_{T}) + \int_{0}^T F(X_{t},\mu_{t},\alpha_{t}) \ud t 
\biggr],
\end{equation*}
and being actually given by the solution of a fixed point problem,
the description of which is taken from \cite{carmona:delarue:sicon}:

(i) Given 
the family $(\mu_{t})_{t \in [0,T]}$, solve the optimization problem
\begin{equation*}
\inf_{(\alpha_{t})_{t \in [0,T]}} J\bigl( (\alpha_{t})_{t \in [0,T]} \bigr). 
\end{equation*}
Assume that the optimal path is uniquely defined and denote it by 
$(\hat{X}_{t}^{(\mu_{s})_{s \in [0,T]}})_{t \in [0,T]}$. \label{MFG:page}

(ii) Find $(\mu_{s})_{s \in [0,T]}$ such that  
$[\hat{X}_{t}^{(\mu_{s})_{s \in [0,T]}}]=\mu_{t}$ for all $t \in [0,T]$. 
\vspace{5pt}

{
Generally speaking, there are two ways 
to characterize the optimal paths in 
(i) by means of an FBSDE. The first one is to 
represent the value function of the optimization problem (i) as 
the decoupling field of a forward-backward system, in which 
case equilibria solving (ii) may be described
through a McKean-Vlasov FBSDE along the lines 
of 
\cite{carmona:lacker}. Another way is to make use 
of the stochastic Pontryagin principle
to represent directly the optimal path 
in (i) as the forward component of the solution of a forward-backward system, in which 
case equilibria solving 
(ii) may be described through a McKean-Vlasov FBSDE along 
the lines of \cite{carmona:delarue:sicon}.
When using the stochastic Pontryagin principle, the decoupling field of the underlying forward-backward system is then understood 
as the gradient of the value function of the optimization problem (i).}

{
Here we are willing to show that, in both cases, 
the decoupling field of the McKean-Vlasov FBSDE used to characterize equilibria of the game is indeed a classical solution 
of a master PDE of the type 
\eqref{eq:master:PDE}
and, then, to make the connection with the so-called 
\textit{master equation} presented in Lions' lectures 
at the \textit{Coll\`ege de France}. In each case, 
we exhibit sufficient conditions under which the master PDE is 
solvable for an arbitrary time horizon $T$. 
In short, the two types of representation apply under slightly different assumptions. 
The direct representation of the value function 
is well-fitted to cases when $\sigma$ is uniformly non-degenerate, 
since standard theory for uniformly parabolic semilinear PDEs then applies. 
The stochastic Pontryagin principle is more adapted to cases when the 
underlying Hamiltonian is convex in both the space and control variables, 
$\sigma$ being possible degenerate. 
In both cases, we shall implement the Lasry-Lions monotonicity condition, 
see \HYP{4}(iii) below,
in order to investigate the Lipschitz property of the
solution of the corresponding master PDE in the direction of the measure. 
}

\subsubsection{Use of the Stochastic Pontryagin Principle}
\textcolor{black}{We first explain how things work when 
using the stochastic Pontryagin principle in order to characterize the optimal paths in (i). 
Then,
following \cite{carmona:delarue:sicon},
 the 
matching problem (ii) is solved by forcing the forward component of the FBSDE derived from 
the Pontryagin principle to have $(\mu_{t})_{t \in [0,T]}$ as marginal laws.} The resulting system
becomes ($(Y_{s})_{s \in [t,T]}$ being seen as a row vector process)
\begin{equation}
\label{eq:FBSDE:pontryagin:MFG}
\begin{split}
&\ud X_{t} = b \bigl(X_{t},[X_{t}],\hat{\alpha}(X_{t},[X_{t}],Y_{t}) \bigr) \ud t + \sigma(X_{t},[X_{t}]) \ud W_{t}
\\
&\ud Y_{t} = - \partial_{x} H \bigl(X_{t},[X_{t}],Y_{t},\hat{\alpha}(X_{t},[X_{t}],Y_{t}) \bigr) \ud t + Z_{t} \ud W_{t},
\end{split}
\end{equation}
with the boundary condition $Y_{T} = \partial_{x} G(X_{T},[X_{T}])$, where $H$ denotes the so-called
\textit{extended Hamiltonian} of the system:
\begin{equation}
\label{eq:hamiltonian}
H(x,\mu,y,\alpha) = y^\dagger b(x,\mu,\alpha) + F(x,\mu,\alpha),
\quad x,y \in \R^d, \ \alpha \in \R^k, \ \mu \in {\mathcal P}_{2}(\R^d),
\end{equation}
and  
$\hat{\alpha}(x,\mu,y)$ denotes the minimizer:
\begin{equation}
\label{eq:minimizer}
\hat{\alpha}(x,\mu,y) = \textrm{argmin}_{\alpha} H(x,\mu,y,\alpha). 
\end{equation} 
We shall specify below assumptions under which the minimizer is indeed well-defined. For the moment, 
we concentrate on the regularity properties we need on the coefficients. 
As we aim at applying Proposition 
\ref{prop:iteration weak tentative}, we let:

\begin{Assumption}[\HYP{4}(i)]
The running cost $F$ may be decomposed as 
\begin{equation}
\label{eq:decomposition}
F(x,\mu,\alpha) = F_{0}(x,\mu) + F_{1}(x,\alpha), \quad x \in \R^d, \ 
\mu \in {\mathcal P}_{2}(\R^d), \ \alpha \in \R^k,
\end{equation}
the function $F_{1}$ being three times differentiable, with bounded and Lipschitz-continuous 
derivatives of order $2$ and $3$. 
The functions $F_{0}$ and $G$ 
are locally Lipschitz continuous in $x$ and $\mu$, 
the Lipschitz constant being at most of linear growth in 
$\vert x \vert$ and in $(\int_{\R^d} \vert x' \vert^2 \ud \mu(x'))^{1/2}$. 
Moreover, $F_{0}$ and $G$
are differentiable with respect to $x$
 and the coefficients $f_{0}=\partial_{x} F_{0}$ and $g=\partial_{x} G$ 
are Lipschitz in $(x,\mu)$ and satisfy
\HYP{1} and \HYP{2} with $h=f_{0},g$ and $w=x$.

In particular, there exists a constant $C$ such that,
for all $x \in \R^d$, 
$\mu \in \cP_{2}(\R^d)$,
$\alpha \in \R^k$,
\begin{equation}
\label{eq:ch:5:quadratic:growth}
\begin{split}
&\vert G(x,\mu) \vert \leq C 
\Bigl[ 1 + \vert x \vert ^2 
+ \int_{\R^d} \vert x' \vert^2 
d\mu(x') 
\Bigr], 
\\
&\vert 
F_{0}(x,\mu) 
\vert 
+ \vert F_{1}(x,\alpha) \vert \leq C 
\Bigl[ 1 + 
\vert x \vert^2 
+ \int_{\R^d}
\vert x' \vert^2 
d\mu(x') 
+ \vert \alpha \vert^2 
\Bigr]. 
\end{split}
\end{equation}
\end{Assumption}
Actually, the decomposition \eqref{eq:decomposition} is motivated by the uniqueness criterion we use below. 
We introduce it now and not later since 
it makes the exposition of the regularity assumption much simpler. 
The growth conditions on the Lipschitz constant of the derivatives are motivated by 
the typical example 
when $F$ and $G$
have a quadratic structure 
in $x$ and $\alpha$ (see \cite{carmona:delarue:lachapelle}). 

The reader may notice that nothing is said about the smoothness of 
$b$ and $\sigma$. The reason is the following. 
Generally speaking, the uniqueness of the minimizer in \eqref{eq:minimizer}
is ensured under strict convexity of the Hamiltonian in the direction 
$\alpha$, but, for our purpose, we will use more. We indeed require 
the \textit{full extended Hamiltonian}
\begin{equation*}
H'(x,\mu,y,z,\alpha) = H(x,\mu,y,\alpha) + {\rm Trace} \bigl( z \sigma(x,\mu) \bigr), 
\end{equation*}
for $x \in \R^d$, $\mu \in \cP_{2}(\R^d)$,
$y \in \R^d$, $z \in \R^{d \times d}$, $\alpha \in \R^k$,
to be convex in $(x,\alpha)$, namely  
\begin{align}
H'(x',\mu,y,z,\alpha') - H'(x,\mu,y,z,\alpha) &- \langle x'-x, \partial_{x} H'(x,\mu,y,z,\alpha)
\rangle 
 \nonumber
\\
&
- \langle \alpha' - \alpha,\partial_{\alpha} H'(x,\mu,y,z,\alpha) \rangle \geq \lambda \vert \alpha' - \alpha \vert^2,
\label{eq:convexity:H}
\end{align}
for some $\lambda >0$. In order to guarantee the convexity of $H$, we must assume
that $b(x,\mu,\alpha)$ is a linear function in $(x,\alpha)$  of the form $b_{0}(\mu) + 
b_{1} x + b_{2}\alpha$,
for some matrices $b_{1} \in \R^{d \times d}$ and $b_{2} \in \R^{d \times k}$ and 
 $b_{0} : {\mathcal P}_{2}(\R^d) \rightarrow \R^d$. Moreover, because of the uniqueness criterion we use below, we shall restrict ourselves 
to the case $b_{0} \equiv 0$ so that the drift reduces to the linear combination 
$b(x,\alpha) = b_{1} x + b_{2} \alpha$.  
\textcolor{black}{Similarly, we must assume that $\sigma(x,\mu)$ is a linear function in $x$, which 
implies that $\sigma$ is independent of $x$ as we need it to be bounded 
(see \HYP{\sigma}). Again, because of
 the uniqueness criterion we use below, we restrict ourselves to 
 the case when $\sigma$ is also independent of $\mu$, namely 
 $\sigma(x,\mu) = \sigma$ for some constant matrix $\sigma$ of dimension 
 $d \times d$. 
 Then, the convexity property 
\eqref{eq:convexity:H} holds provided $F$ satisfies it. 
In particular, 
$H'(x',\mu,y,z,\alpha') - H'(x,\mu,y,z,\alpha)
=H(x',\mu,y,\alpha') - H(x,\mu,y,\alpha)$
so that 
the analysis of the 
\textit{full extended Hamiltonian} $H'$
may be reduced to the analysis of 
the \textit{extended Hamiltonian}
$H$.} 
We thus require 

\begin{Assumption}[\HYP{4}(ii)]
\textcolor{black}{
There exist $b_{1} \in \R^{d \times d}$, $b_{2} \in \R^{d \times k}$
and $\sigma \in \R^{d \times d}$
such that $b(x,\mu,\alpha) = b_{1} x + b_{2} \alpha$ and 
$\sigma(x,\mu)=\sigma$, for any
$x \in \R^d$, $\mu \in \cP_{2}(\R^d)$ and $\alpha \in \R^k$.} Moreover, 
$F$ satisfies \eqref{eq:convexity:H} and the mapping $\R^d \ni x 
\mapsto G(x,\mu) \in \R$
is convex in the $x$-variable for any $\mu \in {\mathcal P}_{2}(\R^d)$.  
\end{Assumption}

We then notice that 
$\hat{\alpha}(x,y,\mu)$ solves the equation:
\begin{equation}
\label{eq:optimizer}
y^\dagger b_{2}  + \partial_{\alpha} F\bigl(x,\mu,\hat{\alpha}(x,\mu,y) \bigr) = 0.  
\end{equation}
Since $\partial_{\alpha} F = \partial_{\alpha} F_{1}$ does not depend upon $\mu$, we deduce
that $\hat{\alpha}(x,\mu,y)$ reduces to $\hat{\alpha}(x,y)$. It is then straightforward to prove from the
implicit function theorem that the mapping $(x,y) \mapsto \hat{\alpha}(x,y)$ is twice differentiable 
with respect to $(x,y)$ with bounded and Lipschitz-continuous derivatives. This says in particular that,
in \eqref{eq:FBSDE:pontryagin:MFG}, there is no McKean-Vlasov interaction in the forward equation.
Moreover, we deduce, by composition, that Assumption 
\HYP{2} is satisfied (and thus 
\HYP{0} and \HYP{1} as well). 
\vspace{5pt}

\noindent \textit{Existence, uniqueness and differentiability of the solution.}
In \cite{carmona:delarue:sicon}, it is proved that \eqref{eq:FBSDE:pontryagin:MFG} admits a unique solution 
provided the following assumption is in force (in addition to \HYP{4}(i) and \HYP{4}(ii)):

\begin{Assumption}[\HYP{4}(iii)]
There exists $c>0$ such that 
\begin{enumerate}
\item 
For all $x \in \R^d$, 
$\vert \partial_{\alpha} F_{1}(x,0) \vert \leq c$,
\item For all $x \in \R^d$, 
$\langle x,\partial_{x} F_{0}(0,\delta_{x}) \rangle 
\geq - c(1+ \vert x \vert)$, 
$\langle x,\partial_{x} G(0,\delta_{x}) \rangle 
\geq - c(1+ \vert x \vert)$,
\end{enumerate}
where $\delta_{x}$ is the Dirac mass at point $x$. 
Moreover, the following Lasry-Lions monotonicity condition is in force:
\begin{equation*}
\int_{\R^d} \bigl( F_{0}(x,\mu) - F_{0}(x,\mu') \bigr) \ud \bigl( \mu - \mu') (x) \geq 0,
\
\int_{\R^d} \bigl( G(x,\mu) - G(x,\mu') \bigr) \ud \bigl( \mu - \mu') (x) \geq 0.
\end{equation*}
\end{Assumption}

Actually, not only existence and uniqueness hold, 
but also the key Lipschitz estimate \eqref{eq:weak lip} is true, 
justifying \HYP{3}. The argument is the same as the one given in 
\cite[Proposition 3.7]{carmona:delarue:sicon} for proving uniqueness. The only 
difference is that initial conditions may be different. More precisely, given $t 
\in [0,T]$ and
two square-integrable ${\mathcal F}_{t}$-measurable random variables
$\xi$ and $\xi'$, the same argument as in \cite{carmona:delarue:sicon}, 
combined with (3.6) therein to take into account  the fact that the initial 
conditions are different,
 shows that 
\begin{equation}
\label{eq:lip alpha MFG}
2 \lambda {\mathbb E} \int_{t}^T \vert \hat{\alpha}(X_{s}^{t,\xi},
Y_{s}^{t,\xi}) - \hat{\alpha}(X_{s}^{t,\xi'},Y_{s}^{t,\xi'})
\vert^2 \ud s \leq {\mathbb E} \bigl[ \langle \xi - \xi',Y_{t}^{t,\xi} - Y_{t}^{t,\xi'} \rangle
\bigr].
\end{equation}
(Here $(X^{t,\xi},Y^{t,\xi},Z^{t,\xi})$ 
and $(X^{t,\xi'},Y^{t,\xi'},Z^{t,\xi'})$ satisfy 
\eqref{eq:FBSDE:pontryagin:MFG}
with $X_{t}^{t,\xi}=\xi$ and $X_{t}^{t,\xi'}=\xi'$.)
Now, it is quite straightforward to see that 
\begin{equation}
\label{eq:lip Y MFG}
\begin{split}
&{\mathbb E} \bigl[ \vert Y_{t}^{t,\xi} - Y_{t}^{t,\xi'} \vert^2 \bigr] 
\\ 
 &\hspace{15pt} \leq C \Bigl( \sup_{s \in [t,T]} 
{\mathbb E} \bigl[ \vert X_{s}^{t,\xi} - X_{s}^{t,\xi'} \vert^2 \bigr] 
+ {\mathbb E}
\int_{t}^T \vert \hat{\alpha}(X_{s}^{t,\xi},Y_{s}^{t,\xi})
- \hat{\alpha}(X_{s}^{t,\xi'},Y_{s}^{t,\xi'})
\vert^2 \ud s \Bigr), 
\end{split}
\end{equation}
and,
\begin{equation}
\label{eq:lip X MFG}
\begin{split}
&\sup_{s \in [t,T]} 
{\mathbb E} \bigl[ \vert X_{s}^{t,\xi} - X_{s}^{t,\xi'} \vert^2 \bigr]
 \\
&\hspace{15pt} \leq C \Bigl( {\mathbb E} \bigl[ \vert \xi - \xi' \vert^2 \bigr] + 
{\mathbb E}
\int_{t}^T \vert \hat{\alpha}(X_{s}^{t,\xi},Y_{s}^{t,\xi})
- \hat{\alpha}(X_{s}^{t,\xi'},Y_{s}^{t,\xi'})
\vert^2 \ud s \Bigr).
\end{split}
\end{equation}
Therefore, from \eqref{eq:lip Y MFG} and \eqref{eq:lip X MFG},
\begin{equation*}
\begin{split}
&{\mathbb E} \bigl[ \vert Y_{t}^{t,\xi} - Y_{t}^{t,\xi'} \vert^2 \bigr] 
 \leq C \Bigl( {\mathbb E} \bigl[ \vert \xi - \xi' \vert^2 \bigr] +
  {\mathbb E}
\int_{t}^T \vert \hat{\alpha}(X_{s}^{t,\xi},Y_{s}^{t,\xi})
- \hat{\alpha}(X_{s}^{t,\xi'},Y_{s}^{t,\xi'})
\vert^2 \ud s \Bigr), 
\end{split}
\end{equation*}
Plugging \eqref{eq:lip alpha MFG} into the above equation, we get \eqref{eq:weak lip}. 
\vspace{5pt}

\noindent \textit{Master equation.}
The fact that \HYP{3} holds permits us to apply Proposition 
\ref{prop:iteration weak tentative}. It follows that the decoupling field $U$ of 
the forward-backward equation \eqref{eq:FBSDE:pontryagin:MFG}
satisfies the corresponding master PDE \eqref{eq:master:PDE}. 

We emphasize that the master PDE that we derive is not the standard master equation in mean-field games theory. Loosely speaking, the master equation in mean-field games is the equation satisfied by $V$, such that $U$ is the gradient of $V$, which stands for the value function of the game, namely 
\begin{equation}
\label{eq:ch:5:value function}
\begin{split}
&V(t,x,\mu) 
\\
&\hspace{15pt} = {\mathbb E} \biggl[ G\bigl(X_{T}^{t,x,\mu},[X_{T}^{t,\xi}]\bigr)
+ \int_{t}^T F\bigl(X_{s}^{t,x,\mu},[X_{s}^{t,\xi}],\hat{\alpha}(X_{s}^{t,x,\mu},Y_{s}^{t,x,\mu}) \bigr) \ud s \biggr],
\quad \xi \sim \mu,
\end{split}
\end{equation} 
in other words
$V(t,x,\mu)$ is the optimal cost when 
the private player is initialized at $x$ 
and 
the equilibrium strategy for the population is initialized at 
$\mu$. (Here $(X^{t,x,\mu},Y^{t,x,\mu},Z^{t,x,\mu})$
solves \eqref{eq X-Y t,x,mu} with the coefficients of 
\eqref{eq:FBSDE:pontryagin:MFG}.)

Now that $U$ is known to belong to $\bigcup_{\beta \geq 0} \cD_{\beta}$, 
we can see $X^{t,\xi}$ and $X^{t,x,\mu}$ as solutions of autonomous forward SDEs
driven by smooth Lipschitz-continuous coefficients (the drift being just obtained by a composition of $b$ with $\alpha(\cdot,U(\cdot,\cdot,\cdot))$). In particular,  
$X^{t,\xi}$ and $X^{t,x,\mu}$ must have the same smoothness properties as in the various 
results of Section \ref{se smoothness}, but for arbitrary time since the backward constraint has been removed. 
Another way to understand that claim is to prove regularity inductively, 
by means of a forward induction, applying successively the results obtained in 
Section \ref{se smoothness} on $[t,T-n\delta]$, $[T-n \delta, T-(n-1)\delta]$, ..., $[T-\delta,T]$, 
for the same $\delta$ as in the proof of Proposition \ref{prop:iteration weak tentative}
and for $n$ such that $t \in [T-(n+1)\delta,T-n\delta)$. The induction is 
then based on the flow property, which says that, for $s \in [T-k \delta,T-(k-1) \delta]$, 
\begin{equation}
\label{eq:ch:5:flow}
X_{s}^{t,\xi} = X_{s}^{T-k \delta,X_{T-k\delta}^{t,\xi}} \quad \text{and}
\quad
X_{s}^{t,x,[\xi]} = X_{s}^{T-k \delta,X_{T-k \delta}^{t,x,[\xi]},[X_{T-k \delta}^{t,\xi}]},
\end{equation}
and, thus, permits the transfer from one interval to another. 

Basically, this permits us to prove that $V$ is smooth in $x$ and $\mu$ by differentiating under the expectation, provided that $G$ and $F_{0}$ are smooth enough in the direction of 
the measure. 
\color{black}
Motivated by the fact that the coefficients are required to satisfy 
the
convexity assumption 
\HYP{4}(ii),
assume for instance that 

\begin{Assumption}[\HYP{4}(iv)]
The functions 
\begin{equation}
\label{eq:quadratic:growth:linear:derivative}
\begin{split}
&\R^d \times \cP_{2}(\R^d) \ni (x,\mu) \mapsto 
\frac{F_{0}(x,\mu)}{\sqrt{1+ \vert x\vert^2 + \int_{\R^d} \vert v\vert^2
\ud \mu(v)}},
\\
&\R^d \times \cP_{2}(\R^d) \ni (x,\mu) \mapsto 
\frac{G(x,\mu)}{\sqrt{1+ \vert x\vert^2 + \int_{\R^d} \vert v\vert^2 \ud \mu(v)}},
\end{split}
\end{equation}
satisfy \HYP{0}{\rm (i)}--\HYP{1}--\HYP{2} (for some values of the 
 parameters therein). 
In particular, $F_{0}$ and $G$ satisfy the same differentiability 
property as in
\HYP{0}{\rm (i)}--\HYP{1}--\HYP{2} but the derivatives 
are locally (instead of globally) controlled.   
\end{Assumption}

Then, 
we can differentiate the representation formula for $V$ as 
we differentiated the backward components of \eqref{eq X-Y t,xi} and \eqref{eq X-Y t,x,mu}
in Section \ref{se smoothness}, up to the slight difference that 
the
derivatives of $G$ and $F(\cdot,\cdot,\hat{\alpha}(\cdot,\cdot))$ 
in $x$, $y$ and $\mu$ may be of linear growth in all the arguments.
The key point to circumvent it is to notice
from \HYP{4}(iv)
 that 
the random variable 
\begin{equation*}
\frac{G(X_{T}^{t,x,\mu},[X_{T}^{t,\xi}])}
{\sqrt{1+\vert X_{T}^{t,x,\mu} \vert^2
+ \| X_{T}^{t,\xi} \|_{2}^2}}
\end{equation*}
satisfies the same first-order and second-order differentiability properties as
$\theta_{T}^{t,x,\xi}$ in Lemmas
\ref{lem:cont:derivatives:x,xi}
 and
\ref{lem:second:t,x,xi}. 
Since all the estimates in 
Lemmas
\ref{lem:cont:derivatives:x,xi}
 and
\ref{lem:second:t,x,xi} hold in $L^2$, it is then 
pretty clear that 
the mapping
\begin{equation*}
\begin{split}
\R^d \times 
\cP_{2}(\R^d)\ni (x,\mu) &\mapsto 
\E \Bigl[
 \sqrt{1+ \vert X_{T}^{t,x,\mu} \vert^2 + \| X_{T}^{t,\xi} \|_{2}^2}
\frac{G(X_{T}^{t,x,\mu},[X_{T}^{t,\xi}])}
{\sqrt{1+ \vert X_{T}^{t,x,\mu} \vert^2 + \| X_{T}^{t,\xi} \|_{2}^2}}
\Bigr]
\\
&= \E \bigl[
G(X_{T}^{t,x,\mu},[X_{T}^{t,\xi}])
\bigr], \quad \textrm{with} \ \  \xi \sim \mu,
\end{split} 
\end{equation*}
satisfies the same assumption as $F_{0}$ and $G$ in \HYP{4}(ii).

We then may proceed in the same way with 
$F(\cdot,\cdot,\hat{\alpha}(\cdot,\cdot))$
instead of $G(\cdot,\cdot)$
(recalling that $F_{1}$
has bounded 
derivatives of order $2$ and $3$, that $\hat{\alpha}$ has bounded
derivatives of order $1$ and $2$ and that 
$U(t,\cdot,\cdot)$ satisfies 
\HYP{0}{\rm (i)}--\HYP{1}--\HYP{2}, 
the values of the parameters therein being uniform 
in $t \in [0,T]$). 

In the spirit of Theorems 
\ref{thm:4:1}
and 
\ref{prop:partial:C^2:u}, this permits to show that, for 
any $t \in [0,T]$, 
$V(t,\cdot,\cdot)$ 
satisfies the same assumption as $F_{0}$ and $G$ in \HYP{4}(ii), 
the parameters that appear 
in \HYP{0}{\rm (i)}--\HYP{1}--\HYP{2}
being uniform
in $t \in [0,T]$.

It thus remains to identify the shape of the master PDE
and, in the same time, to prove the continuity of $V$ and of its derivatives with respect to 
$t$. From the same flow 
property as in 
\eqref{eq:ch:5:flow}, we notice that, 
for any $(t,x,\mu) \in [0,T] 
\times \R^d \times \cP_{2}(\R^d)$
and any 
$s \in [t,T]$, 
\begin{equation*}
\begin{split}
V(t,x,\mu) &=
 {\mathbb E} \biggl[ G\bigl(X_{T}^{s,X_{s}^{t,x,\mu},[X_{s}^{t,\xi}]},
[X_{T}^{s,[X_{s}^{t,\xi}]}]\bigr)
\\
&\hspace{15pt} + \int_{s}^T F\bigl(X_{r}^{s,X_{s}^{t,x,\mu},[X_{s}^{t,\xi}]},
[X_{r}^{s,[X_{s}^{t,\xi}]}],\hat{\alpha}(X_{r}^{s,X_{s}^{t,x,\mu},[X_{s}^{t,\xi}]},
Y_{r}^{s,X_{s}^{t,x,\mu},[X_{s}^{t,\xi}]}) \bigr) \ud r \biggr]
\\
&\hspace{15pt} + 
{\mathbb E} \biggl[ 
\int_{t}^{s} F\bigl(X_{r}^{t,x,\mu},[X_{r}^{t,\xi}],\hat{\alpha}(X_{r}^{t,x,\mu},Y_{r}^{t,x,\mu}) \bigr) \ud r \biggr]
\\
&= 
{\mathbb E} \biggl[ V\bigl(s,X_{s}^{t,x,\mu},[X_{s}^{t,\xi}]\bigr)
+
\int_{t}^{s} F\bigl(X_{r}^{t,x,\mu},[X_{r}^{t,\xi}],\hat{\alpha}(X_{r}^{t,x,\mu},Y_{r}^{t,x,\mu}) \bigr) \ud r \biggr],
\end{split}
\end{equation*}
from
which we may repeat the arguments 
from Theorems 
\ref{thm:4:1},
\ref{prop:partial:C^2:u}
and \ref{main:thm:short:time} (see also Subsection \ref{subse:solution:PDE}). We
finally obtain
\begin{Theorem}
\label{thm:MFG:1}
Under \HYP{4}{\rm (i--iv)}, the function $V(t,\cdot,\cdot)$ 
satisfies
the same assumption as $F_{0}$ and $G$ in 
\HYP{4}{\rm (iv)},
the parameters that appear 
in \HYP{0}{\rm (i)}--\HYP{1}--\HYP{2}
being uniform
in $t \in [0,T]$. 
Moreover, for any $x \in \R^d$ and $\mu \in \cP_{2}(\R^d)$, the function 
$[0,T] \ni t \mapsto V(t,x,\mu)$ is continuously differentiable,
\textcolor{black}{the derivative being 
continuous in  $(t,x,\mu)$}.
For any $x \in \R^d$ and $\xi \in L^2(\Omega,\cA,\P;\R^d)$, the functions
$[0,T] \times \R^d \times L^2(\Omega,\cA,\P;\R^d) \ni (t,x,\xi) \mapsto 
\partial_{\mu} V(t,x,[\xi])(\xi) \in L^2(\Omega,\cA,\P;\R^d)$
and $[0,T] \times
\R^d \times L^2(\Omega,\cA,\P;\R^d) \ni (t,x,\xi)
 \mapsto 
\partial_{v} [
\partial_{\mu} V(t,x,[\xi])](\xi) \in L^2(\Omega,\cA,\P;\R^d)$ are continuous. 
\color{black}

Finally, $V$ satisfies the master equation
\begin{equation}
\label{eq:full master PDE:MFG:1}
\begin{split}
&\partial_{t} V(t,x,\mu) 
+
\partial_{x} V(t,x,\mu)
b \bigl(x,\hat{\alpha}(x,U(t,x,\mu))\bigr)
+ 
F\bigl(x,\mu,\hat{\alpha}(x,U(t,x,\mu))
\bigr)
\\
&\hspace{5pt}
  + \int_{\R^d} 
  \partial_{\mu} V(t,x,\mu)(v) b\bigl(v,\hat{\alpha}(v,U(t,v,\mu)\bigr) \, \ud \mu(v) 
\\
 &\hspace{5pt} +  \frac12{\rm Tr} \biggl[
 \biggl(  
 \partial_{xx}^2 V(t,x,\mu)
 +
 \int_{\R^d} \partial_{v} \bigl(\partial_{\mu} V(t,x,\mu)\bigr)(v) \ud \mu(v)
\biggr) \sigma \sigma^{\dagger} 
\biggr]
=0,
\end{split}
\end{equation}
with $V(T,x,\mu) = G(x,\mu)$ as terminal condition and with $U$
denoting the decoupling field of \eqref{eq:FBSDE:pontryagin:MFG}.
\end{Theorem}
\color{black}

\begin{Remark}
\label{rem:identification}
The identification $U(t,x,\mu) = \partial_{x} V(t,x,\mu)$ can be checked directly as:
\begin{equation*}
\begin{split}
&\partial_{x} V(t,x,\mu) = 
{\mathbb E} \biggl[ \partial_{x}
G \bigl(X_{T}^{t,x,\mu},[X_{T}^{t,\xi}]\bigr) \partial_{x} X_{T}^{t,x,\mu}
\\
&\hspace{15pt}+ \int_{t}^T \partial_{x}
F\bigl(X_{s}^{t,x,\mu},[X_{s}^{t,\xi}],\hat{\alpha}(X_{s}^{t,x,\mu},Y_{s}^{t,x,\mu}) \bigr) \partial_{x}X_{s}^{t,x,\mu}
 \ud s 
 \\
&\hspace{15pt} +  
\int_{t}^T \partial_{\alpha}
F\bigl(X_{s}^{t,x,\mu},[X_{s}^{t,\xi}],\hat{\alpha}(X_{s}^{t,x,\mu},Y_{s}^{t,x,\mu}) \bigr) \partial_{x}
\bigl( \hat{\alpha}(X_{s}^{t,x,\mu},Y_{s}^{t,x,\mu})
\bigr)
\ud s 
 \biggr], \quad \xi \sim \mu. 
\end{split}
\end{equation*}
Now,  \eqref{eq:optimizer} says that $
\partial_{\alpha}
F (X_{s}^{t,x,\mu},[X_{s}^{t,\xi}],\hat{\alpha}(X_{s}^{t,x,\mu},Y_{s}^{t,x,\mu})) = - b_{2}^{\dagger} Y_{s}^{t,x,\mu}$, so that 
\begin{equation*}
\begin{split}
\partial_{x} V(t,x,\mu) &= 
{\mathbb E} \biggl[ Y_{T}^{t,x,\mu} \partial_{x} X_{T}^{t,x,\mu}
+ \int_{t}^T \partial_{x}
F\bigl(X_{s}^{t,x,\mu},[X_{s}^{t,\xi}],\hat{\alpha}(X_{s}^{t,x,\mu},Y_{s}^{t,x,\mu}) \bigr) \partial_{x}X_{s}^{t,x,\mu}
 \ud s 
 \\
&\hspace{15pt} - b_{2}^{\dagger}  
\int_{t}^T Y_{s}^{t,x,\mu} \partial_{x}
\bigl( \hat{\alpha}(X_{s}^{t,x,\mu},Y_{s}^{t,x,\mu})
\bigr)
 \ud s 
 \biggr].
\end{split}
\end{equation*}
 Using  \eqref{eq:FBSDE:pontryagin:MFG}
and It\^o's formula
  to expand $(Y_{s}^{t,x,\mu} \partial_{x} X_{s}^{t,x,\mu})_{t \leq s \leq T}$, we get that the right-hand side is equal to $Y_{t}^{t,x,\mu}$. We omit the details of the computation here. 
\end{Remark}

\subsubsection{Direct approach}
\textcolor{black}{
Theorem \ref{thm:MFG:1} is specifically designed to handle the case when 
the coefficients may be quadratic in $x$, provided 
that the \textit{extended Hamiltonian} has a convex structure in 
$(x,\alpha)$. When the coefficients are bounded in $x$ and $\mu$
and $\sigma$ is non-degenerate, we can give 
a direct proof of the solvability of the master 
equation
\eqref{eq:full master PDE:MFG:1} under the weaker 
assumption that  the \textit{extended Hamiltonian}
is convex in $\alpha$.
The key point is to represent directly
the value function 
$V$ in 
\eqref{eq:ch:5:value function}
by means of a McKean-Vlasov FBSDE, 
and thus to avoid any further reference 
to the stochastic Pontryagin principle. In particular, contrary to the last paragraph,
we shall prove existence and uniqueness to \eqref{eq X-Y t,xi} without relying on results in \cite{carmona:delarue:sicon}.
Of course, 
as previously, we shall need to check that
the processes that enter the representation of the value function 
satisfy \HYP{3}, or equivalently, that 
the key estimate \eqref{eq control recu U} holds true.}
We shall assume: 

\color{black}
\begin{Assumption}[\HYP{5}]
The coefficient $\sigma$ 
has the form $\sigma : \R^d \ni x \mapsto \sigma(x) \in \R^{d \times d}$,  
is bounded, twice differentiable, with bounded and Lipschitz-continuous derivatives of order 1 and 2,
and, for any $x \in \R^d$, the matrix $\sigma(x)$ is invertible
with $\sup_{x \in \R^d} | \sigma^{-1}(x) | < \infty$. 

The parameter $k$ is equal to $d$ and $b$ may decomposed as
\begin{equation*}
b(x,\alpha) = b_{0}(x) + \alpha, \quad x \in \R^d,
\quad \alpha \in \R^d,
\end{equation*}
the function $b_{0}$ being bounded and twice continuously differentiable with bounded and Lipschitz-continuous derivatives of order 1 and 2.  

The running cost $F$ may be decomposed as 
\begin{equation*}
F(x,\mu,\alpha) = F_{0}(x,\mu) + F_{1}(x,\alpha), \quad x \in \R^d, \ 
\mu \in {\mathcal P}_{2}(\R^d), \ \alpha \in \R^d,
\end{equation*}  
where
\begin{itemize}
\item[$\bullet$]
the functions $F_{0}$ and $G$ 
are bounded and satisfy \HYP{0}{\rm (i)}, \HYP{1}, \HYP{2};
\item[$\bullet$]
the function $F_{1}$ is bounded in $x$ and 
at most of quadratic growth in $\alpha$, uniformly in $x$; it is
is three times differentiable in $(x,\alpha)$,
the derivatives of order 2 and 3 being bounded and Lipschitz-continuous, 
the derivative of order $1$ in $x$ being bounded and the derivative of 
order $1$ in $\alpha$ being at most of linear growth in $\alpha$, uniformly in 
$x$; 
there exists $\lambda >0$ such that it 
satisfies the convexity assumption 
 \begin{align*}
F_{1}(x,\alpha') - F_{1}(x,\alpha) 
- \langle \alpha' - \alpha,\partial_{\alpha} F_{1}(x,\alpha) \rangle \geq \lambda \vert \alpha' - \alpha \vert^2,
\end{align*}
\end{itemize}

And, the Lasry-Lions monotonicity condition 
in the last line of \HYP{4}{\rm (iii)}
holds true. 
\end{Assumption}

We here prove that

\begin{Theorem}
\label{thm:main:MFG:2}
For a given $T>0$ and under \HYP{5}, the master PDE 
 \eqref{eq:full master PDE:MFG:1}
has a unique classical solution in the space $\bigcup_{\beta} \cD_{\beta \geq 0}$
\end{Theorem}

\proof
In comparison with the proof of Theorem
\ref{thm:MFG:1},
the difficulty here is that  
we do not have an \textit{a priori} existence and uniqueness result for 
the McKean-Vlasov FBSDE system representing  
the master PDE \eqref{eq:full master PDE:MFG:1}. 
In order to proceed, we thus revisit the proof of Proposition 
\ref{prop:iteration weak tentative} and show, by induction, that 
there exist an integer $N \geq 1$ 
and a constant $\tilde{\Lambda} \geq 0$
such that, 
with $\delta = T/N$, the following holds true for any $n \in \{1,\dots,N\}$: 

$(\mathcal{I}_n):$ 
There exists a mapping $V : [T-n \delta,T] \times \R^d \times \cP_{2}(\R^d)
\ni (t,x,\mu) \mapsto V(t,x,\mu) \in \R$
that 
belongs to $\bigcup_{\beta \geq 0} \cD_{\beta}([T-n\delta,T])$ 
such that

\textit{(i)}
 for any $t \in [T-n \delta,T]$, the function 
 $V(t,\cdot,\cdot)$ satisfies
 the same assumption as $g$ in 
 \HYP{0}{\rm (i)}, \HYP{1} \HYP{2}, 
 but with the constant $L$ replaced by $\tilde{\Lambda}$ 
and $\tilde{L}$ and $\alpha$ being replaced by some 
$\tilde{L}_{n}$ and $\tilde{\alpha}_{n}$;

\textit{(ii)}
$V$ satisfies the master PDE \eqref{eq:full master PDE:MFG:1} on $[T-n\delta,T] \times \R^d \times 
\cP_{2}(\R^d)$. 
\vspace{5pt}

\textit{First step.}
We start with the following observation. 
Equation \eqref{eq:full master PDE:MFG:1} is of the type 
\eqref{eq:master:PDE}, with $m=1$, 
$b(x,y,z,\nu)= b_{0}(x) + \hat{\alpha}(x,z \sigma^{-1}(x))$
\footnote{Pay attention that the letter $b$ is used both to denote 
the 
first-order coefficient in 
\eqref{eq:master:PDE}
and the drift in 
\eqref{eq:controlled SDE}. We feel that the reader can easily 
make the distinction between the two of them.}, $\sigma(x,\nu)=\sigma(x)$,
$f(x,y,z,\nu) = F_{0}(x,\mu) + 
F_{1}(x,\hat{\alpha}(x,z \sigma^{-1}(x)))$ and 
$g(x,\mu) = G(x,\mu)$ (recall that the product $z\sigma^{-1}(x)$ makes sense since
 $z$ reads as an element 
of $\R^{1 \times m}$, that is a row vector). 
Since $b$ does not rely on $y$ and $\nu$, we shall write 
$b(x,z)$ for $b(x,y,z,\nu)$. Similarly, since 
$f$ is independent of the variable $y$ and depends on 
the variable 
$\nu \in {\mathcal P}_{2}(\R^{d} \times \R)$
 through its first marginal $\mu \in {\mathcal P}_{2}(\R^d)$
 only 
(recall that, 
formally, $\nu$ is understood as the joint marginal law of 
the process $(X,Y)$), we shall write 
$f(x,z,\mu)$ for $f(x,y,z,\nu)$.
Now, recalling  \eqref{eq:optimizer}, we know that 
$(x,z) \mapsto \hat{\alpha}(x,z)$ is twice differentiable 
with respect to $(x,z)$ with bounded and Lipschitz-continuous derivatives of order 1 and 2.
In particular, we can find a constant $C$ such that, 
for all $x,z \in \R^d$ and $\mu \in \cP_{2}(\R^d)$,
\begin{equation}
\label{eq:ch:5:linear:growth:driver}
\vert \partial_{z} f(x,z,\mu) \vert \leq C(1+\vert z \vert),
\end{equation}
which plays an important role below. 

Of course, the assumption \HYP{0}{\rm (i)}
is not satisfied because of the quadratic growth of $f$ in the variable $z$. 
In order to apply Theorem \ref{main:thm:short:time}, we shall make use of a truncation argument. 
Considering a smooth function 
$\varphi_{R} : \R^d \rightarrow \R^d$ that matches the identity on the ball of center $0$
and of radius $R$, that is zero outside the ball of center $0$ and of radius $2R$
and that satisfies $\vert \nabla \varphi_{R} \vert \leq C$ with $C$ independent of $R$,
we let $b_{R}(x,z)=
b(x,\varphi_{R}(z))$
and
$f_{R}(x,z,\mu) = 
f(x,\varphi_{R}(z),\mu)$.

In particular,  
for any 
$(t,x) \in [0,T] \times \R^d$
and any 
flow of probability measures $(\mu_{u})_{u \in [t,T]}$
with values in $\cP_{2}(\R^d)$, 
we know 
from 
\cite{del02}
that the FBSDE system 
\begin{align}
\left\{
\begin{array}{ll}
&X_{s}
^{t,x,(\mu_{r})_{r \in [t,T]}}
\\
&\hspace{10pt}= x + \int_{t}^s 
b_{R}\bigl(X_{r}^{t,x,(\mu_{u})_{u \in [t,T]}}, 
Z_{r}^{t,x,(\mu_{u})_{u \in [t,T]}}\bigr) \ud r
+ \int_{t}^s 
\sigma\bigl(X_{r}^{t,x,(\mu_{u})_{u \in [t,T]}} \bigr) \ud W_{r},
\\
&Y_{s}^{t,x,(\mu_{r})_{r \in [t,T]}}
\\
&\hspace{10pt}= g\bigl(X_{T}^{t,x,(\mu_{u})_{u \in [t,T]}},\mu_{T})
+ \int_{s}^T 
f_{R}\bigl(X_{r}^{t,x,(\mu_{u})_{u \in [t,T]}}, 
Z_{r}^{t,x,(\mu_{u})_{u \in [t,T]}},\mu_{r}\bigr)
\ud r 
\\
&\hspace{177pt}- \int_{s}^T Z_{r}^{t,x,(\mu_{u})_{u \in [t,T]}}
\ud W_{r}, \quad s \in [t,T],
\end{array}
\right.
\label{eq:FBSDE:local}
\end{align}
admits a unique solution. 
It satisfies
$\vert Z_{s}^{t,x,(\mu_{u})_{u \in [t,T]}} \vert \leq C_{R}$ $\ud s \otimes \ud \P$ 
almost everywhere, for a constant $C_{R}$ that may depend upon $R$
(but not on $(\mu_{u})_{u \in [t,T]}$). 

We now prove that $C_{R}$ may be chosen independently of $R$. 
The proof is as follows. 
We write
\begin{equation*}
f(x,\varphi_{R}(z),\mu)
= f(x,0,\mu)
+ \biggl( \int_{0}^1 \partial_{z} f (x,\varphi_{R}(\lambda z),\mu)
\nabla \varphi_{R}(\lambda z) \ud \lambda  \biggr) z^{\dagger} .
\end{equation*}
By a standard Girsanov argument, the above decomposition of $f_{R}$ says that the FBSDE
\eqref{eq:FBSDE:local} may be written, under a new probability,
as a new FBSDE system with $f(X^{t,x,
(\mu_{u})_{u \in [t,T]}}_{s},0,\mu_{r})
$ as driver in the backward component 
and with 
\begin{equation*}
\begin{split}
&b_{R}\bigl(X^{t,x,
(\mu_{u})_{u \in [t,T]}}_{r},
Z^{t,x,
(\mu_{u})_{u \in [t,T]}}_{r}\Bigr)
\\
&\hspace{30pt}+
\biggl( \int_{0}^1 \partial_{z} f \bigl(
X^{t,x,
(\mu_{u})_{u \in [t,T]}}_{r},\varphi_{R}(\lambda 
Z^{t,x,(\mu_{u})_{u \in [t,T]}}_{r}
),\mu)
\nabla \varphi_{R}(\lambda Z^{t,x,(\mu_{u})_{u \in [t,T]}}_{r}) \ud \lambda 
 \biggr)^\dagger
 \end{split}
 \end{equation*}
as drift in the forward component: 
The driver in the backward component is bounded and, by 
\eqref{eq:ch:5:linear:growth:driver}, 
the drift in the forward component is bounded in the variable 
$x$ and at most of linear growth in the variable $z$. 
In particular, by \cite{del03}, 
there exists a constant $\Gamma$, independent of $R$
and 
$(\mu_{u})_{u \in [t,T]}$,  
such that, we indeed have $\vert Z_{s}^{t,x,(\mu_{u})_{u \in [t,T]}} \vert \leq \Gamma$. 

The coefficients $G$ and $F_{0}$ being bounded, we also have $\vert Y_{s}^{t,x,(\mu_{u})_{u \in [t,T]}} \vert \leq C$, 
for $C$ independent of $R$ and of 
$(\mu_{u})_{u \in [t,T]}$. 
\vspace{5pt}

\textit{Second step.}
We now construct $\delta >0$ such that 
the master PDE 
\eqref{eq:full master PDE:MFG:1}
admits a solution in 
$\bigcup_{\beta \geq 0} \cD_{\beta}([T-\delta,T])$
on $[T-\delta] \times \R^d \times \cP_{2}(\R^d)$. 
With the same $\Gamma$ as in the previous step, 
we indeed 
apply
Theorem \ref{main:thm:short:time}
with 
$(b_{R},\sigma,f_{R},g)$
instead of $(b,\sigma,f,g)$, for some $R >\Gamma \| \sigma^{-1} \|_{\infty}$. 
This says that, for some 
$\delta \in (0,T]$, there exists a function  
$$V : [T-\delta,T] \times \R^d \times \cP_{2}(\R^d)
\rightarrow \R$$ 
in $\bigcup_{\beta \geq 0} \cD_{\beta}([T-\delta,T])$
that solves
\eqref{eq:full master PDE:MFG:1}
with $b$ replaced by $b_{R}$ and $f$ replaced by $f_{R}$. 

Now, for any $(t,x,\mu) \in [T-\delta,T] \times \R^d 
\times \cP_{2}(\R^d)$,  
for any $s \in [t,T]$,
$$\partial_{x} V(s,X_{s}^{t,x,\mu},[X_{s}^{t,\xi}])=
Z_{s}^{t,x,\mu} \sigma^{-1}(X_{s}^{t,x,\mu}),$$ 
where $\xi \sim \mu$ and 
$(X^{t,\xi},Y^{t,\xi},Z^{t,\xi})$
and 
$(X^{t,x,\mu},Y^{t,x,\mu},Z^{t,x,\mu})$
solve 
\eqref{eq X-Y t,xi}
and 
\eqref{eq X-Y t,x,mu}
with $(b,f)$ replaced by $(b_{R},f_{R})$. 
In particular, 
$\vert \partial_{x} V(t,x,\mu) \vert \leq \Gamma
\| \sigma^{-1} \|_{\infty} < R$. 
Therefore, 
$V$ also solves 
\eqref{eq:full master PDE:MFG:1}.  
It also satisfies $\vert V \vert \leq C$, for some $C$ independent of $R$. 
Basically, this proves \textit{(ii)} in $(\cI_{1})$. 
\vspace{5pt}

\textit{Third step.}
In order to prove 
\textit{(i)} in $(\cI_{1})$ and more generally in
$({\mathcal I}_{n})$ for any $n=2,\dots,N$, we must identify
the constant $\tilde{\Lambda}$ first. 
We thus proceed as follows. 
%
%
%
%
%
We assume that there exists a time $t \in [0,T]$ such that, 
on $[t,T] \times \R^d \times \cP_{2}(\R^d)$, the master PDE 
 \eqref{eq:full master PDE:MFG:1} has a solution
 $V$ in 
 $\bigcup_{\beta \geq 0} \cD_{\beta}([t,T])$. We 
are then willing to provide a bound for $\sup_{x \in \R^d,
\xi \in L^2(\Omega,\cA,\P;\R^d)} 
\| \partial_{\mu} V(t,x,[\xi])(\xi) \|_{2}$, 
independently  of $t \in [0,T]$. 
 
 Since $V \in \bigcup_{\beta \geq 0} \cD_{\beta}([t,T])$, 
 we can find some $R>0$ such that 
 $\| \partial_{x} V(s,\cdot,\mu) \|_{\infty} \|\sigma \|_{\infty} < R$
 for any 
$s \in [t,T]$ and $\mu \in \cP_{2}(\R^d)$. In particular, 
 $V$ also solves the master PDE associated with 
 $(b_{R},\sigma,f_{R},g)$ instead of 
 $(b,\sigma,f,g)$. 
 Since $(b_{R},\sigma,f_{R},g)$ satisfies 
 \HYP{0}{(i)}--\HYP{1}--\HYP{2}, we can imitate the proof of 
Theorem \ref{main:thm:uniqueness}
and build a solution to
\eqref{eq X-Y t,xi}
for any $\xi \in L^2(\Omega,\cF_{t},\P;\R^d)$. 
The forward process
is defined as a solution of 
\eqref{eq:MKV:SDE}. 
We shall prove right below that 
it is uniquely defined, so that we can denote it by 
$(X_{s}^{t,\xi})_{s \in [t,T]}$. 

Uniqueness is a consequence of a more general result of stability, 
the proof of which is as follows.  
Given $\xi,\xi' \in L^2(\Omega,\cF_{t},\P;\R^d)$, 
we consider two solutions 
$(X_{s}^{t,\xi})_{s \in [t,T]}$
and 
$(X_{s}^{t,\xi'})_{s \in [t,T]}$
to 
the SDE 
\eqref{eq:MKV:SDE},
with $\xi$ and $\xi'$ as respective initial solutions. 
We then expand, by means of It\^o's formula 
$(V(s,X_{s}^{t,\xi},[X_{s}^{t,\xi}])-V(s,X_{s}^{t,\xi'},[X_{s}^{t,\xi}]))_{s \in [t,T]}$
(observe that, in both terms, the measure argument is driven by $\xi$).  
By Proposition \ref{prop:ito:general} (either by generalizing to the case when the process plugged in 
the spatial argument is not the same as the one plugged in the measure argument or by 
extending the dimension in order to see $(X_{s}^{t,\xi},X_{s}^{t,\xi'})_{s \in [t,T]}$ as a single process), 
we get for $s \in [t,T]$ (using the fact that 
$R>
\| \partial_{x} V \|_{\infty} \|\sigma \|_{\infty}$)
\begin{equation}
\label{eq:ito:mfg:1}
\begin{split}
\ud \bigl[ V \bigl( s,X_{s}^{t,\xi},[X_{s}^{t,\xi}] \bigr) \bigr]
&= - f\bigl(X_{s}^{t,\xi},[X_{s}^{t,\xi}],\hat{\alpha}
\bigl(X_{s}^{t,\xi},\partial_{x} V
(s,X_{s}^{t,\xi},[X_{s}^{t,\xi}])\bigr)
 \bigr) \ud s 
 \\
&\hspace{15pt}+
\partial_{x} V
(s,X_{s}^{t,\xi},[X_{s}^{t,\xi}])
\sigma(X_{s}^{t,\xi})
 \ud W_s, 
 \end{split}
\end{equation}
and
\begin{equation}
\label{eq:ito:mfg:2}
\begin{split}
&\ud \bigl[ V \bigl(s,X_{s}^{t,\xi'},[X_{s}^{t,\xi}] \bigr)
\bigr]
= \Bigl[ - f\bigl(X_{s}^{t,\xi'},[X_{s}^{t,\xi}],\hat{\alpha}
\bigl(X_{s}^{t,\xi'},
\partial_{x} V(s,X_{s}^{t,\xi'},[X_{s}^{t,\xi}])
\bigr) 
 \bigr)  
 \\
&\hspace{30pt} +  \partial_{x} V(s,X_{s}^{t,\xi'},[X_{s}^{t,\xi}])
\Bigl( 
\hat{\alpha}\bigl(X_{s}^{t,\xi'},
\partial_{x} V(s,X_{s}^{t,\xi'},[X_{s}^{t,\xi'}])
\bigr)
\\
&\hspace{150pt}
-
\hat{\alpha}\bigl(X_{s}^{t,\xi'},
\partial_{x} V(s,X_{s}^{t,\xi'},[X_{s}^{t,\xi}])
\bigr)
\Bigr)
 \Bigr] \ud s 
\\ 
&\hspace{30pt} + \partial_{x} V(s,X_{s}^{t,\xi'},[X_{s}^{t,\xi}])  
\sigma(X_{s}^{t,\xi'})
 \ud W_s.
 \end{split}
 \end{equation}

  Taking the difference between 
  \eqref{eq:ito:mfg:1}
  and \eqref{eq:ito:mfg:2}
  and using the same notation $H$ 
  for the Hamiltonian 
  as in 
  \eqref{eq:hamiltonian}, we obtain 
 \begin{equation*}
  \begin{split}
  &\ud \bigl[ V \bigl( s,X_{s}^{t,\xi},[X_{s}^{t,\xi}] \bigr) 
  - V \bigl(s,X_{s}^{t,\xi'},[X_{s}^{t,\xi}] \bigr)
  \bigr]
  \\
  &\hspace{15pt} = 
- \Bigl[  
f\bigl(X_{s}^{t,\xi},[X_{s}^{t,\xi}],\hat{\alpha}
\bigl(X_{s}^{t,\xi},\partial_{x} V
(s,X_{s}^{t,\xi},[X_{s}^{t,\xi}])\bigr)
 \bigr) 
 \\
&\hspace{45pt}  -
f\bigl(X_{s}^{t,\xi'},[X_{s}^{t,\xi}],\hat{\alpha}
\bigl(X_{s}^{t,\xi'},\partial_{x} V
(s,X_{s}^{t,\xi'},[X_{s}^{t,\xi'}])\bigr)
 \bigr) 
  \Bigr] \ud s
  \\
  &\hspace{30pt}
  - \Bigl[ H\Bigl(X_{s}^{t,\xi'},[X_{s}^{t,\xi}],
  \partial_{x} V(s,X_{s}^{t,\xi'},[X_{s}^{t,\xi}]),  
  \hat{\alpha}\bigl(X_{s}^{t,\xi'},
  \partial_{x} V(s,X_{s}^{t,\xi'},[X_{s}^{t,\xi'}])  
\bigr)
   \Bigr) 
   \\
&\hspace{45pt} - H\Bigl(X_{s}^{t,\xi'},[X_{s}^{t,\xi}],
  \partial_{x} V(s,X_{s}^{t,\xi'},[X_{s}^{t,\xi}]),  
  \hat{\alpha}\bigl(X_{s}^{t,\xi'},
  \partial_{x} V(s,X_{s}^{t,\xi'},[X_{s}^{t,\xi}])  
\bigr)
   \Bigr)  \Bigr] \ud s
   \\
&\hspace{30pt} + \Bigl[
 \partial_{x} V(s,X_{s}^{t,\xi},[X_{s}^{t,\xi}]) 
 \sigma(X_{s}^{t,\xi})
  -  \partial_{x} V(s,X_{s}^{t,\xi'},[X_{s}^{t,\xi}]) 
  \sigma(X_{s}^{t,\xi'}) \Bigr] \ud W_s.   
\end{split}
\end{equation*}
Therefore, taking the expectation and integrating in $s$ from $t$ to $T$, we get 
from the convexity of $H$ in $\alpha$ 
(that follows from the convexity of $F_{1}$ and the linear
structure of 
the drift 
in $\alpha$ in \eqref{eq:controlled SDE})
that 
\begin{equation*}
\begin{split}
&\E \bigl[ V(t,\xi,[\xi]) - V(t,\xi',[\xi]) \bigr]
\\
&\hspace{15pt} -  \E \int_{t}^T
\Bigl[ 
F_{1}\Bigl(X_{s}^{t,\xi},
\hat{\alpha}\bigl(X_{s}^{t,\xi},
  \partial_{x} V(s,X_{s}^{t,\xi},[X_{s}^{t,\xi}])  
\bigr)
\Bigr)
\\
&\hspace{150pt}-
F_{1}\Bigl(X_{s}^{t,\xi'},
\hat{\alpha}\bigl(X_{s}^{t,\xi'},
  \partial_{x} V(s,X_{s}^{t,\xi'},[X_{s}^{t,\xi'}])  
\bigr)
\Bigr)
\Bigr] \ud s
\\
&\geq \E \bigl[ G(X_{T}^{t,\xi},[X_{T}^{t,\xi}]) - G(X_{T}^{t,\xi'},[X_{T}^{t,\xi}]) \bigr]
+ \E \int_{t}^T \bigl( F_{0}(X_{s}^{t,\xi},[X_{s}^{t,\xi}]) - F_{0}(X_{s}^{t,\xi'},[X_{s}^{t,\xi}])
\bigr) \ud s
\\
&\hspace{15pt} 
+ \lambda \E \int_{t}^T 
\bigl\vert \hat{\alpha}\bigl(X_{s}^{t,\xi'},
  \partial_{x} V(s,X_{s}^{t,\xi'},[X_{s}^{t,\xi'}])  
\bigr)
-  \hat{\alpha}\bigl(X_{s}^{t,\xi'},
  \partial_{x} V(s,X_{s}^{t,\xi'},[X_{s}^{t,\xi}])  
\bigr)
 \vert^2 \ud s.
\end{split}
\end{equation*}
By exchanging the roles of $\xi$ and $\xi'$ and then by summing up, we deduce that
\begin{equation*}
\begin{split}
&\E \bigl[ V(t,\xi,[\xi]) - V(t,\xi',[\xi]) - \bigl( V(t,\xi,[\xi']) - V(t,\xi',[\xi']) \bigr) \bigr] 
\\
&\geq \E \bigl[ G(X_{T}^{t,\xi},[X_{T}^{t,\xi}]) - G(X_{T}^{t,\xi'},[X_{T}^{t,\xi}])
- \bigl( 
G(X_{T}^{t,\xi},[X_{T}^{t,\xi'}]) - G(X_{T}^{t,\xi'},[X_{T}^{t,\xi'}])
\bigr)
 \bigr]
\\
&\hspace{15pt}
+ \E \int_{t}^T 
\bigl[
\bigl( F_{0}(X_{s}^{t,\xi},[X_{s}^{t,\xi}]) - F_{0}(X_{s}^{t,\xi'},[X_{s}^{t,\xi}])
\bigr) 
\\
&\hspace{150pt}- \bigl( F_{0}(X_{s}^{t,\xi},[X_{s}^{t,\xi'}]) - F_{0}(X_{s}^{t,\xi'},[X_{s}^{t,\xi'}])
\bigr) 
\bigr]
\ud s
\\
&\hspace{15pt} +\lambda \E \int_{t}^T 
\bigl\vert \hat{\alpha}\bigl(X_{s}^{t,\xi'},
  \partial_{x} V(s,X_{s}^{t,\xi'},[X_{s}^{t,\xi'}])  
\bigr)
-  \hat{\alpha}\bigl(X_{s}^{t,\xi'},
  \partial_{x} V(s,X_{s}^{t,\xi'},[X_{s}^{t,\xi}])  
\bigr)
 \vert^2 \ud s
\\
&\hspace{15pt} +\lambda \E \int_{t}^T 
\bigl\vert \hat{\alpha}\bigl(X_{s}^{t,\xi},
  \partial_{x} V(s,X_{s}^{t,\xi},[X_{s}^{t,\xi}])  
\bigr)
-  \hat{\alpha}\bigl(X_{s}^{t,\xi},
  \partial_{x} V(s,X_{s}^{t,\xi},[X_{s}^{t,\xi'}])  
\bigr)
 \vert^2 \ud s.
\end{split} 
\end{equation*}
Finally, rearranging the terms, we deduce from the Lasry-Lions condition
that 
\begin{equation}
\label{eq:monotonicity}
\begin{split}
&\E \bigl[ V(t,\xi,[\xi])  - V(t,\xi',[\xi]) - \bigl( V(t,\xi,[\xi']) -  V(t,{\xi'},[\xi']) \bigr) \bigr]
\\
&\hspace{15pt} \geq \lambda \E \int_{t}^T 
\bigl\vert \hat{\alpha}\bigl(X_{s}^{t,\xi'},
  \partial_{x} V(s,X_{s}^{t,\xi'},[X_{s}^{t,\xi'}])  
\bigr)
-  \hat{\alpha}\bigl(X_{s}^{t,\xi'},
  \partial_{x} V(s,X_{s}^{t,\xi'},[X_{s}^{t,\xi}])  
\bigr)
 \vert^2 \ud s
\\
&\hspace{30pt} +\lambda \E \int_{t}^T 
\bigl\vert \hat{\alpha}\bigl(X_{s}^{t,\xi},
  \partial_{x} V(s,X_{s}^{t,\xi},[X_{s}^{t,\xi}])  
\bigr)
-  \hat{\alpha}\bigl(X_{s}^{t,\xi},
  \partial_{x} V(s,X_{s}^{t,\xi},[X_{s}^{t,\xi'}])  
\bigr)
 \vert^2 \ud s.
\end{split}
\end{equation}
When $\xi=\xi'$, the left-hand side is zero. Denoting 
by $X$ and $X'$ two solutions to the SDE 
\eqref{eq:MKV:SDE} 
with the same initial condition $\xi$, 
the above inequality (with the formal identification 
$X \equiv X^{t,\xi}$ and 
$X' \equiv X^{t,\xi'}$)
 says that 
$\hat{\alpha}(X_{s}',\partial_{x} V(s,X_{s}',[X_{s}']))
=
\hat{\alpha}(X_{s}',\partial_{x} V(s,X_{s}',[X_{s}]))$. 
Then, uniqueness to 
\eqref{eq:MKV:SDE} follows from the fact that, by assumption, 
$\partial_{x} V$ is Lipschitz continuous in $x$. 
\vspace{5pt}

\textit{Fourth step.}
Given the flow of probability measures 
$([X_{s}^{t,\xi}])_{s \in [t,T]}$ we just constructed, we 
know from 
\cite{del02}
that, for any $x \in \R^d$, the FBSDE 
 \eqref{eq X-Y t,x,mu}, when 
 driven by 
 $(b_{R},\sigma,f_{R},g)$ and by $\mu = [\xi]$,
 is uniquely solvable. By the first step, 
 the solution must solve 
 \eqref{eq X-Y t,x,mu}, when 
 driven by 
 $(b,\sigma,f,g)$. Moreover, it satisfies 
 $\vert Z_{s}^{t,x,\mu} \vert \leq \Gamma$
 $\ud s \otimes \ud \P$ almost everywhere. 
 Applying 
It\^o's formula to $(V(s,X_{s}^{t,x,\mu},[X_{s}^{t,\xi}]))_{s \in [t,T]}$,
we can check, in the spirit of 
 Theorem \ref{main:thm:uniqueness}, 
 that 
$Y_{s}^{t,x,\mu}=V(s,X_{s}^{t,x,\mu},[X_{s}^{t,\xi}])$
and 
 $Z_{s}^{t,x,\mu} = \partial_{x} V(s,X_{s}^{t,x,\mu},[X_{s}^{t,\xi}])
 \sigma(X_{s}^{t,x,\mu})$, $s \in [t,T]$, so that, on 
 $[t,T] \times \R^d \times \cP_{2}(\R^d)$, 
 \begin{equation*}
 \| \partial_{x} V \|_{\infty} \leq \Gamma \| \sigma^{-1} \|_{\infty}. 
 \end{equation*}

Another way to make the connection with 
 \eqref{eq X-Y t,x,mu} is to 
 see expansions \eqref{eq:ito:mfg:1} and \eqref{eq:ito:mfg:2} as standard verification arguments, as often used in stochastic control theory.
Indeed, we are just using the fact that the mapping $(s,x) \mapsto V(s,x,[X_{s}^{t,\xi}])$
is a solution of a standard HJB equation, corresponding to the optimization problem (i)
in the description of a mean-field game on page  \pageref{MFG:page}.
We can indeed differentiate in time 
$V(s,x,\mu_{s})$ for a given $x \in \R^d$, where $\mu_{s}=[X_{s}^{t,\xi}]$. Applying the chain rule 
proved in Section \ref{se chain rule}
and combining with the master PDE 
 \eqref{eq:full master PDE:MFG:1}, we then recover the HJB
equation:
\begin{equation}
\label{eq:HJB}
\begin{split}
&\partial_{s} \bigl[ V(s,x,\mu_{s}) \bigr]
+
\partial_{x} V(s,x,\mu_{s})
\bigl( b_{0}(x)
+\hat{\alpha}(x,
\partial_{x} V({s},x,\mu_{s})) \bigr)
\\
 &\hspace{5pt} 
 +  \frac12{\rm Tr} \bigl[
\sigma \sigma^\dagger(x) \partial_{xx}^2 V(s,x,\mu_{s})
\bigr]
+ F\bigl(x,\mu_{s},\hat{\alpha}(x,
\partial_{x} V({s},x,\mu_{s}) )
\bigr)
=0,
\end{split}
\end{equation}
for $s \in [t,T]$ and $x \in \R^d$,
with $V(T,x,\mu_{T}) = G(x,\mu_{T})$. 
We  know that $\partial_{x} V$ is bounded by $\Gamma \| \sigma^{-1} \|_{\infty}$. Therefore, 
\eqref{eq:HJB} reads as a standard 
semilinear uniformly parabolic equation driven by 
smooth coefficients in $x$.  
Since $f$ is Lipschitz-continuous in the direction of the measure and $[t,T] \ni s \mapsto \mu_{s}$
is $1/2$-H\"older continuous (the drift of the diffusion $X^{t,\xi}$ being bounded),
the coefficients are $1/2$-H\"older continuous in time. By Schauder's theory 
for semilinear parabolic equation (see \cite[Chapter 7]{friedman}), we can find 
a bound $\Gamma'$ 
for $\partial^2_{xx} V$ that is independent of $t \in [0,T]$.

Now, going back to \eqref{eq:MKV:SDE},
we may use the 
bound for $\partial_{xx}^2 V$ as a 
Lipschitz bound for $\partial_{x} V$ in the direction $x$. 
It is then pretty standard to deduce, from Gronwall's lemma, 
that, for any $\xi,\xi' \in L^2(\Omega,\cF_{t},\P;\R^d)$,
\begin{equation}
\label{eq:gross:preuve:1}
\begin{split}
&{\mathbb E}
\bigl[ \sup_{s \in [t,T]}
\vert X_{s}^{t,\xi} - X_{s}^{t,\xi'}
\vert^2 \bigr] \leq 
C \Bigl( {\mathbb E} \bigl[
\vert \xi - \xi' \vert ^2
\bigr]
\\
&\hspace{30pt} + \E \int_{t}^T 
\bigl\vert \hat{\alpha}\bigl(X_{s}^{t,\xi},
  \partial_{x} V(s,X_{s}^{t,\xi},[X_{s}^{t,\xi}])  
\bigr)
-  \hat{\alpha}\bigl(X_{s}^{t,\xi},
  \partial_{x} V(s,X_{s}^{t,\xi},[X_{s}^{t,\xi'}])  
\bigr)
 \vert^2 \ud s
\Bigr),
\end{split}
\end{equation} 
for a constant $C$ that is independent of $t$, $\xi$ and $\xi'$
and the value of which is allowed to increase from line to line. 
In particular, using the Lipschitz property of $\hat{\alpha}$
and once again the bound for $\partial^2_{xx} V$, we deduce that 
\begin{equation}
\label{eq:gross:preuve:2}
\begin{split}
&\E \int_{t}^T 
\bigl\vert \hat{\alpha}\bigl(X_{s}^{t,\xi},
  \partial_{x} V(s,X_{s}^{t,\xi},[X_{s}^{t,\xi}])  
\bigr)
-  \hat{\alpha}\bigl(X_{s}^{t,\xi'},
  \partial_{x} V(s,X_{s}^{t,\xi'},[X_{s}^{t,\xi'}])  
\bigr)
 \bigr\vert^2 \ud s
 \\
 &\hspace{0pt}
\leq 
 C \Bigl( {\mathbb E} \bigl[
\vert \xi - \xi' \vert ^2
\bigr]
\\
&\hspace{15pt} + \E \int_{t}^T 
\bigl\vert \hat{\alpha}\bigl(X_{s}^{t,\xi},
  \partial_{x} V(s,X_{s}^{t,\xi},[X_{s}^{t,\xi}])  
\bigr)
-  \hat{\alpha}\bigl(X_{s}^{t,\xi},
  \partial_{x} V(s,X_{s}^{t,\xi},[X_{s}^{t,\xi'}])  
\bigr)
 \vert^2 \ud s
\Bigr)
\\
&\leq 
C \Bigl( 
 {\mathbb E} \bigl[
\vert \xi - \xi' \vert ^2
\bigr]
\\
&\hspace{15pt}+
\E \bigl[ V(t,\xi,[\xi])  - V(t,\xi',[\xi]) - \bigl( V(t,\xi,[\xi']) -  V(t,{\xi'},[\xi']) \bigr) \bigr]
\Bigr),
\end{split}
\end{equation}
the last line following from 
\eqref{eq:monotonicity} (paying attention that the last term in the right-hand side 
is non-negative). 

We now make use of 
Remark \ref{rem:identification}. By differentiating 
$(Y^{t,x,\mu}_{s})_{s \in [t,T]}$
with respect to $x$
(which is licit as it reads 
$(Y^{t,x,\mu}_{s}=V(s,X_{s}^{t,x,\mu},[X_{s}^{t,\xi}]))_{s \in [t,T]}$
and $(X_{s}^{t,x,\mu})_{s \in [t,T]}$ solves a standard SDE
with smooth coefficients)
and then, by applying It\^o's formula, 
we can indeed check that
$(\partial_{x} Y^{t,x,\mu}_{s} (\partial_{x} X_{s}^{t,x,\mu})^{-1})_{s \in [t,T]}$,
solves the backward SDE in 
\eqref{eq:FBSDE:pontryagin:MFG}, so that 
\begin{equation*}
\begin{split}
&\partial_{x} V(t,x,\mu) = \E 
\Bigl[ \partial_{x} G\bigl(X_{T}^{t,x,\mu},[X_{T}^{t,\xi}]\bigr)
\\
&\hspace{15pt} + \int_{t}^T 
\partial_{x} H\bigl(X_{s}^{t,x,\mu},[X_{s}^{t,\xi}],
\partial_{x} V(s,X_{s}^{t,x,\mu},[X_{s}^{t,\xi}]),
\hat{\alpha}(X_{s}^{t,x,\mu},\partial_{x} V(s,
X_{s}^{t,x,\mu},[X_{s}^{t,\xi}])) \bigr) \ud s \Bigr],
\end{split}
\end{equation*}
and thus
\begin{equation*}
\begin{split}
&\partial_{x} V(t,\xi,[\xi]) = \E 
\Bigl[ \partial_{x} G\bigl(X_{T}^{t,\xi},[X_{T}^{t,\xi}]\bigr)
\\
&\hspace{15pt} + \int_{t}^T 
\partial_{x} H\bigl(X_{s}^{t,\xi},[X_{s}^{t,\xi}],
\partial_{x} V(s,X_{s}^{t,\xi},[X_{s}^{t,\xi}]),
\hat{\alpha}(X_{s}^{t,\xi},\partial_{x} V(s,
X_{s}^{t,\xi},[X_{s}^{t,\xi}])) \bigr) \ud s \vert \cF_{t} \Bigr].
\end{split}
\end{equation*}
Therefore, 
thanks to
\eqref{eq:gross:preuve:1}
and 
\eqref{eq:gross:preuve:2}, 
and by the Lipschitz property of 
$\partial_{x} F$ and $\partial_{x} G$ in the variables 
$x$, $\mu$ and $\alpha$, 
we get
that, for any $\xi,\xi' \in L^2(\Omega,\cF_{t},\P;\R^d)$,
\begin{equation*}
\begin{split}
&\E \bigl[ \vert \partial_{x} V(t,\xi,[\xi])
- 
\partial_{x} V(t,\xi',[\xi'])
\vert^2 \bigr] 
\\
&\hspace{15pt}
\leq C \Bigl(
 {\mathbb E} \bigl[
\vert \xi - \xi' \vert ^2
\bigr]
+
\E \bigl[ V(t,\xi,[\xi])  - V(t,\xi',[\xi]) - \bigl( V(t,\xi,[\xi']) -  V(t,{\xi'},[\xi']) \bigr) \bigr]
\Bigr). 
\end{split}
\end{equation*}
Since $V$ is smooth in $x$
and $\partial_{x} V$ is $\Gamma'$-Lipschitz in $x$, 
we can write
\begin{equation*}
\begin{split}
&\E \bigl[ \vert \partial_{x} V(t,\xi,[\xi])
- 
\partial_{x} V(t,\xi,[\xi'])
\vert^2 \bigr] 
\\
&\leq C \Bigl(
 {\mathbb E} \bigl[
\vert \xi - \xi' \vert ^2
\bigr]
\\
&\hspace{15pt}
 +
\int_{0}^1 \E \bigl[ \bigl( \partial_{x} V\bigl(t,\lambda \xi + (1-\lambda) \xi',[\xi] \bigr)
- \partial_{x} V\bigl(t, \lambda \xi + (1-\lambda) \xi',[\xi'] \bigr)
\bigr) \bigl( \xi-\xi' \bigr) \bigr] \ud \lambda 
\Bigr)
\\
&\leq C \Bigl(
 {\mathbb E} \bigl[
\vert \xi - \xi' \vert ^2
\bigr] +
\E \bigl[ \bigl(
  \partial_{x} V\bigl(t,\xi,[\xi] \bigr)
- \partial_{x} V\bigl(t, \xi,[\xi'] \bigr)
\bigr) \bigl( \xi-\xi' \bigr) 
\bigr] \Bigr). 
\end{split}
\end{equation*}
We finally get that 
\begin{equation*}
\E \bigl[ \vert \partial_{x} V(t,\xi,[\xi])
- 
\partial_{x} V(t,\xi,[\xi'])
\vert^2 \bigr] 
\leq C \| \xi' - \xi \|_{2}^2,
\end{equation*}
the constant $C$ being independent of $t$, $\xi$ and $\xi'$. 
Plugging into \eqref{eq:MKV:SDE}, 
we can deduce that 
\begin{equation}
\label{eq:regul}
\sup_{s \in [t,T]} \E \bigl[ \vert X_{s}^{t,\xi} - X_{s}^{t,\xi'} \vert^2 \bigr]
\leq C \| \xi' - \xi \|_{2}^2.
\end{equation} 
We now look at the backward equation in \eqref{eq X-Y t,xi}
(driven by $(b_{R},\sigma,f_{R},g)$). Now that we have proven 
a Lipschitz estimate for the forward component, it is standard to prove 
a similar estimate for the backward one. We deduce that 
\eqref{eq:weak lip} and thus \eqref{eq control recu U}
hold true. Applying 
Lemma \ref{le control lip U}, we get the 
required $\tilde{\Lambda}$ in 
\textit{(i)} of the induction 
property 
$(\mathcal{I}_n)$. 
\vspace{5pt}

\textit{Last step.}
From the second and fourth steps, it is clear that 
\textit{(i)} in $({\mathcal I}_{1})$ holds true, which completes the 
proof of $({\mathcal I}_{1})$. 

We then apply Theorem 
\ref{main:thm:short:time}
iteratively along the lines 
of the proof of Proposition \ref{prop:iteration weak tentative}. 
Notice that here there is no need of the assumption \textit{(iii)} in the induction 
scheme used in the proof 
of Proposition \ref{prop:iteration weak tentative}. Indeed, by the fourth step above, we have 
a direct way to establish \eqref{eq control recu U}, whereas, in 
the proof 
of Proposition \ref{prop:iteration weak tentative}, 
the bound \eqref{eq control recu U} is obtained by means of
the induction assumption \textit{(iii)}.

Uniqueness follows from Theorem 
\ref{main:thm:uniqueness}, observing that the quadratic term in the equation may be truncated 
(as any solution in the class $\bigcup_{\beta \geq 0} \cD_{\beta}$
has a bounded gradient). 
\eproof

\color{black}
\subsection{Control of McKean-Vlasov equations}
\label{subse ussr}
\subsubsection{General set-up}
Another example taken from large population stochastic control is the 
optimal control of McKean-Vlasov equations. We refer to 
\cite{carmona:delarue:aop,carmona:delarue:lachapelle} for a complete review. The idea here is to 
minimize the cost functional 
\begin{equation*}
J\bigl( (\alpha_{t})_{t \in [0,T]} \bigr) 
= {\mathbb E} \biggl[ G(X_{T},[X_{T}]) + \int_{0}^T F(X_{t},[X_{t}],\alpha_{t}) dt 
\biggr],
\end{equation*}
over controlled McKean-Vlasov diffusion processes
of the form
\begin{equation}
\label{eq:controlled MKV SDE}
dX_{t} = b(X_{t},[X_{t}],\alpha_{t}) dt  + \sigma dW_{t}, \quad t \in [0,T],
\end{equation}
for some possibly random initial condition $X_{0}$. As in \eqref{eq:controlled SDE}, 
$(W_{t})_{t \in [0,T]}$
is an $\R^{d}$-valued Brownian motion, $b: \R^d \times {\mathcal P}_{2}(\R^d)
\times \R^k \rightarrow \R^d$ is Lipschitz-continuous on the same model as in 
\textcolor{black}{\HYP{0}(i)}
and 
$(\alpha_{t})_{t \in [0, T]}$ 
denotes the progressively-measurable square-integrable control process.
Note that we shall only consider the case $\sigma$ constant. 
 
Unlike the mean-field games example, 
in which the McKean-Vlasov constraint is imposed in step (ii) only, 
see page \pageref{MFG:page}, the McKean-Vlasov prescription is here given first. In particular, the problem now consists of a true 
optimization problem. 

Below, we make use of the stochastic Pontryagin principle in order to 
characterize the optimal paths. 
Although 
the form of the Pontryagin principle is different from what it is in mean-field games, 
it imposes, in a rather similar way, 
restrictive conditions on the structure of the SDE
\eqref{eq:controlled MKV SDE}, among which 
the fact that $\sigma$ has to be constant.  
The Hamiltonian is defined in the same way as before,
see \eqref{eq:hamiltonian}, 
but 
 the FBSDE derived from the stochastic Pontryagin principle has a more complicated 
 form (see \cite{carmona:delarue:aop}):
\begin{equation}
\label{eq:FBSDE:pontryagin:MKV}
\begin{split}
&dX_{t} = b \bigl(X_{t},[X_{t}],\hat{\alpha}(X_{t},[X_{t}],Y_{t}) \bigr) dt + \sigma dW_{t}
\\
&dY_{t} = - \partial_{x} H  \bigl(X_{t},[X_{t}],Y_{t},\hat{\alpha}(X_{t},[X_{t}],Y_{t}) \bigr) dt 
\\
&\hspace{30pt} - \hat{\mathbb E} \bigl[
\partial_{\mu} H  \bigl(\cc{{X}_{t}},[X_{t}],\cc{{Y}_{t}},\hat{\alpha}(\cc{{X}_{t}},
[X_{t}],\cc{{Y}_{t}}) \bigr)(X_{t}) \bigr]
+ Z_{t} dW_{t},
\end{split}
\end{equation}
with the boundary condition $Y_{T} = \partial_{x} G(X_{T},[X_{T}])
+ \hat{\mathbb E}[\partial_{\mu} G(\cc{X_{T}},[X_{T}])(X_{T})]$.
The reason is that the state space over which the optimization is performed 
is the enlarged space $\R^d \times \cP_{2}(\R^d)$. This means that, 
in the extended Hamiltonian, the state variable is the pair $(x,\mu)$
and not $x$ itself.  The additional terms in the driver and in the boundary condition 
deriving from the stochastic Pontryagin principle thus
express the sensitivity of the Hamiltonian with respect to the measure argument. We notice that these two terms may be 
reformulated as
\begin{equation*}
\begin{split}
&\hat{\mathbb E} \bigl[
\partial_{\mu} H \bigl(\cc{{X}_{t}},[X_{t}],\cc{Y_{t}},\hat{\alpha}(\cc{{X}_{t}},
[X_{t}],\cc{{Y}_{t}}) \bigr)(X_{t}) \bigr]
= \tilde{h}\bigl(X_{t},[X_{t},Y_{t}]),
\\
&\hat{\mathbb E} \bigl[
\partial_{\mu} G \bigl(\cc{{X}_{T}},[X_{T}] \bigr)(X_{T}) \bigr]
= \tilde{g}\bigl(X_{T},[X_{T}] \bigr),
\end{split}
\end{equation*}
where
\begin{equation*}
\begin{split}
&\tilde{h}(x,\nu) = \int_{\R^d \times \R^d}
\partial_{\mu} H\bigl(v,\pi_{1} \sharp \nu,w,\hat{\alpha}(v,\pi_{1} \sharp \nu,w)\bigr)(x) \ud \nu(v,w),
\\
&\tilde{g}(x,\mu) = \int_{\R^d} \partial_{\mu} G(v,\mu)(x) \ud \mu(v), 
\end{split}
\end{equation*}
with $x \in \R^d$, $\nu \in {\mathcal P}_{2}(\R^d \times \R^d)$, 
$\mu \in {\mathcal P}_{2}(\R^d)$ and $\pi_{1} : \R^d \times \R^d
\ni (x,y) 
\mapsto x \in \R^d$. 

Existence and uniqueness of a solution to \eqref{eq:FBSDE:pontryagin:MKV}
have been established under the following assumption (see \cite{carmona:delarue:aop}):

\begin{Assumption}[\HYP{6}(i)]
The drift $b$ is of the linear form 
$b(x,\mu,\alpha) = b_{0} x + b_{1} \int_{\R^d} v \ud \mu(v) + 
b_{2} \alpha$. The cost functions 
$F$ and $G$
are locally Lipschitz continuous in 
$(x,\mu,\alpha)$, the local Lipschitz constant
being at most of linear growth 
in 
$\vert x \vert$, $(\int_{\R^d} \vert v \vert^2 \ud \mu(v))^{1/2}$
and $\vert \alpha \vert$. 
Moreover, $F$ and $G$ are also
$\cC^1$ in $(x,\mu,\alpha)$, the 
derivative in $(x,\alpha)$ being 
Lipschitz continuous in $(x,\mu,\alpha)$
and the functions $\partial_{\mu} F$ and 
$\partial_{\mu} G$ satisfying
(with $h=F$ and $w=(x,\alpha)$ or 
$h=g$ and $w=x$)
\begin{equation*}
{\mathbb E} \bigl[ \bigl\vert 
 \partial_{\mu} h (w,[\xi])(\xi)
- \partial_{\mu} h (w',[\xi']) (\xi')
\bigr\vert^2 \bigr]^{1/2}
\leq \tilde{L} \bigl\{ 
\vert w- w' \vert 
+ \E \bigl[ \vert \xi - \xi' \vert^2 \bigr]^{1/2}
 \bigr\}.
\end{equation*}
Finally, there exists $\lambda >0$ such that
\begin{equation}
\label{eq:convexity:H=MKV}
\begin{split}
&F(x',\mu',\alpha') - F(x,\mu,\alpha)
 - \langle x'-x, \partial_{x} F(x,\mu,\alpha)
\rangle 
\\
&\hspace{30pt}- \langle \alpha' - \alpha,\partial_{\alpha} F(x,\mu,\alpha) \rangle 
- {\mathbb E} \bigl[ \langle {\xi}' - {\xi},\partial_{\mu} F(x,\mu,\alpha)({\xi})
\rangle \bigr] \geq \lambda \vert \alpha' - \alpha \vert^2,
\end{split}
\end{equation}
for any pair $({\xi},{\xi}')$ with $\mu$ and $\mu'$ as marginal distributions,
where $x,x' \in \R^d$, $\mu,\mu' \in {\mathcal P}_{2}(\R^d)$ and
$\alpha,\alpha' \in \R^k$.

In a similar way, the function $(x,\mu) \mapsto G(x,\mu)$ is convex in 
the joint variable $(x,\mu)$. 
\end{Assumption}

Of course, the Hamiltonian is convex in $\alpha$ under \HYP{6}(i) so that 
the minimizer
\eqref{eq:minimizer} is well-defined. 
By \eqref{eq:optimizer}
and by a suitable version of the implicit function theorem, 
the function $\hat{\alpha}$ inherits the smoothness of 
$\partial_{\alpha} H$. For instance, assume that

\begin{Assumption}[\HYP{6}(ii)]
The function $\R^d \times \cP_{2}(\R^d) \times \R^k \ni (x,\mu,\alpha)
\mapsto \partial_{\alpha} F(x,\mu,\alpha)$ satisfy 
\HYP{2}
(and thus 
\HYP{0} and \HYP{1}
as well)
 (with $w=(x,\alpha)$ in the notations used in \HYP{1} and \HYP{2}).
\end{Assumption}

Then,
\begin{Lemma}
\label{lem:reg:alpha}
 Under \HYP{6}{\rm (i)} and \HYP{6}{\rm (ii)}, 
 the function $\R^d \times \cP_{2}(\R^d) \times \R^k \ni (x,\mu,\alpha)
\mapsto \hat{\alpha}(x,\mu,y)$ satisfies \HYP{0}, \HYP{1}
and \HYP{2} (with $w=(x,y)$ in the notations used in \HYP{1} and \HYP{2}).
\end{Lemma}

\proof The starting point is 
\eqref{eq:optimizer}.
By \eqref{eq:convexity:H=MKV}
and by the Lipschitz property of $\partial_{\alpha} F$, we can reproduce the argument 
used in \cite{carmona:delarue:aop} to prove that $\hat{\alpha}$ is also Lipschitz continuous.  More generally, the smoothness in $x,y$ follows from a standard application of the implicit function theorem.

We now discuss the regularity of $\hat{\alpha}$ in the direction $\mu$. Given $\xi,\chi \in L^2(\Omega,\cA,\P)$, we 
deduce from 
\eqref{eq:optimizer} that, for any $t \in \R$,
\begin{equation*}
y^{\dagger} b_{2} + \partial_{\alpha} F\bigl(x,[ \xi + t \chi],\hat{\alpha}(x,[\xi + t \chi],y) \bigr) = 0.
\end{equation*}
By the standard implicit function theorem, we deduce
that the function $\R \ni t \mapsto \hat{\alpha}(x,[\xi+t \chi],y) \in \R^k$
is differentiable and that 
\begin{equation*}
\begin{split}
&\E \bigl\{ 
\partial_{\mu} \bigl[ \partial_{\alpha} F \bigl( x,[\xi],\hat{\alpha}(x,[\xi],y) \bigr)(\xi) \chi \bigr] \bigr\}
\\
&\hspace{15pt}+ \partial^2_{\alpha \alpha} F \bigl(x, [\xi], \hat{\alpha}(x,[\xi],y) \bigr) 
\frac{\ud}{\ud t}_{\vert t=0}
\bigl[
\hat{\alpha}(x,[\xi+t \chi],y)
\bigr]
 = 0.
 \end{split}
\end{equation*}
By strict convexity, 
the matrix $\partial_{\alpha \alpha}^2 F(x,[\xi],\hat{\alpha}(x,[\xi],y))$ is invertible.
We easily deduce that the mapping 
$L^2(\Omega,{\mathcal A},\P;\R^d) \ni \xi \mapsto \hat{\alpha}(x,[\xi],y) \in \R^k$
is Fr\'echet differentiable.
In particular, the mapping 
$\cP_{2}(\R^d) \ni \mu \mapsto \hat{\alpha}(x,\mu,y)$
is differentiable in Lions' sense and 
\begin{equation*}
\partial_{\mu} \hat{\alpha}(x,\mu,y)(v) = 
\bigl[ 
\partial_{\alpha \alpha}^2 F\bigl(x,\mu,\hat{\alpha}(x,\mu,y) \bigr) \bigr]^{-1}
\partial_{\mu} \bigl[ \partial_{\alpha} F \bigl( x,\mu,\hat{\alpha}(x,\mu,y) \bigr) \bigr](v).
\end{equation*}
The corresponding bounds in \HYP{1} together with the uniform integrability property are easily checked. 
Now, the smoothness in $v$ follows from that 
one of $\partial_{\mu} [\partial_{\alpha} F]$ and 
the related bounds in \HYP{2} hold true. 
The smoothness of $\partial_{\mu} \hat{\alpha}$ in $x,y$ 
is satisfied once we have the smoothness of 
$\hat{\alpha}$ in $x,y$.  
\eproof

\subsubsection{Master equation}
The point is now to apply 
Proposition \ref{prop:iteration weak tentative}
with $b$ as above, $\sigma$ constant
and 
\begin{equation*}
\begin{split}
&f(x,y,\nu) = 
\partial_{x} H\bigl(x,\pi_{1} \sharp \nu,\hat{\alpha}(x,\pi_{1} \sharp \nu,y) \bigr)
+ \tilde{h}(x,\nu), \quad x,y \in \R^{d}, \ \nu \in {\mathcal P}_{2}(\R^d \times \R^d),
\\
&g(x,\mu) = \partial_{x} G(x,\mu) + \tilde{g}(x,\mu), \ x \in \R^d, \ \mu \in {\mathcal P}_{2}(\R^d).
\end{split}
\end{equation*}
Notice also that \HYP{3} is satisfied, see again \cite{carmona:delarue:aop}. 
It thus remains to check that \HYP{2} is satisfied. 

We thus assume that 
\begin{Assumption}[\HYP{6}(iii)]
The functions $\R^d \times \cP_{2}(\R^d) \times \R^k \ni (x,\mu,\alpha)
\mapsto
\partial_{x} F(x,\mu,\alpha)$ and 
$\R^d \times \cP_{2}(\R^d) \ni (x,\mu) \mapsto
\partial_{x} G(x,\mu)$ satisfy \HYP{0}{\rm (i)}, \HYP{1} and \HYP{2}
(with $w=(x,\alpha)$ and $w=x$ respectively).

\textcolor{black}{
For any $(x,\mu,\alpha) \in \R^d \times \cP_{2}(\R^d) \times 
\R^k$,  
there exist versions of 
$\partial_{\mu} F(x,\mu,\alpha)(\cdot)$
and of $\partial_{\mu} G(x,\mu)(\cdot)$
such that $\R^d \times \cP_{2}(\R^d) \times \R^k \times \R^d 
\ni (x,\mu,\alpha,v) \mapsto
\partial_{\mu} F(x,\mu,\alpha)(v)$ and 
$\R^d \times \cP_{2}(\R^d) \times \R^d \ni (x,\mu,v)
\mapsto 
\partial_{\mu} G(x,\mu)(v)$ that satisfy} \HYP{0}{\rm (i)}, \HYP{1} and \HYP{2}
(with $w=(x,\alpha,v)$ and $w=(x,v)$ respectively).
\end{Assumption}

Under \HYP{6}(i-ii-iii), by Lemma \ref{lem:reg:alpha},
the function $\R^d \times \cP_{2}(\R^d) \times \R^d 
\ni (x,\mu,y) \mapsto \partial_{x} H(x,\mu,\hat{\alpha}(x,\mu,y))$
satisfies \HYP{0}(i), \HYP{1} and \HYP{2}.
We now discuss $\tilde{h}$. By linearity of $b$, we first observe that
(recalling that $\partial_{\mu} [\int_{\R^d} h(v') \ud \mu(v')](v)
 = \nabla h(v)$, see \cite{carmona:delarue:aop})
\begin{equation*}
\tilde{h}(x,\nu) = \int_{\R^d \times \R^d} w^\dagger b_{1} \ud 
\nu(v,w) + \int_{\R^d \times \R^d}
\partial_{\mu} F\bigl(v,\pi_{1} \sharp \nu,\hat{\alpha}(v,\pi_{1} \sharp \nu,w)\bigr)(x) \ud \nu(v,w),
\end{equation*}
The smoothness of the first term is easily handled, the smoothness of the second one in $x$ as well. 
The difficulty is to differentiate the second one with respect to $\nu$. 
We get 
\begin{equation*}
\begin{split}
&\partial_{\nu}
\biggl[ \int_{\R^d \times \R^d}
\partial_{\mu} F\bigl(v',\pi_{1} \sharp \nu,\hat{\alpha}(v',\pi_{1} \sharp \nu,w')\bigr)(x) \ud \nu(v',w')
\biggr](v,w)
\\
&= \partial_{(v,w)} \bigl[ 
\partial_{\mu} F\bigl(v,\pi_{1} \sharp \nu,\hat{\alpha}(v,\pi_{1} \sharp \nu,w)\bigr)(x)
\bigr] 
\\
&\hspace{15pt} + \biggl( \int_{\R^d \times \R^d}
\partial_{\mu} 
\bigl[ \partial_{\mu} F\bigl(v',\pi_{1} \sharp \nu,\hat{\alpha}(v',\pi_{1} \sharp \nu,w')\bigr)(x) \bigr](v) \ud \nu(v',w'),0 \biggr)
\\
&\hspace{15pt} + \biggl( \int_{\R^d \times \R^d}
\Bigl[\partial_{\alpha} 
\bigl[ \partial_{\mu} F\bigl(v',\pi_{1} \sharp \nu,\hat{\alpha}(v',\pi_{1} \sharp \nu,w')\bigr)(x) \bigr]
\partial_{\mu} \bigl[\hat{\alpha}(v',\pi_{1} \sharp \nu,w')
\bigr] \Bigr]
(v) \ud \nu(v',w'),0 \biggr),
\end{split}
\end{equation*}
where the `$0$' indicates that the derivative in the direction $w$ is zero. 
We let the reader check the required conditions for the derivative in the direction 
$\nu$ in \HYP{1} and \HYP{2} are indeed satisfied. 
Derivatives in the direction $x$ are easily handled. 

We deduce that Proposition \ref{prop:iteration weak tentative} applies.
As for mean-field games, the master PDE satisfied by $U$ is not the `natural' equation associated with the optimization problem. Following the previous subsection, we thus define
\begin{equation*}
V(t,x,\mu) = {\mathbb E} \biggl[ G\bigl(X_{T}^{t,x,\mu},[X_{T}^{t,\xi}]\bigr)
+ \int_{t}^T F\bigl(X_{s}^{t,x,\mu},[X_{s}^{t,\xi}],\hat{\alpha}(X_{s}^{t,x,\mu},[X_{s}^{t,\xi}],Y_{s}^{t,x,\mu}) \bigr) ds \biggr],
\end{equation*}
where $\xi \sim \mu$, $(X_{s}^{t,\xi})_{s \in [t,T]}$ denotes the forward component in 
\eqref{eq:FBSDE:pontryagin:MKV}, under the initial condition $X_{t}^{t,\xi}=\xi$, 
and $(X_{s}^{t,x,\mu})_{s \in [t,T]}$ denotes the corresponding solution 
of \eqref{eq X-Y t,x,mu}. 

\textcolor{black}{
As in the proof of Theorem \ref{thm:MFG:1},
we are willing to apply the results from
Section 
\ref{se smoothness}
in order to investigate the smoothness of $V$. 
Again,  this 
requires some precaution as the coefficients may be of quadratic growth in the 
space variable and in the measure argument. 
Proceeding as in the proof of 
Theorem \ref{thm:MFG:1}, we have
\footnote{Pay attention that there is no need for
an analogue of \HYP{4}(iv), 
since
\HYP{4}(iv) is necessarily true
under \HYP{6}(i--iii),
with 
$F_{0}(x,\mu)$
in \eqref{eq:quadratic:growth:linear:derivative}
 replaced by 
$F_{0}(x,\mu,\hat{\alpha}(x,\mu,U(t,x,\mu)))$,
the constants appearing in \HYP{0}(i)--\HYP{1}--\HYP{2}
being uniform in $t \in [0,T]$.}
}

\begin{Theorem}
\label{thm:MKV:1}
Under \HYP{6}{\rm (i--iii)}, the function $V$ 
satisfies the statement of Theorem \ref{thm:MFG:1}, 
with the same master equation expect that $U$ inside 
is the decoupling field of \eqref{eq:FBSDE:pontryagin:MKV}.
\end{Theorem}
On the model of Remark \ref{rem:identification}, the identification of $U(t,x,\mu)$ in terms of $V(t,x,\mu)$ now reads
\begin{equation}
\label{eq:identification:MKV}
U(t,x,\mu) = \partial_{x} V(t,x,\mu) + \int_{\R^d} \partial_{\mu} V(t,x',\mu)(x) d\mu(x'),
\end{equation}
which can be proved by differentiating 
the map $L^2(\Omega,\cF_{t},\P;\R^d) \ni \xi \mapsto \E[V(t,\xi,[\xi])]
\in \R$ in the direction $\chi \in L^2(\Omega,\cF_{t},\P;\R^d)$.
By the same kind of expansion as in Remark 
\ref{rem:identification}, we get
\begin{equation*}
\E\Bigl[
\partial_{x} V(t,\xi,[\xi]) \chi + \hat \E \bigl[ \partial_{\mu} V(t,\xi,[\xi])(\cc{\xi}) \cc{\chi} \bigr] \Bigr]
= \E \bigl[ Y_{t}^{t,\xi} \chi \bigr] = \E \bigl[ U(t,\xi,[\xi]) \chi \bigr],
\end{equation*}
where $Y_{t}^{t,\xi}$ and $U(t,\xi,[\xi])$ are seen as row vectors. 
By Fubini's theorem, this identifies 
$U(t,\xi,[\xi])$ with 
$\partial_{x} V(t,\xi,[\xi]) + \hat \E [\partial_{\mu} V(t,\cc{\xi},[\xi])(\xi)]$. 
This proves 
\eqref{eq:identification:MKV} when the law of $\xi$ has $\R^d$ as support. 
In the general case, we can approximate $\xi$ by random variables
with $\R^d$ as support. Passing to the limit in \eqref{eq:identification:MKV}, 
this completes the proof of the identification. 

We refer to 
\cite{carmona:delarue:lyons} for additional comments about the differences between 
the shapes of the master equation in mean-field games and in the control of McKean-Vlasov equations.

\section{Appendix}
\label{se:appendix}

\subsection{Proof of Proposition \ref{prop:lipschitz:lifted}}
\label{subse:appendix:1}
The proof is a straightforward adaptation of Lemma 3.3
in \cite{carmona:delarue:aop}. 
Basically, it suffices to prove the result when 
$\mu$ has a smooth positive density denoted by $p$, and
$p$ and its derivatives being at most of exponential decay at the infinity. 
It is then possible to construct a quantile function 
$U : (0,1)^d \ni (z_{1},\dots,z_{d}) \mapsto U(z_{1},\dots,z_{d}) \in \R^d$ (this {\color{black}is} the notation 
used in \cite{carmona:delarue:aop}, but this 
has nothing to do with the generic notation 
$U$ used in the paper for denoting a function 
of the measure) such that 
$U(\eta_{1},\dots,\eta_{d})$
has law $\mu$ when $\eta_{1},\dots,\eta_{d}$
are i.i.d. random variables with uniform distribution 
on $(0,1)$. Moreover, 
$\partial U_{i}/\partial z_{i} \not =0$
and $\partial U_{j}/\partial z_{i} = 0$ if $i <j$.

Going to (69) therein, 
we see 
from the assumption imposed on $V$
that the bound becomes
\begin{equation*}
\begin{split}
&\int_{\vert r \vert <h}
\bigl\vert V_{n}\bigl[ U \bigl(z^0+r-2r_{d} e_{d} \bigr)
        \bigr] - V_{n} \bigl( U(z^0+r) \bigr) \bigr\vert^2 \ud r
\\
&\leq C_{n}^2 
\int_{\vert r \vert <h}
\biggl[ 1 + \vert 
 U \bigl(z^0+r-2r_{d} e_{d} \bigr) \vert^{2\alpha} 
 +\vert 
 U (z^0+r) \vert^{2\alpha} 
 +
  \biggl( \int_{\R^d} \vert x\vert^2 d\mu(x) \biggr)^{2\alpha}
 \biggr]
 \\
 &\hspace{190pt} \times
\bigl\vert   U \bigl(z^0+r-2r_{d} e_{d} \bigr)
         - ( U(z^0+r)  \bigr\vert^2 \ud r,
\end{split}
\end{equation*}
where $V_{n}$ is a mollification of $V$
that satisfies \eqref{eq:assumption:prop:lipschitz:lifted}
with respect to a constant $C_{n}$
that converges to $C$
as $n$ tends to the infinity. 
Dividing by $h^d$ and following the lines
 of the original argument, we 
get, for a given $z^0 \in \R^d$, 
\begin{equation*} 
\begin{split}
&\Bigl\vert \frac{\partial V_{n}}{\partial x_{d}}
\bigl( U( z^0) \bigr) 
\frac{\partial U_{d}}{\partial z_{d}}
\bigl( z^0 \bigr) \Bigr\vert^2
\leq C_{n}^2 
\biggl[ 1 + 2\vert 
 U \bigl(z^0 \bigr) \vert^{2\alpha} 
 +
  \biggl( \int_{\R^d} \vert x\vert^2 d\mu(x) \biggr)^{2\alpha}
 \biggr]
\Bigl\vert 
\frac{\partial U_{d}}{\partial z_{d}}
\bigl( U( z^0) \bigr)
\Bigr\vert^2. 
\end{split}
\end{equation*}
Dividing by 
$\vert 
[\partial U_{d}/\partial z_{d}]
( U( z^0) )
\vert$ and letting $n$ tend to the infinity, we complete the proof. \qed

\subsection{Differentiability lemma}

\begin{Lemma}
\label{lem:appendix:2}
Consider a function $V : \R^d \times \cP_{2}(\R^d) \times \R^d \rightarrow \R^d$ such 
that, for any $\xi,\chi \in L^2(\Omega,\cA,\P;\R^d)$,
the mapping $\R^d \ni x \mapsto  \E[ \langle V(x,[\xi],\xi), \chi\rangle]$
is differentiable (where $\langle \cdot,\cdot \rangle$ denotes the inner product in $\R^d$). Assume moreover that there exist a constant 
$C \geq 0$ and a function $\Phi_{\alpha}$ as in \HYP{1}
such that, for all $x,x' \in \R^d$ and $\xi,\xi',\chi \in L^2(\Omega,\cA,\P;\R^d)$: 
\begin{equation*}
\begin{split}
&\bigl\vert \frac{\ud}{\ud x} \E\bigl[ \langle V(x,[\xi],\xi), \chi \rangle \bigr] \bigr\vert 
\leq C \| \chi \|_{2},
\\
&\bigl\vert \frac{\ud}{\ud x} \E\bigl[ \langle V(x,[\xi],\xi) ,\chi \rangle \bigr] 
- \frac{\ud}{\ud x} \E\bigl[ \langle V(x',[\xi'],\xi') ,\chi \rangle \bigr]
\bigr\vert 
\leq C \bigl( \vert x - x' \vert + \Phi_{\alpha}(\xi,\xi') \bigr) \| \chi \|_{2}.
\end{split}
\end{equation*}
\textcolor{black}{
Then, for any 
$x \in \R^d$ and any 
$\mu \in \cP_{2}(\R^d)$, 
we can find a continuous 
version of $V(x,\mu,\cdot)$, uniquely defined on 
$\textrm{\rm Supp}(\mu)$, such that 
the mapping
$\R^d \times \textrm{\rm Supp}(\mu) 
\ni (x,v) \mapsto V(x,\mu,v)$ is 
differentiable with respect to 
$x$. Moreover, 
we can find a mapping $\R^d \times \R^d \ni (x,v) 
\mapsto
\partial V(x,\mu,v)$, continuous in $v$ for any given 
$x \in \R^d$, jointly continuous at any point $(x,v)$
with $v \in \textrm{\rm Supp}(\mu)$, 
such that
$\partial V(x,\mu,v)$ identifies with 
$\partial_{x} V(x,\mu,v)$ whenever 
$v \in \textrm{\rm Supp}(\mu)$. In particular, 
$\partial_{x} V(\cdot,\mu,\cdot)$ is continuous 
on $\R^d \times \textrm{\rm Supp}(\mu)$.} 
\end{Lemma}
\proof
By Riesz' theorem, 
for any $i \in \{1,\dots,d\}$,
for any $x \in \R^d$ and any $\xi \in L^2(\Omega,\cA,\P;\R^d)$, 
we can find an element $V^i_{x,\xi} \in L^2(\Omega,\cA,\P;\R^d)$
such that 
\begin{equation*}
\frac{\ud}{\ud x_{i}} \E\bigl[ \langle V(x,[\xi],\xi), \chi \rangle \bigr] = 
\E \bigl[ \langle V^i_{x,\xi} ,\chi \rangle \bigr]. 
\end{equation*}
Now, for $h \not =0$, denoting by $e_{i}$ the $i$th vector of the canonical basis,
\begin{equation*}
\begin{split}
&\E \Bigl[ \Bigl\langle h^{-1} \bigl( 
 V(x + h e_{i},[\xi],\xi)
-  V(x,[\xi],\xi), \chi  
 \bigr) - V^i_{x,\xi}, \chi \Big\rangle \Bigr] 
 \\
 \hspace{15pt} &= \int_{0}^1 
 \Bigl(
 \frac{\ud}{\ud x_{i}}  \E\bigl[ \langle V(x +  sh e_{i},[\xi],\xi), \chi \rangle \bigr]
-
 \frac{\ud}{\ud x_{i}}  \E\bigl[ \langle V(x ,[\xi],\xi), \chi \rangle \bigr]
 \Bigr)
\ud s. 
\end{split}
\end{equation*}
By assumption, we thus get 
that 
$h^{-1} ( 
 V(x + h e_{i},[\xi],\xi)
-  V(x,[\xi],\xi)) - V^i_{x,\xi}$
tends to $0$ in $L^2(\Omega,\cA,\P;\R^d)$. 
Therefore $V^i_{x,\xi}$ is a random variable in 
\new{$L^2(\Omega,\sigma(\xi),\P;\R^d)$}
and we can
express it as  $\partial_{i} V(x,[\xi],\xi)$
where  $\partial_{i} V(x,[\xi],\cdot)$ is a function in  
$L^2(\R^d,[\xi];\R^d)$.  

We have 
\begin{equation*}
\E \bigl[ \vert \partial V(x,[\xi],\xi) - \partial V(x',[\xi'],\xi') \vert^2 \bigr] \leq 
C \bigl( \vert x-x' \vert^2 + \Phi_{\alpha}^2(\xi,\xi') \bigr).
\end{equation*}
Choosing $x=x'$, 
\textcolor{black}{
we deduce from Proposition \ref{prop:lipschitz:lifted} that, 
for any $x \in \R^d$ and any $\xi \in L^2(\Omega,\cA,\P;\R^d)$,
there exists a version of the mapping
 $\R^d \ni v \mapsto 
\partial V(x,[\xi],v)=(\partial_1V(x,[\xi],v),...,\partial_dV(x,[\xi],v)) \in \R^{d \times d}$ 
that is continuous on compact subsets of $\R^d$,
uniformly in $x \in \R^d$, such a version being uniquely defined on 
the support of $[\xi]$.}  
By the same method as in \eqref{eq:trick:markov}, we deduce that
the family $(\R^d \ni v \mapsto \partial V(x,[\xi],v) \in \R^{d \times d})_{x \in \R^d}$
is relatively compact for the topology of uniform convergence on compact subsets. 
Considering a sequence $(x_{n})_{n \geq 1}$ that converges to $x \in \R^d$, 
we already know that 
the sequence of functions $(\R^d \ni v \mapsto \partial V(x_{n},[\xi],v) \in \R^{d \times d})_{n \geq 1}$
converges in $L^2(\R^d,[\xi];\R^{d \times d})$ to 
$\R^d \ni v \mapsto \partial V(x,[\xi],v) \in \R^{d \times d}$. 
\textcolor{black}{Since $\partial V(x,[\xi],\cdot)$ is uniquely defined on 
the support of $[\xi]$,
the limit of any converging subsequence (for the topology of uniform convergence 
on compact subsets of $\R^d$) of
$(\partial V(x_{n},[\xi],\cdot))_{n \geq 1}$ 
coincides with 
$\partial V(x,[\xi],\cdot)$ on the support of $[\xi]$. 
We easily deduce that the function $\R^d \times \R^d \ni 
(x,v) \mapsto \partial V(x,[\xi],v) \in \R^{d \times d}$ is continuous
at any point $(x,v)$ such that $v \in \textrm{Supp}([\xi])$.}

Similarly, we deduce from the identity
\begin{equation*}
\begin{split}
&\E \Big[ \Big\langle \frac{1}{h} \Bigl(  \bigl( V(x+h e_{i},[\xi],\xi) - V(x,[\xi],\xi)
\bigr) - \bigl( V(x'+h e_{i},[\xi'],\xi') - V(x',[\xi'],\xi')
\bigr) \Bigr), \chi \Big\rangle \Bigr]
\\
&= \int_{0}^1 
\E \Bigl\{ \Big\langle \bigl(  \partial_{i} V(x+ sh,[\xi],\xi)  -  \partial_{i} V(x'+sh,[\xi'],\xi') \bigr),
 \chi \Big\rangle \Bigr\} \ud s,
\end{split}
\end{equation*}
that 
$\| h^{-1} 
[( V(x+h e_{i},[\xi],\xi) - V(x,[\xi],\xi)
) - ( V(x'+h e_{i},[\xi'],\xi') - V(x',[\xi'],\xi')
)] \|_{2} \leq C (\vert x - x' \vert + \Phi_{\alpha}(\xi,\xi'))$,
from which we get that, for any 
$x \in \R^d$, any $h \not =0$ and any $\mu \in \cP_{2}(\R^d)$,
there exists a version of the mapping $\R^d \ni v \mapsto h^{-1}[ V(x+he_{i},\mu,v) - V(x,\mu,v)]$ that
is continuous on compact subsets of $\R^d$, uniformly in $x \in \R^d$ and in $h \not =0$. 
As above, we deduce that the family 
$(\R^d \ni v \mapsto h^{-1}[ V(x+he_{i},\mu,v) - V(x,\mu,v)])_{x \in \R^d,h \not =0}$
is relatively compact, for the topology of uniform convergence on compact subsets. 
Once again, following the same argument as above, this says that, for any $x \in \R^d$, 
the functions $(\textrm{Supp}(\mu) \ni v \mapsto 
h^{-1}[ V(x+he_{i},\mu,v) - V(x,\mu,v)] \in \R^d )_{h \not =0}$
\textcolor{black}{
converge
uniformly on compact subsets as 
$h$ tends to $0$ 
 to 
some derivative function, which identifies
with
$\textrm{Supp}(\mu) \ni v \mapsto \partial_{i} V(x,\mu,v) \in \R^d$.}
\eproof
\vspace{10pt}

\noindent
{\bf Acknowledgement.} The authors would like to thank Pierre Cardaliaguet for fruitful discussions. The financial support of a CNRS-Royal Society International exchange  grant is also acknowledged.

\end{document}